\documentclass{amsart}
\pdfoutput=1
\usepackage{amssymb}
\usepackage{fullpage}
\usepackage[outline]{contour}
\usepackage{tikz}
\usepackage{pdflscape}
\usepackage{enumitem}
\usepackage[all]{xy}
\usetikzlibrary{decorations.pathreplacing,decorations.markings,decorations.pathreplacing,arrows,shapes}
\usepackage[bookmarks=false]{hyperref}
\hypersetup{colorlinks, linkcolor=blue}

\newtheorem{thm}{Theorem}[section]
\newtheorem{prop}[thm]{Proposition}
\newtheorem{cor}[thm]{Corollary}
\newtheorem{lem}[thm]{Lemma}

\theoremstyle{definition}
\newtheorem{defn}[thm]{Definition}

\newtheorem{exmp}[thm]{Example}
\newtheorem{notn}[thm]{Notation}

\newtheorem{strg}[thm]{Strategy}
\theoremstyle{remark}
\newtheorem{rmk}[thm]{Remark}

\makeatletter
\let\c@equation\c@thm
\makeatother
\numberwithin{equation}{section}

\newcommand{\conf}{\mathrm{Conf}}

\newcommand{\GL}{\mathrm{GL}}
\newcommand{\PGL}{\mathrm{PGL}}
\newcommand{\SL}{\mathrm{SL}}

\newcommand{\sgn}{\mathrm{sign}}

\newcommand{\DT}{\mathrm{DT}}
\newcommand{\ad}{\mathrm{ad}}
\newcommand{\diag}{\mathrm{diag}}
\newcommand\mapsfrom{\mathrel{\reflectbox{\ensuremath{\mapsto}}}}
\newcommand{\tw}{\mathrm{tw}}
\newcommand{\ord}{\mathrm{ord}}

\renewcommand{\vec}[1]{\mathbf{#1}}
\newcommand{\inprod}[2]{\left\langle#1,#2\right\rangle}

\tikzset{
  -z>/.style={
    decoration={
      show path construction,
      lineto code={
        \path (\tikzinputsegmentfirst) -- (\tikzinputsegmentlast) coordinate[pos=#1] (mid);
        \draw (\tikzinputsegmentfirst) -- (mid);
        \draw[double distance=2pt, -implies] (mid) -- (\tikzinputsegmentlast);      }
    },decorate
  }, -z>/.default=.5,
  z->/.style={
    decoration={
      show path construction,
      lineto code={
          \path (\tikzinputsegmentfirst) -- (\tikzinputsegmentlast) coordinate[pos=#1] (mid);
                \draw[double distance=2pt] (\tikzinputsegmentfirst) -- (mid);
                \draw[decoration={markings, mark=at position 1 with {\arrow[scale=2]{>}}},postaction={decorate}] (mid) -- (\tikzinputsegmentlast);
      }
    },decorate
  }, z->/.default=.5,
  --z>/.style={
    decoration={
      show path construction,
      lineto code={
        \path (\tikzinputsegmentfirst) -- (\tikzinputsegmentlast) coordinate[pos=#1] (mid);
        \draw [dashed](\tikzinputsegmentfirst) -- (mid);
        \draw [double distance=2pt, dashed, -implies] (mid) -- (\tikzinputsegmentlast);      }
    },decorate
  }, --z>/.default=.5,
    z-->/.style={
    decoration={
      show path construction,
      lineto code={
          \path (\tikzinputsegmentfirst) -- (\tikzinputsegmentlast) coordinate[pos=#1] (mid);
                \draw[double distance=2pt, dashed] (\tikzinputsegmentfirst) -- (mid);
                \draw[dashed, decoration={markings, mark=at position 1 with {\arrow[scale=2]{>}}},postaction={decorate}] (mid) -- (\tikzinputsegmentlast);
      }
    },decorate
  }, z-->/.default=.5,
   x->/.style={
    decoration={
      show path construction,
      lineto code={
          \path (\tikzinputsegmentfirst) -- (\tikzinputsegmentlast) coordinate[pos=#1] (mid);
                \draw[double distance=3pt] (\tikzinputsegmentfirst) -- (mid);
                \draw[decoration={markings, mark=at position 1 with {\arrow[scale=2]{>}}},postaction={decorate}] (\tikzinputsegmentfirst) -- (\tikzinputsegmentlast);
      }
    },decorate
  }, x->/.default=.5,
  -x>/.style={
    decoration={
      show path construction,
      lineto code={
        \path (\tikzinputsegmentfirst) -- (\tikzinputsegmentlast) coordinate[pos=#1] (mid);
        \draw[double distance=3pt, shorten >=2pt] (mid) -- (\tikzinputsegmentlast); 
        \draw[decoration={markings, mark=at position 1 with {\arrow[scale=2]{>}}},postaction={decorate}] (\tikzinputsegmentfirst) -- (\tikzinputsegmentlast);}
    },decorate
  }, -x>/.default=.5,
   x-->/.style={
    decoration={
      show path construction,
      lineto code={
          \path (\tikzinputsegmentfirst) -- (\tikzinputsegmentlast) coordinate[pos=#1] (mid);
                \draw[double distance=3pt, dashed] (\tikzinputsegmentfirst) -- (mid);
                \draw[dashed, decoration={markings, mark=at position 1 with {\arrow[scale=2]{>}}},postaction={decorate}] (\tikzinputsegmentfirst) -- (\tikzinputsegmentlast);
      }
    },decorate
  }, x-->/.default=.5,
  --x>/.style={
    decoration={
      show path construction,
      lineto code={
        \path (\tikzinputsegmentfirst) -- (\tikzinputsegmentlast) coordinate[pos=#1] (mid);
        \draw[double distance=3pt, dashed, shorten >=2pt] (mid) -- (\tikzinputsegmentlast); 
        \draw[dashed, decoration={markings, mark=at position 1 with {\arrow[scale=2]{>}}},postaction={decorate}] (\tikzinputsegmentfirst) -- (\tikzinputsegmentlast);}
    },decorate
  }, --x>/.default=.5,
}

\allowdisplaybreaks

\title{Donaldson-Thomas Transformation of Double Bruhat Cells in Semisimple Lie Groups}
\author{Daping Weng}
\date{\today}

\begin{document}

\begin{abstract} Double Bruhat cells $G^{u,v}$ were first studied by Fomin and Zelevinsky \cite{FZ}. They provide important examples of cluster algebras \cite{BFZ} and cluster Poisson varieties \cite{FGamalgamation}. Cluster varieties produce examples of 3d Calabi-Yau categories with stability conditions, and their Donaldson-Thomas invariants, defined by Kontsevich and Soibelman \cite{KS}, are encoded by a formal automorphism on the cluster variety known as the Donaldson-Thomas transformation. Goncharov and Shen conjectured in \cite{GS} that for any semisimple Lie group $G$, the Donaldson-Thomas transformation of the cluster Poisson variety $H\backslash G^{u,v}/H$ is a slight modification of Fomin and Zelevinsky's twist map \cite{FZ}. In this paper we prove this conjecture, using crucially Fock and Goncharov's cluster ensembles \cite{FGensemble} and the amalgamation construction \cite{FGamalgamation}. Our result, combined with the work of Gross, Hacking, Keel, and Kontsevich \cite{GHKK}, proves the duality conjecture of Fock and Goncharov \cite{FGensemble} in the case of $H\backslash G^{u,v}/H$.
\end{abstract}

\maketitle

\tableofcontents

\section{Introduction}

Cluster algebras were defined by Fomin and Zelevinsky in \cite{FZI}. Cluster varieties were introduced by Fock and Goncharov in \cite{FGensemble}. They can be used to construct examples of 3d Calabi-Yau categories with stability conditions. One important object of study for such categories are their Donaldson-Thomas invariants, introduced by Kontsevich and Soibelman \cite{KS}, which generalize geometric invariants of Calabi-Yau manifolds. For a 3d Calabi-Yau category with stability condition constructed from a quiver (seed) together with a generic potential, its Donaldson-Thomas invariants are encoded by a single formal automorphism on the corresponding cluster variety, which is also known as the Donaldson-Thomas transformation \cite{KS}. Keller \cite{KelDT} gave a combinatorial characterization of a certain class of Donaldson-Thomas transformations based on quiver mutation. Goncharov and Shen gave an equivalent definition of the Donaldson-Thomas transformations using tropical points of cluster varieties in \cite{GS}, which we will use in this paper.

Double Bruhat cells form an important family of examples in the study of cluster algebras and cluster Poisson varieties since the very beginning of the subject. On the one hand, Berenstein, Fomin, and Zelevinsky \cite{BFZ} proved that the algebra of regular functions on double Bruhat cells in simply connected semisimple Lie groups are upper cluster algebras. On the other hand, Fock and Goncharov \cite{FGamalgamation} showed that double Bruhat cells in adjoint semisimple Lie groups are cluster Poisson varieties. Furthermore, Fock and Goncharov proved in the same paper that the Poisson structure in the biggest double Bruhat cell, which is a Zariski open subset of the Lie group, coincides with the Poisson-Lie structure defined by Drinfeld in \cite{D}. Williams \cite{Wilkacmoody} showed that these two constructions can be combined into a cluster ensemble in the sense of Fock and Goncharov \cite{FGensemble}, which will play a central role in our construction of Donaldson-Thomas transformation on the double quotient $H\backslash G^{u,v}/H$.

Recall the flag variety $\mathcal{B}$ associated to a semisimple Lie group $G$. Generic configurations of flags were studied by Fock and Goncharov in \cite{FGteich}. The cluster Donaldson-Thomas transformation of such configuration space was constructed by Goncharov and Shen in \cite{GS}. In this paper we make use of the configuration space of quadruple of flags with certain degenerate conditions depending on a pair of Weyl group elements $(u,v)$, which we call $\conf^{u,v}(\mathcal{B})$, and show that such configuration space is isomorphic to the quotient $H\backslash G^{u,v}/H$ of double Bruhat cells. 

Our main result is to show that the Donaldson-Thomas transformation on the cluster Poisson variety associated to $H\backslash G^{u,v}/H$ is equivalent to a modified version of Fomin and Zelevinsky's twist map \cite{FZ}, which is in turn equivalent to some explicit automorphism on the configuration space $\conf^{u,v}(\mathcal{B})$ as well. Related versions of the twist map in the case of Grassmannian have been studied independently by Marsh and Scott \cite{MStw}, by Muller and Speyer \cite{MSpos}, and by the author \cite{Weng}.

\subsection{Main Result} To state our main result, we need to first introduce the related spaces and maps between them. 

Let $G$ be a semisimple Lie group. Fix a pair of opposite Borel subgroups $B_\pm$ of $G$ and let $H:=B_+\cap B_-$ be the corresponding maximal torus. Then with respect to these two Borel subgroups, $G$ admits two Bruhat decompositions 
\[
G=\bigsqcup_{w\in W}B_+wB_+=\bigsqcup_{w\in W}B_-wB_-,
\]
where $W$ denotes the Weyl group with respect to the maximal torus $H$. For a pair of Weyl group elements $(u,v)$, we define the \textit{double Bruhat cell} $G^{u,v}$ to be the intersection
\[
G^{u,v}:=B_+uB_+\cap B_-vB_-.
\]

One of the most important map in the study of double Bruhat cells is Fomin and Zelevinsky's twist map, which they introduced in \cite{FZ}
\begin{align*}
    \tw:G^{u,v}& \rightarrow G^{u^{-1},v^{-1}}\\
    x & \mapsto \left(\left[\overline{u}^{-1}x\right]_-^{-1}\overline{u}^{-1}x\overline{v^{-1}}\left[x\overline{v^{-1}}\right]_+^{-1}\right)^{t\circ \iota}.
\end{align*}
Here $x=[x]_-[x]_0[x]_+$ denotes the Gaussian decomposition, $\overline{w}$ denotes the lift of a Weyl group element $w$ defined by Equation \eqref{wbar} (see also \cite{FZ}, \cite{BFZ}, and \cite{GS}), and $\iota$ is anti-involution on $G$ that maps $G^{u,v}$ biregularly to $G^{u^{-1},v^{-1}}$ (see Equation \eqref{iotadef}).

Double Bruhat cells are closely related to cluster varieties. On the one hand, Berenstein, Fomin, and Zelevinsky \cite{BFZ} showed that when $G$ is a simply connected semisimple Lie group (which we will indicate by a subscript $sc$ in this papre), the coordinate ring $\mathcal{O}\left(\mathcal{G}^{u,v}_{sc}\right)$ has the structure of a cluster algebra, which can be interpreted as a birational equivalence with a cluster $\mathcal{A}$-variety $\mathcal{A}^{u,v}$
\[
\psi:G_{sc}^{u,v}\dashrightarrow \mathcal{A}^{u,v}.
\]
On the other hand, Fock and Goncharov \cite{FGamalgamation} defined amalgamation map for double Bruhat cells in an adjoint semisimple Lie group, which induces a birational equivalence between the cluster $\mathcal{X}$-variety $\mathcal{X}^{u,v}$ corresponding to $\mathcal{A}^{u,v}$ and the double quotient $H\backslash G^{u,v}/H$
\[
\chi:\mathcal{X}^{u,v}\rightarrow H\backslash G/H.
\]

On the cluster Poisson variety $\mathcal{X}^{u,v}$, one question we can ask is whether it possesses a Donaldson-Thomas transformation $\DT$ in the sense of Goncharov and Shen \cite{GS}, and if yes, a follow-up question to ask is whether the Donaldson-Thomas transformation is a cluster transformation. Our main result provides positive answers to both questions. 

The next space we would like to introduce is the configuration space of quadruple of flags. Let $\mathcal{B}$ be the flag variety associated to the semisimple Lie group $G$. By using the Borel subgroup $B_+$ we can define a bijection between the $G$-orbits in $\mathcal{B}\times \mathcal{B}$ and the Weyl group $W$; in particular, it is known that two Borel subgroups $B_1$ and $B_2$ are opposite if and only if the pair $\left(B_1,B_2\right)$ is in the orbit corresponding to the longest Weyl group element $w_0$. For any two Borel subgroups $B_1$ and $B_2$, we write $\xymatrix{B_1 \ar[r]^w & B_2}$ if the pair $\left(B_1,B_2\right)$ is in the orbit corresponding to the Weyl group element $w$, and write $\xymatrix{B \ar@{-}[r] & B'}$ if $B$ and $B'$ are opposite Borel subgroups. Then we define the configuration space $\conf^{u,v}(\mathcal{B})$ to be the configuration space of quadruple of Borel subgroups $(B_1,B_2,B_3,B_4)$ satisfying the relative condition
\[
\xymatrix{B_1 \ar[r]^u \ar@{-}[d] & B_2 \ar@{-}[d] \\ B_3 \ar[r]_{v^*} & B_4}
\]
with $v^*:=w_0vw_0$, modulo the diagonal adjoint action by $G$. As it turns out, $\conf^{u,v}(\mathcal{B})$ is naturally isomorphic to the double quotient of the double Bruhat cell $H\backslash G^{u,v}/H$ (Proposition \ref{flag-bruhat}).

Starting with a quadruple of Borel subgroups $\left(B_1,B_2,B_3,B_4\right)$ whose configuration is in $\conf^{u,v}(\mathcal{B})$, we can find unique Borel subgroups $B_5$, $B_6$, $B_7$, and $B_8$ that fit in the following two relative position diagrams, with $w^l:=w_0w^{-1}$ and $w^r:=w^{-1}w_0$.
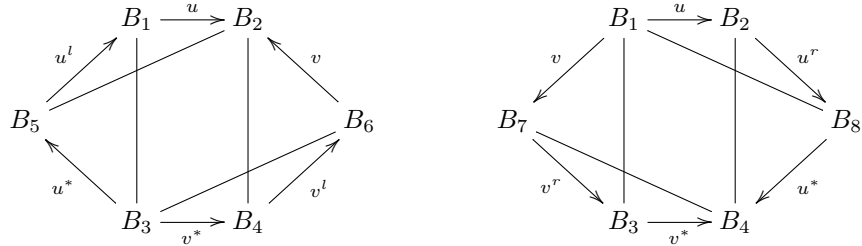
\[
\xymatrix{&B_1 \ar[r]^u \ar@{-}[dd] & B_2 \ar@{-}[dd] & \\
B_5 \ar[ur]^{u^l}\ar@{-}[urr] & & & B_6 \ar[ul]_{v} \\
& B_3 \ar[r]_{v^*} \ar@{-}[urr] \ar[ul]^{u^*} & B_4 \ar[ur]_{v^l} &}
 \quad \quad \quad \quad 
\xymatrix{&B_1 \ar[r]^u \ar[dl]_v \ar@{-}[dd] \ar@{-}[drr] & B_2 \ar@{-}[dd] \ar[dr]^{u^r} & \\
B_7 \ar[dr]_{v^r} \ar@{-}[drr] & & & B_8 \ar[dl]^{u^*} \\
& B_3 \ar[r]_{v^*} & B_4 &}
\]
Using the two hexagon diagrams above we define two biregular automorphisms on $\conf^{u,v}(\mathcal{B})$
\[
\xi:\left[B_1,B_2,B_3,B_4\right] \mapsto \left[B_3^*,B_5^*,B_6^*,B_2^*\right] \quad \quad \text{and} \quad \quad \eta:\left[B_1,B_2,B_3,B_4\right]\mapsto \left[B_8^*,B_4^*,B_1^*,B_7^*\right],
\]
where $*$ is induced by an involution on $G$ (see Equation \ref{starinv}).

Now we are ready to state our main theorems.

\begin{thm}\label{mainthm} (1) The Donaldson-Thomas transformation $\DT$ exists on $\mathcal{X}^{u,v}$ and is a cluster transformation; moreover, the following diagram commutes.
\[
\xymatrix{ \mathcal{X}^{u,v} \ar[r]^\DT \ar[d]_{\chi}^\cong & \mathcal{X}^{u,v} \ar[d]_{\chi}^\cong \\
H\backslash G^{u,v} /H \ar[r]^{\tw\circ \iota} \ar[d]_\cong & H\backslash G^{u,v}/H \ar[d]^\cong \\
\conf^{u,v}(\mathcal{B})\ar[r]_\eta & \conf^{u,v}(\mathcal{B})}
\]

(2) By intertwining the top level with the cluster theoretical map $i_\mathcal{X}$ \eqref{ixdef} and the two lower levels with the anti-involutions $\iota$ \eqref{iotadef}, we get another commutative diagram, with $\Xi:=i_\mathcal{X}\circ \DT\circ i_\mathcal{X}$.
\[
\xymatrix{ \mathcal{X}^{u,v} \ar[r]^\Xi \ar[d]_{\chi}^\cong & \mathcal{X}^{u,v} \ar[d]_{\chi}^\cong \\
H\backslash G^{u,v} /H \ar[r]^{ \iota\circ \tw} \ar[d]_\cong & H\backslash G^{u,v}/H \ar[d]^\cong \\
\conf^{u,v}(\mathcal{B})\ar[r]_\xi & \conf^{u,v}(\mathcal{B})}
\]

(3) The top square in the last commutative diagram fits into a bigger commutative diagram as below, with $s$ being any arbitrary section and $p$ being the canonical map $p:\mathcal{A}^{u,v}\rightarrow \mathcal{X}^{u,v}$ defined for any cluster ensemble.
\[
\xymatrix{ & G_{sc}^{u,v} \ar@{-->}[r]^\psi  \ar@<.5ex>[d] & \mathcal{A}^{u,v}\ar[d]^p & \\
\mathcal{X}^{u,v} \ar[r]^\chi \ar@/_5ex/[rr]_\Xi & H\backslash G^{u,v}/H \ar@<.5ex>[u]^s \ar@/_5ex/[rr]_{\iota \circ \tw} & \mathcal{X}^{u,v} \ar[r]^(0.4){\chi} & H\backslash G^{u,v}/H
}
\]
\end{thm}

Part (3) of our theorem is similar to a result of Williams (\cite{Wilkacmoody}, Theorem 4.9). 

The special case of our main theorem for double Bruhat cells in $\GL_n$ was solved in another paper of the author \cite{WengGL}; however, the bipartite graph method (which was also used in the computation of Donaldson-Thomas transformation of Grassmannian \cite{Weng}) does not apply in the general case of double Bruhat cells in semisimple Lie groups.

The following is an important application of our main result. We proved in our main theorem that the Donaldson-Thomas transformation of the cluster Poisson variety $H\backslash G^{u,v}/H$ is a cluster transformation. Combined with the work of Gross, Hacking, Keel, and Kontsevich (Theorem 0.10 of \cite{GHKK}), our result proves the Fock and Goncharov's conjecture \cite{FGensemble} in the case of $H\backslash G^{u,v}/H$.

\subsection{Structure of the Paper} We divide the rest of the paper into two sections. Section 2 contains all the preliminaries necessary for our proof of the main theorem. Subsection \ref{bruhat} introduces double Bruhat cells $G^{u,v}$ and related structures, most of which are similar to the original work of Fomin and Zelevinsky \cite{FZ} and that of Fock and Goncharov \cite{FGamalgamation}. Subsection \ref{2.2} introduces the configuration spaces of quadruples of Borel subgroups and describes the link between such configuration spaces and double quotients of double Bruhat cells, which is the key to proving the cluster nature of our candidate map for Donaldson-Thomas transformation; such link and the proof of cluster nature are due to Shen. Subsections \ref{cluster}, \ref{1.4}, and \ref{amal} review Fock and Goncharov's theory of cluster ensemble, tropicalization, and amalgamation respectively; the main source of reference of which is \cite{FGensemble} and \cite{FGamalgamation}. Subsection \ref{cells} focuses on the cluster structures related to the main object of our study, namely the double Bruhat cells $G^{u,v}$, the main resources of reference of which are \cite{BFZ} and \cite{FGamalgamation}.

Section 3 is the proof of our main theorem itself, and we will prove the three parts in a backward order. Subsection \ref{3.1} uses cluster ensemble to give a factorization of the composite automorphism $\iota \circ \tw$ on $H\backslash G^{u,v}/H$ and proves part (3) of our main theorem. Subsections \ref{3.2} and \ref{3.3} prove part (2) of our main theorem, and the idea presented in Subsection \ref{3.3} is due to a private conversation with Shen. Subsection \ref{3.4} summarizes the relation among all the maps we have discussed thusfar in the paper, and use the intertwining maps to recover part (1) from part (2) of our main theorem.

In addition to the two main sections of the paper, we have also included three subsections in the appendix. Appendix \ref{A} describes a sequence of mutations can be used to realize the braid relation in the $G_2$ case. Appendix \ref{B} is a table consisting of Lie theoretical identities used in the description of cluster ensemble structure on double Bruhat cells. Lastly, Appendix \ref{C} gives a computational example for the Donaldson-Thomas transformation in the case of $H\backslash \PGL_3^{w_0,w_0}/H$.

\subsection{Acknowledgements} The author is deeply grateful to his advisor Alexander Goncharov for his enlightening guidance and strong encouragement during the process of solving this problem as well as his detailed advice in the revision of the paper. The author also would like to thank both Alexander Goncharov and Linhui Shen for their help on understanding the relation between configuration space of flags and double Bruhat cells and their inspiring idea on proving the cluster nature of our candidate map using the configuration of quadruples of Borel subgroups. This paper uses many ideas from the work of Goncharov \cite{Gon}, the work of Fock and Goncharov \cite{FGensemble}, \cite{FGamalgamation}, \cite{FGteich}, the work of Goncharov and Shen \cite{GS}, and the work of Fomin and Zelevinsky \cite{FZ}; without the work of all these pioneers in the field this paper would not have been possible.

\section{Preliminaries}

\subsection{Structures in Semisimple Lie Groups}\label{bruhat}

This subsection serves as a brief review of some structure theory of semisimple Lie groups with focus on structures that will play important roles in our paper. Please see \cite{Hum} or other standard Lie group textbook for more details.

Let $G$ be a semisimple Lie group and let $\mathfrak{g}$ be its Lie algebra. Fix a pair of opposite Borel subgroups $B_\pm$ in $G$. Then $H:=B_+\cap B_-$ is a maximal torus in $G$. The Weyl group $W$ of $G$ is defined to be the quotient $N_GH/H$. It is known that for any Borel subgroup $B$ of $G$, there is a \textit{Bruhat decomposition} $G=\bigsqcup_{w\in W} BwB$. Taking the two Bruhat decomposition corresponding to the pair of opposite Borel subgroups $B_\pm$ and intersecting them, we obtain our object of interest.

\begin{defn} For a pair of Weyl group elements $(u,v)\in W\times W$, the \textit{double Bruhat cell} $G^{u,v}$ is defined to be
\[
G^{u,v}:=\left(B_+uB_+\right)\cap\left( B_-vB_-\right).
\]
\end{defn}

It is known that the commutator subgroup of a Borel subgroup is a maximal unipotent subgroup. We will denote the commutator subgroup of each of the pair of opposite subgroups by $N_\pm:=[B_\pm,B_\pm]$. 

\begin{defn} The subset of \textit{Gaussian decomposable} elements in $G$ is defined to be the Zariski open subset $G_0:=N_-HN_+$. A \textit{Gaussian decomposition} of an element $x\in G_0$ is the factorization
\[
x=[x]_-[x]_0[x]_+
\]
where $[x]_\pm\in N_\pm$ and $[x]_0\in H$. 
\end{defn}

\begin{prop} The Gaussian decomposition of a Gaussian decomposable element is unique.
\end{prop}
\begin{proof} It follows from the standard facts that $N_-H=B_-$, $HN_+=B_+$, and $N_\pm\cap B_\mp=\{e\}$.
\end{proof}

The adjoint action of $H$ on $\mathfrak{g}$ admits a root space decomposition, and the choice of the pair of opposite Borel subgroups $B_\pm$ defines a subset of \textit{simple roots} $\Pi$ in the root system. The root system spans a lattice called the \textit{root lattice} $Q$. There is a dual notion called \textit{simple coroots}, which can be identified with elements inside the Cartan subalgebra $\mathfrak{h}$ of $\mathfrak{g}$. We will denote the simple coroot dual to $\alpha$ as $H_\alpha$. The \textit{Cartan matrix} of $G$ can then be defined as 
\[
C_{\alpha\beta}:= \inprod{H_\alpha}{\beta}.
\]
The Cartan matrix $C_{\alpha\beta}$ of a semisimple Lie group $G$ is known to have 2 along the diagonal and non-positive integers entries elsewhere. In particular, the Cartan matrix $C_{\alpha\beta}$ matrix is invertible and symmetrizable, i.e., there exists a diagonal matrix $D:=\diag(D_\alpha)$ with integer entries and a symmetric matrix $B_{\alpha\beta}$ with rational entries such that 
\begin{equation}\label{symmetrizable}
C_{\alpha\beta}=D_\alpha B_{\alpha\beta}.
\end{equation}

The Lie algebra $\mathfrak{g}$ can be generated by the Chevalley generators $E_{\pm \alpha}$ and $H_\alpha$; the relations among the Chevalley generators are
\begin{align*}
    [H_\alpha,H_\beta]=& 0,\\
    [H_\alpha,E_{\pm \beta}]=& \pm C_{\alpha\beta}E_{\pm\beta},\\
    [E_{\pm \alpha}, E_{\mp \beta}]=& \pm \delta_{\alpha\beta} H_\alpha,\\
    \left(\ad_{E_{\pm\alpha}}\right)^{1-C_{\alpha\beta}}E_{\pm \beta}=&0 \quad \quad \text{for $\alpha\neq \beta$}.
\end{align*}

The simple coroots $H_\alpha$ are also cocharacters of $H$, and hence they define group homomorphisms $\mathbb{C}^*\rightarrow H$; we will denote such homomorphisms by 
\[
t\mapsto t^{H_\alpha}
\]
for any $a\in \mathbb{C}^*$. Using the exponential map $\exp:\mathfrak{g}\rightarrow G$ we also define a group homomorphisms $e_{\pm \alpha}:\mathbb{C}\rightarrow G$ by
\[
e_{\pm \alpha}(t):=\exp\left(tE_{\pm \alpha}\right).
\]
Particularly when $t=1$ we will omit the argument and simply write $e_{\pm\alpha}$. The arguments $t$ in $e_{\pm \alpha}(t)$ are known as \textit{Lusztig factorization coordinates}, which can be used to define coordinate system on double Bruhat cells as well (see \cite{FZ} for details).

It is known that a semisimple Lie group $G$ is generated by elements of the form $e_{\pm \alpha}(t)$ and the maximal torus $H$. We can then define an anti-involution $t$ on $G$ called \textit{transposition}, which acts on the generators by
\[
\left(e_{\pm \alpha}(t)\right)^t=e_{\mp \alpha}(t) \quad \quad \text{and} \quad \quad a^t=a \quad \forall a\in H.
\]
Note that if $G$ is a matrix group, then the transposition anti-involution we defined above is identical to the transposition of matrices.

A less well-known anti-involution on a semisimple Lie group $G$ is the anti-involution $\iota$, which acts on the generators by
\begin{equation} \label{iotadef}
\left(e_{\pm\alpha}(t)\right)^\iota =e_{\pm\alpha}(t) \quad \quad \text{and} \quad \quad a^\iota=a^{-1} \quad \forall a\in H.
\end{equation}

It is not hard to verify on generators that the three anti-involutions: group inverse, transposition, and $\iota$, commute with each other. In particular, the composition of any two of them gives an involution on $G$. 

The set of simple roots also defines a Coxeter generating set $S=\{s_\alpha\}$ for the Weyl group $W$. The braid relations among these Coxeter generators are the following:
\begin{equation} \label{braid}
\left\{\begin{array}{ll}
    s_\alpha s_\beta s_\alpha = s_\beta s_\alpha s_\beta & \text{if $C_{\alpha\beta}C_{\beta\alpha}=1$;} \\
    s_\alpha s_\beta s_\alpha s_\beta = s_\beta s_\alpha s_\beta s_\alpha & \text{if $C_{\alpha\beta}C_{\beta\alpha}=2$;} \\
    s_\alpha s_\beta s_\alpha s_\beta s_\alpha s_\beta = s_\beta s_\alpha s_\beta s_\alpha s_\beta s_\alpha & \text{if $C_{\alpha\beta}C_{\beta\alpha}=3$.}
\end{array}\right.
\end{equation}

\begin{defn} For an element $w$ in the Weyl group $W$, a \textit{reduced word} of $w$ (with respect to the choice of Coxeter generating set $S$) is a sequence of simple roots $\vec{i}=(\alpha(1),\alpha(2),\dots, \alpha(l))$ which is the shortest among all sequences satisfying $s_{\alpha(1)}s_{\alpha(2)}\dots s_{\alpha(l)}=w$. Entries of a reduced word are called \textit{letters}, and the number $l$ is called the \textit{length} of $w$. 
\end{defn}

One important fact about reduced words is that any two reduced words of the same Weyl group element can be obtained from each other via a finite sequence of braid relations. It is also known that there exists a unique longest element inside the Weyl group $W$ of any semisimple Lie group $G$, and we will denote this longest element by $w_0$. It follows easily from the uniqueness that $w_0^{-1}=w_0$. Conjugation by any lift of $w_0$ swaps the pair of opposite Borel subgroups, i.e., $w_0B_\pm w_0=B_\mp$.

\begin{defn} For a pair of Weyl group elements $(u,v)\in W\times W$, a \textit{reduced word} of $(u,v)$ is a sequence $\vec{i}=(\alpha(1),\alpha(2),\dots, \alpha(l))$ in which every letter is either a simple root or the opposite of a simple root, satisfying the conditions that if we pick out the subsequence consisting of simple roots we get back a reduced word of $v$, and if we pick out the subsequence consisting of the opposite simple roots and then drop the minus sign in front of every letter we get back a reduced word of $u$. The number $l$ is again called the \textit{length} of the pair $(u,v)$.
\end{defn}

Generally speaking, the Weyl group $W$ does not live inside the Lie group $G$. However, one can define a lift of each Coxeter generator $s_\alpha$ by
\begin{equation}\label{wbar}
\overline{s}_\alpha:=e_\alpha^{-1}e_{-\alpha}e_\alpha^{-1}=e_{-\alpha}e_\alpha^{-1}e_{-\alpha}.
\end{equation}
It is not hard to verify that these lifts of the Coxeter generators satisfy the braid relations \ref{braid}. Thus by using a reduced word one can define a lift $\overline{w}$ for any Weyl group element $w$, which is independent of the choice of the reduced word (see also \cite{FZ} and \cite{GS}). 

From Equation \eqref{wbar} we also see that $\overline{s}_\alpha$ are fixed by the anti-involution $\iota$. Then it follows that $\overline{w}^\iota=\overline{w^{-1}}$, and hence $\left(B_\pm wB_\pm\right)^\iota=B_\pm w^{-1}B_\pm$ and $\left(G^{u,v}\right)^\iota=G^{u^{-1},v^{-1}}$. Later we will see in Subsection \ref{cells} that the restriction of the anti-involution $\iota$ to each double Bruhat cell $G^{u,v}$ also has cluster theoretical interpretation.

Conjugation by the longest element $w_0$ defines an involution $w^*:=w_0ww_0$ on the Weyl group $W$. We can lift such involution up to the semisimple Lie group $G$ by defining 
\begin{equation}\label{starinv}
x^*:= \overline{w}_0 x^{-1\circ t}\overline{w}_0^{-1}.
\end{equation}
Since $w_0B_\pm w_0=B_\mp$ and $B_\pm^t=B_\mp$, we know that the pair of opposite Borel subgroups $B_\pm$ is invariant under the involution $*$, i.e., $B_\pm^*=B_\pm$. The reason we call the involution $*$ on $G$ a lift of the involution $*$ on $W$ is because of the following proposition.

\begin{prop} For a Weyl group element $w$, $\overline{w^*}=\overline{w}^*$.
\end{prop}
\begin{proof} We only need to show it for the Coxeter generators. Let $s_\alpha$ be a Coxeter generator. From the definition of the lift $\overline{s}_\alpha$ we see that $\overline{s}_\alpha^t=\overline{s}_\alpha^{-1}$. By a length argument we also know that $s_\alpha^*$ is also a Coxeter generator, say $s_\beta$. Then it follows that $w_0s_\alpha=s_\beta w_0$. Since $l(w_0s_\alpha)=l(s_\beta w_0)=l(w_0)-1$, it follows that
\[
\overline{w}_0\overline{s}_\alpha^{-1}=\overline{w_0s_\alpha}=\overline{s_\beta w_0}=\overline{s}_\beta^{-1}\overline{w}_0.
\]
Therefore we know that
\[
\overline{s}_\alpha^*=\overline{w}_0\overline{s}_\alpha^{-1\circ t}=\overline{w}_0\overline{s}_\alpha\overline{w}_0^{-1}=\overline{s}_\beta. \qedhere
\]
\end{proof}

Each simple root $\alpha$ also defines a group homomorphism $\varphi_\alpha:\SL_2\rightarrow G$, which maps the generators as follows:
\[
\begin{pmatrix} 1 & t \\ 0 & 1\end{pmatrix} \mapsto e_\alpha(t), \quad \quad \begin{pmatrix} 1 & 0 \\ t & 1 \end{pmatrix} \mapsto e_{-\alpha}(t), \quad \quad \begin{pmatrix} t & 0 \\ 0 & t^{-1}\end{pmatrix} \mapsto t^{H_\alpha}.  
\]
Using this group homomorphism, the lift $\overline{s}_\alpha$ can be alternatively defined as 
\[
\overline{s}_\alpha:=\varphi_\alpha\begin{pmatrix} 0 & -1 \\ 1 & 0 \end{pmatrix}.
\]
The following identities are due to Fomin and Zelevinsky \cite{FZ}, which can be easily verified within $\SL_2$ and then mapped over to $G$ under $\varphi_\alpha$:
\begin{equation}\label{e_+e_-}
e_\alpha(p)e_{-\alpha}(q)=e_{-\alpha}\left(\frac{q}{1+pq}\right)(1+pq)^{H_\alpha}e_\alpha\left(\frac{p}{1+pq}\right);
\end{equation}
\begin{equation}\label{e_+s}
e_\alpha(t)\overline{s}_\alpha=e_{-\alpha}\left(t^{-1}\right)t^{H_\alpha} e_\alpha\left(-t^{-1}\right);
\end{equation}
\begin{equation}\label{se_-}
\overline{s}_\alpha^{-1}e_{-\alpha}(t)=e_{-\alpha}\left(-t^{-1}\right)t^{H_\alpha} e_\alpha\left(t^{-1}\right).
\end{equation}

In general there is more than one Lie group associated to the same semisimple Lie algebra $\mathfrak{g}$. Among such a family of Lie groups, there is a simply connected one $G_{sc}$ and a centerless one $G_{ad}$ (which is also known as the \textit{adjoint form} and hence the subscript), and they are unique up to isomorphism. Any Lie group $G$ associated to the Lie algebra $\mathfrak{g}$ is a quotient of $G_{sc}$ by some subgroup of the center $C(G_{sc})$, and $G_{ad}\cong G/C(G)$. For the rest of this subsection we will discuss some structures special to each of these two objects.

Let's start with the simply connected case $G_{sc}$. One important fact from representation theory is that the finite dimensional irreducible representations of a simply connected semisimple Lie group $G_{sc}$ are classified by dominant \textit{weights}, which form a cone inside the \textit{weight lattice} $P$. The weight lattice $P$ can be identified with the lattice of characters of the maximal torus $H$, and contains the root lattice $Q$ as a sublattice. The dual lattice of $P$ is spanned by the simple coroots $H_\alpha$. Hence by dualizing $\{H_\alpha\}$ we obtain a basis $\{\omega_\alpha\}$ of $P$, and we call the weights $\omega_\alpha$ \textit{fundamental weights}. The fundamental weights are dominant, and their convex hull is precisely the cone of dominant weights. Since fundamental weights are elements in the lattice of characters of the maximal torus $H$, they define group homomorphisms $H\rightarrow \mathbb{C}^*$ which we will denote as 
\[
a\mapsto a^{\omega_\alpha}.
\]
Further, it is known that for each fundamental weight $\omega_\alpha$, there is a regular function $\Delta_\alpha$ on $G_{sc}$ uniquely determined by the equation
\[
\Delta_\alpha(x)=[x]_0^{\omega_\alpha}
\]
when restricted to the subset of Gaussian decomposable elements of $G_{sc}$ (see also \cite{Hum} Section 31.4). In particular, if we consider $\mathbb{C}[G_{sc}]$ as a representation of $G_{sc}$ under the natural action $(x.f)(y):=f(x^{-1}y)$, then $\Delta_\alpha$ is a highest weight vector of weight $\omega_\alpha$ (see \cite{FZ} Proposition 2.2). Such regular functions $\Delta_\alpha$ are called \textit{generalized minors}, and they are indeed the principal minors in the case of $\SL_n$.

\begin{prop}\label{minor} $\Delta_\alpha(x)=\Delta_\alpha\left(x^t\right)$. If $\beta\neq \alpha$, then $\Delta_\alpha\left(\overline{s}_\beta^{-1}x\right)=\Delta_\alpha(x)=\Delta_\alpha\left(x\overline{s}_\beta\right)$.
\end{prop}
\begin{proof} The first part follows from the fact that transposition swaps $N_\pm$ while leaving $H$ invariant. To show the second part, recall that $\overline{s}_\beta=e_\beta^{-1}e_{-\beta}e_\beta^{-1}$; therefore
\[
\Delta_\alpha\left(\overline{s}_\beta^{-1}x\right)=\Delta_\alpha\left(e_\beta e_{-\beta}^{-1}e_\beta x\right).
\]
But then since $\Delta_\alpha$ is a highest weight vector of weight $\omega_\alpha$, it is invariant under the action of $e_\beta^{-1}$. Therefore we can conclude that
\[
\Delta_\alpha\left(\overline{s}_\beta^{-1}x\right)=\Delta_\alpha\left(e_\beta e_{-\beta}^{-1}e_\beta x\right)=\Delta_\alpha\left(e_{-\beta}^{-1}e_\beta x\right)=\left(e_{-\beta}.\Delta_\alpha\right)\left(e_\beta x\right).
\]
But then since $\Delta_\alpha$ is a highest weight vector with weight $\omega_\alpha$ in $\mathbb{C}[G]$ as a representation of $G$, the action of $e_{-\beta}$ leaves $\Delta_\alpha$ invariant when $\beta\neq \alpha$. Therefore we have
\[
\Delta_\alpha\left(\overline{s}_\beta^{-1}x\right)=\left(e_{-\beta}.\Delta_\alpha\right)\left(e_\beta x\right)=\Delta_\alpha\left(e_\beta x\right)=\Delta_\alpha(x).
\]
The other equality can be obtained from this one by taking transposition.
\end{proof}

Now let's turn to the adjoint form $G_{ad}$. Since the Cartan matrix $C_{\alpha\beta}$ is invertible, and we can use it to construct another basis $\left\{H^\alpha\right\}$ of the Cartan subalgebra $\mathfrak{h}$, which is defined by
\[
H_\alpha=\sum_\beta C_{\alpha\beta}H^\beta.
\]
Replacing $H_\alpha$ with $H^\alpha$ we can rewrite the relations among the Chevalley generators as
\begin{align*}
    [H^\alpha,H^\beta]=& 0, \\
    [H^\alpha, E_{\pm \beta}]=&\pm \delta_{\alpha\beta}E_{\pm \beta},\\
    [E_{\pm \alpha}, E_{\mp \beta}]=& \pm \delta_{\alpha\beta}C_{\alpha\gamma}H^\gamma,\\
    \left(\ad_{E_{\pm\alpha}}\right)^{1-C_{\alpha\beta}}E_{\pm \beta}=&0 \quad \quad \text{for $\alpha\neq \beta$.}
\end{align*}
It turns out that $H^\alpha$ are cocharacters of the maximal torus $H$ of $G_{ad}$, and hence we can extend our earlier notation $t\mapsto t^{H_\alpha}$ to $t\mapsto t^{H^\alpha}$. In particular the following statement is true.

\begin{prop}\label{commute} If $\beta\neq \alpha$, then $e_{\pm \beta} t^{H^\alpha}=t^{H^\alpha}e_{\pm \beta}$.
\end{prop}
\begin{proof} This follows the fact that $[H^\alpha,E_{\pm \beta}]=0$ whenever $\beta\neq \alpha$.
\end{proof}

The basis $\left\{H^\alpha\right\}$ of $\mathfrak{h}$ was first introduced by Fock and Goncharov \cite{FGamalgamation} to give the cluster Poisson structure on the double Bruhat cell $G_{ad}^{u,v}$, which will play an important role in our paper. 

Lastly, we should point out that $H^\alpha$ are not cocharacters of the maximal torus $H$ of a non-adjoint-form semisimple Lie group $G$. Thus for a non-adjoint-form semisimple Lie group $G$, $a^{H^\alpha}$ is only well-defined up to a center element; however, since we will be considering the double quotient $H\backslash G^{u,v}/H$, this center element ambiguity will be gone because the center $C(G)$ is contained in the maximal torus $H$. 

\subsection{Configuration of Quadruples of Borel Subgroups} \label{2.2}

In this subsection we will look at the double quotients of double Bruhat cells $H\backslash G^{u,v}/H$ from a different angle. Recall that a semisimple Lie group $G$ acts transitively on the space of its Borel subgroups $\mathcal{B}$ by conjugation, with the stable subgroup of the point $B$ being the Borel subgroup $B$ itself. For notation simplicity, for any $x\in G$ and any $B\in \mathcal{B}$ we will denote the Borel subgroup $xBx^{-1}$ as $x.B$. In particular, since lifts of Weyl group elements are unique up to an $H$-multiple and $H\subset B_\pm$, there is no ambiguity in writing $\overline{w}.B_\pm$ as $w.B_\pm$.

It is known that the orbits of the diagonal $G$-action on $\mathcal{B}\times \mathcal{B}$ can be identified with elements of an abstract Weyl group, which can in turn be identified with the Weyl group $W$ we choose using the Borel subgroup $B_+$ (see Chriss and Ginzburg \cite{CG} Chapter 3 for more details). For future convenience we define the following notation.

\begin{notn} For a pair of Borel subgroups $\left(B_1,B_2\right)$ we write $\xymatrix{B_1 \ar[r]^w & B_2}$ if $\left(B_1,B_2\right)\sim \left(B_+,w.B_+\right)$ in $\mathcal{B}\times \mathcal{B}$ (we use $\sim$ to denote the equivalence relation of being in the same $G$-orbit of $\mathcal{B}\times \mathcal{B}$). It follows that if $\xymatrix{B_1\ar[r]^w & B_2}$ then $\xymatrix{B_2 \ar[r]^{w^{-1}} & B_1}$. 
\end{notn} 

\begin{prop}\label{eqposition} For two Borel subgroups $x.B_+$ and $y.B_+$, $\xymatrix{x.B_+\ar[r]^w & y.B_+}$ if and only if $x^{-1}y\in B_+wB_+$.
\end{prop}
\begin{proof} Suppose $x^{-1}y$ belongs to the Bruhat cell $B_+wB_+$ under the Bruhat decomposition. Then we know that $\left(x.B_+,y.B_+\right)\sim \left(B_+,\left(x^{-1}y\right).B_+\right)\sim \left(B_+,w.B_+\right)$.
\end{proof}

Recall that we have an involution $*$ defined on $G$, which can be extended naturally to the space of Borel subgroups $\mathcal{B}$ with $B_\pm^*=B_\pm$. With respect to the involution $*$ we have the following observation.

\begin{prop} $\xymatrix{B_1 \ar[r]^w & B_2}$ if and only if $\xymatrix{B_1^*\ar[r]^{w^*} & B_2^*}$.
\end{prop}
\begin{proof} Suppose $B_1=x.B_+$ and $B_2=y.B_+$. Then $B_1^*=x^*.B_+$ and $B_2=y^*.B_+$. From our last proposition we also know that $x^{-1}y\in B_+wB_+$. Therefore $\left(x^*\right)^{-1}y^*=\left(x^{-1}y\right)^*$ is in the Bruhat cell $B_+w^*B_+$.
\end{proof}

Recall that we also have an anti-involution $\iota$ defined on $G$, and composing with taking inverses we get an group involution $x\mapsto x^{-1\circ \iota}$. We can extend this group involution naturally to the space of Borel subgroups $\mathcal{B}$, and by an abuse of notation we denote the image of a Borel subgroup $B$ under such involution as $B^\iota$ (this notation makes sense because taking inverses does not affect $B$ at all because it is a subgroup). Note that $B_\pm^\iota=B_\pm$ and if $B=x.B_+$ then $B^\iota=x^{-1\circ \iota}.B_+$.

\begin{prop}\label{commutinginvolution} The two involutions $*$ and $\iota$ on $\mathcal{B}$ commute.
\end{prop}
\begin{proof} It suffices to show on the group level that $\left(x^*\right)^{-1\circ \iota}=\left(x^{-1\circ \iota}\right)^*$, which can be done with the following computation:
\[
\left(x^*\right)^{-1\circ \iota}=\left(\overline{w}_0\left(x^{-1\circ t}\right)\overline{w}_0^{-1}\right)^{-1\circ \iota}=\overline{w_0^{-1}} \left(x^{-1\circ \iota}\right)^{-1\circ t}\overline{w_0^{-1}}^{-1}=\left(x^{-1\circ \iota}\right)^*.
\]
Note that the last equality holds because $w_0^{-1}=w_0$.
\end{proof}

\begin{prop}\label{iotadual} If $\xymatrix{B_1\ar[r]^w & B_2}$ then $\xymatrix{B_1^\iota \ar[r]^w & B_2^\iota}$, or equivalently $\xymatrix{B_2^\iota \ar[r]^{w^{-1}}& B_1^\iota}$.
\end{prop}
\begin{proof} From the assumption we know that $\left(B_1, B_2\right)\sim \left(B_+, w.B_+\right)\sim \left(B_+,\overline{w}.B_+\right)$. Recall that when we introduce $\overline{w}$ we said that $\overline{w}^\iota=\overline{w}^{-1}$ (as a direct corollary of Equation \eqref{wbar}); therefore 
\[
\left(B_1^\iota,B_2^\iota\right)\sim \left(B_+,\left(\overline{w}^\iota\right)^{-1}.B_+\right)\sim \left(B_+,w.B_+\right). \qedhere
\]
\end{proof}

\begin{notn} Since $w_0^{-1}=w_0$ and $w_0^*=w_0$, we will also denote $\xymatrix{B_1 \ar[r]^{w_0} &B_2}$ simply as $\xymatrix{B_1 \ar@{-}[r] & B_2}$ without writing $w_0$ explicitly.
\end{notn}

\begin{prop}\label{opposite flag} Let $B_1$ and $B_2$ be two Borel subgroups. Then the followings are equivalent:
\begin{enumerate}
\item $B_1$ and $B_2$ are opposite Borel subgroups;
\item $\xymatrix{B_1 \ar@{-}[r] & B_2}$;
\item there exists an element $z\in G$ such that $B_1=z.B_+$ and $B_2=z.B_-$.
\end{enumerate}
Further the choice of $z$ in (3) is unique up to a right multiple of an element from $H$.
\end{prop}
\begin{proof} (1)$\implies$(2). Suppose the pair $\left(B_1,B_2\right)\sim \left(B_+,x.B_+\right)$. From (1) we know that $B_1\cap B_2$ is a maximal torus, and thus $B_+\cap x.B_+$ is also a maximal torus in $B_+$. But then since any two maximal tori in $B_+$ are conjugate to each other, by conjugating by some $b\in B_+$ we may further impose the condition that $B_+\cap x.B_+=H$. This forces $x.B_+=B_-=w_0.B_+$, which implies (2).

(2)$\implies$(3). Suppose $B_1=x.B_+$ and $B_2=y.B_+$. Then from (2) we know that $x^{-1}y\in B_+w_0B_+$. Thus we can find $b$ and $b'$ from $B_+$ such that $x^{-1}y=b\overline{w}_0b'$. Let $z:=xb$; then 
\[
z.B_+=(xb).B_+=x.B_+=B_1 \quad \quad \text{and} \quad \quad z.B_-=\left(z\overline{w}_0\right).B_+=\left(yb'^{-1}\right).B_+=y.B_+=B_2.
\]

(3)$\implies$(1) is trivial since $z.B_+$ and $z.B_-$ are obviously opposite Borel subgroups.

For the remark on the uniqueness of $z$, note that if $z.B_+=z'.B_+$ and $z.B_-=z'.B_-$, then $z^{-1}z'$ is in both $B_+$ and $B_-$; since $B_+\cap B_-=H$, it follows that $z$ and $z'$ can only differ by a right multiple of an element from $H$.
\end{proof}

The reason we introduce the notation $\xymatrix{B_1\ar[r]^w& B_2}$ is because we want to be able to break the arrow into two in a unique way whenever $w=uv$ with $l(w)=l(u)+l(v)$. We begin by proving the following lemma.

\begin{lem} Suppose $\left(\alpha(1),\alpha(2),\dots, \alpha(l)\right)$ is a reduced word of a Weyl group element $w\in W$ with $l(w)>1$. Then for any $x\in B_+wB_+$ there exists a factorization $x=x_1x_2\dots x_l$ with $x_k\in B_+s_{\alpha(k)}B_+$ for each $k$. Moreover, if $x=x'_1x'_2\dots x'_l$ is another such factorization, then there exists $b_1,b_2,\dots, b_{l-1}\in B_+$ such that
\begin{align*}
    x'_1=&x_1b_1,\\
    x'_k=&b_{k-1}^{-1}x_kb_k \quad \quad \text{for} \quad 2\leq k\leq l-1,\\
    x'_l=&b_{l-1}x_l.
\end{align*}
\end{lem}
\begin{proof} The existence part follows from the fact that $B_+uvB_+=\left(B_+uB_+\right)\left(B_+vB_+\right)$ whenever $l(uv)=l(u)+l(v)$ (see for example \cite{Hum} Section 29.3 Lemma A). For the uniqueness (up to $B_+$-multiples) part we will do an induction on the length of $w$. Suppose $x=x_1x_2\dots x_l=x'_1x'_2\dots x'_l$ are two such factorizations. Then 
\[
x'_lx_l^{-1}\in \left(B_+s_{\alpha(l)}B_+\right)^2\subset B_+\cup \left(B_+s_{\alpha(l)}B_+\right).
\]
To show that $x'_lx_l^{-1}\in B_+$, we just need to rule out the possibility that $x'_lx_l^{-1}\in B_+s_{\alpha(l)}B_+$. Suppose $x'_lx_l^{-1}\in B_+s_{\alpha(l)}B_+$; then by considering the element
\[
x_1x_2\dots x_{l-1}=xx_l^{-1}=x'_1x'_2\dots x'_{l-1}\left(x'_lx_l^{-1}\right),
\]
we see from the left hand side that this element is in $B_+ws_{\alpha(l)}B_+$ and from the right hand side that this element is in $B_+wB_+$. This is absurd because the Bruhat cells are disjoint. Therefore we can conclude that $x'_lx_l^{-1}\in B_+$.

Now find $b_{l-1}\in B_+$ such that $x'_l=b_{l-1}^{-1}x_l$ and consider $y:=xx_l^{-1}=x_1x_2\dots x_{l-1}=x'_1x'_2\dots x'_{l-1}b_{l-1}^{-1}$. By using the same argument, we can deduce that $x'_{l-1}b_{l-1}^{-1}x_{l-1}^{-1}\in B_+$. But this is saying that there exists $b_{l-2}\in B_+$ such that $x'_{l-1}=b_{l-2}^{-1}x_{l-1}b_{l-1}$. Continue in this fashion, we will eventually reach the state that $x'_1b_1^{-1}=x_1$, which is the same as saying that $x'_1=x_1b_1$, and hence the proof is finished.
\end{proof}

\begin{cor} Suppose $u$ and $v$ are two Weyl group elements with $l(uv)=l(u)+l(v)$. Then for any $x\in B_+uvB_+$ there exists a factorization $x=yz$ with $y\in B_+uB_+$ and $z\in B_+vB_+$. Moreover, if $x=y'z'$ is another such factorization, then there exists $b\in B_+$ such that $y'=yb$ and $z'=b^{-1}z$.
\end{cor}
\begin{proof} Again the existence follows the fact that $B_+uvB_+=\left(B_+uB_+\right)\left(B_+vB_+\right)$ whenever $l(uv)=l(u)+l(v)$. For the uniqueness part we can apply the lemma above. Choose any reduced word $\left(\alpha(1),\alpha(2),\dots, \alpha(m)\right)$ for $u$ and any reduced word $\left(\beta(1),\beta(2),\dots, \beta(n)\right)$ for $v$. Then both $y$ and $y'$ can be factorized with respect to the first reduced word and both $z$ and $z'$ can be factorized with respect to the second reduced word. But since $l(uv)=l(u)+l(v)$, $\left(\alpha(1),\alpha(2),\dots, \alpha(m), \beta(1),\beta(2),\dots, \beta(n)\right)$ is a reduced word of $uv$, and the factorizations we get for $y$, $y'$, $z$, and $z'$ can be concatenated to become two factorizations of $x$ with respect to the concatenated reduced word, and it follows from the above lemma that there exists some $b\in B_+$ such that $y'=yb$ and $z'=b^{-1}z$. 
\end{proof}

\begin{cor}\label{2.8} Suppose $u$ and $v$ are two Weyl group elements with $l(uv)=l(u)+l(v)$. Then for any pair of Borel subgroups $\xymatrix{B_1 \ar[r]^{uv} & B_2}$ there exists a unique Borel subgroup $B_3$ such that $\xymatrix{B_1 \ar[r]^u & B_3 \ar[r]^v & B_2}$.
\end{cor}
\begin{proof} Without loss of generality we may assume that $B_1=B_+$ and $B_2=x.B_+$; then it follows that $x\in B_+uvB_+$. By the above corollary we know that $x=yz$ for some $y\in B_+uB_+$ and $z\in B_+vB_+$, with $y$ unique up to a right $B_+$-multiple and $z$ unique up to a left $B_+$-multiple. Therefore $B_3=y.B_+$ is one and the only one Borel subgroup satisfying $\xymatrix{B_1\ar[r]^u & B_3 \ar[r]^v & B_2}$.
\end{proof}

After all the basic facts and notations, we are ready to link a configuration space of points in $\mathcal{B}$ to the double Bruhat cell $G^{u,v}$.

\begin{defn} For a pair of Weyl group elements $(u,v)$ we define the \textit{configuration space of Borel subgroups} $\conf^{u,v}(\mathcal{B})$ to be the quotient space of quadruples of Borel subgroups $(B_1,B_2,B_3,B_4)$ satisfying the following relation
\[
\xymatrix{B_1 \ar[r]^u \ar@{-}[d] & B_2 \ar@{-}[d] \\
B_3 \ar[r]_{v^*} & B_4}
\]
modulo the diagonal action by $G$, i.e., $(x.B_1,x.B_2,x.B_3,x.B_4)\sim (B_1,B_2,B_3,B_4)$ for any $x\in G$. We also call diagrams like the one above \textit{square diagrams}.
\end{defn}

As it turns out, such configuration space $\conf^{u,v}(\mathcal{B})$ is just our old friend $H\backslash G^{u,v}/H$ in disguise.

\begin{prop}\label{flag-bruhat} There is a natural isomorphism $\conf^{u,v}(\mathcal{B})\overset{\cong}{\longrightarrow} H\backslash G^{u,v}/H$.
\end{prop}
\begin{proof} By Proposition \ref{opposite flag}, any element $[B_1,B_2,B_3,B_4]$ in $\conf^{u,v}(\mathcal{B})$ can be represented by the square diagram
\[
\xymatrix{B_+ \ar[r]^u \ar@{-}[d] & x.B_+ \ar@{-}[d] \\
B_- \ar[r]_{v^*} & x.B_-}
\]
for some $x\in G$, and the choice of $x$ is unique up to a left multiple and a right multiple by elements in $H$. Note that by definition of $\conf^{u,v}(\mathcal{B})$, $x\in B_+uB_+\cap B_+v^*B_+=B_+uB_+\cap B_-vB_-$. Thus the map 
\[
[B_1,B_2,B_3,B_4]\mapsto H\backslash x/H
\]
is a well-defined map from $\conf^{u,v}(\mathcal{B})$ to $H\backslash G^{u,v}/H$, and it is not hard to see that this is indeed an isomorphism.
\end{proof}

By identifying $H\backslash G^{u,v}/H$ with $\conf^{u,v}(\mathcal{B})$, we can use properties of the latter space to relate maps of the two spaces. For example, we can define the following map on configuration spaces, and show that it is equivalent to the map $\iota:H\backslash G^{u,v}/H \rightarrow H\backslash G^{u^{-1},v^{-1}}/H$.

\begin{defn} Let $(u,v)$ be a pair of Weyl group elements. We define the map $\iota:\conf^{u,v}(\mathcal{B})\rightarrow \conf^{u^{-1},v^{-1}}(\mathcal{B})$ to be the following.
\[
\left[\vcenter{\vbox{\xymatrix{B_1\ar[r]^u \ar@{-}[d] & B_2 \ar@{-}[d] \\
B_3 \ar[r]_{v^*} & B_4}}}\right]\quad \mapsto \quad 
\left[\vcenter{\vbox{\xymatrix{B_2^\iota\ar[r]^{u^{-1}} \ar@{-}[d] & B_1^\iota \ar@{-}[d] \\
B_4^\iota \ar[r]_{\left(v^{-1}\right)^*} & B_3^\iota}}}\right].
\]
Note that the image lies inside $\conf^{u^{-1},v^{-1}}(\mathcal{B})$ by Proposition \ref{iotadual}. Moreover, it is not hard to see that $\iota$ is of order 2.
\end{defn}

\begin{prop}\label{iotacommute} The following diagram commutes, where the vertical maps are the natural isomorphism stated in Proposition \ref{flag-bruhat}.
\[
\xymatrix{\conf^{u,v}(\mathcal{B}) \ar[r]^\iota \ar[d]_{\cong} & \conf^{u^{-1},v^{-1}}(\mathcal{B}) \ar[d]^{\cong} \\
H\backslash G^{u,v}/H \ar[r]_\iota & H\backslash G^{u^{-1},v^{-1}}/H}
\]
\end{prop}
\begin{proof} Just note that if $\left[\vcenter{\vbox{\xymatrix{B_1\ar[r]^u \ar@{-}[d] & B_2 \ar@{-}[d] \\
B_3 \ar[r]_{v^*} & B_4}}}\right] \quad = \quad \left[\vcenter{\vbox{\xymatrix{B_+\ar[r]^u \ar@{-}[d] & x.B_+ \ar@{-}[d] \\
B_- \ar[r]_{v^*} & x.B_-}}}\right]$, then applying $\iota$ will give 
\[
\left[\vcenter{\vbox{\xymatrix{x^{-1\circ \iota}.B_+\ar[r]^{u^{-1}} \ar@{-}[d] & B_+ \ar@{-}[d] \\
x^{-1\circ \iota}.B_- \ar[r]_{\left(v^*\right)^{-1}} & B_-}}}\right]\quad =\quad \left[\vcenter{\vbox{\xymatrix{B_+\ar[r]^{u^{-1}} \ar@{-}[d] & x^\iota.B_+ \ar@{-}[d] \\
B_- \ar[r]_{\left(v^{-1}\right)^*} & x^\iota.B_-}}}\right],
\]
which is exactly the configuration corresponds to $H\backslash x^\iota/H$.
\end{proof}

We will construct two more automorphisms below, both of which are closely related to Fomin and Zelevinsky's twist map introduced in \cite{FZ}, and we will later show that one of them is a geometric interpretation of a cluster Donaldson-Thomas transformation.

Earlier we have defined a Weyl group element $w^*$ for every Weyl group element $w$; we will define two more here: let $w^l:=w_0w^{-1}$ and let $w^r:=w^{-1}w_0$. It then follows that 
\begin{equation}\label{w0factor}
w_0=w^lw=w^*w^l=ww^r=w^rw^*.
\end{equation}
(The way to remember the superscripts is that in order to multiply to $w_0$, $w^r$ has to be on the right of $w$ and $w^l$ has to be on the left of $w$.)

Now suppose we begin with the square diagram 
\[
\vcenter{\vbox{\xymatrix{B_1 \ar[r]^u \ar@{-}[d] & B_2 \ar@{-}[d] \\
B_3 \ar[r]_{v^*} & B_4}}}.
\]
We can use Corollary \ref{2.8} and Equation \eqref{w0factor} to break the two vertical edges in two ways.
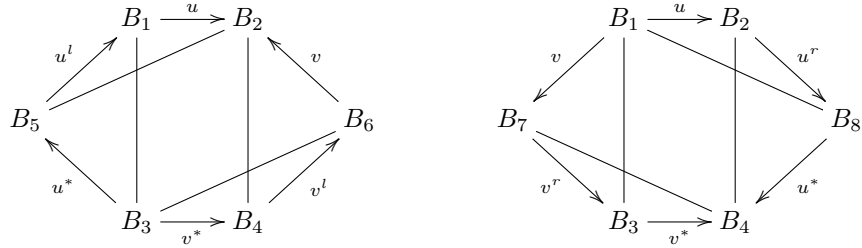
\[
\xymatrix{&B_1 \ar[r]^u \ar@{-}[dd] & B_2 \ar@{-}[dd] & \\
B_5 \ar[ur]^{u^l}\ar@{-}[urr] & & & B_6 \ar[ul]_{v} \\
& B_3 \ar[r]_{v^*} \ar@{-}[urr] \ar[ul]^{u^*} & B_4 \ar[ur]_{v^l} &}
 \quad \quad \quad \quad 
\xymatrix{&B_1 \ar[r]^u \ar[dl]_v \ar@{-}[dd] \ar@{-}[drr] & B_2 \ar@{-}[dd] \ar[dr]^{u^r} & \\
B_7 \ar[dr]_{v^r} \ar@{-}[drr] & & & B_8 \ar[dl]^{u^*} \\
& B_3 \ar[r]_{v^*} & B_4 &}
\]
But then we have two new square diagrams that satisfy the configuration condition of $\conf^{u^*,v^*}(\mathcal{B})$, namely
\[
\xymatrix{B_3 \ar[r]^{u^*} \ar@{-}[d] & B_5 \ar@{-}[d] \\ B_6 \ar[r]_v & B_2} \quad \quad \quad \quad \quad \quad \quad \quad \quad \quad \quad \quad \xymatrix{B_8 \ar[r]^{u^*} \ar@{-}[d] & B_4 \ar@{-}[d] \\ B_1 \ar[r]_v & B_7}
\]
To get back to $\conf^{u,v}(\mathcal{B})$, all we need to do is to apply the involution $*$; the resulting square diagrams are the followings.
\[
\xymatrix{B_3^* \ar[r]^{u} \ar@{-}[d] & B_5^* \ar@{-}[d] \\ B_6^* \ar[r]_{v^*} & B_2^*} \quad \quad \quad \quad \quad \quad \quad \quad \quad \quad \quad \quad \xymatrix{B_8^* \ar[r]^{u} \ar@{-}[d] & B_4^* \ar@{-}[d] \\ B_1^* \ar[r]_{v^*} & B_7^*}
\]
These two square diagrams give rise to two automorphisms on $\conf^{u,v}(\mathcal{B})$, which we will denote as $\eta$ and $\xi$:
\begin{equation}\label{etaxi}
\begin{split}
\conf^{u,v}(\mathcal{B}) \quad  \overset{\xi}{\leftarrow} \quad & \quad \conf^{u,v}(\mathcal{B}) \quad  \quad \overset{\eta}{\rightarrow} \quad \conf^{u,v}(\mathcal{B}) \\
\left[B_3^*,B_5^*,B_6^*,B_2^*\right] \quad \mapsfrom \quad &\left[B_1,B_2, B_3,B_4\right] \quad \mapsto \quad \left[B_8^*,B_4^*,B_1^*,B_7^*\right].
\end{split}
\end{equation}

\begin{prop}\label{xieaaintertwine} The maps $\xi$ and $\eta$ are intertwined by the map $\iota$, i.e., the following diagram commutes
\[
\xymatrix{ \conf^{u,v}(\mathcal{B}) \ar[r]^\xi \ar@{<->}[d]_\iota & \conf^{u,v}(\mathcal{B}) \ar@{<->}[d]^\iota \\ 
\conf^{u^{-1},v^{-1}}(\mathcal{B})\ar[r]_\eta & \conf^{u^{-1},v^{-1}}(\mathcal{B})}
\]
\end{prop}
\begin{proof} Note that if we take the hexagon diagram we used to define $\xi$ (the one below on the left) and apply $\iota$ to it, we get the hexagon diagram on the right.
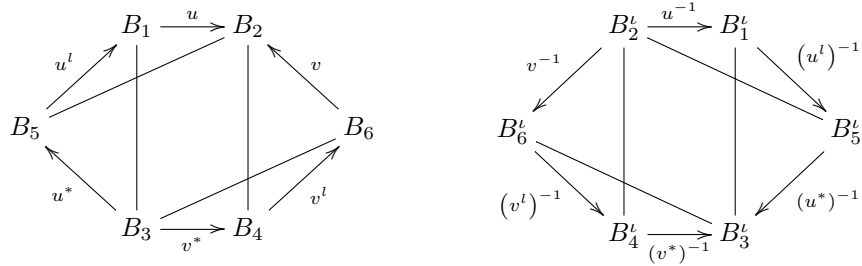
\[
\xymatrix{&B_1 \ar[r]^u \ar@{-}[dd] & B_2 \ar@{-}[dd] & \\
B_5 \ar[ur]^{u^l}\ar@{-}[urr] & & & B_6 \ar[ul]_{v} \\
& B_3 \ar[r]_{v^*} \ar@{-}[urr] \ar[ul]^{u^*} & B_4 \ar[ur]_{v^l} &}
\quad \quad \quad \quad
\xymatrix{&B_2^\iota \ar[r]^{u^{-1}} \ar[dl]_{v^{-1}} \ar@{-}[dd] \ar@{-}[drr] & B_1^\iota \ar@{-}[dd] \ar[dr]^{\left(u^l\right)^{-1}} & \\
B_6^\iota \ar[dr]_{\left(v^l\right)^{-1}} \ar@{-}[drr] & & & B_5^\iota \ar[dl]^{\left(u^*\right)^{-1}} \\
& B_4^\iota \ar[r]_{\left(v^*\right)^{-1}} & B_3^\iota &}
\]
Note that for any Weyl group element,
\[
\left(w^*\right)^{-1}=\left(w_0ww_0\right)^{-1}=w_0w^{-1}w_0=\left(w^{-1}\right)^*,
\]
\[
\left(w^l\right)^{-1}=\left(w_0w^{-1}\right)^{-1}=\left(w^{-1}\right)^{-1}w_0=\left(w^{-1}\right)^r.
\]
Therefore we can relabel the arrows in the hexagon diagram on the above right as
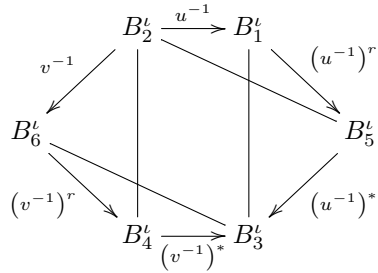
\[
\xymatrix{&B_2^\iota \ar[r]^{u^{-1}} \ar[dl]_{v^{-1}} \ar@{-}[dd] \ar@{-}[drr] & B_1^\iota \ar@{-}[dd] \ar[dr]^{\left(u^{-1}\right)^r} & \\
B_6^\iota \ar[dr]_{\left(v^{-1}\right)^r} \ar@{-}[drr] & & & B_5^\iota \ar[dl]^{\left(u^{-1}\right)^*} \\
& B_4^\iota \ar[r]_{\left(v^{-1}\right)^*} & B_3^\iota &}
\]
from which (together with Proposition \ref{commutinginvolution}) we can conclude that
\begin{align*}
(\eta\circ \iota)\left[B_1,B_2,B_3,B_4\right]=&\eta\left[B_2^\iota,B_1^\iota,B_4^\iota, B_3^\iota\right]\\
=&\left[\left(B_5^\iota\right)^*,\left(B_3^\iota\right)^*,\left(B_2^\iota\right)^*,\left(B_6^\iota\right)^*\right]\\
=&\left[\left(B_5^*\right)^\iota,\left(B_3^*\right)^\iota,\left(B_2^*\right)^\iota,\left(B_6^*\right)^\iota\right]\\
=&\iota\left[B_3,B_5,B_6,B_2\right]\\
=&\iota \circ \xi \left[B_1,B_2,B_3,B_4\right].\qedhere
\end{align*}
\end{proof}

Before we end this subsection, we would like to relate the autormophisms $\xi$ and $\eta$ on $\conf^{u,v}$ to Fomin and Zelevinsky's twist map. In \cite{FZ}, Fomin and Zelevinsky defined a twist map 
\begin{align*} \tw:G^{u,v}&\rightarrow G^{u^{-1},v^{-1}}\\
 x& \mapsto \left( \left[\overline{u}^{-1}x\right]_-^{-1}\overline{u}^{-1}x\overline{v^{-1}}\left[x\overline{v^{-1}}\right]_+^{-1}\right)^{t\circ \iota},
\end{align*}
which is a biregular map of order 2. By an abuse of notation we will denote the induced map $H\backslash G^{u,v}/H\rightarrow H\backslash G^{u^{-1},v^{-1}}/H$ by $\tw$ as well.

\begin{prop}\label{twistflag} The following diagram commutes.
\[
\xymatrix{ \conf^{u,v}(\mathcal{B}) \ar[r]^\xi \ar[d]_\cong & \conf^{u,v}(\mathcal{B}) \ar[d]^\cong \\ H\backslash G^{u,v}/H \ar[r]_{\iota\circ \tw} & H\backslash G^{u,v}/H}
\]
Combining this result with Propositions \ref{iotacommute} and \ref{xieaaintertwine} we can deduce that the following diagram commutes as well.
\[
\xymatrix{ \conf^{u,v}(\mathcal{B}) \ar[r]^\eta \ar[d]_\cong & \conf^{u,v}(\mathcal{B}) \ar[d]^\cong \\ H\backslash G^{u,v}/H \ar[r]_{\tw\circ\iota} & H\backslash G^{u,v}/H}
\]
\end{prop}
\begin{proof} Take a point $H\backslash x/H$ in $H\backslash G^{u,v}/H$. The left vertical map says that it correspond to the configuration
\[
\vcenter{\vbox{\xymatrix{ B_+ \ar@{-}[d] \ar[r]^u & x.B_+\ar@{-}[d] \\
B_- \ar[r]_{v^*} & x.B_-}}}
\]
which only depends on the double coset $H\backslash x/H$ rather than $x$ itself due to Proposition \ref{flag-bruhat}.

Now let's try to break the two vertical edges into two arrows according to the definition of $\xi$. We claim the unique choices for $B_5$ and $B_6$ are the following.
\[
\xymatrix{&B_+ \ar[r]^u \ar@{-}[dd] & x.B_+ \ar@{-}[dd] & \\
\overline{u}.B_- \ar[ur]^{u^l}\ar@{-}[urr] & & & \left(x\overline{v^{-1}}\right).B_+ \ar[ul]_{v} \\
& B_- \ar[r]_{v^*} \ar@{-}[urr] \ar[ul]^{u^*} & x.B_- \ar[ur]_{v^l} &}
\]
To verify one just need to compute the labelling of the new arrows:
\[
\left(\overline{u}.B_-,B_+\right)\sim \left(\left(uw_0\right).B_+,B_+\right)\sim \left(B_+,u^l.B_+\right);
\]
\[
\left(B_-,\overline{u}.B_-\right)\sim \left(w_0.B_+,\left(uw_0\right).B_+\right)\sim \left(B_+,u^*.B_+\right)
\]
\[
\left(\left(x\overline{v^{-1}}\right).B_+,x.B_+\right)\sim \left(B_+,\overline{v^{-1}}^{-1}.B_+\right)\sim \left(B_+,v.B_+\right);
\]
\[
\left(x.B_-,\left(x\overline{v^{-1}}\right).B_+\right)\sim \left(B_+,\left(w_0v^{-1}\right).B_+\right)\sim \left(B_+,v^l.B_+\right).
\]
Next we take out the new square diagram with vertices $B_-$, $\overline{u}.B_-$, $\left(x\overline{v^{-1}}\right).B_+$, and $x.B_+$:
\begin{align*}
\vcenter{\vbox{\xymatrix{B_- \ar[r]^{u^*} \ar@{-}[d] & \overline{u}.B_- \ar@{-}[d] \\
\left(x\overline{v^{-1}}\right).B_+ \ar[r]_(0.6){v} & x.B_+}}}\quad =&\quad \vcenter{\vbox{\xymatrix{\left[x\overline{v^{-1}}\right]_-.B_- \ar[r]^{u^*} \ar@{-}[d] & \overline{u}.B_- \ar@{-}[d] \\
\left[x\overline{v^{-1}}\right]_-.B_+ \ar[r]_v & \left(\overline{u}\overline{u}^{-1}x\right).B_+}}}\\
\quad =& \quad \vcenter{\vbox{\xymatrix{\left[x\overline{v^{-1}}\right]_-.B_- \ar[r]^{u^*} \ar@{-}[d] & \left(\overline{u}\left[\overline{u}^{-1}x\right]_-\right).B_- \ar@{-}[d] \\
\left[x\overline{v^{-1}}\right]_-.B_+ \ar[r]_v & \left(\overline{u}\left[\overline{u}^{-1}x\right]_-\right).B_+}}}.
\end{align*}
Note that the last square diagram represents the same configuration as the square diagram
\[
\vcenter{\vbox{\xymatrix{B_+ \ar[r]^(0.2){u^*} \ar@{-}[d] & \left(\overline{w}_0\left[x\overline{v^{-1}}\right]_-^{-1}\overline{u}\left[\overline{u}^{-1}x\right]_-\overline{w}_0^{-1}\right).B_+ \ar@{-}[d]\\
B_- \ar[r]_(0.2){v} & \left(\overline{w}_0\left[x\overline{v^{-1}}\right]_-^{-1}\overline{u}\left[\overline{u}^{-1}x\right]_-\overline{w}_0^{-1}\right).B_-
}}}.
\]
Lastly, we need to apply the $*$ involution, which yields the square diagram
\[
\vcenter{\vbox{\xymatrix{B_+ \ar[r]^(0.2){u} \ar@{-}[d] & \left(\left[\overline{u}^{-1}x\right]_-^{-1}\overline{u}^{-1}\left[x\overline{v^{-1}}\right]_-\right)^t.B_+ \ar@{-}[d]\\
B_- \ar[r]_(0.2){v^*} & \left(\left[\overline{u}^{-1}x\right]_-^{-1}\overline{u}^{-1}\left[x\overline{v^{-1}}\right]_-\right)^t.B_-
}}}
\]
(we have also used the identity that $\overline{w}_0^t=\overline{w}_0^{-1}$). This shows that the counterpart of the automorphism $\xi$ on $H\backslash G^{u,v}/H$ can be defined as
\[
H\backslash x/H\mapsto H\left\backslash \left(\left[\overline{u}^{-1}x\right]_-^{-1}\overline{u}^{-1}\left[x\overline{v^{-1}}\right]_-\right)^t\right/H.
\]
But we can further split the factor 
\[
\left[x\overline{v^{-1}}\right]_-=x\overline{v^{-1}}\left[x\overline{v^{-1}}\right]_+^{-1}\left[x\overline{v^{-1}}\right]_0^{-1}
\]
and absorb $\left[x\overline{v^{-1}}\right]_0^{-1}$ into the right $H$-quotient. Thus in conclusion, the counterpart of the automorphism $\xi$ can also be written as
\[
H\backslash x/H \mapsto H\left\backslash\left( \left[\overline{u}^{-1}x\right]_-^{-1}\overline{u}^{-1}x\overline{v^{-1}}\left[x\overline{v^{-1}}\right]_+^{-1}\right)^t\right/H,
\]
which is precisely the map $\iota\circ \tw$.
\end{proof}

\subsection{Cluster Ensemble}\label{cluster} We will give a brief review of Fock and Goncharov's theory of cluster ensemble. We will mainly follow the coordinate description presented in \cite{FGensemble}.

\begin{defn} A \textit{seed} $\vec{i}$ is a quadruple $(I,I_0, \epsilon, d)$ satisfying the following properties:
\begin{enumerate}
    \item $I$ is a finite set;
    \item $I_0\subset I$;
    \item $\epsilon=\left(\epsilon_{ab}\right)_{a,b\in I}$ is a $\mathbb{Q}$-coefficient matrix in which $\epsilon_{ab}$ is always an integer unless $a,b\in I_0$;
    \item $d=\left(d_a\right)_{a\in I}$ is a $|I|$-tuple of positive integers such that $\hat{\epsilon}_{ab}:=\epsilon_{ab}d_b^{-1}$ is a skew-symmetric matrix.
\end{enumerate}
\end{defn}

In the special case where $\epsilon_{ab}$ is itself skew-symmetric with integer entries, the data of a seed defined as above is equivalent to the data of a quiver $Q$ with vertex set $I$, with no loops or 2-cycles, and with the exchange matrix 
\[
\epsilon_{ab}=\#\left\{\text{arrows from $a$ to $b$}\right\}-\#\left\{\text{arrows from $b$ to $a$}\right\}.
\]
Thus by extending the terminology from quivers to seeds, we call elements of $I$ \textit{vertices}, call elements of $I_0$ \textit{frozen vertices}, and call $\epsilon$ the \textit{exchange matrix}. 

\begin{defn} Let $\vec{i}=(I,I_0,\epsilon,d)$ be a seed and let $c$ be a non-frozen vertex. Then the \textit{mutation} of $\vec{i}$ at $c$, which we will denote as $\mu_c$, gives rise to new seed $\vec{i}'=\left(I',I'_0, \epsilon', d'\right)$ where $I'=I$, $I'_0=I_0$, $d'=d$, and
\begin{equation}\label{mutation}
\epsilon'_{ab}=\left\{\begin{array}{ll}
-\epsilon_{ab} & \text{if $c\in \{a,b\}$;} \\
\epsilon_{ab}+\left[\epsilon_{ac}\right]_+\left[\epsilon_{cb}\right]_+-\left[-\epsilon_{ac}\right]_+\left[-\epsilon_{cb}\right]_+ & \text{if $c\notin \{a,b\}$;}
\end{array}\right.
\end{equation}
where $[n]_+:=\max\{0,n\}$.
\end{defn}

It is not hard to check that mutating twice at the same vertex gives back the original seed. In fact, if we start with a skew-symmetric matrix $\epsilon$ so the data of a seed can be translated into a quiver, then seed mutation precisely corresponds to quiver mutation defined by Derksen, Weyman, and Zelevinsky \cite{DWZ}.

Starting with an initial seed $\vec{i}_0$, we say that a seed $\vec{i}$ is \textit{mutation equivalent} to $\vec{i}_0$ if there is a sequence of seed mutations that turns $\vec{i}_0$ into $\vec{i}$; we denote the set of all seeds mutation equivalent to $\vec{i}_0$ by $|\vec{i}_0|$. To each seed $\vec{i}$ in $|\vec{i}_0|$ we associate two algebraic tori $\mathcal{A}_\vec{i}=\mathbb{G}_m^{|I|}$ and $\mathcal{X}_\vec{i}=\mathbb{G}_m^{|I|}$, which are equipped with canonical coordinates $(A_a)$ and $(X_a)$ indexed by the set $I$. These two algebraic tori are linked by a map $p_\vec{i}:\mathcal{A}_\vec{i}\rightarrow \mathcal{X}_\vec{i}$ given by
\[
p_\vec{i}^*(X_a)=\prod_{j\in I} A_b^{\epsilon_{ab}}.
\]
The algebraic tori $\mathcal{A}_\vec{i}$ and $\mathcal{X}_\vec{i}$ are called a \textit{seed $\mathcal{A}$-torus} and a \textit{seed $\mathcal{X}$-torus} respectively.

A seed mutation $\mu_c:\vec{i}\rightarrow \vec{i}'$ gives rise to birational equivalences between the corresponding seed tori, which by an abuse of notation we also denote both as $\mu_c$; in terms of the canonical coordinates $(A'_a)$ and $(X'_a)$ they can be expressed as
\begin{align*}
\mu_c^*(A'_a)=&\left\{\begin{array}{ll} \displaystyle A_c^{-1}\left(\prod_{\epsilon_{cb}>0} A_b^{\epsilon_{cb}}+\prod_{\epsilon_{cb}<0} A_b^{-\epsilon_{cb}}\right) & \text{if $a=c$,}\\
A_a & \text{if $a\neq c$,}\end{array}\right. \\
\text{and} \quad \quad  \mu_c^*(X'_a)=&\left\{\begin{array}{l l} X_c^{-1} & \text{if $a=c$,} \\
X_a\left(1+X_c^{-\sgn (\epsilon_{ac})}\right)^{-\epsilon_{ac}}& \text{if $a\neq c$.}\end{array}\right.
\end{align*}
These two birational equivalences are called \textit{cluster $\mathcal{A}$-mutation} and \textit{cluster $\mathcal{X}$-mutation} respectively. One important feature about cluster mutations is that they commute with the respective $p$ maps.
\[
\xymatrix{\mathcal{A}_\vec{i} \ar[d]_{p_\vec{i}} \ar@{-->}[r]^{\mu_k} & \mathcal{A}_{\vec{i}'} \ar[d]^{p_{\vec{i}'}}\\
\mathcal{X}_\vec{i} \ar@{-->}[r]_{\mu_k} & \mathcal{X}_{\vec{i}'}} 
\]

Besides cluster mutations between seed tori we also care about cluster isomorphisms induced by seed isomorphisms. A \textit{seed isomorphism} $\sigma:\vec{i}\rightarrow \vec{i}'$ is a bijection $\sigma:I\rightarrow I'$ that fixes the subset $I_0\subset I\cap I'$ such that $\epsilon'_{\sigma(i)\sigma(j)}=\epsilon_{ij}$. Given a seed isomorphism $\sigma:\vec{i}\rightarrow \vec{i}'$ between two seeds in $|\vec{i}_0|$, we obtain isomorphisms on the corresponding seed tori, which by an abuse of notation we also denote by $\sigma$:
\[
\sigma^*(A'_{\sigma(a)})=A_a \quad \quad \text{and} \quad \quad \sigma^*(X'_{\sigma(a)})=X_a.
\]
We call these isomorphisms \textit{cluster isomorphisms}. It is not hard to see that cluster isomorphisms also commute with the $p$ maps.
\[
\xymatrix{\mathcal{A}_\vec{i} \ar[d]_{p_\vec{i}} \ar[r]^\sigma & \mathcal{A}_{\vec{i}'} \ar[d]^{p_{\vec{i}'}}\\
\mathcal{X}_\vec{i} \ar[r]_\sigma & \mathcal{X}_{\vec{i}'}} 
\]

\begin{defn} By gluing the seed tori via cluster mutations we obtain the corresponding \textit{cluster varieties}, which will be denoted as $\mathcal{A}_{|\vec{i}_0|}$ and $\mathcal{X}_{|\vec{i}_0|}$ respectively. Each seed torus is called a \textit{cluster coordinate chart}, and the functions $\left(A_a\right)$ and $\left(X_a\right)$ defined on each cluster coordinate chart are called \textit{cluster coordinates}. Since the maps $p_\vec{i}$ commute with cluster mutations, they naturally glue into a map $p:\mathcal{A}_{|\vec{i}_0|}\rightarrow \mathcal{X}_{|\vec{i}_0|}$ of cluster varieties. The triple $\left(\mathcal{A}_{|\vec{i}_0|},\mathcal{X}_{|\vec{i}_0|}, p\right)$ associated to a mutation equivalent family of seeds $|\vec{i}_0|$ is called a \textit{cluster ensemble}.
\end{defn}

\begin{defn}\label{clustertran} Cluster isomorphisms can be seen as birational automorphisms on these cluster varieties which map one cluster coordinate chart to another associated to an isomorphic seed, and we call these birational automorphisms \textit{cluster transformations}. 
\end{defn}

\begin{rmk}\label{clustertran'} In practice, one can also fix a cluster coordinate chart and try to find the rational expressions of the pull-backs of the cluster coordinates on the same cluster coordinate chart (which is what we will do when computing the Donaldson-Thomas transformation); the resulting expression can then be factored as a composition consisting of a sequence of cluster mutations followed by a cluster isomorphism and then followed by another sequence of cluster mutations; this is why sometimes cluster transformations are defined as such compositions as well. The picture below also conveys the idea; note that the different styles represent different cluster coordinate charts.
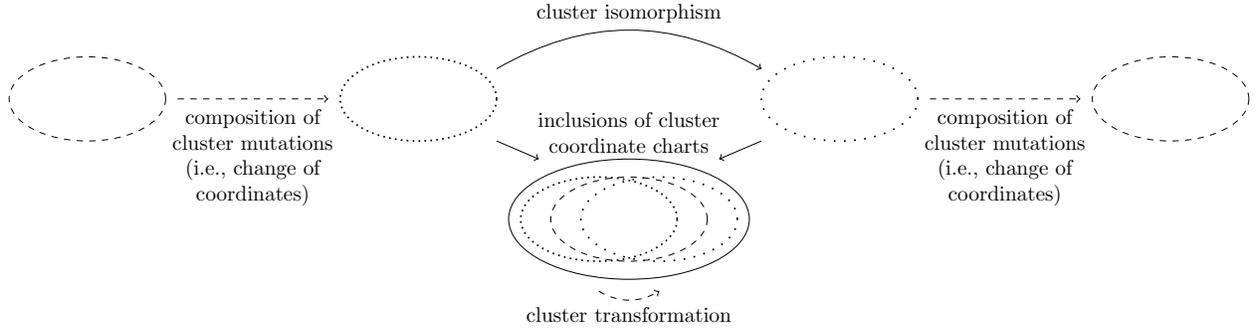
\[
\scalebox{0.8}{$\tikz{
\draw (0,0) ellipse (2 and 1);
\draw [dashed] (0,0) ellipse (1.3 and 0.7);
\draw [loosely dotted, thick] (0.5,0) ellipse (1.3 and 0.7);
\draw [dotted, thick] (-0.5,0) ellipse (1.3 and 0.7);
\draw [dotted, thick] (-3.5,2) ellipse (1.3 and 0.7);
\draw [loosely dotted, thick] (3.5,2) ellipse (1.3 and 0.7);
\draw [->] (-2.2,1.3) -- (-1.5,1);
\draw [->] (2.2,1.3) -- (1.5,1);
\node at (0,1.4) [] {$\begin{array}{c}\text{inclusions of cluster} \\  \text{coordinate charts}\end{array}$};
\draw [->] (-2.2,2.5) to [bend left] node [above] {cluster isomorphism} (2.2,2.5);
\draw [->, dashed] (-0.5,-1.2) to [bend right] node [below] {cluster transformation} (0.5,-1.2);
\draw [dashed] (-9,2) ellipse (1.3 and 0.7);
\draw [->, dashed] (-7.5,2) -- node [below] {$\begin{array}{c} \text{composition of} \\ \text{cluster mutations} \\ \text{(i.e., change of} \\ \text{coordinates)}\end{array}$} (-5,2);
\draw [dashed] (9,2) ellipse (1.3 and 0.7);
\draw [->, dashed] (5,2) -- node [below] {$\begin{array}{c}\text{composition of} \\ \text{cluster mutations} \\ \text{(i.e., change of} \\ \text{coordinates)}\end{array}$} (7.5,2);
}$}
\]
\end{rmk}

Cluster ensembles connect the theory of cluster algebras with the Poisson geometry as well: on the one hand, the coordinate rings on cluster $\mathcal{A}$-varieties are upper cluster algebras \cite{BFZ}, and on the other hand, cluster $\mathcal{X}$-varieties carry natural Poisson variety structures given by
\[
\{X_a,X_b\}=\hat{\epsilon}_{ab}X_aX_b.
\]
Thus a cluster $\mathcal{X}$-variety is also known as a \textit{cluster Poisson variety}. More details are available in \cite{FGensemble}.

Given a seed $\vec{i}=\left(I,I_0,\epsilon,d\right)$ we define its \textit{opposite seed} to be the seed $\vec{i}^\circ=\left(I^\circ, I_0^\circ, \epsilon^\circ, d^\circ\right)$ with
\[
I^\circ:=I,\quad I_0^\circ:=I_0, \quad \epsilon^\circ=-\epsilon^\circ, \quad d^\circ=d.
\]

\begin{prop} Taking the opposite seed commutes with seed mutation, i.e., if $\vec{i}=\left(I,I_0,\epsilon, d\right)$ is a seed and $c$ is an unfrozen vertex, then $\left(\mu_c\vec{i}\right)^\circ=\mu_c\left(\vec{i}^\circ\right)$.
\end{prop}
\begin{proof} The only thing we need to check is that Equation \eqref{mutation} is invariant under taking opposite seeds, which is obvious.
\end{proof}

By going through the same construction we get the opposite cluster ensemble $\left(\mathcal{A}_{\left|\vec{i}_0^\circ\right|}, \mathcal{X}_{\left|\vec{i}_0^\circ\right|},p^\circ\right)$. Since taking the opposite seed commutes with seed mutation, there is a one-to-one correspondence between the cluster coordinate charts on $\mathcal{A}_{\left|\vec{i}_0\right|}$ and those on $\mathcal{A}_{\left|\vec{i}_0^\circ\right|}$, and a one-to-one correspondence between the cluster coordinate charts on $\mathcal{X}_{\left|\vec{i}_0\right|}$ and those on $\mathcal{X}_{\left|\vec{i}_0^\circ\right|}$. Moreover, Fock and Goncharov defined a pair of variety isomorphisms between the corresponding cluster coordinate charts 
\begin{equation}\label{ixdef}
i_\mathcal{A}:\mathcal{A}_\vec{i}\rightarrow \mathcal{A}_{\vec{i}^\circ} \quad \text{and} \quad i_\mathcal{X}:\mathcal{X}_\vec{i}\rightarrow \mathcal{X}_{\vec{i}^\circ}
\end{equation}
by defining the pull-backs
\[
i_\mathcal{A}^*\left(A^\circ_a\right):=A_a \quad \text{and} \quad i_\mathcal{X}^*\left(X^\circ_a\right)=X_a^{-1}.
\]

\begin{prop} $i_\mathcal{A}$ and $i_\mathcal{X}$ commute with cluster mutation maps and cluster isomorphisms, and the following diagram commutes.
\[
\xymatrix{ \mathcal{A}_\vec{i} \ar[r]^{i_\mathcal{A}} \ar[d]_p & \mathcal{A}_{\vec{i}^\circ} \ar[d]^{p^\circ}\\
\mathcal{X}_{\vec{i}} \ar[r]_{i_\mathcal{X}} & \mathcal{X}_{\vec{i}_0^\circ}}
\]
As a corollary, $i_\mathcal{A}$ can be glued into a variety isomorphism $i_\mathcal{A}:\mathcal{A}_{\left|\vec{i}_0\right|}\rightarrow \mathcal{A}_{\left|\vec{i}_0^\circ\right|}$ and $i_\mathcal{X}$ can be glued into a variety isomorphism $i_\mathcal{X}:\mathcal{X}_{\left|\vec{i}_0\right|}\rightarrow \mathcal{X}_{\left|\vec{i}_0^\circ\right|}$, and they form a commutative diagram together with the respective $p$ maps similar to the one above.
\end{prop}
\begin{proof} This follows from the definition that $\epsilon_{ab}^\circ=-\epsilon_{ab}$.
\end{proof}

\begin{rmk}\label{promise} We will see later that the biregular variety morphism $\iota:H\backslash G^{u,v}/H\rightarrow H\backslash G^{u^{-1},v^{-1}}/H$ induced by the anti-involution $\iota$ on $G$ can be interpreted as $i_\mathcal{X}$ on the associated cluster $\mathcal{X}$-varieties. 
\end{rmk}

\begin{notn}\label{reduced} (\textbf{Important!}) There is a \textit{reduced} version of a seed $\mathcal{X}$-torus, which is obtained as the image of $\mathcal{X}_\vec{i}$ under the projection to the coordinates corresponding to non-frozen vertices (and hence is isomorphic to $\mathbb{G}_m^{|I\setminus I_0|}$). If we need to distinguish the two seed $\mathcal{X}$-tori, we will denote the reduced one as $\underline{\mathcal{X}}_\vec{i}$. One can view the reduced seed $\mathcal{X}$-torus $\underline{\mathcal{X}}_\vec{i}$ as a seed $\mathcal{X}$-torus for the seed $\left(I\setminus I_0,\emptyset, \underline{\epsilon}, \underline{d}\right)$ where $\underline{\epsilon}$ and $\underline{d}$ are obtained by deleting the data entries of $\epsilon$ and $d$ that involve $I_0$ respectively. By composing the $p$ map $p_\vec{i}:\mathcal{A}_\vec{i}\rightarrow \mathcal{X}_\vec{i}$ and the projection map $\mathcal{X}_\vec{i}\rightarrow \underline{\mathcal{X}}_\vec{i}$ we obtain another $p$ map, which by an abuse of notation we also denote as $p_\vec{i}:\mathcal{A}_\vec{i}\rightarrow \underline{\mathcal{X}}_\vec{i}$. In fact, the reduced $p$ map makes more sense since it is guaranteed to be algebraic, whereas the unreduced one may have fractional exponents.

In addition, by looking at the formulas for cluster $\mathcal{X}$-mutation we see that the frozen coordinates never get into the formula of unfrozen ones; therefore we can conclude that the projection $\mathcal{X}_\vec{i}\rightarrow \underline{\mathcal{X}}_\vec{i}$ commutes with cluster transformation. In particular, the reduced seed $\mathcal{X}$-tori also glue together into a \textit{reduced cluster $\mathcal{X}$-variety} (\textit{reduced} in the cluster sense, not in the algebraic geometric sense), which we may denote as $\underline{\mathcal{X}}_{|\vec{i}_0|}$, together with a $p$ map $p:\mathcal{A}_{|\vec{i}_0|}\rightarrow \underline{\mathcal{X}}_{|\vec{i}_0|}$. We call the triple $\left(\mathcal{A}_{|\vec{i}_0|}, \underline{\mathcal{X}}_{|\vec{i}_0|}, p\right)$ a \textit{reduced cluster ensemble}. 

The reduced version is actually more useful and more relevant to our story, but the unreduced one also makes our life easier as it allows us to use amalgamation to define the cluster structure of double Bruhat cells later. We therefore decide to include both in this section and use them to define the cluster structures on double Bruhat cells later; however, after we finish the construction of the cluster structures on double Bruhat cells we will drop the underline notation and only use $\mathcal{X}_{|\vec{i}_0|}$ to denote the \textbf{reduced} cluster $\mathcal{X}$-variety.
\end{notn}

\subsection{Tropicalization and Cluster Donaldson-Thomas Transformation}\label{1.4} One important feature of the theory of cluster ensemble is that all the maps present in the construction (cluster transformation and the $p$ map) are positive, which enables us to tropicalize a cluster ensemble. For the rest of this subsection we will make this statement precise, and use the tropical language to define cluster Donaldson-Thomas transformation.

Let's start with what we mean by tropicalization. Consider a split algebraic torus $\mathcal{X}$. The semiring of \textit{positive rational functions} on $\mathcal{X}$, which we denote as $P(\mathcal{X})$, is the semiring consisting of elements in the form $f/g$ where $f$ and $g$ are linear combinations of characters on $\mathcal{X}$ with positive integral coefficients. A rational map $\phi:\mathcal{X}\dashrightarrow \mathcal{Y}$ between two split algebraic tori is said to be \textit{positive} if it induces a semiring homomorphism $\phi^*:P(\mathcal{Y})\rightarrow P(\mathcal{X})$. It then follows that composition of positive rational maps is still a positive rational map.

One typical example of a positive rational map is a cocharacter $\chi$ of a split algebraic torus $\mathcal{X}$: the induced map $\chi^*$ pulls back an element $f/g\in P(\mathcal{X})$ to $\frac{\langle f, \chi\rangle}{\langle g,\chi\rangle}$ in $P(\mathbb{G}_m)$, where $\langle f, \chi\rangle$ and $\langle g,\chi\rangle$ are understood as linear extensions of the canonical pairing between characters and cocharacters with values in powers of $z$. We will denote the lattice of cocharacters of a split algebraic torus $\mathcal{X}$ by $\mathcal{X}^t$ for reasons that will become clear in a moment.

Note that $P(\mathbb{G}_m)$ is the semiring of rational functions in a single variable $z$ with positive integral coefficients. Thus if we let $\mathbb{Z}^t$ be the semiring $(\mathbb{Z}, \min, +)$, then we can define a semiring homomorphism call the \textit{order} of $z$
\[
\ord_z:P(\mathbb{G}_m)\rightarrow \mathbb{Z}^t,
\]
which is defined on any polynomial $f$ as the lowest power of $z$ present in $f$, and defined on any rational function $f/g$ as $\ord_z(f/g):=\ord_zf-\ord_zg$. Therefore a cocharacter $\chi$ on $\mathcal{X}$ gives rise to a natural semiring homomorphism 
\[
\ord_z \langle \cdot, \chi\rangle:P(\mathcal{X})\rightarrow \mathbb{Z}^t
\]

\begin{prop} The map $\chi\mapsto \ord_z\langle \cdot, \chi\rangle$ is a bijection between the lattice of cocharacters and set of semiring homomorphisms from $P(\mathcal{X})$ to $\mathbb{Z}^t$.
\end{prop}
\begin{proof} Note that $P(\mathcal{X})$ is a free commutative semiring generated by any basis of the lattice of characters, and in particular any choice of coordinates $(X_i)_{i=1}^r$ on the split algebraic torus. Therefore to define a semiring homomorphism from $P(\mathcal{X})$ to $\mathbb{Z}^t$ we just need to assign to each $X_i$ some integer $a_i$. But for any such $r$-tuple $(a_i)$ there exists a unique cocharacter $\chi$ such that $\langle X_i,\chi\rangle=z^{a_i}$. Therefore $\chi\mapsto \ord_z\langle \cdot, \chi\rangle$ is indeed a bijection.
\end{proof}

\begin{cor} A positive rational map $\phi:\mathcal{X}\dashrightarrow \mathcal{Y}$ between split algebraic tori gives rise to a natural map $\phi^t:\mathcal{X}^t\rightarrow \mathcal{Y}^t$ between the respective lattice of cocharacters.
\end{cor}
\begin{proof} Note that $\phi$ induces a semiring homomorphism $\phi^*:P(\mathcal{Y})\rightarrow P(\mathcal{X})$. Therefore for any cochcaracter $\chi$ of $\mathcal{X}$, the map $f\mapsto \ord_z\langle \phi^*f,\chi\rangle$ is a semiring homomorphism from $P(\mathcal{Y})\rightarrow \mathbb{Z}^t$. By the above proposition there is a unique cocharacter $\eta$ of $\mathcal{Y}$ representing this semiring homomorphism, and we assign $\phi^t(\chi)=\eta$.
\end{proof}

We also want to give an explicit way to compute the induced map $\phi^t$. Fix two coordinate charts $(X_i)$ on $\mathcal{X}$ and $(Y_j)$ on $\mathcal{Y}$. Then $(X_i)$ gives rise to a dual basis $\{\chi_i\}$ of the lattice of cocharacters $\mathcal{X}^t$, which is defined by
\[
\chi_i^*(X_k):=\left\{\begin{array}{ll}
z & \text{if $k=i$;} \\
1 & \text{if $k\neq i$.}
\end{array}\right.
\]
This basis allows us to write each cocharacter $\chi$ of $\mathcal{X}$ as a linear combination $\sum x_i\chi_i$. It is not hard to see that
\[
x_i=\ord_z\langle X_i, \chi\rangle.
\]
Similarly the coordinate chart $(Y_j)$ also gives rise to a basis $\{\eta_j\}$ of the lattice of cocharacters $\mathcal{Y}^t$, and we can write each cocharacter of $\mathcal{Y}$ as a linear combination $\sum y_j\eta_j$. 
On the other hand, for any positive rational expression $q$ in $r$ variables $X_1, \dots, X_r$ we have the so-called \textit{na\"{i}ve tropicalization}, which turns $q$ into a map from $\mathbb{Z}^r$ to $\mathbb{Z}$ via the following process:
\begin{enumerate}
\item replace addition in $q(X_1,\dots, X_r)$ by taking minimum;
\item replace multiplication in $q(X_1,\dots, X_r)$ by addition;
\item replace division in $q(X_1,\dots, X_r)$ by subtraction;
\item replace every constant by zero;
\item replace $X_i$ by $x_i$.
\end{enumerate}
It is not hard to see that, given a positive rational map $\phi:\mathcal{X}\dashrightarrow\mathcal{Y}$, the induced map $\phi^t$ maps $x=\sum x_i\chi_i$ to $\sum y_j\eta_j$ where
\begin{equation}\label{cocharacter}
y_j=(\phi^*(Y_j))^t(x).
\end{equation}
In particular, this shows that $\phi^t:\mathcal{X}^t\rightarrow \mathcal{Y}^t$ is a piecewise linear map.

Now we are ready to define tropicalization. 

\begin{defn} The \textit{tropicalization} of a split algebraic torus $\mathcal{X}$ is defined to be its lattice of cocharacters $\mathcal{X}^t$ (and hence the notation). For a positive rational map $\phi:\mathcal{X}\dashrightarrow \mathcal{Y}$ between split algebraic tori, the \textit{tropicalization} of $\phi$ is defined to be the map $\phi^t:\mathcal{X}^t\rightarrow \mathcal{Y}^t$. The basis $\{\chi_i\}$ of $\mathcal{X}^t$ corresponding to a coordinate system $(X_i)$ on $\mathcal{X}$ is called the \textit{basic laminations} associated to $(X_i)$.
\end{defn}

\begin{exmp} Let's give an example of tropicalization of a positive rational map between two split algebraic tori. Let $\mathcal{X}=\mathbb{G}_m^3$ and let $\mathcal{Y}=\mathbb{G}_m^2$; define a rational map $\phi:\mathcal{X}\dashrightarrow \mathcal{Y}$ via the pull-back of coordinate functions:
\[
\phi^*\left(Y_1\right)=\frac{X_2^{-2}+2X_1X_2+3}{\left(X_1+X_2\right)^2} \quad \quad \text{and} \quad \quad \phi^*\left(Y_2\right)=\frac{2X_1^{-1}+X_2}{4X_2+\left(X_1+2\right)^3+1}.
\]
The tropicalization of $\mathcal{X}$ and $\mathcal{Y}$ are then $\mathcal{X}^t=\mathbb{Z}^3$ and $\mathcal{Y}=\mathbb{Z}^2$, and the tropicalization $\phi^t:\mathcal{X}^t\rightarrow \mathcal{Y}^t$ can be expressed via the pull-back of coordinate functions:
\[
\left(\phi^t\right)^*(y_1)=\min \{-2x_2,x_1+x_2,0\}-2\min\{x_1,x_2\},
\]
\[
\left(\phi^t\right)^*(y_2)=\min\{-x_1,x_2\}-\min\{x_2,3\min\{x_1,0\},0\}.
\]
\end{exmp}

Now let's go back to the cluster varieties $\mathcal{A}_{\left|\vec{i}_0\right|}$ and $\mathcal{X}_{\left|\vec{i}_0\right|}$. Since both cluster varieties are obtained by gluing seed tori via positive birational equivalences, we can tropicalize everything and obtain two new glued objects which we call \textit{tropicalized cluster varieties} and denote as $\mathcal{A}_{\left|\vec{i}_0\right|}^t$ and $\mathcal{X}_{\left|\vec{i}_0\right|}^t$. Note that the tropicalization of the mutation maps are piecewise linear and bijective; therefore the resulting tropical cluster varieties $\mathcal{A}_{\left|\vec{i}_0\right|}^t$ and $\mathcal{X}_{\left|\vec{i}_0\right|}^t$ are themselves lattices. 

Since each seed $\mathcal{X}$-torus $\mathcal{X}_\vec{i}$ is a split algebraic torus, it has a set of basic laminations associated to the canonical coordinates $(X_a)$; we will call this set of basic laminations the \textit{positive basic $\mathcal{X}$-laminations} and denote them as $l_a^+$. Note that $\{-l_a^+\}$ is also a set of basic laminations on $\mathcal{X}_\vec{i}$, which will be called the \textit{negative basic $\mathcal{X}$-laminations} and denote them as $l_a^-$.

With all the terminologies developed, we can now state the definition of Goncharov and Shen's cluster Donaldson-Thomas transformation as follows.

\begin{defn}[Definition 2.15 in \cite{GS}] A \textit{cluster Donaldson-Thomas transformation} (of a seed $\mathcal{X}$-torus $\mathcal{X}_\vec{i}$) is a cluster transformation $\mathrm{DT}:\mathcal{X}_\vec{i}\dashrightarrow \mathcal{X}_\vec{i}$ whose tropicalization $\mathrm{DT}^t:\mathcal{X}_\vec{i}^t\rightarrow \mathcal{X}_\vec{i}^t$ maps each positive basic $\mathcal{X}$-laminations $l_a^+$ to its corresponding negative basic $\mathcal{X}$-laminations $l_a^-$. \end{defn}

Goncharov and Shen proved that a cluster Donaldson-Thomas transformation enjoys the following properties.

\begin{thm} [Goncharov-Shen, Theorem 2.16 in \cite{GS}] \label{gs} A cluster Donaldson-Thomas transformation $\mathrm{DT}:\mathcal{X}_\vec{i}\rightarrow \mathcal{X}_\vec{i}$ is unique if it exists. If $\vec{i}'$ is another seed in $|\vec{i}|$ (the collection of seeds mutation equivalent to $\vec{i}$) and $\tau:\mathcal{X}_\vec{i}\rightarrow \mathcal{X}_{\vec{i}'}$ is a composition of cluster mutations, then the conjugation $\tau \circ \mathrm{DT}\circ \tau^{-1}$ is the cluster Donaldson-Thomas transformation of $\mathcal{X}_{\vec{i}'}$. Therefore it makes sense to view the cluster Donaldson-Thomas transformation $\mathrm{DT}$ as a birational automorphism on a cluster $\mathcal{X}$-variety without referring to any one specific seed $\mathcal{X}$-torus.
\end{thm}

From our discussion on tropicalization above, we can translate the definition of a cluster Donaldson-Thomas transformation into the following equivalent one, which we will use to prove our main theorem.

\begin{prop}\label{lem0} A cluster transformation $\mathrm{DT}:\mathcal{X}_{\left|\vec{i}_0\right|}\dashrightarrow \mathcal{X}_{\left|\vec{i}_0\right|}$ is a cluster Donaldson-Thomas transformation if and only if on one (any hence any) cluster coordinate chart $\left(X_a\right)$, we have
\[
\ord_{X_a}\mathrm{DT}^*(X_b)=-\delta_{ab}
\]
where $\delta_{ab}$ denotes the Kronecker delta.
\end{prop}
\begin{proof} From Equation \eqref{cocharacter} we see that $\mathrm{DT}^t(l_a^+)= l_a^-$ for all $a$ if and only if $\ord_{X_a}\mathrm{DT}^*(X_b)=-\delta_{ab}$.
\end{proof}

\begin{defn}\label{dtxi} Recall that we have a map $i_\mathcal{X}:\mathcal{X}_{\left|\vec{i}_0^\circ\right|}\rightarrow \mathcal{X}_{\left|\vec{i}_0^\circ\right|}$. Suppose we have a cluster Donaldson-Thomas transformation $\DT$ defined on $\mathcal{X}_{\left|\vec{i}_0^\circ\right|}$, then by using the map $i_\mathcal{X}$, we can induce a cluster transformation on $\mathcal{X}_{\left|\vec{i}_0^\circ\right|}$ that is
\[
\Xi:=i_\mathcal{X}\circ \DT\circ i_\mathcal{X}.
\]
\end{defn}

The following picture is a pictorial description of the relation between $\DT$ and $\Xi$.
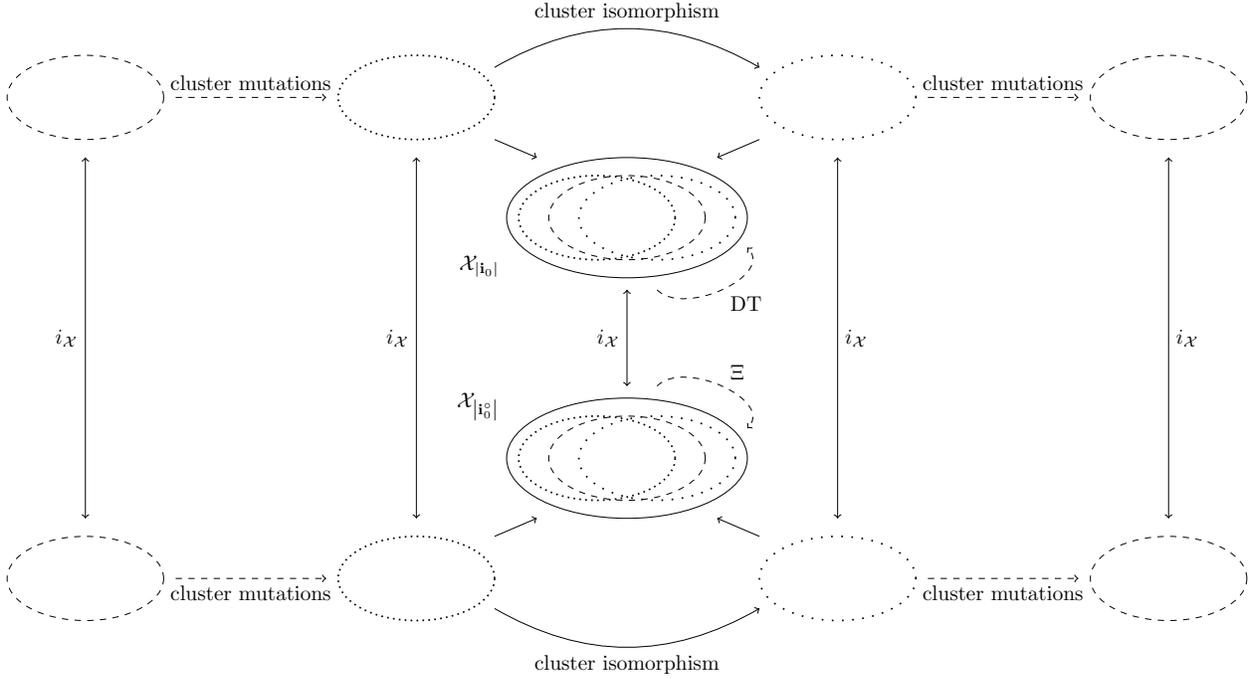
\[\scalebox{0.8}{$
\tikz{
\draw (0,0) ellipse (2 and 1);
\draw [dashed] (0,0) ellipse (1.3 and 0.7);
\node at (-2,-0.5) [below left] {$\mathcal{X}_{\left|\vec{i}_0\right|}$};
\draw [loosely dotted, thick] (0.5,0) ellipse (1.3 and 0.7);
\draw [dotted, thick] (-0.5,0) ellipse (1.3 and 0.7);
\draw [dotted, thick] (-3.5,2) ellipse (1.3 and 0.7);
\draw [loosely dotted, thick] (3.5,2) ellipse (1.3 and 0.7);
\draw [->] (-2.2,1.3) -- (-1.5,1);
\draw [->] (2.2,1.3) -- (1.5,1);
\draw [->] (-2.2,2.5) to [bend left] node [above] {cluster isomorphism} (2.2,2.5);
\draw [->,dashed] (0.5,-1.2) to [out=-45, in=-45] node [below right] {$\DT$} (2,-0.5);
\draw [dashed] (-9,2) ellipse (1.3 and 0.7);
\draw [->, dashed] (-7.5,2) -- node [above] {cluster mutations} (-5,2);
\draw [dashed] (9,2) ellipse (1.3 and 0.7);
\draw [->, dashed] (5,2) -- node [above] {cluster mutations} (7.5,2);
\draw (0,-4) ellipse (2 and 1);
\draw [dashed] (0,-4) ellipse (1.3 and 0.7);
\node at (-2,-3.5) [above left] {$\mathcal{X}_{\left|\vec{i}_0^\circ\right|}$};
\draw [loosely dotted, thick] (0.5,-4) ellipse (1.3 and 0.7);
\draw [dotted, thick] (-0.5,-4) ellipse (1.3 and 0.7);
\draw [dotted, thick] (-3.5,-6) ellipse (1.3 and 0.7);
\draw [loosely dotted, thick] (3.5,-6) ellipse (1.3 and 0.7);
\draw [->] (-2.2,-5.3) -- (-1.5,-5);
\draw [->] (2.2,-5.3) -- (1.5,-5);
\draw [->] (-2.2,-6.5) to [bend right] node [below] {cluster isomorphism} (2.2,-6.5);
\draw [->,dashed] (0.5,-2.8) to [out=45, in=45] node [above right] {$\Xi$} (2,-3.5);
\draw [dashed] (-9,-6) ellipse (1.3 and 0.7);
\draw [dashed] (9,-6) ellipse (1.3 and 0.7);
\draw [->, dashed] (-7.5,-6) -- node [below] {cluster mutations} (-5,-6);
\draw [->, dashed] (5,-6) -- node [below] {cluster mutations} (7.5,-6);
\draw [<->] (0,-1.2) -- node [left] {$i_\mathcal{X}$} (0,-2.8);
\draw [<->] (-3.5,1) -- node [left] {$i_\mathcal{X}$} (-3.5,-5);
\draw [<->] (3.5,1) -- node [right] {$i_\mathcal{X}$} (3.5,-5);
\draw [<->] (-9,1) -- node [left] {$i_\mathcal{X}$} (-9,-5);
\draw [<->] (9,1) -- node [right] {$i_\mathcal{X}$} (9,-5);
}$}
\]

For a polynomial $f$ in terms of variables $\left(X_a\right)$, we define the \textit{degree} of a variable $X_a$ in $f$, denoted as $\deg_{X_a}f$, to be the highest power of $X_a$ in $f$; for a rational expression $f/g$ where $f$ and $g$ are both polynomials in variables $\left(X_a\right)$ we define the \textit{degree} of a variable $X_a$ in $f/g$ as $\deg_{X_a}\left(f/g\right):=\deg_{X_a}f-\deg_{X_a}g$.

\begin{prop}\label{uniquexi} The cluster transformation $\Xi:\mathcal{X}_{\left|\vec{i}_0^\circ\right|}\dashrightarrow \mathcal{X}_{\left|\vec{i}_0^\circ\right|}$ induced by $\DT$ has the property that
\[
\deg_{X_a}\Xi^*\left(X_b\right)=-\delta_{ab}.
\]
Moreover, a cluster transformation satisfying the above property is unique (as a birational automorphism) if it exists.
\end{prop}
\begin{proof} Suppose $\DT^*\left(X_b\right)$ is a rational expression of the form $\frac{f\left(X_c\right)}{g\left(X_c\right)}$. Then we have $\Xi^*\left(X_b\right)=\frac{g\left(1/X_c\right)}{f\left(1/X_c\right)}$. Thus
\begin{align*}
\deg_{X_a}\Xi^*\left(X_b\right)=&\deg_{X_a} g\left(1/X_c\right)-\deg_{X_a} f\left(1/X_c\right)\\
=& -\ord_{X_a}g\left(X_c\right)+\ord_{X_a}f\left(X_c\right)\\
=& \ord_{X_a} \DT^*\left(X_b\right)\\
=&-\delta_{ab}.
\end{align*}
The uniqueness of $\Xi$ follows from the uniqueness of cluster Donaldson-Thomas transformation (Theorem \ref{gs}).
\end{proof}

\begin{cor} If we have a cluster transformation $\Xi$ on $\mathcal{X}_{\left|\vec{i}_0^\circ\right|}$ satisfying $\deg_{X_a}\Xi^*\left(X_b\right)=-\delta_{ab}$ for all vertices $a$ and $b$, then $\DT:=i_\mathcal{X}\circ \Xi\circ i_\mathcal{X}$ is the cluster Donaldson-Thomas transformation on $\mathcal{X}_{\left|\vec{i}_0\right|}$.
\end{cor}
\begin{proof} Note that $i_\mathcal{X}\circ \Xi\circ i_\mathcal{X}$ is automatically a cluster transformation. The order condition can be obtained in the exact same way as the identity in the proof above.
\end{proof}

\begin{rmk} In an earlier unpublished version of this paper we used $\left(\mathbb{Z},\max,+\right)$ instead of $\left(\mathbb{Z},\min,+\right)$ for the semiring of tropical integers. Then instead of taking the order of a variable in a rational expression we took the degree, and hence in that earlier version we were in fact computing the map $\Xi$. This is why in that earlier unpublished version of this paper we claimed that $\DT$ is biratinoally isomorphic to $\iota\circ \tw$ but in this version we are proving that $\DT$ is birationally isomorphic to $\tw\circ \iota$. The reason we make this change is because the minimum convention is compatible with other versions of cluster Donaldson-Thomas transformations such as maximal green sequences (see Appendix \ref{C}), but the maximum convention is not.
\end{rmk}

\subsection{Amalgamation}\label{amal} Our goal of this subsection is to describe a process of constructing a cluster ensemble from smaller pieces, which is known as \textit{amalgamation}; it was first introduced by Fock and Goncharov in \cite{FGamalgamation} when they studied the Poisson structures on double Bruhat cells. This will be the key ingredient in our construction of the cluster Donaldson-Thomas transformations of double Bruhat cells.

\begin{defn}\label{amalgamation} Let $\left\{\vec{i}^s\right\}=\left\{\left(I^s,I_0^s, \epsilon^s, d^s\right)\right\}$ be a finite collection of seeds, together with a collection of injective maps $i^s:I^s\rightarrow K$ for some finite set $K$ satisfying the following conditions:
\begin{enumerate}
    \item the images of $i^s$ cover $K$;
    \item $\left(i^s\right)^{-1}\left(i^t\left(I^t\right)\right)\subset I_0^s$ for any $s\neq t$;
    \item if $i^s(a)=i^t(b)$ then $d^s_a=d^t_b$.
\end{enumerate}
Then the \textit{amalgamation} of such collection of seeds is defined to be a new seed $(K,K_0,\epsilon, d)$ where
\[
\epsilon_{ab}:=\sum_{\substack{i^s(a^s)=a\\ i^s(b^s)=b}} \epsilon^s_{a^sb^s}, \quad \quad \quad \quad d_a:=d^s_{a^s} \quad \text{for any $a^s$ with $i^s(a^s)=a$},
\]
\[
\text{and} \quad \quad K_0:=\left(\bigcup_s i^s\left(I_0^s\right)\right)\setminus L.
\]
The set $L$ in the last line can be any subset of the set
\[
\{a\in K\mid  \text{both $\epsilon_{ab}$ and $\epsilon_{ba}$ are integers for all $b\in K$}\}.
\]
In particular, elements of the set $L$ are called \textit{defrosted vertices}.
\end{defn}

Observe that if $a=i^s(a^s)$ for some non-frozen vertex $a^s\in I^s\setminus I_0^s$, then $a$ cannot possibly lie inside the image of any other $i^t$. Therefore mutation at $a$ after amalgamation will give the same seed as amalgamation after mutation at $a^s$. This shows that amalgamation commutes with mutation at non-frozen vertices (\cite{FGamalgamation} Lemma 2.2).

If the seed $\vec{k}$ is the amalgamation of seeds $\vec{i}^s$, then on the seed torus level we can induce two \textit{amalgamation maps}:
\[
     \Delta:\mathcal{A}_\vec{k}\rightarrow \prod_s\mathcal{A}_{\vec{i}^s} \quad \quad \text{and} \quad \quad  m:\prod_s\mathcal{X}_{\vec{i}^s}  \rightarrow \mathcal{X}_\vec{k},
\]
whose pull-backs are
\[
\Delta^*\left(A_{a^s}\right)=A_{i^s(a^s)} \quad \quad \text{and} \quad \quad m^*\left(X_a\right)=\prod_{i^s(a^s)=a} X_{a^s}.
\]
One should think of $\Delta$ as some sort of diagonal embedding and $m$ as some sort of multiplication. This point will become much clearer when we construct the cluster structures on double Bruhat cells of semisimple Lie groups.

\begin{prop} The map $p_\vec{k}:\mathcal{A}_\vec{k} \rightarrow \mathcal{X}_\vec{k}$ can be factored as the composition
\[
\xymatrix{ \mathcal{A}_\vec{k} \ar[r]^(0.4){\Delta} & \prod_s \mathcal{A}_{\vec{i}^s} \ar[d]^{\prod_sp_{\vec{i}^s}} & \\
& \prod_s\mathcal{X}_{\vec{i}^s} \ar[r]_(0.6){m}& \mathcal{X}_\vec{k}
}
\]
\end{prop}
\begin{proof} This can be verified by direction computation using the definitions of the $p$ maps and amalgamation maps $\Delta$ and $m$.
\end{proof}

\subsection{Double Bruhat Cells as Cluster Varieties}\label{cells}

Recall that there are a pair of semisimple Lie groups $G_{sc}$ and $G_{ad}$ associated to each semisimple Lie algebra $\mathfrak{g}$, with $G_{sc}$ being simply connected and $G_{ad}\cong G_{sc}/C(G_{sc})$ being centerless. In this subsection we will describe a reduced cluster ensemble $\left(\mathcal{A}^{u,v},\underline{\mathcal{X}}^{u,v},p\right)$ (see Remark \ref{reduced}) together with two rational maps
\[
 \psi:G_{sc}^{u,v}\dashrightarrow \mathcal{A}^{u,v} \quad \quad \text{and} \quad \quad\chi:\underline{\mathcal{X}}^{u,v}\dashrightarrow H\backslash G^{u,v}/H 
\]
for any given pair of Weyl group elements $(u,v)$. The first map $\psi$ can be obtained from a cluster algebra result of Berenstein, Fomin, and Zelevinsky \cite{BFZ}, whereas the second map $\chi$ is the amalgamation map introduced by Fock and Goncharov in \cite{FGamalgamation}. Unfortunately the seed data used in the two references differ by a sign; we choose to follow Fock and Goncharov's treatment and later will comment on its relation to Berenstein, Fomin, and Zelevinsky's result.

The amalgamation that produces the cluster varieties $\mathcal{A}^{u,v}$ and $\underline{\mathcal{X}}^{u,v}$ comes from the spelling of a (equivalently any) reduced word of the pair of Weyl group elements $(u,v)$. To describe it more precisely, we start with the building block pieces, namely seeds that correspond to letters. Recall that the spelling alphabet is $-\Pi\sqcup \Pi$ where $\Pi$ is the set of simple roots, so the letters naturally come in two types: the ones that are simple roots and the ones that are opposite to simple roots. We will describe the seed data associated to each of these two types.

Let $\alpha$ be a simple root. We define a new set $\Pi^\alpha:=\left(\Pi\setminus\{\alpha\}\right)\cup \{\alpha_-,\alpha_+\}$. One should think of $\Pi^\alpha$ as almost the same as the set of simple roots $\Pi$ except the simple root $\alpha$ splits into two copies $\alpha_-$ and $\alpha_+$. 

Now we define the seed associated to the simple root $\alpha$ to be $\vec{i}^\alpha:=\left(\Pi^\alpha,\Pi^\alpha, \epsilon^\alpha, d^\alpha\right)$ where 
\begin{align*}
\epsilon^\alpha_{ab}:=&\left\{\begin{array}{ll}
    \pm 1 & \text{if $a=\alpha_\pm$ and $b=\alpha_\mp$;} \\
    \pm C_{\beta\alpha}/2 & \text{if $a=\alpha_\pm$ and $b=\beta$ that is neither of the two split copies of $\alpha$;}\\
    \pm C_{\alpha\beta}/2 & \text{if $a=\beta$ that is neither of the two split copies of $\alpha$ and $b=\alpha_\mp$;}\\
    0 & \text{otherwise},
\end{array}\right. \\
d^\alpha_a:=&\left\{\begin{array}{ll}
    D_\alpha & \text{if $a=\alpha_\pm$;} \\
    D_\beta & \text{if $a=\beta$ that is neither of the two split copies of $\alpha$.}
\end{array}\right.
\end{align*}
Here $D$ is the diagonal matrix that symmetrize the Cartan matrix $C$ \eqref{symmetrizable}. In contrast, we define the seed associated to $-\alpha$ to be $\vec{i}^{-\alpha}:=\left(\Pi^{-\alpha},\Pi^{-\alpha}, \epsilon^{-\alpha}, d^{-\alpha}\right)$ where 
\[
\Pi^{-\alpha}:=\Pi^\alpha, \quad \quad d^{-\alpha}:=d^\alpha,\quad \quad\text{ and}\quad \quad \epsilon^{-\alpha}:=-\epsilon^\alpha.
\]

It is straightforward to verify that $\vec{i}^\alpha$ and $\vec{i}^{-\alpha}$ are seeds. Note that all vertices of these two seeds are frozen, so there is no mutation available. Before we start amalgamation, we would like to introduce two families rational maps that will go along with amalgamation:
\[
\psi^{\pm \alpha}: G_{sc}\dashrightarrow \mathcal{A}_{\vec{i}^{\pm\alpha}} \quad \quad \text{and} \quad \quad \chi^{\pm \alpha}: \mathcal{X}_{\vec{i}^{\pm\alpha}}\rightarrow G_{ad}.
\]
The rational maps $\psi^{\pm \alpha}$ are defined by the pull-backs
\[
\psi^{-\alpha*}\left(A_a\right)=\left\{\begin{array}{ll}
    \Delta_\alpha & \text{if $a=\alpha_-$;} \\
    \Delta_\alpha\left(\overline{s}_\alpha^{-1} \quad \cdot \quad \right) & \text{if $a=\alpha_+$;}\\
    \Delta_\beta & \text{if $a=\beta$ that is neither of the two split copies of $\alpha$;}
\end{array}\right.
\]
\[
\psi^{\alpha*}\left(A_a\right)=\left\{\begin{array}{ll}
    \Delta_\alpha\left(\quad \cdot \quad \overline{s}_\alpha\right) & \text{if $a=\alpha_-$;} \\
    \Delta_\alpha & \text{if $a=\alpha_+$;}\\
    \Delta_\beta & \text{if $a=\beta$ that is neither of the two split copies of $\alpha$.}
\end{array}\right.
\]
The maps $\chi^{\pm \alpha}$ are defined by
\begin{equation} \label{chi^pm}
\chi^{\pm \alpha}\left(X_a\right):= X_{\alpha_-}^{H^\alpha}e_{\pm \alpha} X_{\alpha_+}^{H^\alpha} \prod_{\beta\neq \alpha} X_\beta^{H^\beta},
\end{equation}
where the product on the right hand side takes place inside the Lie group $G_{ad}$.

Now we are ready for amalgamation. Fix a pair of Weyl group elements $(u,v)$ and a reduced word $\vec{i}:=\left(\alpha(1),\dots, \alpha(l)\right)$ of the pair $(u,v)$. The family of seeds that we are amalgamating is 
\[
\left\{\vec{i}^{\alpha(k)}\right\}_{k=1}^l:=\left\{\left(\Pi^{\alpha(k)}, \Pi^{\alpha(k)}, \epsilon^{\alpha(k)}, d^{\alpha(k)}\right)\right\}_{k=1}^l,
\]
one for each letter $\alpha(k)$ in the reduced word. To define the amalgamation, we also need a finite set $K$ and a collection of injective maps $i^k:\Pi^{\alpha(k)}\rightarrow K$. Let $n_\alpha$ be the total number of times that letters $\pm \alpha$ appear in the reduced word $\vec{i}$; then we define
\[
K:=\left\{\textstyle\binom{\alpha}{i} \ \middle| \ \alpha\in \Pi, 0\leq i\leq n_\alpha\right\}.
\]
Further we define
\[
|\alpha(k)|=\left\{\begin{array}{ll}
    \alpha(k) & \text{if $\alpha(k)$ is a simple root},  \\
    -\alpha(k) & \text{if $\alpha(k)$ is opposite to a simple root},
\end{array}\right.
\]
which we then use to define the injective maps $i^k:\Pi^{\alpha(k)}\rightarrow K$ as follows:
\[
i^k(a)=\left\{\begin{array}{ll}
    \binom{|\alpha(k)|}{i-1} & \text{if $a=|\alpha(k)|_-$ and $\alpha(k)$ is the $i$th appearance of letters $\pm |\alpha(k)|$;} \\
    \binom{|\alpha(k)|}{i} & \text{if $a=|\alpha(k)|_+$ and $\alpha(k)$ is the $i$th appearance of lettes $\pm |\alpha(k)|$;}\\
    \binom{\beta}{j} & \begin{tabular}{l} if $a=\beta$ that is neither of the two split copies of $|\alpha(k)|$ and there \\
    have been $j$ numbers of $\pm\beta$ appearing before $\alpha(k)$.\end{tabular}
\end{array}\right.
\]

\begin{exmp} Although the amalgamation data above is heavy on notations, the idea is intuitive. One should think of $\binom{\alpha}{i}$ as the space before the first appearance of $\pm \alpha$, the gap between every two appearances of $\pm \alpha$, or the space after the last appearance of $\pm \alpha$. Then the injective map $i^k$ basically replaces the letter $\alpha(k)$ with the seed $\vec{i}^{\alpha(k)}$, and the gluing connects to the left via $\left(\Pi\setminus\{|\alpha(k)|\}\right)\cup\{|\alpha(k)|_-\}$ and connects to the right via $\left(\Pi\setminus\{|\alpha(k)|\}\right)\cup\{|\alpha(k)|_+\}$. 

To better convey the idea, consider a rank 3 semisimple group $G$ whose simple roots are $\{\alpha,\beta, \gamma\}$. Let $\vec{i}:=(\alpha, -\beta, -\alpha,\gamma, \beta, -\beta)$ be a reduced word of a pair of Weyl group elements. By writing different letters (disregard the signs) on different horizontal lines, the elements of $K$ then become very clear, as illustrated below.
\[
\tikz{
\node (1) at (1,2) [] {$\alpha$};
\node (2) at (2,1) [] {$-\beta$};
\node (3) at (3,2) [] {$-\alpha$};
\node (4) at (4,0) [] {$\gamma$};
\node (5) at (5,1) [] {$\beta$};
\node (6) at (6,1) [] {$-\beta$};
\draw (0,2) -- node[above]{$\binom{\alpha}{0}$} (1) -- node[above]{$\binom{\alpha}{1}$} (3) -- node[above]{$\binom{\alpha}{2}$} (7,2);
\draw (0,1) -- node[above]{$\binom{\beta}{0}$} (2) -- node[above]{$\binom{\beta}{1}$} (5) -- node[above]{$\binom{\beta}{2}$} (6) -- node[above]{$\binom{\beta}{3}$} (7,1);
\draw (0,0) -- node[above]{$\binom{\gamma}{0}$} (4) -- node[above]{$\binom{\gamma}{1}$} (7,0);
\node at (3,-1) [] {$\vec{i}$};
}
\]
On the other hand, if we also draw the seed associated to each letter the same way, i.e.,
\[
\tikz[baseline=0ex]{
\node (0) at (1,2) [] {$\pm \alpha$};
\draw (0,2) -- node[above]{$\alpha_-$} (0) -- node[above]{$\alpha_+$} (2,2);
\draw (0,1) -- node[above]{$\beta$} (2,1);
\draw (0,0) -- node[above]{$\gamma$} (2,0);
\node at (1,-1) [] {$\vec{i}^{\pm \alpha}$};
}
\quad \quad \quad 
\tikz[baseline=0ex]{
\node (0) at (1,1) [] {$\pm \beta$};
\draw (0,1) -- node[above]{$\beta_-$} (0) -- node[above]{$\beta_+$} (2,1);
\draw (0,2) -- node[above]{$\alpha$} (2,2);
\draw (0,0) -- node[above]{$\gamma$} (2,0);
\node at (1,-1) [] {$\vec{i}^{\pm \beta}$};
}
\quad \quad \quad
\tikz[baseline=0ex]{
\node (0) at (1,0) [] {$\pm \gamma$};
\draw (0,0) -- node[above]{$\gamma_-$} (0) -- node[above]{$\gamma_+$} (2,0);
\draw (0,1) -- node[above]{$\beta$} (2,1);
\draw (0,2) -- node[above]{$\alpha$} (2,2);
\node at (1,-1) [] {$\vec{i}^{\pm \gamma}$};
}
\]
then the injective maps $i^k$ are just placing pieces of building blocks of $\vec{i}$ into the right positions. We will call the diagram where we put letters associated to different simple roots (disregarding the sign) on different horizontal lines a \textit{string diagram}. In a string diagram we call the letters \textit{nodes} and horizontal lines cut out by nodes \textit{strings}. We say a string is on \textit{level} $\alpha$ if it is on the horizontal line corresponding to the simple root $\alpha$. In particular, we say that a string is \textit{closed} if it is cut out by letters on both ends, and otherwise we say that it is \textit{open}. Obviously a string is open if and only if it is at either end of the string diagram.
\end{exmp}

\begin{prop} The set $K$ and injective maps $i^k$ satisfy the conditions for amalgamation (Definition \ref{amalgamation}).
\end{prop}
\begin{proof} (1) and (2) are obvious, and (3) holds since the data of $d^{\alpha(k)}$ for every seed $\vec{i}^{\alpha(k)}$ is identical to the diagonal matrix $D$ which symmetrizes the Cartan matrix $C$.  
\end{proof}

The amalgamated seed obtained this way is not very interesting unless we defrost some of its vertices. As it turns out, the right choice of the set of defrosted vertices is
\[
L:=\left\{\textstyle\binom{\alpha}{i}\ \middle| \ 0<i<n_\alpha\right\}.
\]
Of course, we need to show the following statement which is demanded by Definition \ref{amalgamation}.

\begin{prop} In the exchange matrix $\epsilon$ of the amalgamated seed, any entry involving elements of $L$ is an integer.
\end{prop}
\begin{proof} The only possible non-integer entry of $\epsilon$ always comes from entries $\pm C_{\alpha\beta}/2$ of the exchange matrices of the seeds for individual letters. Notice that if $0<i<n_\alpha$, then $\binom{\alpha}{i}$ is a gap between two appearances of $\pm \alpha$. Thus any entry of $\epsilon$ involving $\binom{\alpha}{i}$ will have contributions from the both ends of this gap, which will always have the same absolute value with denominator 2 according to the construction. Therefore any entry of $\epsilon$ involving $\binom{\alpha}{i}$ has to be an integer.
\end{proof}

By now it should be clear that vertices of the amalgamated seed are in bijection with strings in the corresponding string diagram. Since the defrosted set of vertices $L$ are by definition the set of closed strings, it follows that the frozen vertices in the amalgamated seed are in bijection with the open strings.

From this amalgamation procdure we obtain the amalgamated seed $(K, K_0, \epsilon,d)$, which by an abuse of notation we will also denote as $\vec{i}$. Recall that amalgamation can also be lifted to the seed torus level, and hence we have maps
\[
\Delta:\mathcal{A}_\vec{i}\rightarrow \prod_{k=1}^l \mathcal{A}_{\vec{i}^{\alpha(k)}} \quad \quad \text{and} \quad \quad m:\prod_{k=1}^l \mathcal{X}_{\vec{i}^{\alpha(k)}}\rightarrow \mathcal{X}_\vec{i}.
\]
Our next goal is to find maps to complete the following commutative diagram.
\[
\xymatrix{ G_{sc} \ar@{-->}[d]_{\psi_\vec{i}} \ar[r]^(0.4){\Delta} & \prod_{k=1}^l G_{sc} \ar@{-->}[d]^{\prod_k \psi^{\alpha(k)}}\\
\mathcal{A}_\vec{i} \ar[r]_(0.3){\Delta} & \prod_{k=1}^l \mathcal{A}_{\vec{i}^{\alpha(k)}}
}\quad \quad \quad \quad 
\xymatrix{ \prod_{k=1}^l \mathcal{X}_{\vec{i}^{\alpha(k)}} \ar[r]^(0.7){m} \ar[d]_{\prod_k \chi^{\alpha(k)}} & \mathcal{X}_\vec{i} \ar[d]^{\chi_\vec{i}} \\
\prod_{k=1}^l G_{ad} \ar[r]_(0.6){m} & G_{ad}}
\]

Among the new maps we need to define, the easiest one is $m:\prod_{k=1}^l G_{ad}\rightarrow G_{ad}$: as the notation suggested, it is just multiplication in $G_{ad}$ following the order $1\leq k\leq l$. 

However, it requires a bit more work to define $\Delta:G_{sc}\rightarrow \prod_{k=1}^l G_{sc}$. Recall the choice of reduced word $\vec{i}=\left(\alpha(1),\dots, \alpha(l)\right)$ of our pair of Weyl group elements $(u,v)$. For every $1\leq k\leq l$ we define two new Weyl group elements $u_{<k}$ and $v_{>k}$ as following: pick out all the entries of $\vec{i}$ before $\alpha(k)$ (not including $\alpha(k)$) that are opposite to simple roots, and then multiply the corresponding simple reflections to get $u_{<k}$; similarly, pick out all the entries of $\vec{i}$ after $\alpha(k)$ (not including $\alpha(k)$) that are simple roots, and then multiply the corresponding simple reflections to get $v_{>k}$. The map $\Delta$ is then defined to be
\[
\Delta: x\mapsto \left(\overline{u_{<k}}^{-1} x \overline{\left(v_{>k}\right)^{-1}}\right)_{k=1}^l.
\]

Now comes the most difficult part, which is completing the square with the remaining yet-to-define map $\psi_\vec{i}$ and $\chi_\vec{i}$.

\begin{prop}\label{psiandchi} There exists unique maps $\psi_\vec{i}$ and $\chi_\vec{i}$ that fit into the commutative diagrams above and make them commute.
\end{prop}
\begin{proof} Let's first consider the commutative diagram on the left. On the one hand, by following the top arrow and then the right arrow of the square we get a rational map $G_{sc}\dashrightarrow \prod_k \mathcal{A}_{\vec{i}^{\alpha(k)}}$. On the other hand, the bottom arrow is some sort of diagonal embedding and is injective. Therefore all we need to show is that the image of $G_{sc}\dashrightarrow \prod_k \mathcal{A}_{\vec{i}^{\alpha(k)}}$ lies inside the image of the bottom arrow. In other words, we need to show that if the vertex $a$ of $\vec{i}^{\alpha(k)}$ is glued to vertex $b$ of $\vec{i}^{\alpha(k+1)}$ during amalgamation, then 
\[
A_a^{\alpha(k)}\left(\psi^{\alpha(k)}\left(\overline{u_{<k}}^{-1} x \overline{\left(v_{>k}\right)^{-1}}\right)\right)=A_b^{\alpha(k+1)}\left(\psi^{\alpha(k+1)}\left(\overline{u_{<k+1}}^{-1} x \overline{\left(v_{>k+1}\right)^{-1}}\right)\right). 
\]
There are a few possible cases to analyze, but the arguments are all analogous. Thus without loss of generality let's assume that $\alpha(k)=\alpha$ and $\alpha(k+1)=\beta$ are both simple roots and $\alpha\neq \beta$. If $a=\beta$ and $b=\beta_-$, then we see right away from the definition of $\psi^{\alpha(k+1)}$ that both sides of the above equation are exactly the same. If $a=b=\gamma$ for some $\gamma$ other than $\beta$, we then know that the left hand side of the desired equality above is $\Delta_\gamma\left(\overline{u_{<k}}^{-1} x \overline{\left(v_{>k}\right)^{-1}}\right)$ whereas the right hand side is $\Delta_\gamma\left(\overline{u_{<k+1}}^{-1} x \overline{\left(v_{>k+1}\right)^{-1}}\right)$. Since we have assumed that both $\alpha$ and $\beta$ are simple roots, it follows that $u_{<k}=u_{<k+1}$ and $v_{>k}=s_\beta v_{>k+1}$. Therefore we have
\[
\Delta_\gamma\left(\overline{u_{<k+1}}^{-1} x \overline{\left(v_{>k+1}\right)^{-1}}\right)=\Delta_\gamma\left(\overline{u_{<k}}^{-1}x\overline{\left(v_{>k}\right)^{-1}}\overline{s}_\beta\right)=\Delta_\gamma\left(\overline{u_{<k}}^{-1} x \overline{\left(v_{>k}\right)^{-1}}\right),
\]
where the last equality is due to Proposition \ref{minor}.

Now let's turn to the commutative diagram on the right. Observe that the top arrow is surjective. So to define $\chi_\vec{i}:\mathcal{X}_\vec{i}\rightarrow G_{ad}$, we can first lift the input to $\prod_k \mathcal{X}_{\vec{i}^{\alpha(k)}}$ and then follow the left arrow and the bottom arrow to arrive at $G_{ad}$. Then all we need to do is the verify that such map is well-defined, i.e., the final output does not depend on the lift. But this is obvious since
\[
X_a^{H^\alpha}X_b^{H^\alpha}=\left(X_aX_b\right)^{H^\alpha}
\]
and Proposition \ref{commute} tells us that whenever $\alpha\neq \beta$,
\[
e_{\pm \beta} X^{H^\alpha}=X^{H^\alpha}e_{\pm \beta}. \qedhere
\]
\end{proof}

Note that the only frozen vertices of the amalgamated seed are of the form $\binom{\alpha}{0}$ and $\binom{\alpha}{n_\alpha}$. According to the map $\chi_\vec{i}$ we constructed above, the coordinates corresponding to these frozen vertices are equivalent to multiplication by elements of the maximal torus $H$ on the left or on the right. Therefore the map $\chi_\vec{i}$ can be passed to a map
\[
\chi_\vec{i}:\underline{\mathcal{X}}_\vec{i}\rightarrow H\backslash G/H.
\]
Note that we don't keep the subscript of $G$ any more because the double quotient of any Lie group $G$ with the same Lie algebra by its maximal torus on both sides is the same variety.

At this point it is natural to ask the question: what if we start with a different reduced word? Is there any relation between the resulting seed $\mathcal{A}$-tori and the resulting seed $\mathcal{X}$-tori? For the rest of this subsection we will explore the relations among these seed tori, which ultimately will give us the reduced cluster ensemble $\left(\mathcal{A}^{u,v}, \mathcal{X}^{u,v},p\right)$ and the maps $\psi:G_{sc}\dashrightarrow\mathcal{A}^{u,v}$ and $\chi:\underline{\mathcal{X}}^{u,v}\rightarrow H\backslash G/H$.

We start with the following elementary observation.

\begin{prop}\label{move} If $\vec{i}$ and $\vec{i}'$ are two reduced words for the same pair of Weyl group elements $(u,v)$, then one can transform $\vec{i}$ to $\vec{i}'$ or vice versa via a finite sequence of the following moves:
\begin{itemize}
    \item move (1): exchange two neighboring letters of opposite signs, i.e., 
    \[
    (\dots, \alpha, -\beta,\dots)\sim (\dots, - \beta,  \alpha,\dots)
    \]
    (we assume that both $\alpha$ and $\beta$ are simple roots);
    \item move (2): replace a consecutive segment of letters with the same sign by their equivalent using the braid relations \eqref{braid}, for example, if $C_{\alpha\beta}C_{\beta\alpha}=-1$, then 
    \[
    (\dots, \pm \alpha,\pm \beta,\pm \alpha,\dots)\sim (\dots, \pm \beta,\pm \alpha,\pm\beta,\dots).
    \]
\end{itemize}
\end{prop}
\begin{proof} Recall the general fact that any two reduced words of the same Weyl group element $w$ can be obtained from one another via a finite sequence of moves described by the braid relations. Therefore we can first use move (1) to separate the letters from $\Pi$ and the letters from $-\Pi$ and get reduced words of $u$ and $v$ respectively, and then use move (2) to rearrange the letters in each of them, and lastly use move (1) to mix them up again. 
\end{proof}

Since these two moves are responsible for the transformation between reduced words, we may also want to investigate their induced transformation on seeds.

\begin{prop} Move (1) does not induce any change on the seed data unless $\alpha=\beta$, which then induces a mutation at the vertex corresponding to the gap between these two letters. Each case of move (2) is a composition of seed mutations; in particular, it is a single mutation if $C_{\alpha\beta}C_{\beta\alpha}=1$, a composition of 3 mutations if $C_{\alpha\beta}C_{\beta\alpha}=2$, and a composition of 10 mutations if $C_{\alpha\beta}C_{\beta\alpha}=3$.
\end{prop}
\begin{proof} Since both the amalgamation and the two moves as described above are local operations, without loss of generality we may assume that there are no other letters present in the reduced words other than the ones involved in the moves. Let's consider move (1) first. If $\alpha\neq \beta$, then the string diagram representing move (1) will look like the following.
\[
\tikz[baseline=4ex]{
\node (a) at (1,1) [] {$ \alpha$};
\node (b) at (2,0) [] {$-\beta$};
\draw (0,1) -- node[above]{$\binom{\alpha}{0}$} (a) -- node[above]{$\binom{\alpha}{1}$} (3,1);
\draw (0,0) -- node[above]{$\binom{\beta}{0}$} (b) -- node[above]{$\binom{\beta}{1}$} (3,0);
}\quad \quad \sim \quad \quad 
\tikz[baseline=4ex]{
\node (a) at (2,1) [] {$ \alpha$};
\node (b) at (1,0) [] {$- \beta$};
\draw (0,1) -- node[above]{$\binom{\alpha}{0}$} (a) -- node[above]{$\binom{\alpha}{1}$} (3,1);
\draw (0,0) -- node[above]{$\binom{\beta}{0}$} (b) -- node[above]{$\binom{\beta}{1}$} (3,0);
}
\]
Thus the only pieces of seed data that may be affected by such a move are the entries $\epsilon_{\binom{\alpha}{0}\binom{\beta}{1}}$, $\epsilon_{\binom{\alpha}{1}\binom{\beta}{0}}$, $\epsilon_{\binom{\beta}{0}\binom{\alpha}{1}}$, and $\epsilon_{\binom{\beta}{1}\binom{\alpha}{0}}$. But the amalgamation construction tells us that they all vanish before and after the move. Thus move (1) does not induce any change on the seed data if $\alpha\neq \beta$.

If we are applying move (1) in the special case $(\alpha,-\alpha)\sim (-\alpha,\alpha)$, then the string diagram will look like the following ($\beta$ can be any simple root other than $\alpha$).
\[
\tikz[baseline=4ex]{
\node (a) at (1,1) [] {$\alpha$};
\node (b) at (3,1) [] {$-\alpha$};
\draw (0,1) -- node[above]{$\binom{\alpha}{0}$} (a) -- node[above]{$\binom{\alpha}{1}$} (b) -- node[above]{$\binom{\alpha}{2}$} (4,1);
\draw (0,0) -- node[above]{$\binom{\beta}{0}$} (4,0);
}\quad \quad \sim \quad \quad 
\tikz[baseline=4ex]{
\node (a) at (1,1) [] {$-\alpha$};
\node (b) at (3,1) [] {$\alpha$};
\draw (0,1) -- node[above]{$\binom{\alpha}{0}$} (a) -- node[above]{$\binom{\alpha}{1}$} (b) -- node[above]{$\binom{\alpha}{2}$} (4,1);
\draw (0,0) -- node[above]{$\binom{\beta}{0}$} (4,0);
}
\]
If we use the unprime notation to denote the exchange matrix of the left and the prime notation to denote that of the right, then we see that the entries of the exchange matrix that got changed under such a move are
\[
\epsilon_{\binom{\alpha}{1}\binom{\beta}{0}}=C_{\beta\alpha}, \quad \quad \epsilon_{\binom{\beta}{0}\binom{\alpha}{1}}=-C_{\alpha\beta},
\]
\[
\epsilon_{\binom{\alpha}{1}\binom{\alpha}{0}}=-\epsilon_{\binom{\alpha}{0}\binom{\alpha}{1}}=\epsilon_{\binom{\alpha}{1}\binom{\alpha}{2}}=-\epsilon_{\binom{\alpha}{2}\binom{\alpha}{1}}=1,
\]
\[
\epsilon_{\binom{\beta}{0}\binom{\alpha}{0}}=\epsilon_{\binom{\beta}{0}\binom{\alpha}{2}}=\frac{C_{\alpha\beta}}{2}, \quad \quad \epsilon_{\binom{\alpha}{0}\binom{\beta}{0}}=\epsilon_{\binom{\alpha}{2}\binom{\beta}{0}}=-\frac{C_{\beta\alpha}}{2},
\]
which change to 
\[
\epsilon'_{\binom{\alpha}{1}\binom{\beta}{0}}=-C_{\beta\alpha}, \quad \quad \epsilon'_{\binom{\beta}{0}\binom{\alpha}{1}}=C_{\alpha\beta},
\]
\[
\epsilon'_{\binom{\alpha}{1}\binom{\alpha}{0}}=-\epsilon'_{\binom{\alpha}{0}\binom{\alpha}{1}}=\epsilon'_{\binom{\alpha}{1}\binom{\alpha}{2}}=-\epsilon'_{\binom{\alpha}{2}\binom{\alpha}{1}}=-1,
\]
\[
\epsilon'_{\binom{\beta}{0}\binom{\alpha}{0}}=\epsilon'_{\binom{\beta}{0}\binom{\alpha}{2}}=-\frac{C_{\alpha\beta}}{2}, \quad \quad \epsilon'_{\binom{\alpha}{0}\binom{\beta}{0}}=\epsilon'_{\binom{\alpha}{2}\binom{\beta}{0}}=\frac{C_{\beta\alpha}}{2},
\]
One can verify easily that such change is exactly the same as a mutation at the vertex $\binom{\alpha}{1}$.

Now let's consider move (2). Due to symmetry, we will only prove for the case where all the letters are simple roots; the case where the letters are opposite to simple roots is completely analogous. Let's start with the simplest case where $C_{\alpha\beta}C_{\beta\alpha}=1$. Then move (2) says that $(\alpha, \beta, \alpha)\sim (\beta, \alpha,\beta)$. The following is the corresponding string diagram.
\[
\tikz[baseline=4ex]{
\node (a) at (1,1) [] {$\alpha$};
\node (b) at (3,1) [] {$\alpha$};
\node (c) at (2,0) [] {$\beta$};
\draw (0,1) -- node[above]{$1$} (a) -- node[above]{$0$} (b) -- node[above]{$2$} (4,1);
\draw (0,0) -- node[above]{$3$} (c) -- node[above]{$4$} (4,0);
}\quad \quad \sim \quad \quad 
\tikz[baseline=4ex]{
\node (a) at (1,0) [] {$\beta$};
\node (b) at (3,0) [] {$\beta$};
\node (c) at (2,1) [] {$\alpha$};
\draw (0,0) -- node[above]{$3$} (a) -- node[above]{$0$} (b) -- node[above]{$4$} (4,0);
\draw (0,1) -- node[above]{$1$} (c) -- node[above]{$2$} (4,1);
}
\]
To avoid possible confusion we rename the vertices (gaps) of the seeds with numbers. Note that the vertex 0 in either picture only has non-vanishing exchange matrix entries with the other 4 vertices that are present but nothing else. We claim that the move (2) in this case induces a single seed mutation at the vertex 0. In fact since $C_{\alpha\beta}C_{\beta\alpha}=1$ we can present the seed data with a quiver, and it is obvious from the quiver presentation that such a move is indeed a seed (quiver) mutation. (We use the convention that dashed arrows represent half weight exchange matrix in either direction).
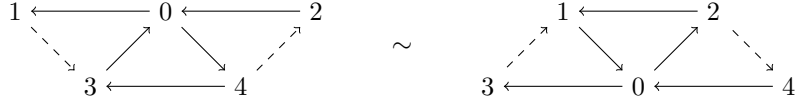
\[
\tikz[baseline=3ex]{
\node (1) at (0,1) [] {$1$};
\node (0) at (2,1) [] {$0$};
\node (2) at (4,1) [] {$2$};
\node (3) at (1,0) [] {$3$};
\node (4) at (3,0) [] {$4$};
\draw [<-] (1) -- (0);
\draw [<-] (0) -- (2);
\draw [<-] (3) -- (4);
\draw [<-] (0) -- (3);
\draw [<-] (4) -- (0);
\draw [dashed, ->] (1) -- (3);
\draw [dashed, ->] (4) -- (2);
}\quad \quad \sim \quad \quad
\tikz[baseline=3ex]{
\node (1) at (1,1) [] {$1$};
\node (0) at (2,0) [] {$0$};
\node (2) at (3,1) [] {$2$};
\node (3) at (0,0) [] {$3$};
\node (4) at (4,0) [] {$4$};
\draw [<-] (1) -- (2);
\draw [<-] (0) -- (1);
\draw [<-] (2) -- (0);
\draw [<-] (3) -- (0);
\draw [<-] (0) -- (4);
\draw [dashed, <-] (1) -- (3);
\draw [dashed, <-] (4) -- (2);
}
\]

Next let's consider the case $C_{\alpha\beta}C_{\beta\alpha}=2$, for which move (2) says $(\alpha,\beta,\alpha,\beta)\sim (\beta,\alpha,\beta,\alpha)$. Without loss of generality let's assume that $C_{\alpha\beta}=-2$ and $C_{\beta\alpha}=-1$. The corresponding string diagram is the following.
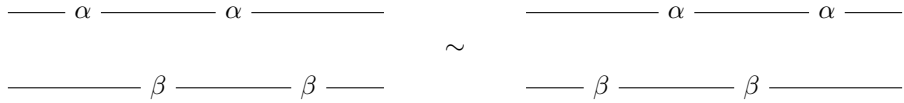
\[
\tikz[baseline=3ex]{
\node (a) at (1,1) [] {$\alpha$};
\node (b) at (3,1) [] {$\alpha$};
\node (c) at (2,0) [] {$\beta$};
\node (d) at (4,0) [] {$\beta$};
\draw (0,1) -- (a) --  (b) --  (5,1);
\draw (0,0) -- (c) -- (d) -- (5,0);
}\quad \quad \sim \quad \quad
\tikz[baseline=3ex]{
\node (a) at (2,1) [] {$\alpha$};
\node (b) at (4,1) [] {$\alpha$};
\node (c) at (1,0) [] {$\beta$};
\node (d) at (3,0) [] {$\beta$};
\draw (0,1) -- (a) --  (b) --  (5,1);
\draw (0,0) -- (c) -- (d) -- (5,0);
}
\]
Unfortunately in this case the exchange matrix is not antisymmetric, so we cannot use an ordinary quiver to present the seed data. One way to go around is to introduce the following new notation to replace arrows in a quiver
\[
\tikz[baseline=-0.5ex]{
    \node (a) at (0,0) [] {$\bullet$};
    \node (b) at (2,0) [] {$\bullet$};
    \node at (0,0) [left] {$a$};
    \node at (2,0) [right] {$b$};
    \draw [z->] (a) -- (b);
}
\quad \quad \text{means} \quad \epsilon_{ab}=1 \quad \text{and} \quad \epsilon_{ba}=-2,
\]
\[
\tikz[baseline=-0.5ex]{
    \node (a) at (0,0) [] {$\bullet$};
    \node (b) at (2,0) [] {$\bullet$};
    \node at (0,0) [left] {$a$};
    \node at (2,0) [right] {$b$};
    \draw [-z>] (a) -- (b);
}
\quad \quad \text{means} \quad \epsilon_{ab}=2 \quad \text{and} \quad \epsilon_{ba}=-1;
\]
we call the resulting picture a \textit{quasi-quiver}, which is a useful tool to make the mutation computation more illustrative. Translating the picture on the left and the picture on the right into the language of quasi-quiver (again we will use dashed line to represent half weight), we get
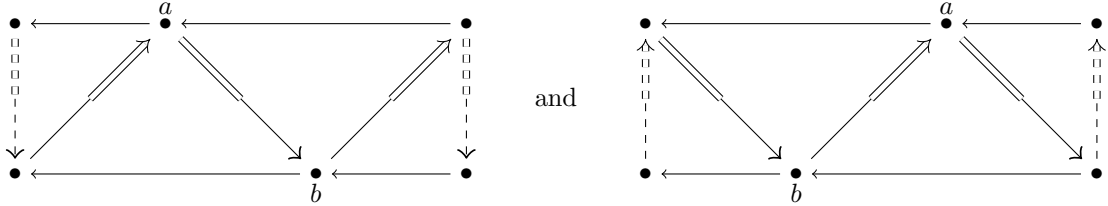
\[
\tikz[baseline=6ex]{
\node (a) at (0,2) [] {$\bullet$};
\node (b) at (2,2) [] {$\bullet$};
\node (c) at (6,2) [] {$\bullet$};
\node (d) at (0,0) [] {$\bullet$};
\node (e) at (4,0) [] {$\bullet$};
\node (f) at (6,0) [] {$\bullet$};
\draw [->] (b) -- (a);
\draw [->] (c) -- (b);
\draw [->] (f) -- (e);
\draw [->] (e) -- (d);
\draw [-z>] (d) -- (b);
\draw [z->] (b) -- (e);
\draw [-z>] (e) -- (c);
\draw [z-->] (a) -- (d);
\draw [z-->] (c) -- (f);
\node at (b) [above] {$a$};
\node at (e) [below] {$b$};
} \quad \quad \text{and} \quad \quad
\tikz[baseline=6ex]{
\node (a) at (0,2) [] {$\bullet$};
\node (b) at (4,2) [] {$\bullet$};
\node (c) at (6,2) [] {$\bullet$};
\node (d) at (0,0) [] {$\bullet$};
\node (e) at (2,0) [] {$\bullet$};
\node (f) at (6,0) [] {$\bullet$};
\draw [->] (b) -- (a);
\draw [->] (c) -- (b);
\draw [->] (f) -- (e);
\draw [->] (e) -- (d);
\draw [z->] (a) -- (e);
\draw [-z>] (e) -- (b);
\draw [z->] (b) -- (f);
\draw [--z>] (d) -- (a);
\draw [--z>] (f) -- (c);
\node at (b) [above] {$a$};
\node at (e) [below] {$b$};
} 
\]
We claim that either seed (quasi-quiver) can be obtained from the other via a sequence of 3 mutations at the vertices labeled $a$ and $b$. The following is the quasi-quiver illustration of such mutation sequence.
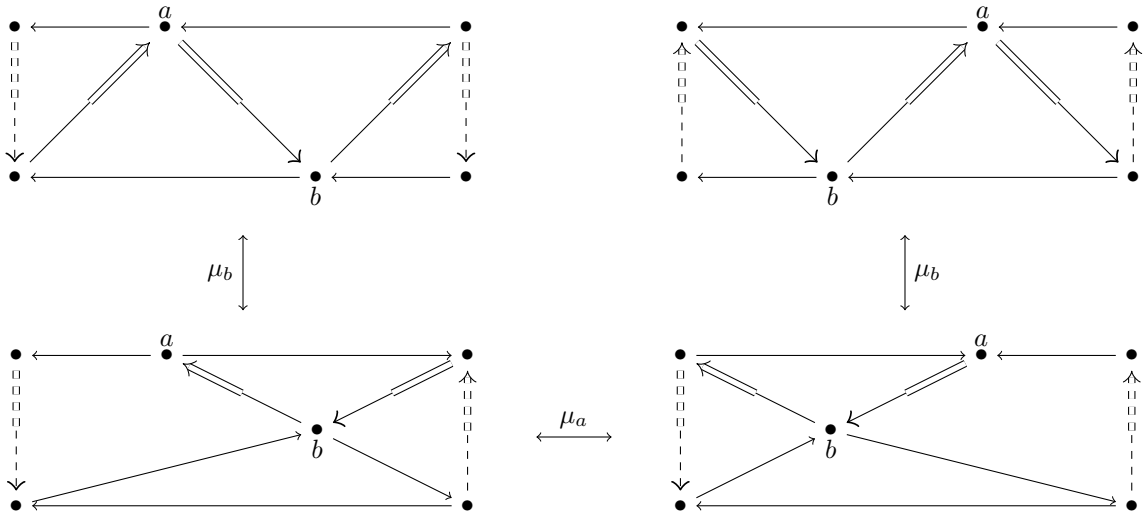
\[
\tikz[baseline=6ex]{
\node (a) at (0,2) [] {$\bullet$};
\node (b) at (2,2) [] {$\bullet$};
\node (c) at (6,2) [] {$\bullet$};
\node (d) at (0,0) [] {$\bullet$};
\node (e) at (4,0) [] {$\bullet$};
\node (f) at (6,0) [] {$\bullet$};
\draw [->] (b) -- (a);
\draw [->] (c) -- (b);
\draw [->] (f) -- (e);
\draw [->] (e) -- (d);
\draw [-z>] (d) -- (b);
\draw [z->] (b) -- (e);
\draw [-z>] (e) -- (c);
\draw [z-->] (a) -- (d);
\draw [z-->] (c) -- (f);
\node at (b) [above] {$a$};
\node at (e) [below] {$b$};
} \quad \quad \quad \quad \quad \quad \quad
\tikz[baseline=6ex]{
\node (a) at (0,2) [] {$\bullet$};
\node (b) at (4,2) [] {$\bullet$};
\node (c) at (6,2) [] {$\bullet$};
\node (d) at (0,0) [] {$\bullet$};
\node (e) at (2,0) [] {$\bullet$};
\node (f) at (6,0) [] {$\bullet$};
\draw [->] (b) -- (a);
\draw [->] (c) -- (b);
\draw [->] (f) -- (e);
\draw [->] (e) -- (d);
\draw [z->] (a) -- (e);
\draw [-z>] (e) -- (b);
\draw [z->] (b) -- (f);
\draw [--z>] (d) -- (a);
\draw [--z>] (f) -- (c);
\node at (b) [above] {$a$};
\node at (e) [below] {$b$};
}
\]
\[
 \tikz{\draw [<->] (0,0) -- node[left]{$\mu_b$} (0,1);} \quad\quad \quad\quad\quad \quad\quad \quad\quad \quad\quad \quad\quad \quad\quad \quad\quad \quad\quad \quad \quad\quad \quad\quad \quad \tikz{\draw [<->] (0,0) -- node[right]{$\mu_b$} (0,1);}
\]
\[
\tikz[baseline=6ex]{
\node (a) at (0,2) [] {$\bullet$};
\node (b) at (2,2) [] {$\bullet$};
\node (c) at (6,2) [] {$\bullet$};
\node (d) at (0,0) [] {$\bullet$};
\node (e) at (4,1) [] {$\bullet$};
\node (f) at (6,0) [] {$\bullet$};
\draw [->] (b) -- (a);
\draw [->] (b) -- (c);
\draw [->] (e) -- (f);
\draw [->] (d) -- (e);
\draw [->] (f) -- (d);
\draw [-z>] (e) -- (b);
\draw [z->] (c) -- (e);
\draw [z-->] (a) -- (d);
\draw [--z>] (f) -- (c);
\node at (b) [above] {$a$};
\node at (e) [below] {$b$};
}
\quad\quad \tikz{\draw [<->] (0,0) -- node[above]{$\mu_a$} (1,0);}\quad \quad
\tikz[baseline=6ex]{
\node (a) at (0,2) [] {$\bullet$};
\node (b) at (4,2) [] {$\bullet$};
\node (c) at (6,2) [] {$\bullet$};
\node (d) at (0,0) [] {$\bullet$};
\node (e) at (2,1) [] {$\bullet$};
\node (f) at (6,0) [] {$\bullet$};
\draw [->] (a) -- (b);
\draw [->] (c) -- (b);
\draw [->] (e) -- (f);
\draw [->] (d) -- (e);
\draw [->] (f) -- (d);
\draw [z->] (b) -- (e);
\draw [-z>] (e) -- (a);
\draw [z-->] (a) -- (d);
\draw [--z>] (f) -- (c);
\node at (b) [above] {$a$};
\node at (e) [below] {$b$};
}
\]

The case $C_{\alpha\beta}C_{\beta\alpha}=3$ can be done similarly. To save space, we will put the quasi-quiver demonstration of one possible sequence of 10 mutations corresponding move (2) in Appendix \ref{A}.
\end{proof}

The last two propositions together implies the following statement, which allows us to glue the seed tori $\mathcal{A}_\vec{i}$ and $\underline{\mathcal{X}}_\vec{i}$ into cluster varieties $\mathcal{A}^{u,v}$ and $\underline{\mathcal{X}}^{u,v}$.

\begin{cor} Any two seeds associated to reduced words of the same pair of Weyl group elements $(u,v)$ are mutation equivalent.
\end{cor}

After obtaining the cluster varieties $\mathcal{A}^{u,v}$ and $\underline{\mathcal{X}}^{u,v}$, the next natural question to ask is whether the following two diagrams commute, where the map $\mu$ is the cluster transformation induced by a sequence of move (1) and move (2) that transforms the reduced word $\vec{i}$ into the reduced word $\vec{i}'$ of the pair of Weyl group elements $(u,v)$.
\[
\xymatrix{& \mathcal{A}_\vec{i} \ar@{-->}[dd]^\mu \\ G_{sc} \ar@{-->}[ur]^{\psi_\vec{i}} \ar@{-->}[dr]_{\psi_{\vec{i}'}} & \\ & \mathcal{A}_{\vec{i}'}}\quad \quad \quad \quad \quad \quad 
\xymatrix{\underline{\mathcal{X}}_\vec{i} \ar[dr]^{\chi_\vec{i}} \ar@{-->}[dd]_\mu & \\ & H\backslash G/H  \\ \underline{\mathcal{X}}_{\vec{i}'} \ar[ur]_{\chi_{\vec{i}'}} & }
\]

\begin{prop}\label{birational} The two diagrams above commute, and therefore they pass to maps $\psi:G_{sc}\dashrightarrow\mathcal{A}^{u,v}$ and $\chi:\underline{\mathcal{X}}^{u,v}\rightarrow H\backslash G/H$. In particular $\psi$ restricts to a birational equivalence $\psi:G_{sc}^{u,v}\dashrightarrow \mathcal{A}^{u,v}$, and $\chi$ is a birational equivalence onto $H\backslash G^{u,v}/H$.
\end{prop}
\begin{proof} The commutativity of these two diagrams follows from a collection of identities corresponding move (1) and move (2), which have been laid out by Fomin and Zelevinsky on the $\mathcal{A}$ side \cite{FZ} and Fock and Goncharov on the $\mathcal{X}$ side \cite{FGamalgamation}. We will simply quote some of these identities from their papers (in particular we omit the formulas for the $C_{\alpha\beta}C_{\beta\alpha}=3$ case) in Appendix \ref{B}; readers who are interested in these formulas should go to the respective paper to find them.

The claim that $\psi$ restricts to a birational equivalence $\psi:G_{sc}^{u,v}\dashrightarrow \mathcal{A}^{u,v}$ is a reinterpretation of a result by Berenstein, Fomin, and Zelevinsky in \cite{BFZ}: in their paper they showed that the coordinate ring $\mathbb{C}[G_{sc}^{u,v}]$ is isomorphic to the upper cluster algebra generated by the seed data associated to any reduced word of the pair $(u,v)$, which is exactly the coordinate ring of our cluster variety $\mathcal{A}^{u,v}$.

The fact that the image of $\chi$ lies inside $H\backslash G^{u,v}/H$ follows from the fact that $e_{\pm \alpha}\in B_\pm \cap B_\mp s_\alpha B_\mp$ and $B_\pm uB_\pm vB_\pm=B_\pm uvB_\pm$ if $l(u)+l(v)=l(uv)$ (\cite{Hum} Section 29.3 Lemma A), and Fock and Goncharov gave a proof of the map $\chi:\underline{\mathcal{X}}^{u,v}\rightarrow H\backslash G^{u,v}/H$ being a birational equivalence in their paper \cite{FGamalgamation} by relating the cluster $\mathcal{X}$-coordinates to the Lusztig factorization coordinates.
\end{proof}

Now we finally obtained what we want at the beginning of this subsection, namely the reduced cluster ensemble $(\mathcal{A}^{u,v},\underline{\mathcal{X}}^{u,v},p)$ and maps
\[
\psi:G_{sc}^{u,v}\dashrightarrow \mathcal{A}^{u,v} \quad \quad \text{and} \quad \quad \chi:\underline{\mathcal{X}}^{u,v}\rightarrow H\backslash G^{u,v}/H.
\]
These structures enable us to rewrite the map $\iota\circ \tw:H\backslash G^{u,v}/H\rightarrow H\backslash G^{u,v}/H$ and equivalently $\xi:\conf^{u,v}(\mathcal{B})\rightarrow \conf^{u,v}(\mathcal{B})$ in a new way, as we will soon see in the next section.

Before we end this subsection, we should also show the equivalence between the map $\iota:H\backslash G^{u,v}/H\rightarrow H\backslash G^{u,v}/H$ induced by the anti-involution $\iota$ and the cluster theoretical map $i_\mathcal{X}$ as promised earlier in Remark \ref{promise}.

\begin{prop}\label{iota} Let $\vec{i}=\left(\alpha(1),\alpha(2),\dots, \alpha(l)\right)$ be a reduced word of a pair of Weyl group elements $(u,v)$. Then $\vec{i}^\circ=\left(\alpha(l),\dots, \alpha(2),\alpha(1)\right)$ is a reduced word of the pair $\left(u^{-1},v^{-1}\right)$, and the seeds constructed from these two words are opposite to each other in a way that is compatible with mutations (and hence the notation). Moreover, the following diagram commutes.
\[
\xymatrix{\underline{\mathcal{X}}^{u,v} \ar[r]^\chi \ar[d]_{i_{\mathcal{X}}} & H\backslash G^{u,v}/H \ar[d]^\iota \\
\underline{\mathcal{X}}^{u^{-1},v^{-1}} \ar[r]_(0.4){\chi} & H\backslash G^{u^{-1},v^{-1}}/H}
\]
\end{prop}
\begin{proof} It is obvious that $\left(\alpha(l),\dots, \alpha(2),\alpha(1)\right)$ is a reduced word for the pair $\left(u^{-1},v^{-1}\right)$. 

For the second claim, recall that the seed $\vec{i}$ is obtained from the reduced word $\vec{i}$ by amalgamating pieces corresponding to letters, which can be pictorially represented by a string diagram; if we reverse the order of the letters, we also reverse the string diagram, which induces a bijection between the vertices (strings) of the seeds (string diagrams), and gives every entry of the exchange matrix a minus sign. 

For the last claim, we just need to first verify it with Equation \eqref{chi^pm}:
\begin{align*}
\left(\iota\circ \chi^{\pm \alpha}\right)\left(X_a\right)=&\left(X_{\alpha_-}^{H^\alpha}e_{\pm\alpha}X_{\alpha_+}^{H^\alpha}\prod_{\beta\neq \alpha}X_\beta^{H^\beta}\right)^\iota\\
=&\left(\prod_{\beta\neq \alpha} \left(X_\beta^{-1}\right)^{H^\beta}\right)\left(X_{\alpha_-}^{-1}\right)^{H^\alpha}e_{\pm\alpha}\left(X_{\alpha_+}^{-1}\right)^{H^\alpha}\\
=&\chi^{\pm \alpha}\left(X_a^{-1}\right)\\
=& \left(\chi^{\pm \alpha}\circ i_\mathcal{X}\right)\left(X_a\right),
\end{align*}
and then note that by reversing the reduced word we have also reversed the order of multiplication when we use the group multiplication to define $\chi_\vec{i}$, which is exactly what the anti-involution $\iota$ does. 
\end{proof}

\begin{exmp}\label{exmp 2.69} Let's give an example of what the maps $\psi$ and $\chi$ are. Consider the Lie algebra $\mathfrak{sl}_3$. The simply connected Lie group associated to $\mathfrak{sl}_3$ is $\SL_3$ and the adjoint form associated to $\mathfrak{sl}_3$ is $\PGL_3$. Consider the pair of Weyl group elements $\left(w_0,w_0\right)$ for either group. Since the cluster varieties are obtained by gluing seed tori via birational maps, it suffices to describe $\psi$ and $\chi$ on one cluster coordinate chart (seed torus), i.e., $\psi_\vec{i}$ and $\chi_\vec{i}$ for some reduced word $\vec{i}$ of $\left(w_0,w_0\right)$. Let's consider the reduced word 
\[
\vec{i}=(-1,-2,-1,2,1,2)
\]
where $1$ and $2$ denote the two simple roots associated to $\mathfrak{sl}_3$. The corresponding string diagram looks like the following.
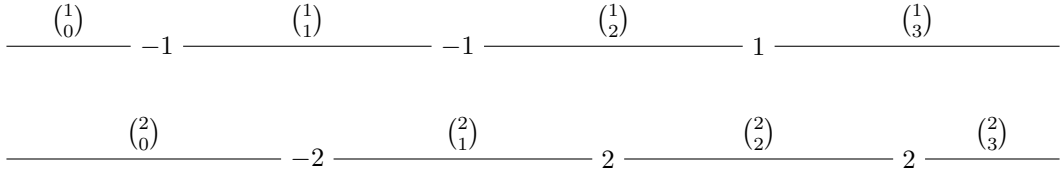
\[
\tikz{
\node (21) at (4,0) [] {$-2$};
\node (11) at (2,1.5) [] {$-1$};
\node (22) at (8,0) [] {$2$};
\node (12) at (6,1.5) [] {$-1$};
\node (23) at (12,0) [] {$2$};
\node (13) at (10,1.5) [] {$1$};
\draw (0,0) -- node [above] {$\binom{2}{0}$} (21) -- node [above] {$\binom{2}{1}$} (22) -- node [above] {$\binom{2}{2}$} (23) -- node [above] {$\binom{2}{3}$} (14,0);
\draw (0,1.5) -- node [above] {$\binom{1}{0}$} (11) -- node [above] {$\binom{1}{1}$} (12) -- node [above] {$\binom{1}{2}$} (13) -- node [above] {$\binom{1}{3}$} (14,1.5);
}
\]
Since $C_{12}C_{21}=1$ for $\mathfrak{sl}_3$, we can represent the data of the seed $\vec{i}$ using a quiver (we have colored the frozen vertices with gray).
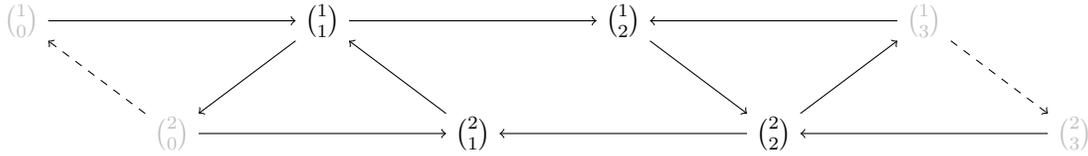
\[
\tikz{
\node (11) at (0,1.5) [lightgray] {$\binom{1}{0}$};
\node (12) at (4,1.5) [] {$\binom{1}{1}$};
\node (13) at (8,1.5) [] {$\binom{1}{2}$};
\node (14) at (12, 1.5) [lightgray] {$\binom{1}{3}$};
\node (21) at (2,0) [lightgray] {$\binom{2}{0}$};
\node (22) at (6,0) [] {$\binom{2}{1}$};
\node (23) at (10,0) [] {$\binom{2}{2}$};
\node (24) at (14,0) [lightgray] {$\binom{2}{3}$};
\draw [->] (11) -- (12);
\draw [->] (12) -- (13);
\draw [->] (14) -- (13);
\draw [->] (21) -- (22);
\draw [->] (23) -- (22);
\draw [->] (24) -- (23);
\draw [dashed, ->] (21) -- (11);
\draw [->] (12) -- (21);
\draw [->] (22) -- (12);
\draw [->] (13) -- (23);
\draw [->] (23) -- (14);
\draw [dashed, ->] (14) -- (24);
}
\]
From the amalgamation procedure we see that the cluster $\mathcal{A}$-coordinates associated to the vertices of the above quiver, which in this example are minors of an $\SL_3$ matrix, are the following (we use the notation $\Delta_{I,J}$ to denote the determinant of the submatrix with rows $I$ and columns $J$). 
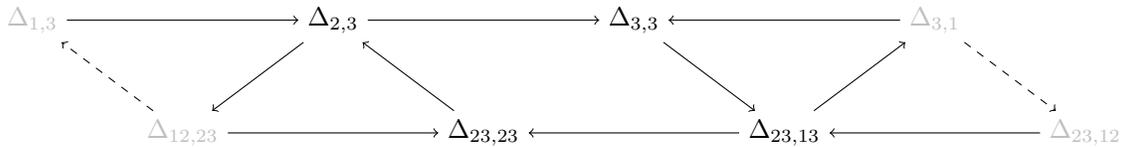
\[
\tikz{
\node (11) at (0,1.5) [lightgray] {$\Delta_{1,3}$};
\node (12) at (4,1.5) [] {$\Delta_{2,3}$};
\node (13) at (8,1.5) [] {$\Delta_{3,3}$};
\node (14) at (12, 1.5) [lightgray] {$\Delta_{3,1}$};
\node (21) at (2,0) [lightgray] {$\Delta_{12,23}$};
\node (22) at (6,0) [] {$\Delta_{23,23}$};
\node (23) at (10,0) [] {$\Delta_{23,13}$};
\node (24) at (14,0) [lightgray] {$\Delta_{23,12}$};
\draw [->] (11) -- (12);
\draw [->] (12) -- (13);
\draw [->] (14) -- (13);
\draw [->] (21) -- (22);
\draw [->] (23) -- (22);
\draw [->] (24) -- (23);
\draw [dashed, ->] (21) -- (11);
\draw [->] (12) -- (21);
\draw [->] (22) -- (12);
\draw [->] (13) -- (23);
\draw [->] (23) -- (14);
\draw [dashed, ->] (14) -- (24);
}
\]
Thus the map $\psi_\vec{i}:\SL_3^{w_0,w_0}\dashrightarrow \mathcal{A}_\vec{i}$ is given by
\[
\psi_\vec{i}(x)= \left(\Delta_{1,3}(x), \Delta_{2,3}(x), \Delta_{3,3}(x), \Delta_{3,1}(x), \Delta_{12,23}(x), \Delta_{23,23}(x), \Delta_{23,13}(x), \Delta_{23,12}(x)\right).
\]

To describe the map $\chi_\vec{i}:\underline{\mathcal{X}}^{u,v}\rightarrow H\backslash \PGL_3^{w_0,w_0}/H$, we first replace the nodes and closed strings in the string diagram with the corresponding matrix factor.
\[
\tikz{
\node (21) at (4,0) [] {$e_{-2}$};
\node (11) at (2,1.5) [] {$e_{-1}$};
\node (22) at (8,0) [] {$e_2$};
\node (12) at (6,1.5) [] {$e_{-1}$};
\node (23) at (12,0) [] {$e_2$};
\node (13) at (10,1.5) [] {$e_1$};
\draw (0,0) -- (21) -- node [above] {$X_3$} (22) -- node [above] {$X_4$} (23) -- (14,0);
\draw (0,1.5) -- (11) -- node [above] {$X_1$} (12) -- node [above] {$X_2$} (13) -- (14,1.5);
}
\]
Note that in the case of $\PGL_3$, 
\[
X^{H^1}=\begin{pmatrix} X & 0 & 0 \\ 0 & 1 & 0 \\ 0 & 0 & 1\end{pmatrix}, \quad e_1=\begin{pmatrix} 1 & 1 & 0 \\ 0 & 1 & 0 \\ 0 & 0 & 1\end{pmatrix}, \quad e_{-1}=\begin{pmatrix} 1 & 0 & 0 \\ 1 & 1 & 0 \\ 0 & 0 & 1\end{pmatrix},
\]
\[
\quad X^{H^2}=\begin{pmatrix} X & 0 & 0 \\ 0 & X & 0 \\ 0 & 0 & 1\end{pmatrix}, \quad e_2=\begin{pmatrix} 1 & 0 & 0 \\ 0 & 1 & 1 \\ 0 & 0 & 1\end{pmatrix}, \quad e_{-2}=\begin{pmatrix} 1 & 0 & 0 \\ 0 & 1 & 0 \\ 0 & 1 & 1\end{pmatrix}.
\]
Therefore to give a representative of the image of the map $\chi_\vec{i}$, all we need to do is multiply all the matrix factors from left to right across the string diagram. Note that due to Proposition \ref{commute}, any possible ambiguity in the order of multiplication does not affect the outcome. As a result, we have
\begin{align*} &\chi_\vec{i}\left(X_1, X_2, X_3, X_4\right)\\
=& H\left\backslash e_{-1}e_{-2} X_1^{H^1}X_3^{H^2} e_{-1}e_2 X_2^{H^2}X_4^{H^2}e_1e_2\right/H\\
=&H\left\backslash\begin{pmatrix} 1 & 0 & 0 \\ 1 & 1 & 0 \\ 0 & 0 & 1\end{pmatrix}\begin{pmatrix} 1 & 0 & 0 \\ 0 & 1 & 0 \\ 0 & 1 & 1\end{pmatrix}\begin{pmatrix} X_1 & 0 & 0 \\ 0 & 1 & 0 \\ 0 & 0 & 1\end{pmatrix}\begin{pmatrix} X_3 & 0 & 0 \\ 0 & X_3 & 0 \\ 0 & 0 & 1\end{pmatrix}\begin{pmatrix} 1 & 0 & 0 \\ 1 & 1 & 0 \\ 0 & 0 & 1\end{pmatrix}\right.\\
&\left.\begin{pmatrix} 1 & 0 & 0 \\ 0 & 1 & 1 \\ 0 & 0 & 1\end{pmatrix}\begin{pmatrix} X_2 & 0 & 0 \\ 0 & 1 & 0 \\ 0 & 0 & 1\end{pmatrix}\begin{pmatrix} X_4 & 0 & 0 \\ 0 & X_4 & 0 \\ 0 & 0 & 1\end{pmatrix}\begin{pmatrix} 1 & 1 & 0 \\ 0 & 1 & 0 \\ 0 & 0 & 1\end{pmatrix}\begin{pmatrix} 1 & 0 & 0 \\ 0 & 1 & 1 \\ 0 & 0 & 1\end{pmatrix}\right/H\\
=&H\left\backslash \begin{pmatrix} X_1X_2X_3X_4 & X_1X_2X_3X_4 & X_1X_2X_3X_4 \\
X_2X_3X_4\left(1+X_1\right) & X_3X_4\left(1+X_2\left(1+X_1\right)\right) & X_3\left(1+X_4\left(1+X_2\left(1+X_1\right)\right)\right) \\
X_2X_3X_4 & X_3X_4\left(1+X_2\right) & 1+X_3\left(1+X_4\left(1+X_2\right)\right) \end{pmatrix}\right/H
\end{align*}

\end{exmp}

\begin{notn} (\textbf{Important!}) Since from now on we will only focus on the reduced cluster ensemble $(\mathcal{A}^{u,v},\underline{\mathcal{X}}^{u,v},p)$, we will drop the underline notation and just write the cluster ensemble as $(\mathcal{A}^{u,v}, \mathcal{X}^{u,v},p)$.
\end{notn}

\section{Proof of Our Main Results}

The first two parts of our main theorem (Theorem \ref{mainthm}) are to prove the commutativity of following two diagrams. We will prove the commutative diagram on the right first, and then use the intertwining maps $i_\mathcal{X}$ and $\iota$ to deduce our main theorem.
\begin{equation}\label{commutative diagram}\vcenter{\vbox{
\xymatrix{\mathcal{X}^{u,v} \ar@{-->}[r]^\DT \ar[d]_\cong^\chi & \mathcal{X}^{u,v} \ar[d]^\cong_\chi \\
H\backslash G^{u,v}/H \ar[r]^{\tw\circ \iota } \ar[d]_\cong & H\backslash G^{u,v}/H \ar[d]^\cong \\
\conf^{u,v}(\mathcal{B})\ar[r]_\eta & \conf^{u,v}(\mathcal{B})}}}
\quad \quad \quad \quad 
\vcenter{\vbox{\xymatrix{\mathcal{X}^{u^{-1},v^{-1}} \ar@{-->}[r]^\Xi \ar[d]_\cong^\chi & \mathcal{X}^{u^{-1},v^{-1}} \ar[d]^\cong_\chi \\
H\backslash G^{u^{-1},v^{-1}}/H \ar[r]^{\iota\circ \tw } \ar[d]_\cong & H\backslash G^{u^{-1},v^{-1}}/H \ar[d]^\cong \\
\conf^{u^{-1},v^{-1}}(\mathcal{B})\ar[r]_\xi & \conf^{u^{-1},v^{-1}}(\mathcal{B})}}}
\end{equation}

The reason we do this is because the horizontal maps in the commutative diagram on the right actually admit a nice factorization that exhibits the interaction between cluster theory and geometry (which is the remaining part of our main theorem). To recall, we claim that the following diagram commutes, where the map $G_{sc}^{u^{-1},v^{-1}}\rightarrow H\backslash G^{u^{-1},v^{-1}}/H$ is just the natural projection map and $s$ is any section in such projection.
\[
\xymatrix{ & G_{sc}^{u^{-1},v^{-1}} \ar@{-->}[r]^\psi  \ar@<.5ex>[d] & \mathcal{A}^{u^{-1},v^{-1}}\ar[d]^p & \\
\mathcal{X}^{u^{-1},v^{-1}} \ar[r]^(0.4){\chi} \ar@{-->}@/_5ex/[rr]_\Xi & H\backslash G^{u^{-1},v^{-1}}/H \ar@<.5ex>[u]^s \ar@/_5ex/[rr]_{\iota \circ \tw} & \mathcal{X}^{u^{-1},v^{-1}} \ar[r]^(0.4){\chi} & H\backslash G^{u^{-1},v^{-1}}/H
}
\]

\begin{strg} To better convey the idea, we give an outline of the structure of the proof; there are in total four steps, and each step will be a subsection on its own.
\begin{enumerate}
    \item Prove that the composition $\xymatrix{H\backslash G^{u,v} /H \ar[r]^(0.6){s} & G_{sc}^{u,v} \ar@{-->}[r]^\psi & \mathcal{A}^{u,v} \ar[r]^p & \mathcal{X}^{u,v}  \ar[r]^(0.4){\chi} & H\backslash G^{u,v}/H}$ is birationally equivalent to the map $\iota \circ \tw$, and hence the composition is independent of the choice of $s$; as a result we also obtain a birational map $\psi:H\backslash G^{u,v}/H\dashrightarrow \mathcal{X}^{u,v}$ that is birationally equivalent to the composition $p\circ \psi\circ s$.
    \item Define birational automorphism $\xymatrix{\Xi:\mathcal{X}^{u,v}\ar[r]^(0.4){\chi} & H\backslash G^{u,v}/H \ar@{-->}[r]^(0.6){\psi} & \mathcal{X}^{u,v}}$ and prove that $\Xi$ has the property that $\deg_{X_a}\Xi^*\left(X_b\right)=-\delta_{ab}$.
    \item Prove that $\Xi$ is a cluster transformation.
    \item Apply the intertwining maps $i_\mathcal{X}$ and $\iota$ to relate the two commutative diagrams in \eqref{commutative diagram}.
\end{enumerate}
\end{strg}

Note that in steps (2) and (3) $\Xi$ is only defined as the composition $\psi\circ \chi$ and nothing else. Therefore one should not assume what we say about the $\Xi$ map in Subsection \ref{1.4}. Instead, one should view steps (2) and (3) together as the proof of the claim that the birational automorphism $\Xi:=\psi\circ \chi$ is indeed identical to the $\Xi$ map we have defined in Subsection \ref{1.4}.

\subsection{Factorization of the Twist \texorpdfstring{$\iota\circ \tw$}{}}\label{3.1}

In this subsection we will prove that $\iota\circ \tw$ is birationally equivalent to the composition $\chi\circ p \circ \psi \circ s$ (Proposition \ref{clustertwist}). The key to prove this proposition is the following two lemmas, equivalent versions of which have been proved by Fomin and Zelevinsky (\cite{FZ}, Theorem 1.9) and H. Williams (\cite{Wilkacmoody}, Proposition 3.23).

\begin{lem} If $x\in B_+\cap B_-vB_-$ and $(\alpha(1),\dots, \alpha(n))$ is a reduced word of $v$, then
\[
\left[x\overline{v^{-1}}\right]_-^t=\prod_{i=1}^n e_{\alpha(i)}(t_i)
\]
where
\[
t_i:=\frac{\prod_{\beta\neq \alpha(i)} \left(\Delta_\beta\left(x\overline{\left(v_{>i}\right)^{-1}}\right)\right)^{-C_{\beta\alpha(i)}}}{\Delta_{\alpha(i)}\left(x\overline{\left(v_{>i}\right)^{-1}}\right)\Delta_{\alpha(i)}\left(x\overline{\left(v_{>i-1}\right)^{-1}}\right)}.
\]
\end{lem}
\begin{proof} We will prove by induction on $n$. There is nothing to show when $n=0$. For $n>0$, we may assume without loss of generality that
\[
x=\left(\prod_\beta a_\beta^{H_\beta} \right)e_{\alpha(1)}(p_1)\dots e_{\alpha(n)}(p_n).
\]
Then it is obvious that for any simple root $\beta$,
\[
\Delta_\beta(x)=a_\beta.
\]
Next let's consider $x\overline{\left(v_{>n-1}\right)^{-1}}=x\overline{s}_{\alpha(n)}$. By using the identity
\[
e_\beta(p)\overline{s}_\beta=e_{-\beta}(p^{-1})p^{H_\beta}e_\beta(-p^{-1})
\]
we get
\[
x\overline{s}_{\alpha(n)}=\left(\prod_\beta a_\beta^{H_\beta} \right)e_{\alpha(1)}(p_1)\dots e_{\alpha(n-1)}(p_{n-1})e_{-\alpha(n)}(p_n^{-1})p_n^{H_{\alpha(n)}}e_{\alpha(n)}(-p_n^{-1}).
\]
Now our mission is to move the factor $e_{-\alpha(n)}(p_n^{-1})p_n^{H_{\alpha(n)}}$ all the way to the front so that we can use induction. The way to do this is to use the fact that $e_\beta(p)$ commutes with $e_{-\alpha}(q)$ whenever $\beta\neq \alpha$ plus the following identities (the first one is identity \eqref{e_+e_-} and the second one comes from the Lie algebra identity $[H_\alpha, E_\beta]=C_{\alpha\beta}E_\beta$):
\begin{align*}
e_\beta(q)e_{-\beta}(p)=&e_{-\beta}\left(\frac{p}{1+pq}\right)(1+pq)^{H_\beta} e_\beta\left(\frac{q}{1+pq}\right);\\
e_\beta(q)a^{H_\alpha}=&a^{H_\alpha}e_\beta\left(a^{-C_{\alpha\beta}}q\right).
\end{align*}
Combining these facts we see that
\[
e_\beta(q)e_{-\alpha(n)}(p^{-1})p^{H_{\alpha(n)}}=\left\{\begin{array}{ll}
    e_{-\alpha(n)}(p^{-1})p^{H_{\alpha(n)}}e_\beta(\cdots) & \text{if $\beta\neq \alpha(n)$;} \\
    e_{-\alpha(n)}\left(\frac{1}{p+q}\right)(p+q)^{H_{\alpha(n)}}e_\beta(\cdots) & \text{if $\beta=\alpha(n)$.} 
\end{array}\right.
\]
Using the above identity recursively, we get
\[
x\overline{s}_{\alpha(n)}=e_{-\alpha(n)}\left(\frac{\prod_\beta a^{-C_{\beta \alpha(n)}}}{\sum_{\alpha(i_k)=\alpha(n)} p_{i_k}}\right)\left(\sum_{\alpha(i_k)=\alpha(n)} p_{i_k}\right)^{H_{\alpha(n)}}\left(\prod_\beta a_\beta^{H_\beta}\right)e_{\alpha(1)}(\cdots)\dots e_{\alpha(n-1)}(\cdots)e_{\alpha(n)}(-p_n^{-1}).
\]
Thus it follows that
\[
\Delta_{\alpha(n)}\left(x\overline{s}_{\alpha(n)}\right)=a_{\alpha(n)}\sum_{\alpha(i_k)=\alpha(n)} p_{i_k}.
\]
Note that if we define $t_n:=\frac{\prod_\beta a_\beta^{-C_{\beta \alpha(n)}}}{\sum_{\alpha(i_k)=\alpha(n)} p_{i_k}}$, then it follows that
\[
t_n=\frac{\prod_{\beta\neq \alpha(n)} a_\beta^{-C_{\beta\alpha(n)}}}{a_{\alpha(n)}^2\sum_{\alpha(i_k)=\alpha(n)}p_{i_k}}=\frac{\prod_{\beta\neq \alpha(n)} \left(\Delta_\beta (x)\right)^{-C_{\beta\alpha(n)}}}{\Delta_{\alpha(n)}(x)\Delta_{\alpha(n)}\left(x\overline{s}_{\alpha(n)}\right)}.
\]

On the other hand, if we define $v':=vs_{\alpha(n)}$ and
\begin{align*}
x':=&e_{-\alpha(n)}\left(-t_n\right)x\overline{s}_{\alpha(n)}e_{\alpha(n)}(p_n^{-1})\\
=&\left(\sum_{\alpha(i_k)=\alpha(n)} p_{i_k}\right)^{H_{\alpha(n)}}\left(\prod_\beta a_\beta^{H_\beta}\right)e_{\alpha(1)}(\cdots)\dots e_{\alpha(n-1)}(\cdots),
\end{align*}
then it follows that $x'\in B_+\cap B_-v'B_-$ and
\[
\left[x\overline{v^{-1}}\right]_-=e_{-\alpha(n)}\left(t_n\right)\left[x'\overline{v'^{-1}}\right]_-.
\]
Note that $l(v')=l(v)-1$; hence we can use induction to finish the proof. The only remaining thing one needs to realize is that
\[
\Delta_\beta\left(x'\overline{\left(v'_{>l}\right)^{-1}}\right)=\Delta_\beta\left(e_{-\alpha(n)}(-t_n)x\overline{s}_{\alpha(n)}e_{\alpha(n)}(p_n^{-1})\overline{\left(v'_{>l}\right)^{-1}}\right)=\Delta_\beta\left(x\overline{\left(v_{>l}\right)^{-1}}\right).
\]
The last equality holds because $v'_{>l}(\alpha(n))$ is a positive root and hence 
\[
e_{\alpha(n)}(p_n^{-1})\overline{\left(v'_{>l}\right)^{-1}}=\overline{\left(v'_{>l}\right)^{-1}}n_+
\]
for some unipotent element $n_+\in N_+$.
\end{proof}

Analogously, one can also prove the following lemma.

\begin{lem} If $x\in B_+uB_+\cap B_-$ and $(\alpha(1),\dots, \alpha(m))$ is a reduced word of $u$, then
\[
\left[\overline{u}^{-1}x\right]_+^t=\prod_{i=1}^m e_{-\alpha(i)}(t_i)
\]
where
\[
t_i:=\frac{\prod_{\beta\neq \alpha(i)} \left(\Delta_\beta\left(\overline{u_{< i}}^{-1}x\right)\right)^{-C_{\beta\alpha(i)}}}{\Delta_{\alpha(i)}\left(\overline{u_{< i}}^{-1}x\right)\Delta_{\alpha(i)}\left(\overline{u_{< i+1}}^{-1}x\right)}.
\]
\end{lem}

\begin{prop}\label{clustertwist} The composition $\chi\circ p\circ \psi \circ s$ is birationally equivalent to the composition $\iota\circ \tw$. In particular, the former composition does not depend on the choice of lift.
\end{prop}
\begin{proof} Since every seed torus is a Zariski open subset of the corresponding cluster variety and we only care about rational equivalence, we can reduced Proposition \ref{clustertwist} to proving that the twist map is rationally equivalent to the composition 
\[
\xymatrix{G_{sc}^{u,v} \ar@{-->}[r]^{\psi_\vec{i}} \ar@<.5ex>[d]  & \mathcal{A}_\vec{i} \ar[d]^{p_\vec{i}}  & \\
H\backslash G^{u,v}/H  \ar@<.5ex>[u]^s & \mathcal{X}_\vec{i} \ar[r]^(0.3){\chi_\vec{i}} & H\backslash G^{u,v}/H}
\]
for some nice reduced word $\vec{i}$ of the pair of Weyl group elements $(u,v)$. Our choice of reduced word for this proof is any reduced word $\vec{i}:=(\alpha(1),\dots, \alpha(n), \alpha(n+1),\dots, \alpha(l))$ satifying the fact that $(\alpha(1),\dots, \alpha(n))$ is a reduced word of $v$ (so all the letters $\alpha(1),\dots, \alpha(n)$ are simple roots) and $-(\alpha(n+1),\dots, \alpha(l))$ is a reduced word of $u$ (so all the letters $\alpha(n+1),\dots, \alpha(l)$ are opposite to simple roots). In other words, our reduced word $\vec{i}$ can be broken down into two parts: the $v$ part and the $u$ part, as depicted below
\[
\vec{i}:=(\underbrace{\alpha(1),\dots, \alpha(n)}_\text{$v$ part}, \underbrace{\alpha(n+1),\dots, \alpha(l)}_\text{$u$ part}).
\]

Let's ignore the lifting at the very first step for now. Suppose we start with an element $x\in G_{sc}^{u,v}$. As the space of Gaussian decomposable elements $G_0$ is dense in $G_{sc}^{u,v}$ (see \cite{FZ} Proposition 2.14), we may further assume that $x$ is Gaussian decomposable, i.e., $x=[x]_-[x]_0[x]_+$. In particular, from the assumption that $x\in G_{sc}^{u,v}$ we know that $[x]_-[x]_0\in B_+u B_+\cap B_-$ and $[x]_0[x]_+\in B_+\cap B_-vB_-$. Now the composition $\iota\circ \tw$ maps $H\backslash x/H$ to
\begin{align*}
&H\left\backslash \left(\left[\overline{u}^{-1}x\right]_-^{-1}\overline{u}^{-1}x\overline{v^{-1}}\left[x\overline{v^{-1}}\right]_+^{-1}\right)^t\right/H\\
=& H\left\backslash\left(\left[\overline{u}^{-1}[x]_-[x]_0\right]_-^{-1}\overline{u}^{-1}[x]_-[x]_0[x]_+\overline{v^{-1}}\left[[x]_0[x]_+\overline{v^{-1}}\right]_+^{-1}\right)^t\right/H\\
=&H\left\backslash \left(\left[\overline{u}^{-1}[x]_-[x]_0\right]_+[x]_0^{-1}\left[[x]_0[x]_+\overline{v^{-1}}\right]_-\right)^t\right/H\\
=& H\left\backslash \left(\left[[x]_0[x]_+\overline{v^{-1}}\right]_-^t [x]_0^{-1} \left[\overline{u}^{-1}[x]_-[x]_0\right]_+^t\right)\right/H.
\end{align*}

By applying the two lemmas we had above, we see that the element in the middle can be rewritten as
\[
\left(\prod_{i=1}^n e_{\alpha(i)}(t_i)\right)[x]_0^{-1}\left(\prod_{i=n+1}^l e_{\alpha(i)}(t_i)\right)
\]
where
\[
t_i:=\left\{\begin{array}{ll}
    \displaystyle\frac{\prod_{\beta\neq \alpha(i)} \left(\Delta_\beta\left([x]_0[x]_+\overline{\left(v_{>i}\right)^{-1}}\right)\right)^{-C_{\beta\alpha(i)}}}{\Delta_{\alpha(i)}\left([x]_0[x]_+\overline{\left(v_{>i}\right)^{-1}}\right)\Delta_{\alpha(i)}\left([x]_0[x]_+\overline{\left(v_{>i-1}\right)^{-1}}\right)} & \text{if $1\leq i\leq n$}; \\
   \displaystyle\frac{\prod_{\beta\neq \alpha(i)} \left(\Delta_\beta\left(\overline{u_{< i}}^{-1}[x]_-[x]_0\right)\right)^{-C_{\beta\alpha(i)}}}{\Delta_{\alpha(i)}\left(\overline{u_{< i}}^{-1}[x]_-[x]_0\right)\Delta_{\alpha(i)}\left(\overline{u_{< i+1}}^{-1}[x]_-[x]_0\right)}  & \text{if $n+1\leq i\leq l$}.
\end{array}\right.
\]
But then since $\Delta_\gamma \left([x]_0[x]_+\overline{\left(v_{>i}\right)^{-1}}\right)=\Delta_\beta \left(x\overline{\left(v_{>i}\right)^{-1}}\right)$ and $\Delta_\gamma\left(\overline{u_{< i}}^{-1}[x]_-[x]_0\right)=\Delta_\gamma\left(\overline{u_{< i}}^{-1}x\right)$ for any $\gamma$, we can rewrite the $t_i$'s as
\[
t_i:=\left\{\begin{array}{ll}
    \displaystyle\frac{\prod_{\beta\neq \alpha(i)} \left(\Delta_\beta\left(x\overline{\left(v_{>i}\right)^{-1}}\right)\right)^{-C_{\beta\alpha(i)}}}{\Delta_{\alpha(i)}\left(x\overline{\left(v_{>i}\right)^{-1}}\right)\Delta_{\alpha(i)}\left(x\overline{\left(v_{>i-1}\right)^{-1}}\right)} & \text{if $1\leq i\leq n$}; \\
   \displaystyle\frac{\prod_{\beta\neq \alpha(i)} \left(\Delta_\beta\left(\overline{u_{< i}}^{-1}x\right)\right)^{-C_{\beta\alpha(i)}}}{\Delta_{\alpha(i)}\left(\overline{u_{< i}}^{-1}x\right)\Delta_{\alpha(i)}\left(\overline{u_{< i+1}}^{-1}x\right)}  & \text{if $n+1\leq i\leq l$}.
\end{array}\right.
\]
Note that every generalized minor factor present in the above expression is a cluster coordinate of the seed torus $\mathcal{A}_\vec{i}$!

To put things into the right places, we need the following additional identities:
\[
e_\alpha(t)=t^{H^\alpha}e_\alpha t^{-H^\alpha}, \quad \quad \quad \quad e_{-\alpha}(t)=t^{-H^\alpha}e_{-\alpha}t^{H^\alpha}.
\]
Note that we have secretly moved from the $G_{sc}$ territory into the $G_{ad}$ territory: the right hand side of either identities above obviously lives in $G_{ad}$. This is okay because what we care at the end is the image of $\iota\circ \tw$ in the double quotient $H\backslash G^{u,v}/H$, and it's completely fine to project $\left(\prod_{i=1}^n e_{\alpha(i)}(t_i)\right)[x]_0^{-1}\left(\prod_{i=n+1}^l e_{\alpha(i)}(t_i)\right)$ into $G_{ad}$ before taking the double quotient. 

Now for each string $a$ between two nodes $e_{\alpha(i)}$ and $e_{\alpha(j)}$ with $i\leq j$ and $\alpha(i)=\alpha(j)=\alpha$ for a simple root $\alpha$, we see that there are two factors $e_{\alpha(i)}\left(t_i\right)=t_i^{H^\alpha}e_\alpha t_i^{-H^\alpha}$ and $e_{\alpha(j)}\left(t_j\right)=t_j^{H^\alpha}e_\alpha t_j^{-H^\alpha}$ in the factorization $\left(\prod_{k=1}^n e_{\alpha(k)}(t_k)\right)[x]_0^{-1}\left(\prod_{k=n+1}^l e_{\alpha(k)}(t_k)\right)$; but then since $t_i^{-H^\alpha}$ and $t_j^{H^\alpha}$ commute with every factor between in the factorization $\left(\prod_{k=1}^n e_{\alpha(k)}(t_i)\right)[x]_0^{-1}\left(\prod_{k=n+1}^l e_{\alpha(k)}(t_i)\right)$, we can combine them into one single factor as $\left(t_i^{-1}t_j\right)^{H^\alpha}$, and we see that
\begin{align*}
t_i^{-1}t_j=&\displaystyle\frac{\Delta_{\alpha}\left(x\overline{\left(v_{>i}\right)^{-1}}\right)\Delta_{\alpha}\left(x\overline{\left(v_{>i-1}\right)^{-1}}\right)}{\prod_{\beta\neq \alpha} \left(\Delta_\beta\left(x\overline{\left(v_{>i}\right)^{-1}}\right)\right)^{-C_{\beta\alpha}}}\frac{\prod_{\beta\neq \alpha} \left(\Delta_\beta\left(x\overline{\left(v_{>j}\right)^{-1}}\right)\right)^{-C_{\beta\alpha}}}{\Delta_{\alpha}\left(x\overline{\left(v_{>j}\right)^{-1}}\right)\Delta_{\alpha}\left(x\overline{\left(v_{>j-1}\right)^{-1}}\right)}\\
=&\displaystyle\frac{\Delta_{\alpha}\left(x\overline{\left(v_{>i-1}\right)^{-1}}\right)\prod_{\beta\neq \alpha} \left(\Delta_\beta\left(x\overline{\left(v_{>j}\right)^{-1}}\right)\right)^{-C_{\beta\alpha}}}{\Delta_{\alpha}\left(x\overline{\left(v_{>j}\right)^{-1}}\right)\prod_{\beta\neq \alpha} \left(\Delta_\beta\left(x\overline{\left(v_{>i}\right)^{-1}}\right)\right)^{-C_{\beta\alpha}}}.
\end{align*}
Note that all minors present in the final expression above are cluster coordinates of the seed torus $\mathcal{A}_\vec{i}$, and the quotient we are taking is exactly the same as $p^*\left(X_a\right)$ for the string $a$ between $\alpha(i)$ and $\alpha(j)$. This implies that the $v$ part of the factorization $\left(\prod_{k=1}^n e_{\alpha(k)}(t_k)\right)[x]_0^{-1}\left(\prod_{k=n+1}^l e_{\alpha(k)}(t_k)\right)$ can be rewritten as an amalgamation product. A similar argument can be applied to the middle part and the $u$ part as well. This allows us to conclude that for any $x\in G_{sc}^{u,v}$,
\[
H\left\backslash \left(\left[\overline{u}^{-1}x\right]_-^{-1}\overline{u}^{-1}x\overline{v^{-1}}\left[x\overline{v^{-1}}\right]_+^{-1}\right)^t\right/H=\chi_\vec{i}\circ p_\vec{i}\circ \psi_\vec{i}(x).
\]

To finish the proof of the proposition, we also need to show that the image does not depend on the lift from $H\backslash G^{u,v}/H$ to $G_{sc}^{u,v}$ in the first place. Suppose instead of $x$, we start with $hx$ for some element $h\in H$. Then
\[
H\left\backslash\left(\left[\overline{u}^{-1}hx\right]_-^{-1}\overline{u}^{-1}hx\overline{v^{-1}}\left[hx\overline{v^{-1}}\right]_+^{-1}\right)^t\right/H=H\left\backslash\left(\left[\overline{u}^{-1}hx\right]_+ \overline{v^{-1}}\left[hx\overline{v^{-1}}\right]_+^{-1}\right)^t\right/H.
\]
But by the definition of the Weyl group $W:=N_GH/H$, we know that $\overline{u}^{-1}$ normalizes $H$; therefore $\overline{u}^{-1}h=h'\overline{u}^{-1}$ for some $h'\in H$ and we can simplify the above equality to
\begin{align*}
H\left\backslash\left(\left[\overline{u}^{-1}hx\right]_-^{-1}\overline{u}^{-1}hx\overline{v^{-1}}\left[hx\overline{v^{-1}}\right]_+^{-1}\right)^t\right/H=&H\left\backslash\left(\left[\overline{u}^{-1}x\right]_+ \overline{v^{-1}}\left[x\overline{v^{-1}}\right]_+^{-1}\right)^t\right/H\\
=&H\left\backslash \left(\left[\overline{u}^{-1}x\right]_-^{-1}\overline{u}^{-1}x\overline{v^{-1}}\left[x\overline{v^{-1}}\right]_+^{-1}\right)^t\right/H.
\end{align*}
A similar argument can also be applied to the replacement of $x$ by $xh$ for some $h\in H$. This finishes our proof of the proposition.
\end{proof}

\begin{cor} There exists a birational equivalence $\psi:H\backslash G^{u,v}/H\dashrightarrow \mathcal{X}^{u,v}$ that makes the following diagram commute.
\[
\xymatrix{G_{sc}^{u,v} \ar@{-->}[r]^\psi \ar@<.5ex>[d]  & \mathcal{A}^{u,v} \ar[d]^{p} \\
H\backslash G^{u,v}/H  \ar@<.5ex>[u]^s \ar@{-->}[r]_\psi & \mathcal{X}^{u,v}}
\]
\end{cor}
\begin{proof} We have shown in last proposition that $\iota\circ \tw$ is birationally equivalent to the composition $\chi\circ p \circ \psi \circ s$. But then since $\chi$ is also a birational equivalence (Proposition \ref{birational}), we can find an inverse birational equivalence $\chi^{-1}$ and define $\psi:=\chi^{-1}\circ \iota\circ \tw$. But then since both $\tw$ and $\iota$ are biregular morphisms, it follows that $\psi$ itself is a birational equivalence.
\end{proof}

\subsection{Verification of the Degree Condition on \texorpdfstring{$\Xi_\vec{i}$}{}}\label{3.2}

In addition to the birational equivalence $\chi:\mathcal{X}^{u,v}\dashrightarrow H\backslash G^{u,v}/H$, we obtained another birational equivalence $\psi:H\backslash G^{u,v}/H\dashrightarrow \mathcal{X}^{u,v}$ at the end of last subsection. We can compose them and get a birational automorphism.

\begin{defn}\label{defnxi} We define $\Xi:=\psi\circ \chi$ as a birational automorphism on $\mathcal{X}^{u,v}$. 
\end{defn}

Note that at this stage $\Xi$ is only given as a birational automorphism on $\mathcal{X}^{u,v}$, and it is not obvious that it has anything to do with the cluster transformation $\Xi$ we introduced in Subsection \ref{1.4}; but recall from Proposition \ref{uniquexi} that the map $\Xi$ introduced in Subsection \ref{1.4} is uniquely characterized by the following two properties:
\begin{itemize}
    \item $\deg_{X_a}\Xi^*\left(X_b\right)=-\delta_{ab}$ for some cluster coordinate chart;
    \item $\Xi$ is a cluster transformation.
\end{itemize}
Verifying these two properties allows us to identify the map $\Xi$ in Definition \ref{defnxi} with the one we introduced in Subsection \ref{1.4}. We will verify the former one in this subsection and verify the latter one in the next.

Following our subscript convention, we define $\Xi_\vec{i}:=\psi_\vec{i}\circ \chi_\vec{i}$ to be the restriction of $\Xi$ to the cluster coordinate chart $\mathcal{X}_\vec{i}$ for any reduced word $\vec{i}$. So now the question is which reduced word $\vec{i}$ we should use. In contrast to our choice in the last subsection, this time our choice of reduced word has the $u$-part coming before the $v$-part:
\[
\vec{i}=(\underbrace{\alpha(1),\dots, \alpha(n)}_\text{$u$-part}, \underbrace{\alpha(n+1),\dots, \alpha(l)}_\text{$v$-part}),
\]
in other words, the first $n$ letters are opposite to simple roots, which will give a reduced word of $u$ if you switch their signs, and the last $l-n$ letters are simple roots, which gives a reduced word of $v$. 

Fix one such reduced word $\vec{i}$. Let $(X_a)$ be a generic point in the seed torus $\mathcal{X}_\vec{i}$ and let $x$ be an element in $G_{sc}^{u,v}$ such that 
\[
H\backslash x/H=\chi_\vec{i}(X_a).
\]
To compute $\Xi_\vec{i}(X_a)$, we need to first compute the image $\psi_\vec{i}(x)$, whose coordinates are generalized minors of the form
\[
\Delta_\alpha\left(\overline{u_{<k}}^{-1}x\overline{v^{-1}}\right) \quad \quad \text{or} \quad \quad \Delta_\alpha\left(\overline{u}^{-1}x\overline{\left(v_{>k}\right)^{-1}}\right),
\]
depending the corresponding vertex in the seed.

The good news is that we sort of already know the Gaussian decomposition of $x$: from the definition $x:=\chi_\vec{i}(x)$ we know that we can write $x$ as a product of $e_{\alpha(i)}$ and $X_a^{H^\alpha}$, and since the $u$-part of $\vec{i}$ comes before the $v$-part, the first half of such product is in $B_-$ while the second half is in $B_+$. But then since $N_\pm$ is normal in $B_\pm$, we know that for any $n_\pm \in N_\pm$ and $h\in H$, there always exists $n'_\pm$ such that
\[
hn_-=n_-'h \quad \quad \text{and} \quad \quad n_+h=hn'_+.
\]
Using these two identities we can move all the $X_{\binom{\alpha}{k}}^{H^\alpha}$ factors to the middle and get
\[
\Delta_\alpha(x)=\Delta_\alpha\left(\prod_{\beta,i} X_{\binom{\beta}{i}}^{H^\beta}\right)=\prod_{\beta,i} X_{\binom{\beta}{i}}^{\inprod{H^\beta}{\omega_\alpha}}=\prod_{\beta,i} X_{\binom{\beta}{i}}^{\left(C^{-1}\right)_{\beta\alpha}},
\]
where $C^{-1}$ denotes the inverse of the Cartan matrix $C$.

A couple of the bad things happen here. One is that although the Cartan matrix $C$ has integer entries, it is generally not true for $C^{-1}$ to have integer entries, and we may run into trouble of taking fraction power of an algebraic variable. However, this may not be as bad as we imagine after all, because at the end of the day what we need is not individual generalized minors but some of their ratios, and hopefully these ratios are algebraic over the cluster variables. The other bad thing is that $\Delta_\alpha(x)$ is not even on the list of the generalized minors we are trying to compute!

Was our effort in vain? Of course not. We can modify it slightly to compute the generalized minors we actually need. The key is again the identities \eqref{e_+s} and \eqref{se_-}, which we will write down again:
\[
e_\alpha(t)\overline{s}_\alpha=e_{-\alpha}\left(t^{-1}\right)t^{H_\alpha} e_\alpha\left(-t^{-1}\right); \quad \quad  \quad \quad \overline{s}_\alpha^{-1}e_{-\alpha}(t)=e_{-\alpha}\left(-t^{-1}\right)t^{H_\alpha} e_\alpha\left(t^{-1}\right).
\]
Note that since we are free to take any lift from $H\backslash x/H:=\chi_\vec{i}\left(X_a\right)$ to a group element $x$, we may assume that $x$ admits a factorization that fits into the Gaussian decomposition pattern $N_-HN_+$, which begins with $e_{\alpha(1)}$ on the left and ends with $e_{\alpha(l)}$ on the right. Now if we were to compute minors of $x\overline{s}_{\alpha(l)}$, we can use identity \eqref{e_+s} to get
\[
x\overline{s}_{\alpha(l)}=\underbrace{e_{\alpha(1)}\dots e_{\alpha(l)}}_{\text{factorization of $x$}}\overline{s}_{\alpha(l)}=e_{\alpha(1)}\dots e_{-\alpha(l)}e_{\alpha(l)}^{-1}.
\]
But then since we begin with a reduced word $\vec{i}$ whose $u$-part comes before its $v$-part, the only factor in the above factorization of $x\overline{s}_{\alpha(l)}$ that stops it from being a Gaussian decomposition is $e_{-\alpha(l)}$. Therefore if we were to compute $\Delta_\alpha \left(x\overline{s}_\alpha\right)$, we just need to move the factor $e_{-\alpha(l)}$ all the way through the $HN_+$-part of the product $\chi_\vec{i}(X_a)$ into the $N_-$-part by using the identity \eqref{e_+e_-}
\[
e_\alpha(p)e_{-\alpha}(q)=e_{-\alpha}\left(\frac{q}{1+pq}\right)(1+pq)^{H_\alpha}e_\alpha\left(\frac{p}{1+pq}\right),
\]
and then pick up whatever is left in the $H$-part and compute the generalized minors. Variation of this observation is also true when we multiply $\overline{s}_{\alpha(l-1)}, \overline{s}_{\alpha(l-2)}, \dots$ on the right and $\overline{s}_{\alpha(1)}^{-1}, \overline{s}_{\alpha(2)}^{-1}, \dots$ on the left of $x$.

But things can still get very messy as we keep applying these identities. Fortunately we don't actually need the explicit expression of $\Xi_\vec{i}^*\left(X_a\right)$ but just to verify the identity $\deg_{X_b}\Xi_\vec{i}^*\left(X_a\right)=-\delta_{ab}$. This allows us to go around the difficulty by only considering the leading power of each variable.

\begin{notn} To better record the data of the leading power of each cluster $\mathcal{X}$-variable, we introduce the following notations replacing the strings and the nodes in the string diagram, each of which carries a monomial inside and represents an element in $G_{ad}$:
\begin{align*}
\text{$\tikz[baseline=0ex]{\draw (-0.75,-0.75) -- (0.75,-0.75) -- (0,0.75) -- cycle;
\node[scale=0.8] at (0,-0.25) [below] {$\textstyle\prod_aX_a^{p_a}$};}$ at level $\alpha$} \quad =& \quad e_\alpha\left(c\prod_a X_a^{p_a}+\text{terms of lower powers}\right),\\
\text{$\tikz[baseline=0ex]{\draw (-0.75,0.75) -- (0.75,0.75) -- (0,-0.75) -- cycle;
\node[scale=0.8] at (0,0.25) [above] {$\textstyle\prod_aX_a^{p_a}$};}$ at level $\alpha$} \quad =& \quad e_{-\alpha}\left(c\prod_a X_a^{p_a}+\text{terms of lower powers}\right),\\
\text{$\tikz[baseline=0ex]{\draw (0,0) circle [radius=0.75];
\node[scale=0.8] at (0,0) [] {$\textstyle\prod_aX_a^{p_a}$};}$ at level $\alpha$} \quad =& \quad \left(c\prod_a X_a^{p_a}+\text{terms of lower powers}\right)^{H^\alpha}.
\end{align*}
We impose the convention that an empty figure means the monomial inside is 1. In particular, we don't draw empty circles since it is just the identity element.
\end{notn}

\begin{defn} We say a picture representing the leading degrees of a variable $X$ is a \textit{final picture} if there is only one circle on each horizontal level, and all triangles of the shape $\tikz[baseline=0ex,scale=0.5]{\draw (-0.75,0.75) -- (0.75,0.75) -- (0,-0.75) -- cycle;}$ are before the circle on that horizontal level and all triangles of the shape $\tikz[baseline=0ex, scale=0.5]{\draw (-0.75,-0.75) -- (0.75,-0.75) -- (0,0.75) -- cycle;}$ are after the circle on that horizontal level.
\end{defn}

\begin{exmp} Consider a rank 2 root system with simple roots $\alpha$ and $\beta$. Then the image of the amalgamation map $\chi_\vec{i}:\mathcal{X}_\vec{i}\rightarrow H\backslash G^{u,v}/H$ corresponding to the reduced word $\vec{i}:=(-\alpha,-\beta,\alpha,\beta)$ can be represented by our notation above as
\[
\tikz[scale=0.8]{
\node at (0,2) [] {$\alpha$};
\node at (0,0) [] {$\beta$};
\draw (1.25,2.75) -- (2.75,2.75) -- (2,1.25) -- cycle;
\draw (4,2) circle [radius=0.75];
\node at (4,2) [] {$X_{\binom{\alpha}{1}}$};
\draw (5.25,1.25) -- (6.75,1.25) -- (6,2.75) -- cycle;
\draw (3.25,0.75) -- (4.75,0.75) -- (4,-0.75) -- cycle;
\draw (6,0) circle [radius=0.75];
\node at (6,0) [] {$X_{\binom{\beta}{1}}$};
\draw (7.25,-0.75) -- (8.75,-0.75) -- (8,0.75) -- cycle;
}
\]
\end{exmp}

Using such notation, we see that to compute the leading powers of cluster $\mathcal{X}$-variables in certain generalized minors, all we need to do is to multiply the corresponding lifts of Weyl group elements on the two sides, and then move figures around so that all triangles of the form $\tikz[baseline=0ex,scale=0.5]{\draw (-0.75,0.75) -- (0.75,0.75) -- (0,-0.75) -- cycle;}$ are on the left of circles and all triangles of the form $\tikz[baseline=0ex, scale=0.5]{\draw (-0.75,-0.75) -- (0.75,-0.75) -- (0,0.75) -- cycle;}$ are on the right of circles, and whatever monomials are left in the circles in the middle at the very end will give us the leading powers of the cluster $\mathcal{X}$-variables.

This may sound more difficult than it really is; in fact, there are many identities we can use to help us achieve the goal.

\begin{prop}\label{figure} The following identities hold (for notation simplicity we only include one cluster $\mathcal{X}$-variable $X$ here).
\begin{enumerate}
    \item Neighboring circles on the same level can be merged into one single circle with the respective monomials multiplied. Different figures on different levels commute with each other. Circles on different levels also commute with each other.
    \vspace{0.5cm}
    \item $\tikz[baseline=0ex, scale=0.7]{
    \draw (-0.75,-0.75) -- (0.75,-0.75) -- (0,0.75) -- cycle;
    \node[scale=0.8] at (0,0) [below] {$X^p$};
    \draw (2,0) circle [radius=0.75];
    \node[scale=0.8] at (2,0) [] {$X^d$};
    }  =  \tikz[baseline=0ex, scale=0.7]{
    \draw (-0.75,-0.75) -- (0.75,-0.75) -- (0,0.75) -- cycle;
    \node[scale=0.8] at (0,0) [below] {$X^{p-d}$};
    \draw (-2,0) circle [radius=0.75];
    \node[scale=0.8] at (-2,0) [] {$X^d$};
    }$.
    \vspace{0.5cm}
    \item $\tikz[baseline=0ex, scale=0.7]{
    \draw (-0.75,0.75) -- (0.75,0.75) -- (0,-0.75) -- cycle;
    \node[scale=0.8] at (0,0) [above] {$X^p$};
    \draw (-2,0) circle [radius=0.75];
    \node[scale=0.8] at (-2,0) [] {$X^d$};
    }  = \tikz[baseline=0ex, scale=0.7]{
    \draw (-0.75,0.75) -- (0.75,0.75) -- (0,-0.75) -- cycle;
    \node[scale=0.8] at (0,0) [above] {$X^{p-d}$};
    \draw (2,0) circle [radius=0.75];
    \node[scale=0.8] at (2,0) [] {$X^d$};
    }$.
    \vspace{0.5cm}
    \item Recall the notation $[n]_+:=\max\{0,n\}$; then 
    
    \vspace{0.5cm}
    $\tikz[baseline=0ex]{
    \draw (-0.75,0.75) -- (0.75,0.75) -- (0,-0.75) -- cycle;
    \node[scale=0.8] at (0,0) [above] {$X^q$};
    \draw (-2,-0.75) -- (-0.5,-0.75) -- (-1.25,0.75) -- cycle;
    \node[scale=0.8] at (-1.25,0) [below] {$X^p$};
    }  =  \tikz[baseline=0ex]{
    \draw (-0.75,0.75) -- (0.75,0.75) -- (0,-0.75) -- cycle;
    \node[scale=0.7] at (0,0.25) [above] {$X^{q-\left[p+q\right]_+}$};
    \draw (2,0) circle [radius=0.75];
    \node[scale=0.8] at (2,0) [] {$X^{2\left[p+q\right]_+}$};
    \draw (3.25,-0.75) -- (4.75,-0.75) -- (4,0.75) -- cycle;
    \node [scale=0.7] at (4,-0.25) [below] {$X^{p-\left[p+q\right]_+}$};
    \draw (2,-2) circle [radius=0.75];
    \node [scale=0.8] at (2,-2) [] {$X^{C_{\alpha\beta}\left[p+q\right]_+}$};
    \node at (5.5,0) [] {$\alpha$};
    \node at (5.5,-2) [] {$\beta$};
    }$.
    \vspace{0.5cm}
    \item $\tikz[baseline=0ex, scale=0.7]{
    \draw (-0.75,-0.75) -- (0.75,-0.75) -- (0,0.75) -- cycle;
    \node[scale=0.8] at (0,0) [below] {$X^p$};}\overline{s}_\alpha =  \tikz[baseline=0ex, scale=0.7]{
    \draw (-0.75,0.75) -- (0.75,0.75) -- (0,-0.75) -- cycle;
    \node[scale=0.8] at (0,0) [above] {$X^{-p}$};
    \draw (2,0) circle [radius=0.75];
    \node[scale=0.8] at (2,0) [] {$X^{2p}$};
    \draw (3.25,-0.75) -- (4.75,-0.75) -- (4,0.75) -- cycle;
    \node[scale=0.8] at (4,0) [below] {$X^{-p}$};
    \draw (2,-2) circle [radius=0.75];
    \node[scale=0.8] at (2,-2) [] {$X^{pC_{\alpha\beta}}$};
    \node at (6,0) [] {$\alpha$};
    \node at (6,-2) [] {$\beta.$};
    }$
    \vspace{0.5cm}
    \item $\overline{s}_\alpha^{-1}\tikz[baseline=0ex, scale=0.7]{
    \draw (-0.75,0.75) -- (0.75,0.75) -- (0,-0.75) -- cycle;
    \node[scale=0.8] at (0,0) [above] {$X^p$};}=  \tikz[baseline=0ex, scale=0.7]{
    \draw (-0.75,0.75) -- (0.75,0.75) -- (0,-0.75) -- cycle;
    \node[scale=0.8] at (0,0) [above] {$X^{-p}$};
    \draw (2,0) circle [radius=0.75];
    \node[scale=0.8] at (2,0) [] {$X^{2p}$};
    \draw (3.25,-0.75) -- (4.75,-0.75) -- (4,0.75) -- cycle;
    \node[scale=0.8] at (4,0) [below] {$X^{-p}$};
    \draw (2,-2) circle [radius=0.75];
    \node[scale=0.8] at (2,-2) [] {$X^{pC_{\alpha\beta}}$};
    \node at (6,0) [] {$\alpha$};
    \node at (6,-2) [] {$\beta.$};
    }$
\end{enumerate}
\end{prop}
\begin{proof} (1) follows from the fact that $\left[E_{\pm \alpha},E_{\mp \beta}\right]=\left[H^\alpha,E_{\pm \beta}\right]=\left[H^\alpha,H^\beta\right]=0$ whenever $\alpha\neq \beta$; (2) and (3) follows from the commutator relation $\left[H^\alpha,E_{\pm\alpha}\right]=\pm E_{\pm\alpha}$; (4), (5), and (6) follow from identities \eqref{e_+e_-}, \eqref{e_+s}, and \eqref{se_-} respectively.
\end{proof}

\begin{strg} To better guide the readers through the rest of this subsection, let's once for all say how to use Proposition \ref{figure} to achieve our goal. Recall that we essentially want to compute generalized minors of group elements of the form $\overline{u_{<k}}^{-1}x\overline{v^{-1}}$ and $\overline{u}^{-1}x\overline{\left(v_{>k}\right)^{-1}}$. If we can somehow draw the final pictures of these elements, then all we need are the exponents of the labellings in the circles in the middle. But this is a managable task: note that the picture of $x:=\psi_\vec{i}\left(X_a\right)$ already have all triangles of the form $\tikz[baseline=0ex,scale=0.5]{\draw (-0.75,0.75) -- (0.75,0.75) -- (0,-0.75) -- cycle;}$ before all triangles of the form $\tikz[baseline=0ex, scale=0.5]{\draw (-0.75,-0.75) -- (0.75,-0.75) -- (0,0.75) -- cycle;}$, with additional circles scatterred in the mix; then by using moves (1), (2), and (3) we can move the circles to the middle and get the final picture of the group element $x$; but then by using the $u$-part and the $v$-part of the reduced word $\vec{i}$ as reduced words for $u$ and $v$ respectively, multiplying $\overline{u_{<k}}^{-1}$ on the left and multiplying $\overline{\left(v_{>k}\right)^{-1}}$ on the right can be done using moves (5) and (6), and the newly emerged triangles and circles can be rearranged into their supposed positions in the final picture by using moves (1) - (4).
\end{strg}

There are further simplifications of the moves in Proposition \ref{figure}. For example, in move (4) we see that if $p+q\leq 0$, then the circles in the middle of the right hand side never appear. As we will see from the analysis in the rest of this subsection, the condition $p+q\leq 0$ is always satisfied, and hence we can replace move (4) by move (4') below.

Another possible simplification is to drop the last triangle of the form $\tikz[baseline=0ex, scale=0.5]{\draw (-0.75,-0.75) -- (0.75,-0.75) -- (0,0.75) -- cycle;}$ when applying (5) of Proposition \ref{figure}: based on the choice of our reduced word $\vec{i}$ we may assume that $(\alpha(n+1),\dots, \alpha(l))$ is a reduced word of $v$; then for any $m>n$ we know that $s_{\alpha(m)}\dots s_{\alpha(l-1)}$ is guaranteed to map $e_{\alpha(l)}$ to some unipotent element $n_+\in N_+$, and hence dropping the last triangle of the form $\tikz[baseline=0ex, scale=0.5]{\draw (-0.75,-0.75) -- (0.75,-0.75) -- (0,0.75) -- cycle;}$ will not affect our computation of generalized minor whatsoever. A similar argument can be applied to (6) of Proposition \ref{figure} to drop the first triangle of the form $\tikz[baseline=0ex,scale=0.5]{\draw (-0.75,0.75) -- (0.75,0.75) -- (0,-0.75) -- cycle;}$ as well. In conclusion, we have the following corollary of Proposition \ref{figure} which can be used to effectively replace moves (4), (5) and (6).

\begin{cor} With the assumption that the reduced word $\vec{i}$ has its $u$-part before its $v$-part, we may replace moves (4), (5), and (6) by the following moves. 

\vspace{0.5cm}

(4') If $p+q\leq 0$, then $\tikz[baseline=0ex,scale=0.7]{
    \draw (-0.75,0.75) -- (0.75,0.75) -- (0,-0.75) -- cycle;
    \node at (0,0) [above] {$X^q$};
    \draw (-1.25,-0.75) -- (-2.75,-0.75) -- (-2,0.75) -- cycle;
    \node at (-2,0) [below] {$X^p$};
    }  =  \tikz[baseline=0ex,scale=0.7]{
    \draw (-0.75,0.75) -- (0.75,0.75) -- (0,-0.75) -- cycle;
    \node at (0,0) [above] {$X^q$};
    \draw (1.25,-0.75) -- (2.75,-0.75) -- (2,0.75) -- cycle;
    \node at (2,0) [below] {$X^p$};
    }$.

\vspace{0.5cm}

\indent (5') $\tikz[baseline=0ex, scale=0.7]{
    \draw (-0.75,-0.75) -- (0.75,-0.75) -- (0,0.75) -- cycle;
    \node[scale=0.8] at (0,0) [below] {$X^p$};}\overline{s}_\alpha =  \tikz[baseline=0ex, scale=0.7]{
    \draw (-0.75,0.75) -- (0.75,0.75) -- (0,-0.75) -- cycle;
    \node[scale=0.8] at (0,0) [above] {$X^{-p}$};
    \draw (2,0) circle [radius=0.75];
    \node[scale=0.8] at (2,0) [] {$X^{2p}$};
    \draw (2,-2) circle [radius=0.75];
    \node[scale=0.8] at (2,-2) [] {$X^{pC_{\alpha\beta}}$};
    \node at (4,0) [] {$\alpha$};
    \node at (4,-2) [] {$\beta.$};
    }$
    
\vspace{0.5cm}

\indent (6') $\overline{s}_\alpha^{-1}\tikz[baseline=0ex, scale=0.7]{
    \draw (-0.75,0.75) -- (0.75,0.75) -- (0,-0.75) -- cycle;
    \node[scale=0.8] at (0,0) [above] {$X^p$};}=  \tikz[baseline=0ex, scale=0.7]{
    \draw (2,0) circle [radius=0.75];
    \node[scale=0.8] at (2,0) [] {$X^{2p}$};
    \draw (3.25,-0.75) -- (4.75,-0.75) -- (4,0.75) -- cycle;
    \node[scale=0.8] at (4,0) [below] {$X^{-p}$};
    \draw (2,-2) circle [radius=0.75];
    \node[scale=0.8] at (2,-2) [] {$X^{pC_{\alpha\beta}}$};
    \node at (6,0) [] {$\alpha$};
    \node at (6,-2) [] {$\beta.$};
    }$
    
\vspace{0.5cm}
\end{cor}

In order to make our statements precise, we should also give the following definition.

\begin{defn} Fix a reduced word $\vec{i}$ whose $u$-part comes before its $v$-part. On the corresponding seed torus $\mathcal{X}_\vec{i}$, a cluster variable $X$ is said to be in the \emph{$v_{>k}$-part} if the left node of the corresponding string is a simple root $\alpha(j)$ for some $j\geq k$, and is said to be in the \emph{$u_{<k}$-part} if the right node of the corresponding string is an opposite simple root $\alpha(i)$ for some $i\leq k$; a cluster variable $X$ is said to be in the \emph{$v$-part} if it is in the $v_{>l(u)+1}$-part, and is said to be in the \emph{$u$-part} if it is in the $u_{<l(u)}$-part. 
\end{defn}

Let's start our detailed analysis. Our first task is to compute the final picture of $x\overline{v^{-1}}$. As a convention, we will use a dashed vertical line to keep track of the separation between the $u$-part and the $v$-part of the reduced word. 

To begin, we notice that there are exactly $l(v)$ number of triangles of the form  $\tikz[baseline=0ex, scale=0.5]{\draw (-0.75,-0.75) -- (0.75,-0.75) -- (0,0.75) -- cycle;}$ on the right of the central dashed line. Since each letter of $l(v)$ flips a triangle from  $\tikz[baseline=0ex, scale=0.5]{\draw (-0.75,-0.75) -- (0.75,-0.75) -- (0,0.75) -- cycle;}$ to $\tikz[baseline=0ex,scale=0.5]{\draw (-0.75,0.75) -- (0.75,0.75) -- (0,-0.75) -- cycle;}$, we would expect that after all the flipping and moving but before we take the minors, there will be no triangles of the form $\tikz[baseline=0ex, scale=0.5]{\draw (-0.75,-0.75) -- (0.75,-0.75) -- (0,0.75) -- cycle;}$ left in the picture.

So now let's do a case-by-case analysis of the eventual degree of each cluster $\mathcal{X}$-variable in $x\overline{v^{-1}}$. Consider the case where $X$ is a cluster $\mathcal{X}$-variable not contained in the $v$-part of the reduced word, such as the $X$ in the following picture of $x$.
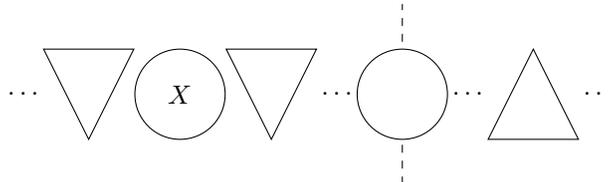
\[
\dots \tikz[scale=0.8, baseline=0ex] {
\draw (-0.75,0.75) -- (0.75,0.75) -- (0,-0.75) -- cycle;
}
\tikz[scale=0.8, baseline=0ex]{
\draw (0,0) circle [radius=0.75];
\node at (0,0) [] {$X$};
} \tikz[scale=0.8, baseline=0ex] {
\draw (-0.75,0.75) -- (0.75,0.75) -- (0,-0.75) -- cycle;
} \dots \tikz[baseline=0ex, scale=0.8]{
\draw[dashed] (0,-1.5) -- (0,1.5);
\draw [fill=white] (0,0) circle [radius=0.75];
}\dots \tikz[scale=0.8, baseline=0ex] {
\draw (-0.75,-0.75) -- (0.75,-0.75) -- (0,0.75) -- cycle;
}\dots
\]
We first need to move it towards to the central dashed line passed all the triangles of the form $\tikz[baseline=0ex,scale=0.5]{\draw (-0.75,0.75) -- (0.75,0.75) -- (0,-0.75) -- cycle;}$ using move (3).
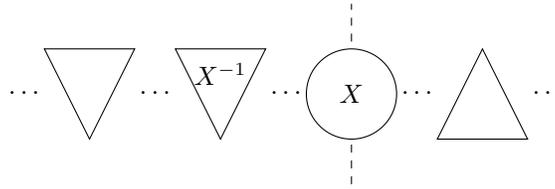
\[
\dots \tikz[scale=0.8, baseline=0ex] {
\draw (-0.75,0.75) -- (0.75,0.75) -- (0,-0.75) -- cycle;
}\dots \tikz[scale=0.8, baseline=0ex] {
\draw (-0.75,0.75) -- (0.75,0.75) -- (0,-0.75) -- cycle;
\node at (0,0) [above] {$X^{-1}$};
}
 \dots \tikz[baseline=0ex, scale=0.8]{
\draw[dashed] (0,-1.5) -- (0,1.5);
\draw [fill=white] (0,0) circle [radius=0.75];
\node at (0,0) [] {$X$};
}\dots \tikz[scale=0.8, baseline=0ex] {
\draw (-0.75,-0.75) -- (0.75,-0.75) -- (0,0.75) -- cycle;
}\dots
\]
Now to compute the labellings involving $X$ in the final picture of $x\overline{v^{-1}}$, we need to use move (5') repeatedly to flip the triangles of the form $\tikz[baseline=0ex, scale=0.5]{\draw (-0.75,-0.75) -- (0.75,-0.75) -- (0,0.75) -- cycle;}$ on the right using the $\overline{s}_\alpha$ factors from $\overline{v^{-1}}$, and pass the flipped triangles to the other side of the central dashed line using move (3) (currently we are only focusing on the variable $X$ so we can safely ignore the labellings emerged from move (5')); note that when a flipped triangle on the same horizontal level as $X$ passes the central dashed line it picks up a label $X^{-1}$ as well; therefore in the final picture of $x\overline{v^{-1}}$, the distribution of the labelling $X$ is the following.
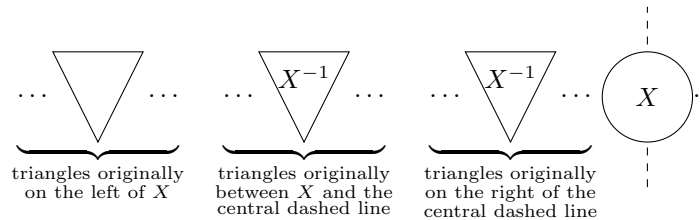
\[
\underbrace{\dots \tikz[scale=0.8, baseline=0ex] {
\draw (-0.75,0.75) -- (0.75,0.75) -- (0,-0.75) -- cycle;
}\dots}_{\substack{\text{triangles originally} \\ \text{on the left of $X$}}}
 \quad \underbrace{\dots \tikz[scale=0.8, baseline=0ex] {
\draw (-0.75,0.75) -- (0.75,0.75) -- (0,-0.75) -- cycle;
\node at (0,0) [above] {$X^{-1}$};
}\dots}_{\substack{\text{triangles originally}  \\ \text{between $X$ and the} \\ \text{central dashed line}}} \quad 
\underbrace{\dots \tikz[scale=0.8, baseline=0ex] {
\draw (-0.75,0.75) -- (0.75,0.75) -- (0,-0.75) -- cycle;
\node at (0,0) [above] {$X^{-1}$};
}\dots}_{\substack{\text{triangles originally} \\ \text{on the right of the} \\ \text{central dashed line}}}
\tikz[baseline=0ex, scale=0.8]{
\draw[dashed] (0,-1.5) -- (0,1.5);
\draw [fill=white] (0,0) circle [radius=0.75];
\node at (0,0) [] {$X$};
}.
\]

Things get a bit trickier when $X$ is a cluster $\mathcal{X}$-variable is contained in the $v$-part of the reduced word. To simply things, let's assume that in the series of diagrams below, the top level is always level $\alpha$ and the bottom level is any other level $\beta$, unless otherwise specified. Suppose we start with the following picture of $x$.
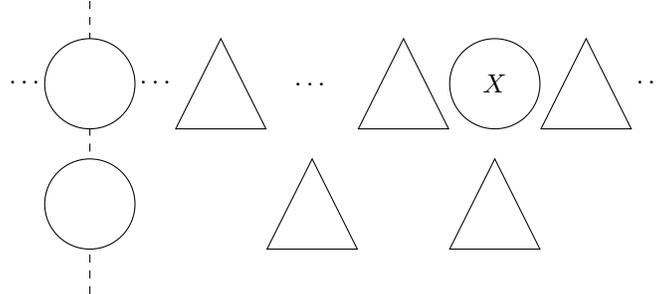
\[
   \dots \tikz[scale=0.8, baseline=0ex]{
    \draw[dashed] (0,-3.5) -- (0,1.5);
    \draw[fill=white] (0,0) circle [radius=0.75];
    \draw[fill=white] (0,-2) circle [radius=0.75];
    }
    \dots \tikz[scale=0.8, baseline=0ex] {
\draw (-0.75,-0.75) -- (0.75,-0.75) -- (0,0.75) -- cycle;
}
\tikz[scale=0.8, baseline=0ex] {
\draw (-0.75,-2.75) -- (0.75,-2.75) -- (0,-1.25) -- cycle;
\node at (0,0) [] {$\dots$};
} \tikz[scale=0.8, baseline=0ex] {
\draw (-0.75,-0.75) -- (0.75,-0.75) -- (0,0.75) -- cycle;
}
\tikz[scale=0.8, baseline=0ex] {
\draw (-0.75,-2.75) -- (0.75,-2.75) -- (0,-1.25) -- cycle;
\draw (0,0) circle [radius=0.75];
\node at (0,0) [] {$X$};
}
\tikz[scale=0.8, baseline=0ex] {
\draw (-0.75,-0.75) -- (0.75,-0.75) -- (0,0.75) -- cycle;
}\dots
\]
The first step is analogous to what we did in last case: we move $X$ to the central dashed line using move (2), which leaves a labeling $X^{-1}$ on every triangle of the form $\tikz[baseline=0ex, scale=0.5]{\draw (-0.75,-0.75) -- (0.75,-0.75) -- (0,0.75) -- cycle;}$ originally on the left of it.
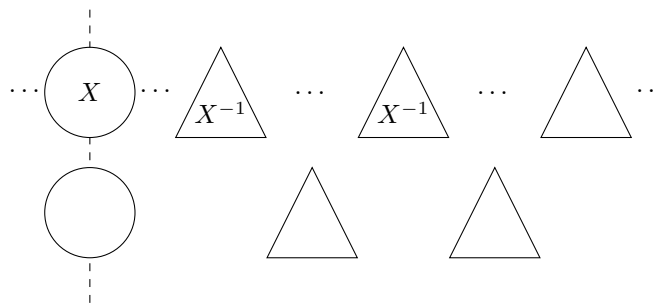
\[
 \dots \tikz[scale=0.8, baseline=0ex]{
    \draw[dashed] (0,-3.5) -- (0,1.5);
    \draw[fill=white] (0,0) circle [radius=0.75];
    \draw[fill=white] (0,-2) circle [radius=0.75];
    \node at (0,0) [] {$X$};
    }
    \dots \tikz[scale=0.8, baseline=0ex] {
\draw (-0.75,-0.75) -- (0.75,-0.75) -- (0,0.75) -- cycle;
\node at (0,0) [below] {$X^{-1}$};
}
\tikz[scale=0.8, baseline=0ex] {
\draw (-0.75,-2.75) -- (0.75,-2.75) -- (0,-1.25) -- cycle;
\node at (0,0) [] {$\dots$};
} \tikz[scale=0.8, baseline=0ex] {
\draw (-0.75,-0.75) -- (0.75,-0.75) -- (0,0.75) -- cycle;
\node at (0,0) [below] {$X^{-1}$};
}
\tikz[scale=0.8, baseline=0ex] {
\draw (-0.75,-2.75) -- (0.75,-2.75) -- (0,-1.25) -- cycle;
\node at (0,0) [] {$\dots$};
}
\tikz[scale=0.8, baseline=0ex] {
\draw (-0.75,-0.75) -- (0.75,-0.75) -- (0,0.75) -- cycle;
}\dots 
\]
Next we start flipping triangles using factors from $\overline{v^{-1}}$ corresponding to its constituent letters. Note that there are two kinds of triangles of the form $\tikz[baseline=0ex, scale=0.5]{\draw (-0.75,-0.75) -- (0.75,-0.75) -- (0,0.75) -- cycle;}$ on level $\alpha$: the ones that have labelling $X^{-1}$ and the ones that do not. Flipping the ones that do not have labelling does not cause complication, as we can see from the formula of move (5'); they only pick up a factor of $X^{-1}$ when they cross the central dashed line, and the final picture after such a flip looks like the following.
\begin{equation}\label{unlabeled}
\dots \tikz[scale=0.8, baseline=0ex] {
\draw (-0.75,0.75) -- (0.75,0.75) -- (0,-0.75) -- cycle;
\node at (0,0) [above] {$X^{-1}$};
} 
\tikz[scale=0.8,baseline=0ex]{
\draw (-0.75,-1.25) -- (0.75,-1.25) -- (0,-2.75) -- cycle;
\node at (0,0) [] {$\dots$};
}\dots\tikz[scale=0.8, baseline=0ex]{
    \draw[dashed] (0,-3.5) -- (0,1.5);
    \draw[fill=white] (0,0) circle [radius=0.75];
    \draw[fill=white] (0,-2) circle [radius=0.75];
    \node at (0,0) [] {$X$};
    }
    \dots \tikz[scale=0.8, baseline=0ex] {
\draw (-0.75,-0.75) -- (0.75,-0.75) -- (0,0.75) -- cycle;
\node at (0,0) [below] {$X^{-1}$};
}
\tikz[scale=0.8, baseline=0ex] {
\draw (-0.75,-2.75) -- (0.75,-2.75) -- (0,-1.25) -- cycle;
\node at (0,0) [] {$\dots$};
} \tikz[scale=0.8, baseline=0ex] {
\draw (-0.75,-0.75) -- (0.75,-0.75) -- (0,0.75) -- cycle;
\node at (0,0) [below] {$X^{-1}$};
}
\dots
\end{equation}
However, flipping a triangle of the form $\tikz[baseline=0ex, scale=0.5]{\draw (-0.75,-0.75) -- (0.75,-0.75) -- (0,0.75) -- cycle;}$ with an $X^{-1}$ labelling introduces more labellings involving $X$. Let's draw what happens when we do such a flip for the first time.
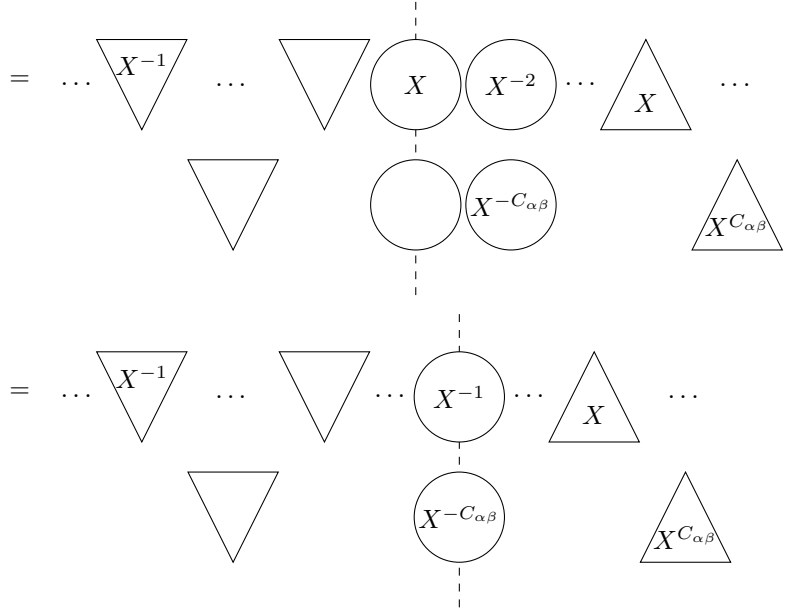
\begin{align*}
&\dots \tikz[scale=0.8, baseline=0ex] {
\draw (-0.75,0.75) -- (0.75,0.75) -- (0,-0.75) -- cycle;
\node at (0,0) [above] {$X^{-1}$};
} 
\tikz[scale=0.8,baseline=0ex]{
\draw (-0.75,-1.25) -- (0.75,-1.25) -- (0,-2.75) -- cycle;
\node at (0,0) [] {$\dots$};
}\dots\tikz[scale=0.8, baseline=0ex]{
    \draw[dashed] (0,-3.5) -- (0,1.5);
    \draw[fill=white] (0,0) circle [radius=0.75];
    \draw[fill=white] (0,-2) circle [radius=0.75];
    \node at (0,0) [] {$X$};
    }
    \dots \tikz[scale=0.8, baseline=0ex] {
\draw (-0.75,-0.75) -- (0.75,-0.75) -- (0,0.75) -- cycle;
\node at (0,0) [below] {$X^{-1}$};
}
\tikz[scale=0.8, baseline=0ex] {
\draw (-0.75,-2.75) -- (0.75,-2.75) -- (0,-1.25) -- cycle;
\node at (0,0) [] {$\dots$};
} \tikz[scale=0.8, baseline=0ex] {
\draw (-0.75,-0.75) -- (0.75,-0.75) -- (0,0.75) -- cycle;
\node at (0,0) [below] {$X^{-1}$};
}
\overline{s}_\alpha\\
=\quad &\dots \tikz[scale=0.8, baseline=0ex] { 
\draw (-0.75,0.75) -- (0.75,0.75) -- (0,-0.75) -- cycle;
\node at (0,0) [above] {$X^{-1}$};
} 
\tikz[scale=0.8,baseline=0ex]{
\draw (-0.75,-1.25) -- (0.75,-1.25) -- (0,-2.75) -- cycle;
\node at (0,0) [] {$\dots$};
}\dots\tikz[scale=0.8, baseline=0ex]{
    \draw[dashed] (0,-3.5) -- (0,1.5);
    \draw[fill=white] (0,0) circle [radius=0.75];
    \draw[fill=white] (0,-2) circle [radius=0.75];
    \node at (0,0) [] {$X$};
    }
    \dots \tikz[scale=0.8, baseline=0ex] {
\draw (-0.75,-0.75) -- (0.75,-0.75) -- (0,0.75) -- cycle;
\node at (0,0) [below] {$X^{-1}$};
}
\tikz[scale=0.8, baseline=0ex] {
\draw (-0.75,-2.75) -- (0.75,-2.75) -- (0,-1.25) -- cycle;
\node at (0,0) [] {$\dots$};
} \tikz[scale=0.8, baseline=0ex] {
\draw (-0.75,0.75) -- (0.75,0.75) -- (0,-0.75) -- cycle;
\node at (0,0) [above] {$X$};
\draw (1.5,0) circle [radius=0.75];
\node at (1.5,0) [] {$X^{-2}$};
\draw (1.5,-2) circle [radius=0.75];
\node at (1.5,-2) [] {$X^{-C_{\alpha\beta}}$};
}\\
=\quad & \dots \tikz[scale=0.8, baseline=0ex] { 
\draw (-0.75,0.75) -- (0.75,0.75) -- (0,-0.75) -- cycle;
\node at (0,0) [above] {$X^{-1}$};
} 
\tikz[scale=0.8,baseline=0ex]{
\draw (-0.75,-1.25) -- (0.75,-1.25) -- (0,-2.75) -- cycle;
\node at (0,0) [] {$\dots$};
}\dots\tikz[scale=0.8, baseline=0ex]{
    \draw[dashed] (0,-3.5) -- (0,1.5);
    \draw[fill=white] (0,0) circle [radius=0.75];
    \draw[fill=white] (0,-2) circle [radius=0.75];
    \node at (0,0) [] {$X$};
    }
 \tikz[scale=0.8, baseline=0ex]{\draw (-0.75,0.75) -- (0.75,0.75) -- (0,-0.75) -- cycle;
\node at (0,0) [above] {$X$};} \dots 
\tikz[scale=0.8, baseline=0ex] {
\draw (-0.75,-0.75) -- (0.75,-0.75) -- (0,0.75) -- cycle;
\node at (0,0) [below] {$X^{-1}$};
}
\tikz[scale=0.8, baseline=0ex] {
\draw (-0.75,-2.75) -- (0.75,-2.75) -- (0,-1.25) -- cycle;
\node at (0,0) [] {$\dots$};
} \tikz[scale=0.8, baseline=0ex] {
\draw (1.5,0) circle [radius=0.75];
\node at (1.5,0) [] {$X^{-2}$};
\draw (1.5,-2) circle [radius=0.75];
\node at (1.5,-2) [] {$X^{-C_{\alpha\beta}}$};
}\\
= \quad & 
 \dots \tikz[scale=0.8, baseline=0ex] {
\draw (-0.75,0.75) -- (0.75,0.75) -- (0,-0.75) -- cycle;
\node at (0,0) [above] {$X^{-1}$};
} 
\tikz[scale=0.8,baseline=0ex]{
\draw (-0.75,-1.25) -- (0.75,-1.25) -- (0,-2.75) -- cycle;
\node at (0,0) [] {$\dots$};
}
\tikz[scale=0.8, baseline=0ex] {
\draw (-0.75,0.75) -- (0.75,0.75) -- (0,-0.75) -- cycle;
}
\tikz[scale=0.8, baseline=0ex]{
    \draw[dashed] (0,-3.5) -- (0,1.5);
    \draw[fill=white] (0,0) circle [radius=0.75];
    \draw[fill=white] (0,-2) circle [radius=0.75];
    \node at (0,0) [] {$X$};
    }
    \tikz[scale=0.8, baseline=0ex] {
\draw (1.5,0) circle [radius=0.75];
\node at (1.5,0) [] {$X^{-2}$};
\draw (1.5,-2) circle [radius=0.75];
\node at (1.5,-2) [] {$X^{-C_{\alpha\beta}}$};
} \dots \tikz[scale=0.8, baseline=0ex] {
\draw (-0.75,-0.75) -- (0.75,-0.75) -- (0,0.75) -- cycle;
\node at (0,0) [below] {$X$};
}
\tikz[scale=0.8, baseline=0ex] {
\draw (-0.75,-2.75) -- (0.75,-2.75) -- (0,-1.25) -- cycle;
\node at (0,0) [] {$\dots$};
\node at (0,-2) [below] {$X^{C_{\alpha\beta}}$};
}\\
=\quad & 
 \dots \tikz[scale=0.8, baseline=0ex] {
\draw (-0.75,0.75) -- (0.75,0.75) -- (0,-0.75) -- cycle;
\node at (0,0) [above] {$X^{-1}$};
} 
\tikz[scale=0.8,baseline=0ex]{
\draw (-0.75,-1.25) -- (0.75,-1.25) -- (0,-2.75) -- cycle;
\node at (0,0) [] {$\dots$};
}
\tikz[scale=0.8, baseline=0ex] {
\draw (-0.75,0.75) -- (0.75,0.75) -- (0,-0.75) -- cycle;
}
\dots\tikz[scale=0.8, baseline=0ex]{
    \draw[dashed] (0,-3.5) -- (0,1.5);
    \draw[fill=white] (0,0) circle [radius=0.75];
    \draw[fill=white] (0,-2) circle [radius=0.75];
    \node at (0,0) [] {$X^{-1}$};
    \node at (0,-2) [] {$X^{-C_{\alpha\beta}}$};
    }
    \dots \tikz[scale=0.8, baseline=0ex] {
\draw (-0.75,-0.75) -- (0.75,-0.75) -- (0,0.75) -- cycle;
\node at (0,0) [below] {$X$};
}
\tikz[scale=0.8, baseline=0ex] {
\draw (-0.75,-2.75) -- (0.75,-2.75) -- (0,-1.25) -- cycle;
\node at (0,0) [] {$\dots$};
\node at (0,-2) [below] {$X^{C_{\alpha\beta}}$};
} 
\end{align*}

We want to pause here and make a few observations before proceeding any further. 

\begin{prop}\label{1stobs} Let $X$ be a cluster $\mathcal{X}$-variable originally on level $\alpha$ contained in the $v_{>k}$-part. Then in the final picture of $x\overline{v_{>k-1}}$, the distribution of $X$ always follows the pattern
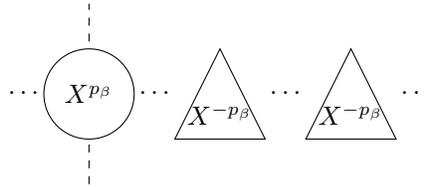
\[
\dots \tikz[baseline=0ex, scale=0.8]{
\draw[dashed] (0,-1.5) -- (0,1.5);
\draw [fill=white] (0,0) circle [radius=0.75];
\node at (0,0) [] {$X^{p_\beta}$};
}\dots \tikz[scale=0.8, baseline=0ex] {
\draw (-0.75,-0.75) -- (0.75,-0.75) -- (0,0.75) -- cycle;
\node at (0,0) [below] {$X^{-p_\beta}$};
}\dots\tikz[scale=0.8, baseline=0ex] {
\draw (-0.75,-0.75) -- (0.75,-0.75) -- (0,0.75) -- cycle;
\node at (0,0) [below] {$X^{-p_\beta}$};
}\dots
\]
on every level for some integer $p_\beta$.
\end{prop}
\begin{proof} This can be done using a backward induction on $k$ (starting from $l$ and working back to the first letter in the $v$-part of $\vec{i}$). The base case is already done just before the statement of the proposition. Now inductively when we multiply $\overline{s}_\gamma$ on the right for some $\gamma$, the last $\tikz[scale=0.7, baseline=0ex] {
\draw (-0.75,-0.75) -- (0.75,-0.75) -- (0,0.75) -- cycle;
\node[scale=0.8] at (0,0) [below] {$X^{-p_\gamma}$};
}$ turns into $\tikz[baseline=0ex, scale=0.7]{\draw (-0.75,0.75) -- (0.75,0.75) -- (0,-0.75) -- cycle; \node[scale=0.8] at (0,0) [above] {$X^{p_\gamma}$}; \draw (1.5,0) circle [radius=0.75]; \node[scale=0.8] at (1.5,0) [] {$X^{-2p_\gamma}$};}$ with some additional $\tikz[baseline=0ex,scale=0.7]{\draw (0,0) circle [radius=0.75]; \node[scale=0.8] at (0,0) [] {$X^{-p_\gamma C_{\gamma\beta}}$};}$ on other $\beta$ levels. But then by move (4') we can move $\tikz[baseline=0ex, scale=0.7]{\draw (-0.75,0.75) -- (0.75,0.75) -- (0,-0.75) -- cycle; \node[scale=0.8] at (0,0) [above] {$X^{p_\gamma}$};}$ through all $\tikz[scale=0.7, baseline=0ex] {
\draw (-0.75,-0.75) -- (0.75,-0.75) -- (0,0.75) -- cycle;
\node[scale=0.8] at (0,0) [below] {$X^{-p_\gamma}$};
}$ on level $\gamma$ without changing anything. Lastly, the circles on the far right, while moving towards the middle, change the remaining triangles in the $v$-part of the reduced word uniformly through each level; hence at the end of the day we have the following.
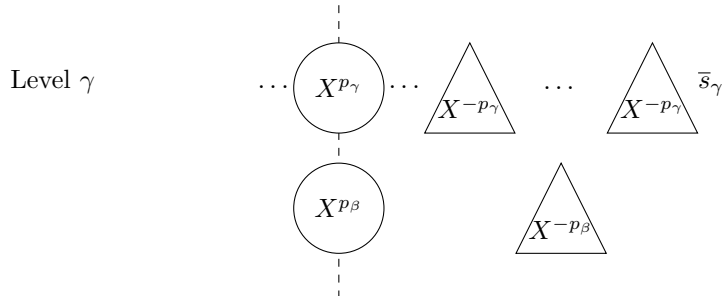
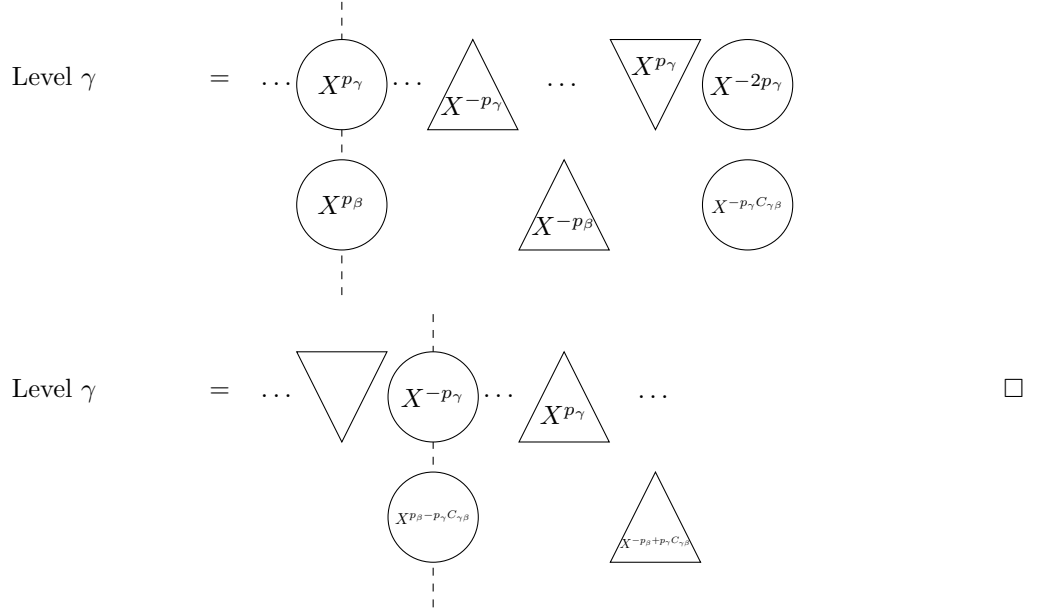
\begin{align*} 
\text{Level $\gamma$} \quad \quad\quad \quad \quad \quad 
&\dots\tikz[scale=0.8, baseline=0ex]{
    \draw[dashed] (0,-3.5) -- (0,1.5);
    \draw[fill=white] (0,0) circle [radius=0.75];
    \draw[fill=white] (0,-2) circle [radius=0.75];
    \node at (0,0) [] {$X^{p_\gamma}$};
    \node at (0,-2) [] {$X^{p_\beta}$};
    }
    \dots \tikz[scale=0.8, baseline=0ex] {
\draw (-0.75,-0.75) -- (0.75,-0.75) -- (0,0.75) -- cycle;
\node at (0,0) [below] {$X^{-p_\gamma}$};
}
\tikz[scale=0.8, baseline=0ex] {
\draw (-0.75,-2.75) -- (0.75,-2.75) -- (0,-1.25) -- cycle;
\node at (0,0) [] {$\dots$};
\node at (0,-2) [below] {$X^{-p_\beta}$};
} 
\tikz[scale=0.8, baseline=0ex] {
\draw (-0.75,-0.75) -- (0.75,-0.75) -- (0,0.75) -- cycle;
\node at (0,0) [below] {$X^{-p_\gamma}$};
}
\overline{s}_\gamma \\
\text{Level $\gamma$} \quad \quad \quad \quad  =\quad & \dots\tikz[scale=0.8, baseline=0ex]{
    \draw[dashed] (0,-3.5) -- (0,1.5);
    \draw[fill=white] (0,0) circle [radius=0.75];
    \draw[fill=white] (0,-2) circle [radius=0.75];
    \node at (0,0) [] {$X^{p_\gamma}$};
    \node at (0,-2) [] {$X^{p_\beta}$};
    }
    \dots \tikz[scale=0.8, baseline=0ex] {
\draw (-0.75,-0.75) -- (0.75,-0.75) -- (0,0.75) -- cycle;
\node at (0,0) [below] {$X^{-p_\gamma}$};
}
\tikz[scale=0.8, baseline=0ex] {
\draw (-0.75,-2.75) -- (0.75,-2.75) -- (0,-1.25) -- cycle;
\node at (0,0) [] {$\dots$};
\node at (0,-2) [below] {$X^{-p_\beta}$};
} 
\tikz[scale=0.8, baseline=0ex] {
\draw (-0.75,0.75) -- (0.75,0.75) -- (0,-0.75) -- cycle;
\node at (0,0) [above] {$X^{p_\gamma}$};
} \tikz[scale=0.8, baseline=0ex] {
\draw (1.5,0) circle [radius=0.75];
\node at (1.5,0) [] {$X^{-2p_\gamma}$};
\draw (1.5,-2) circle [radius=0.75];
\node [scale=0.7] at (1.5,-2) [] {$X^{-p_\gamma C_{\gamma\beta}}$};
}\\
\text{Level $\gamma$} \quad \quad \quad \quad  =\quad & \dots
\tikz[scale=0.8, baseline=0ex] {
\draw (-0.75,0.75) -- (0.75,0.75) -- (0,-0.75) -- cycle;
} 
\tikz[scale=0.8, baseline=0ex]{
    \draw[dashed] (0,-3.5) -- (0,1.5);
    \draw[fill=white] (0,0) circle [radius=0.75];
    \draw[fill=white] (0,-2) circle [radius=0.75];
    \node at (0,0) [] {$X^{-p_\gamma}$};
    \node [scale=0.6] at (0,-2) [] {$X^{p_\beta-p_\gamma C_{\gamma\beta}}$};
    }
    \dots \tikz[scale=0.8, baseline=0ex] {
\draw (-0.75,-0.75) -- (0.75,-0.75) -- (0,0.75) -- cycle;
\node at (0,0) [below] {$X^{p_\gamma}$};
}
\tikz[scale=0.8, baseline=0ex] {
\draw (-0.75,-2.75) -- (0.75,-2.75) -- (0,-1.25) -- cycle;
\node at (0,0) [] {$\dots$};
\node [scale=0.5] at (0,-2.25) [below] {$X^{-p_\beta+p_\gamma C_{\gamma\beta}}$};
}\qedhere
\end{align*}
\end{proof}

\begin{cor}\label{2ndobs} Let $X$ be a cluster $\mathcal{X}$-variable originally on level $\alpha$ contained in the $v$-part. Then in the final picture of $x\overline{v^{-1}}$, every triangle that has been moved from the $v$-part to the $u$-part no longer carries any factor of $X$ in their labellings, except for those that are originally on the right of $X$ on the same level (the ones with thickened boundaries below), each of which carries a labelling $X^{-1}$ in the final picture.
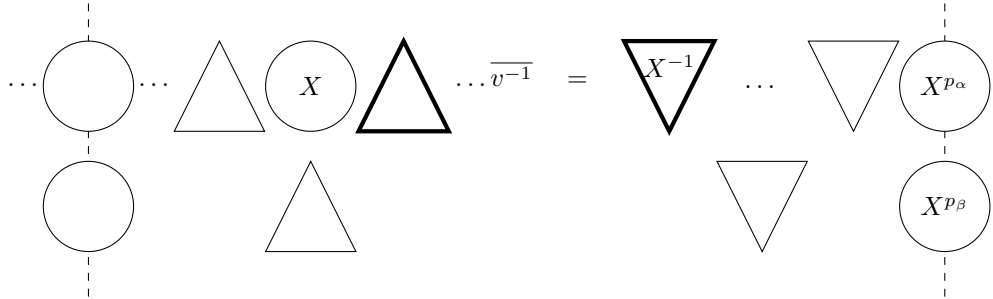
\[
 \dots \tikz[scale=0.8, baseline=0ex]{
    \draw[dashed] (0,-3.5) -- (0,1.5);
    \draw[fill=white] (0,0) circle [radius=0.75];
    \draw[fill=white] (0,-2) circle [radius=0.75];
    }
    \dots 
    \tikz[scale=0.8, baseline=0ex] {
\draw (-0.75,-0.75) -- (0.75,-0.75) -- (0,0.75) -- cycle;
}
\tikz[scale=0.8, baseline=0ex] {
\draw (-0.75,-2.75) -- (0.75,-2.75) -- (0,-1.25) -- cycle;
\draw (0,0) circle [radius=0.75];
\node at (0,0) [] {$X$};
}
\tikz[scale=0.8, baseline=0ex] {
\draw [ultra thick](-0.75,-0.75) -- (0.75,-0.75) -- (0,0.75) -- cycle;
}\dots \overline{v^{-1}} \quad = \quad \tikz[scale=0.8, baseline=0ex] {
\draw [ultra thick] (-0.75,0.75) -- (0.75,0.75) -- (0,-0.75) -- cycle;
\node at (0,0) [above] {$X^{-1}$};
} 
\tikz[scale=0.8,baseline=0ex]{
\draw (-0.75,-1.25) -- (0.75,-1.25) -- (0,-2.75) -- cycle;
\node at (0,0) [] {$\dots$};
}
\tikz[scale=0.8, baseline=0ex] {
\draw (-0.75,0.75) -- (0.75,0.75) -- (0,-0.75) -- cycle;
} \tikz[scale=0.8, baseline=0ex]{
    \draw[dashed] (0,-3.5) -- (0,1.5);
    \draw[fill=white] (0,0) circle [radius=0.75];
    \draw[fill=white] (0,-2) circle [radius=0.75];
    \node at (0,0) [] {$X^{p_\alpha}$};
    \node at (0,-2) [] {$X^{p_\beta}$};
    }
\]
\end{cor}
\begin{proof} From last proposition we know that starting from the very first time when a triangle $\tikz[scale=0.7, baseline=0ex] {
\draw (-0.75,-0.75) -- (0.75,-0.75) -- (0,0.75) -- cycle;
\node[scale=0.8] at (0,0) [below] {$X^{-p_\gamma}$};
}$ is flipped, in the final picture after each additional multiple of $\overline{s}_\gamma$ on the right, the exponents of $X$ in the triangles of the form $\tikz[baseline=0ex, scale=0.5]{\draw (-0.75,-0.75) -- (0.75,-0.75) -- (0,0.75) -- cycle;}$ are opposite to the exponents of $X$ in the circle on the same level. Therefore when these triangles are moved across the central dashed line, they come out carrying no factor of $X$. However, the triangles that were moved before the first time we flipped a triangle with a labelling $X^{-1}$ still carries a labelling of $X^{-1}$, as we have seen from Picture \eqref{unlabeled}.
\end{proof}

Computing the exponents $p_\alpha$ in the final picture of $x\overline{v^{-1}}$ is still non-trivial; however, we can formulate the corollaries that will help us prove statements about such exponents without actually computing them. We start with the following corollary which is just a summary of Picture \eqref{unlabeled} and the picture in the proof of Proposition \ref{1stobs}.

\begin{cor}\label{3.16} Let $X$ be a cluster $\mathcal{X}$-variable originally contained in the $v$-part on level $\alpha$. Let $\gamma=\alpha(k)$ be a simple root (and hence $\alpha(k)$ is a letter in the $v$-part of the reduced word $\vec{i}$). Let $p_\beta$ be the exponent of $X$ in the circle on level $\beta$ in the final picture of $x\overline{\left(v_{>k-1}\right)^{-1}}$. 
\begin{itemize}
    \item If $X$ is not in the $v_{>k}$-part, then $p_\beta=\delta_{\alpha\beta}$.
    \item If $X$ is in the $v_{>k}$-part, then
    \[
p_\beta=q_\beta-C_{\gamma\beta}q_\gamma,
\]
where $q_\beta$ are the counterparts of $p_\beta$ for $x\overline{\left(v_{>k}\right)^{-1}}$.
\end{itemize}
\end{cor}

Based on this corollary we get the following identity.

\begin{cor}\label{3.10} Let $X$ be a cluster $\mathcal{X}$-variable originally contained in the $v_{>k}$-part. Let $\gamma=\alpha(k)$ be a simple root. Fix a simple root $\alpha$ (not necessarily the level where $X$ is originally on). Let $p_\beta$ be the exponent of $X$ in the circle on level $\beta$ in the final picture of $x\overline{\left(v_{>k-1}\right)^{-1}}$, and let $q_\beta$ be their counterparts for $x\overline{\left(v_{>k}\right)^{-1}}$. Then
\[
\sum_\beta p_\beta \left(C^{-1}\right)_{\beta \alpha}=-\delta_{\alpha\gamma}q_\alpha+\sum_\beta q_\beta \left(C^{-1}\right)_{\beta\alpha},
\]
where $C^{-1}$ is the inverse matrix of the Cartan matrix $C$. In other words, unless $\alpha=\gamma$, the quantity $\sum_\beta p_\beta \left(C^{-1}\right)_{\beta\alpha}$ is invariant under an additional multiplication of $\overline{s}_\gamma$ on the right.
\end{cor}
\begin{proof} We just need to plug in the second claim from the last corollary:
\[
\sum_\beta q_\beta \left(C^{-1}\right)_{\beta \alpha}=\sum_\beta \left(p_\beta-C_{\gamma\beta}p_\gamma\right)\left(C^{-1}\right)_{\beta\alpha}=-\delta_{\alpha\gamma}p_\alpha+\sum_\beta p_\beta \left(C^{-1}\right)_{\beta\alpha}.\qedhere
\]
\end{proof}

Now we are more or less done with $x\overline{v^{-1}}$. To finish what we have started, we now need to multiply some part of the reduced word of $u$ on the left to get $\overline{u_{<k}}^{-1}x\overline{v^{-1}}$. This means we start with whatever we get at the end when computing $x\overline{v^{-1}}$, and then flip the triangles on the left one at a time using move (6'), and then move each one of them to the right of the central dashed line. Note that in the final picture of $x\overline{v^{-1}}$ the triangles on the left of the dashed line have two possible origins: the ones that are originally in the $u$-part and the ones that are originally in the $v$-part. We will now draw two dashed lines in the next few pictures, with one representing the separation between the triangles of different origins, and the other one representing the original separation between the $u$-part and the $v$-part of the reduced word.
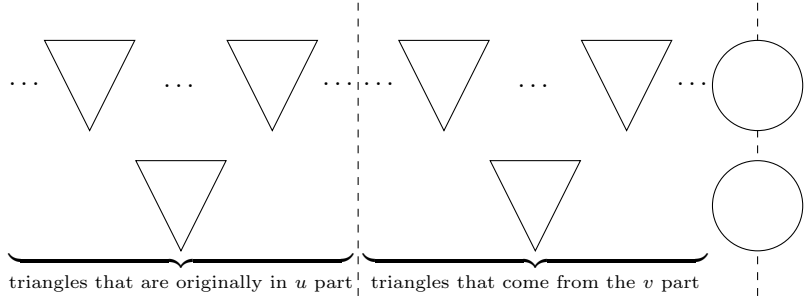
\[
\underbrace{\dots  
\tikz[scale=0.8, baseline=0ex] {
\draw (-0.75,0.75) -- (0.75,0.75) -- (0,-0.75) -- cycle;
}
\tikz[scale=0.8, baseline=0ex] {
\draw (-0.75,-1.25) -- (0.75,-1.25) -- (0,-2.75) -- cycle;
\node at (0,0) [] {$\dots$};
}
\tikz[scale=0.8, baseline=0ex] {
\draw (-0.75,0.75) -- (0.75,0.75) -- (0,-0.75) -- cycle;
}
\dots}_\text{triangles that are originally in $u$ part}
\tikz[scale=0.8,baseline=0ex]{
\draw [dashed] (0,-3.5) -- (0,1.5);
}
\underbrace{\dots  
\tikz[scale=0.8, baseline=0ex] {
\draw (-0.75,0.75) -- (0.75,0.75) -- (0,-0.75) -- cycle;
}
\tikz[scale=0.8, baseline=0ex] {
\draw (-0.75,-1.25) -- (0.75,-1.25) -- (0,-2.75) -- cycle;
\node at (0,0) [] {$\dots$};
}
\tikz[scale=0.8, baseline=0ex] {
\draw (-0.75,0.75) -- (0.75,0.75) -- (0,-0.75) -- cycle;
}
\dots}_\text{triangles that come from the $v$ part}
\tikz[scale=0.8, baseline=0ex]{
\draw[dashed] (0,-3.5) -- (0,1.5);
\draw[fill=white] (0,0) circle [radius=0.75];
\draw[fill=white] (0,-2) circle [radius=0.75];
}
\]

Suppose $X$ is originally a cluster $\mathcal{X}$-variable on level $\alpha$ strictly contained in the $v$-part. Then from Corollary \ref{2ndobs} we know that in the final picture of $x\overline{v^{-1}}$, only the triangles that are originally on level $\alpha$ to the right of $X$ (and hence from the $v$-part) will carry a factor of $X^{-1}$ in them.
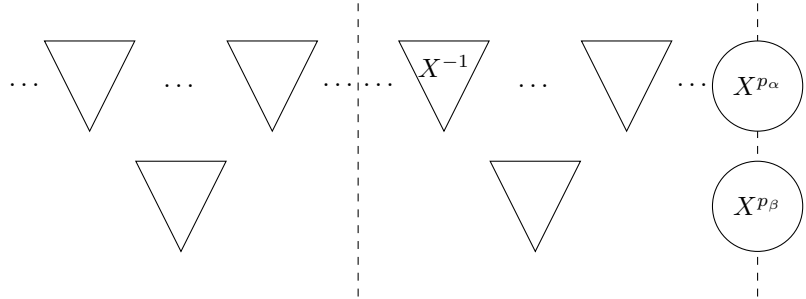
\[
\dots  
\tikz[scale=0.8, baseline=0ex] {
\draw (-0.75,0.75) -- (0.75,0.75) -- (0,-0.75) -- cycle;
}
\tikz[scale=0.8, baseline=0ex] {
\draw (-0.75,-1.25) -- (0.75,-1.25) -- (0,-2.75) -- cycle;
\node at (0,0) [] {$\dots$};
}
\tikz[scale=0.8, baseline=0ex] {
\draw (-0.75,0.75) -- (0.75,0.75) -- (0,-0.75) -- cycle;
}
\dots
\tikz[scale=0.8,baseline=0ex]{
\draw [dashed] (0,-3.5) -- (0,1.5);
}
\dots  
\tikz[scale=0.8, baseline=0ex] {
\draw (-0.75,0.75) -- (0.75,0.75) -- (0,-0.75) -- cycle;
\node at (0,0) [above] {$X^{-1}$};
}
\tikz[scale=0.8, baseline=0ex] {
\draw (-0.75,-1.25) -- (0.75,-1.25) -- (0,-2.75) -- cycle;
\node at (0,0) [] {$\dots$};
}
\tikz[scale=0.8, baseline=0ex] {
\draw (-0.75,0.75) -- (0.75,0.75) -- (0,-0.75) -- cycle;
}
\dots
\tikz[scale=0.8, baseline=0ex]{
\draw[dashed] (0,-3.5) -- (0,1.5);
\draw[fill=white] (0,0) circle [radius=0.75];
\draw[fill=white] (0,-2) circle [radius=0.75];
\node at (0,0) [] {$X^{p_\alpha}$};
\node at (0,-2) [] {$X^{p_\beta}$};
}
\]
Therefore when we apply (4) and (6') repeatedly to compute $\overline{u_{<k}}^{-1}x\overline{v^{-1}}$, no leading power of $X$ is going to change except when the newly flipped triangles cross the central dashed line from the left to the right; but even for that, the labellings on the circles in the middle are still unchanged, which are the only things we are going to use when we compute $\deg_{X_a}\Xi_\vec{i}^*\left(X_b\right)$.
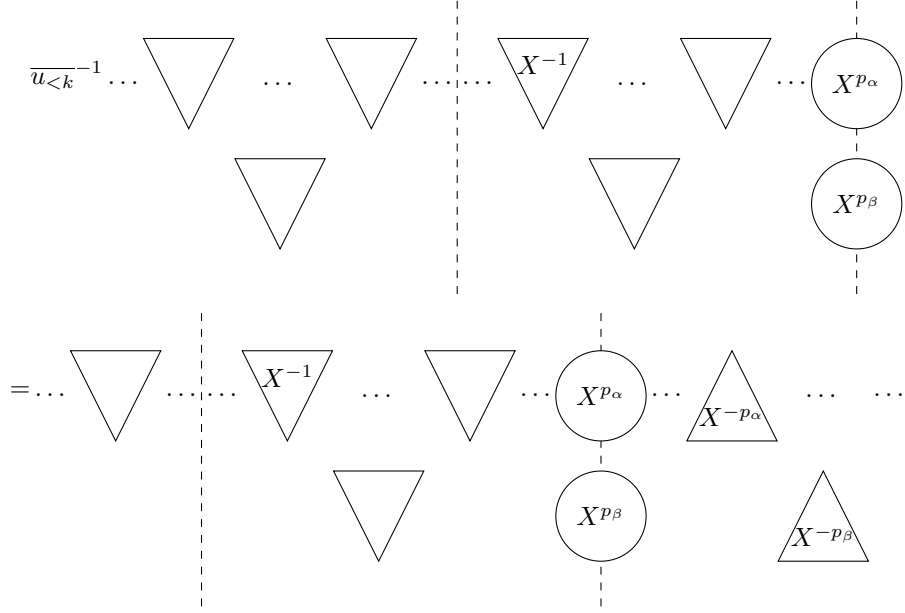
\begin{align*}
&\overline{u_{<k}}^{-1}\dots  
\tikz[scale=0.8, baseline=0ex] {
\draw (-0.75,0.75) -- (0.75,0.75) -- (0,-0.75) -- cycle;
}
\tikz[scale=0.8, baseline=0ex] {
\draw (-0.75,-1.25) -- (0.75,-1.25) -- (0,-2.75) -- cycle;
\node at (0,0) [] {$\dots$};
}
\tikz[scale=0.8, baseline=0ex] {
\draw (-0.75,0.75) -- (0.75,0.75) -- (0,-0.75) -- cycle;
}
\dots
\tikz[scale=0.8,baseline=0ex]{
\draw [dashed] (0,-3.5) -- (0,1.5);
}
\dots  
\tikz[scale=0.8, baseline=0ex] {
\draw (-0.75,0.75) -- (0.75,0.75) -- (0,-0.75) -- cycle;
\node at (0,0) [above] {$X^{-1}$};
}
\tikz[scale=0.8, baseline=0ex] {
\draw (-0.75,-1.25) -- (0.75,-1.25) -- (0,-2.75) -- cycle;
\node at (0,0) [] {$\dots$};
}
\tikz[scale=0.8, baseline=0ex] {
\draw (-0.75,0.75) -- (0.75,0.75) -- (0,-0.75) -- cycle;
}
\dots
\tikz[scale=0.8, baseline=0ex]{
\draw[dashed] (0,-3.5) -- (0,1.5);
\draw[fill=white] (0,0) circle [radius=0.75];
\draw[fill=white] (0,-2) circle [radius=0.75];
\node at (0,0) [] {$X^{p_\alpha}$};
\node at (0,-2) [] {$X^{p_\beta}$};
}\\
=&
\dots  
\tikz[scale=0.8, baseline=0ex] {
\draw (-0.75,0.75) -- (0.75,0.75) -- (0,-0.75) -- cycle;
}
\dots
\tikz[scale=0.8,baseline=0ex]{
\draw [dashed] (0,-3.5) -- (0,1.5);
}
\dots  
\tikz[scale=0.8, baseline=0ex] {
\draw (-0.75,0.75) -- (0.75,0.75) -- (0,-0.75) -- cycle;
\node at (0,0) [above] {$X^{-1}$};
}
\tikz[scale=0.8, baseline=0ex] {
\draw (-0.75,-1.25) -- (0.75,-1.25) -- (0,-2.75) -- cycle;
\node at (0,0) [] {$\dots$};
}
\tikz[scale=0.8, baseline=0ex] {
\draw (-0.75,0.75) -- (0.75,0.75) -- (0,-0.75) -- cycle;
}
\dots
\tikz[scale=0.8, baseline=0ex]{
\draw[dashed] (0,-3.5) -- (0,1.5);
\draw[fill=white] (0,0) circle [radius=0.75];
\draw[fill=white] (0,-2) circle [radius=0.75];
\node at (0,0) [] {$X^{p_\alpha}$};
\node at (0,-2) [] {$X^{p_\beta}$};
}
\dots
\tikz[scale=0.8, baseline=0ex] {
\draw (-0.75,-0.75) -- (0.75,-0.75) -- (0,0.75) -- cycle;
\node at (0,0) [below] {$X^{-p_\alpha}$};
}
\tikz[scale=0.8, baseline=0ex] {
\draw (-0.75,-2.75) -- (0.75,-2.75) -- (0,-1.25) -- cycle;
\node at (0,0) [] {$\dots$};
\node at (0,-2) [below] {$X^{-p_\beta}$};
}
\dots
\end{align*}
Therefore we can formulate the following proposition.

\begin{prop}\label{3.17} Let $X$ be a cluster $\mathcal{X}$-variable originally on level $\alpha$ strictly contained in the $v$-part. If the exponent of $X$ in the circle on level $\beta$ in the final picture of $x\overline{v^{-1}}$ is $p_\beta$ for each $\beta$ (here we include the possibility that $\beta=\alpha$), then the exponent of $X$ in the circle on level $\beta$ in the final picture of $\overline{u_{<k}}^{-1}x\overline{v^{-1}}$ is still $p_\beta$ for each $\beta$.  
\end{prop}

Next if $X$ is originally a cluster $\mathcal{X}$-variable on level $\alpha$ neither in the $v$-part nor in the $u_{<k}$-part, then we can also see easily that
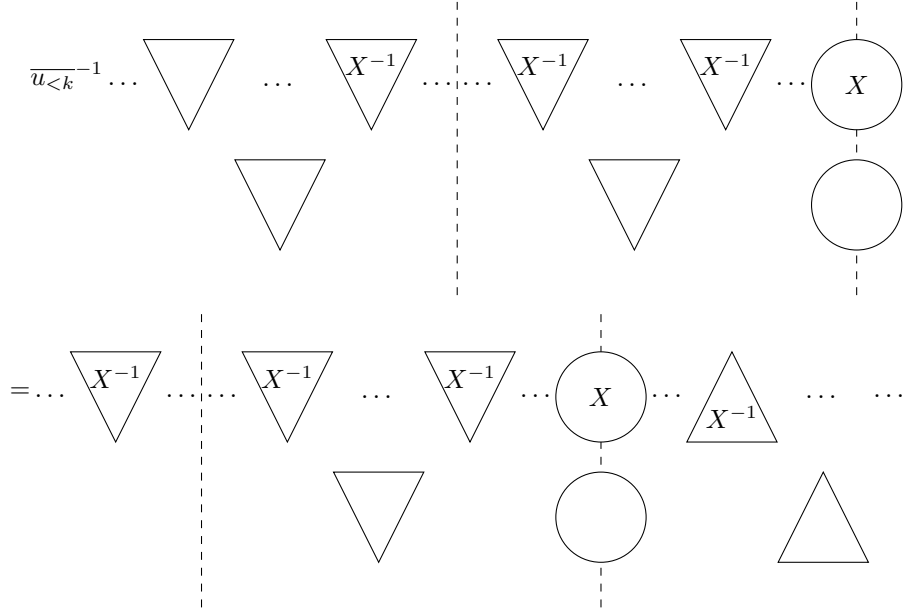
\begin{align*}
&\overline{u_{<k}}^{-1}\dots  
\tikz[scale=0.8, baseline=0ex] {
\draw (-0.75,0.75) -- (0.75,0.75) -- (0,-0.75) -- cycle;
}
\tikz[scale=0.8, baseline=0ex] {
\draw (-0.75,-1.25) -- (0.75,-1.25) -- (0,-2.75) -- cycle;
\node at (0,0) [] {$\dots$};
}
\tikz[scale=0.8, baseline=0ex] {
\draw (-0.75,0.75) -- (0.75,0.75) -- (0,-0.75) -- cycle;
\node at (0,0) [above] {$X^{-1}$};
}
\dots
\tikz[scale=0.8,baseline=0ex]{
\draw [dashed] (0,-3.5) -- (0,1.5);
}
\dots  
\tikz[scale=0.8, baseline=0ex] {
\draw (-0.75,0.75) -- (0.75,0.75) -- (0,-0.75) -- cycle;
\node at (0,0) [above] {$X^{-1}$};
}
\tikz[scale=0.8, baseline=0ex] {
\draw (-0.75,-1.25) -- (0.75,-1.25) -- (0,-2.75) -- cycle;
\node at (0,0) [] {$\dots$};
}
\tikz[scale=0.8, baseline=0ex] {
\draw (-0.75,0.75) -- (0.75,0.75) -- (0,-0.75) -- cycle;
\node at (0,0) [above] {$X^{-1}$};
}
\dots
\tikz[scale=0.8, baseline=0ex]{
\draw[dashed] (0,-3.5) -- (0,1.5);
\draw[fill=white] (0,0) circle [radius=0.75];
\draw[fill=white] (0,-2) circle [radius=0.75];
\node at (0,0) [] {$X$};
}\\
=&
\dots  
\tikz[scale=0.8, baseline=0ex] {
\draw (-0.75,0.75) -- (0.75,0.75) -- (0,-0.75) -- cycle;
\node at (0,0) [above] {$X^{-1}$};
}
\dots
\tikz[scale=0.8,baseline=0ex]{
\draw [dashed] (0,-3.5) -- (0,1.5);
}
\dots  
\tikz[scale=0.8, baseline=0ex] {
\draw (-0.75,0.75) -- (0.75,0.75) -- (0,-0.75) -- cycle;
\node at (0,0) [above] {$X^{-1}$};
}
\tikz[scale=0.8, baseline=0ex] {
\draw (-0.75,-1.25) -- (0.75,-1.25) -- (0,-2.75) -- cycle;
\node at (0,0) [] {$\dots$};
}
\tikz[scale=0.8, baseline=0ex] {
\draw (-0.75,0.75) -- (0.75,0.75) -- (0,-0.75) -- cycle;
\node at (0,0) [above] {$X^{-1}$};
}
\dots
\tikz[scale=0.8, baseline=0ex]{
\draw[dashed] (0,-3.5) -- (0,1.5);
\draw[fill=white] (0,0) circle [radius=0.75];
\draw[fill=white] (0,-2) circle [radius=0.75];
\node at (0,0) [] {$X$};
}
\dots
\tikz[scale=0.8, baseline=0ex] {
\draw (-0.75,-0.75) -- (0.75,-0.75) -- (0,0.75) -- cycle;
\node at (0,0) [below] {$X^{-1}$};
}
\tikz[scale=0.8, baseline=0ex] {
\draw (-0.75,-2.75) -- (0.75,-2.75) -- (0,-1.25) -- cycle;
\node at (0,0) [] {$\dots$};
}
\dots
\end{align*}
Therefore we can formulate the following proposition.

\begin{prop}\label{u-v} Let $X$ be a cluster $\mathcal{X}$-variable originally on level $\alpha$ neither in the $u_{<k}$-part nor in the $v$-part. Then the exponent of $X$ in the circle on level $\beta$ in the final picture of $\overline{u_{<k+1}}^{-1}x\overline{v^{-1}}$ is $\delta_{\alpha\beta}$.
\end{prop}

Lastly, if $X$ is a cluster $\mathcal{X}$-variable on level $\alpha$ in the $u_{<k}$-part, then by going over similar computations as we did for $x\overline{v^{-1}}$, one can arrive at statements analogous to Corollary \ref{3.16} and Corollary \ref{3.10}, which we summarize in the following propositions; their proofs will be left as an exercise for the readers.

\begin{prop}\label{3.21} Let $X$ be a cluster $\mathcal{X}$-variable originally contained in the $u_{<k}$-part on level $\alpha$. Let $\gamma=-\alpha(k)$ be a simple root (and hence $\alpha(k)$ is a letter in the $u$-part of the reduced word $\vec{i}$). Let $p_\beta$ be the exponent of $X$ in the circle on level $\beta$ in the final picture of $\overline{u_{<k+1}}^{-1}x\overline{v^{-1}}$ and let $q_\beta$ are the counterparts of $p_\beta$ for $\overline{u_{<k}}^{-1}x\overline{v^{-1}}$. Then
\[
p_\beta=q_\beta-C_{\gamma\beta}q_\gamma.
\]
\end{prop}

\begin{prop}\label{3.20} Let $X$ be a cluster $\mathcal{X}$-variable originally contained in the $u_{<k}$-part. Let $\gamma=-\alpha(k)$ be a simple root. Fix a simple root $\alpha$ (not necessarily the level where $X$ is originally on). Let $p_\beta$ be the exponent of $X$ in the circle on level $\beta$ in the final picture of $\overline{u_{<k+1}}^{-1}x\overline{v^{-1}}$, and let $q_\beta$ be their counterparts for $\overline{u_{<k+1}}^{-1}x\overline{v^{-1}}$. Then
\[
\sum_\beta p_\beta\left(C^{-1}\right)_{\beta\alpha}=-\delta_{\alpha\gamma}q_\alpha+\sum_\beta q_\beta\left(C^{-1}\right)_{\beta\alpha}.
\]
In other words, unless $\alpha=\gamma$, the quantity $\sum_\beta p_\beta \left(C^{-1}\right)_{\beta\alpha}$ is invariant under an additional multiplication of $\overline{s}_\gamma^{-1}$ on the left.
\end{prop}

Propositions \ref{3.17}, \ref{u-v}, \ref{3.21}, \ref{3.20}, and their symmetric versions for group elements $\overline{u}^{-1}x\overline{\left(v_{>k}\right)^{-1}}$ will play significant roles in the proof of the main theorem of this subsection, which is the following.

\begin{prop}\label{1sthalf} For our choice of the reduced word $\vec{i}$ (whose $u$-part comes before its $v$-part), we have 
\[
\deg_{X_a}\Xi_\vec{i}^*\left(X_b\right)=-\delta_{ab}
\]
for any cluster $\mathcal{X}$-variables $X_a$ and $X_b$.
\end{prop}
\begin{proof} In the proof we will mainly focus on two cases depending on where $X_b$ is: (i) $X_b$ is contained in the $u$-part, and (ii) $X_b$ is in neither the $u$-part nor the $v$-part; the case where $X_b$ is contained in the $v$-part can be proved by a symmetric argument of case (i).

Let's consider case (i) first. Suppose the vertex (string) $b$ is on level $\alpha$ cut out by nodes $\alpha(i)$ and $\alpha(j)$ as below. 
\[ \tikz{
    \node (1) at (1,1) [] {$\alpha(i)$};
    \node (2) at (3,1) [] {$\alpha(j)$};
    \draw (0,1) -- (1) -- node[above]{$b$} (2) -- (4,1);
    }
\]
From the factorization $\Xi_\vec{i}=\psi_\vec{i}\circ \chi_\vec{i}=p_\vec{i}\circ \psi_\vec{i}\circ s\circ \chi_\vec{i}$ we learn that
\[
\Xi_\vec{i}^*\left(X_b\right)=\frac{\Delta_\alpha\left(\overline{u_{<j+1}}^{-1}x\overline{v^{-1}}\right)\prod_{\beta\neq \alpha} \left(\Delta_\beta\left(\overline{u_{<i}}^{-1}x\overline{v^{-1}}\right)\right)^{-C_{\beta\alpha}}}{\Delta_\alpha\left(\overline{u_{<i}}^{-1}x\overline{v^{-1}}\right)\prod_{\beta\neq \alpha} \left(\Delta_\beta\left(\overline{u_{<j+1}}^{-1}x\overline{v^{-1}}\right)\right)^{-C_{\beta\alpha}}}.
\]

If $a$ is a vertex on level $\gamma$ not in the $u_{<j}$-part, then using Propositions \ref{3.17} and \ref{u-v} it is not hard to see from the final pictures that multiplying $\overline{u_{<i}}^{-1}$ and $\overline{u_{<j+1}}^{-1}$ on the left of $x\overline{v^{-1}}$ yield the same degrees of $X_a$ in the middle, and hence 
\[
\deg_{X_a}\Xi_\vec{i}^*\left(X_b\right)=0.
\]

This reduces this case to the situation where $a$ is a vertex contained in the $u_{<j}$-part. Under such circumstance, it is helpful to rewrite $\Xi_\vec{i}^*\left(X_b\right)$ as 
\begin{equation}\label{degDT}
\Xi_\vec{i}^*\left(X_b\right)=\frac{\Delta_\alpha\left(\overline{u_{<i}}^{-1}x\overline{v^{-1}}\right)\prod_\beta\left(\Delta_\beta\left(\overline{u_{<j+1}}^{-1}x\overline{v^{-1}}\right)\right)^{C_{\beta\alpha}}}{\Delta_\alpha\left(\overline{u_{<j+1}}^{-1}x\overline{v^{-1}}\right)\prod_\beta \left(\Delta_\beta\left(\overline{u_{<i}}^{-1}x\overline{v^{-1}}\right)\right)^{C_{\beta\alpha}}}.
\end{equation}
Such expression may still look complicated, but it is actually easy to compute, especially if we only care about the leading power: note that if one tries to compute $\prod_\beta\left( \Delta_\beta \left(n_-\left(\prod_\mu t_\mu^{H^\mu}\right)n_+\right)\right)^{C_{\beta\alpha}}$ for some $n_\pm \in N_\pm$ and $t_\mu\in \mathbb{C}^*$, all you need is just to realize that
\[
\sum_\beta C_{\beta\alpha}\inprod{H^\mu}{\omega_\beta}=\sum_{\beta,\nu} C_{\beta\alpha}\left(C^{-1}\right)_{\mu\nu}\inprod{H_\nu}{\omega_\beta}=\delta_{\alpha\mu},
\]
which implies immediately that
\[
\prod_\beta\left( \Delta_\beta \left(n_-\left(\prod_\mu t_\mu^{H^\mu}\right)n_+\right)\right)^{C_{\beta\alpha}}=t_\alpha.
\]

Now let's suppose that in the final picture of $\overline{u_{<i}}x\overline{v^{-1}}$ the exponent of $X_a$ in the circle on level $\beta$ is $p_\beta$ for each $\beta$, and in the final picture of $\overline{u_{<j+1}}^{-1}x\overline{v^{-1}}$ the exponent of $X_a$ in the circle on level $\beta$ is $q_\beta$ for each $\beta$. From Equation \ref{degDT} we see that 
\begin{equation}
\begin{split}\label{degree}
\deg_{X_a}\Xi_\vec{i}^*\left(X_b\right)=&\inprod{\sum_\beta p_\beta H^\beta}{\omega_\alpha}+q_\alpha-\inprod{\sum_\beta q_\beta H^\beta}{\omega_\alpha}-p_\alpha\\
=&\left(\sum_\beta p_\beta\left(C^{-1}\right)_{\beta\alpha}\right)+q_\alpha-\left(\sum_\beta q_\beta \left(C^{-1}\right)_{\beta\alpha}\right)-p_\alpha.
\end{split}
\end{equation}

Let's further divide this case into three subcases:
\begin{itemize}
    \item $X_a$ is in the $u_{<i}$-part;
    \item $X_a$ is not in $u_{<i}$-part but is in $u_{<j-1}$-part;
    \item $X_a$ is not in $u_{<j-1}$-part but is in $u_{<j}$-part.
\end{itemize}
Note that the second subcase implies that $X_a$ is not on level $\alpha$ (and hence $\gamma\neq \alpha$) since $\alpha(i)$ and $\alpha(j)$ are two consecutive nodes, and the last case allows only one possible choice of $X_a$, namely $X_a=X_b$.

Let's first look at the subcase where $X_a$ is in the $u_{<i}$-part. Note that to go from $\overline{u_{<i}}^{-1}x\overline{v^{-1}}$ to $\overline{u_{<j+1}}^{-1}x\overline{v^{-1}}$ we need to first multiply $\overline{s}_{-\alpha(i)}^{-1}=\overline{s}_\alpha^{-1}$ on the left to get to $\overline{u_{<i+1}}^{-1}x\overline{v^{-1}}$, and then multiply by a sequence of $\overline{s}_\beta^{-1}$ on the left for various $\beta\neq \alpha$ to get to $\overline{u_{<j}}^{-1}x\overline{v^{-1}}$, and then finally multiply another $\overline{s}_{-\alpha(j)}^{-1}=\overline{s}_\alpha^{-1}$ on the left to get to $\overline{u_{<j+1}}^{-1}x\overline{v^{-1}}$. Let's denote the counterparts of $p_\beta$ in $\overline{u_{<i+1}}^{-1}x\overline{v^{-1}}$ as $p'_\beta$ and denote the counterparts of $q_\beta$ in $\overline{u_{<j}}^{-1}x\overline{v^{-1}}$ as $q'_\beta$. Since $X_a$ is in the $u_{<i}$-part and $i<j$, it follows from Proposition \ref{3.20} (with $\alpha$ being the choice of simple root $\alpha$ in the statement) that
\[ \sum_\beta q_\beta \left(C^{-1}\right)_{\beta\alpha}= -q'_\alpha+\sum_\beta q'_\beta \left(C^{-1}\right)_{\beta\alpha} \quad \text{and}\quad \sum_\beta p'_\beta \left(C^{-1}\right)_{\beta\alpha}= -p_\alpha+\sum_\beta p_\beta \left(C^{-1}\right)_{\beta\alpha}.
\]
But then by Proposition \ref{3.21} we also have $q_\alpha=q'_\alpha-C_{\alpha\alpha}q'_\alpha=-q'_\alpha$; therefore the first equality above can be turned into
\[
-q_\alpha+\sum_\beta q_\beta \left(C^{-1}\right)_{\beta\alpha}=\sum_\beta q'_\beta \left(C^{-1}\right)_{\beta\alpha}.
\]
But then we also know from Proposition \ref{3.20} that $\sum_\beta q'_\beta\left(C^{-1}\right)_{\beta\alpha}=\sum_\beta p'_\beta\left(C^{-1}\right)_{\beta\alpha}$ because the sequence of multiplications by $\overline{s}_\beta^{-1}$ happens between the multiplications by $\overline{s}_{-\alpha(i)}^{-1}$ and $\overline{s}_{-\alpha(j)}^{-1}$ only involve $\beta\neq \alpha$. Therefore we can conclude that
\[
-p_\alpha+\sum_\beta p_\beta\left(C^{-1}\right)_{\beta\alpha}=\sum_\beta p'_\beta\left(C^{-1}\right)_{\beta\alpha}=\sum_\beta q'_\beta\left(C^{-1}\right)_{\beta\alpha}=-q_\alpha+\sum_\beta q_\beta \left(C^{-1}\right)_{\beta\alpha},
\]
which implies that $\deg_{X_a}\Xi_\vec{i}^*\left(X_b\right)=0$.

Next let's turn to the second subcase, with $X_a$ being not in $u_{<i}$-part but in $u_{<j-1}$-part. By Proposition \ref{u-v} we know that $p_\beta=\delta_{\beta\gamma}$; in particular, since $\gamma\neq \alpha$ in this subcase, $p_\alpha=0$. Again let $q'_\beta$ as defined in the last subcase; then by the same argument we can deduce that
\[
-p_\alpha+\sum_\beta p_\beta\left(C^{-1}\right)_{\beta\alpha}=\sum_\beta p_\beta\left(C^{-1}\right)_{\beta\alpha}=\sum_\beta q'_\beta\left(C^{-1}\right)_{\beta\alpha}=-q_\alpha+\sum_\beta q_\beta \left(C^{-1}\right)_{\beta\alpha},
\]
which also implies that $\deg_{X_a}\Xi_\vec{i}^*\left(X_b\right)=0$.

This leaves us with one remaining subcase, namely $X_a=X_b$. By Proposition \ref{u-v} and Proposition \ref{3.20} we know that $p_\beta=\delta_{\alpha\beta}$ and $q_\beta=\delta_{\alpha\beta}-C_{\alpha\beta}$; plug these into Equation \eqref{degree} we get
\begin{align*}
\deg_{X_b}\Xi_\vec{i}^*\left(X_b\right)=&\left(\sum_\beta \delta_{\alpha\beta}\left(C^{-1}\right)_{\beta\alpha}\right)+(-1)-\left(\sum_\beta \left(\delta_{\alpha\beta}-C_{\alpha\beta}\right)\left(C^{-1}\right)_{\beta\alpha}\right)-1\\ =&\left(C^{-1}\right)_{\alpha\alpha}-1-\left(\left(C^{-1}\right)_{\alpha\alpha}-1\right)-1\\
=&-1,
\end{align*}
which concludes the proof of case (i).

Let's now turn to case (ii). Again suppose the vertex (string) $b$ is on level $\alpha$ cut out by nodes $\alpha(i)$ and $\alpha(j)$, but this time with $-\alpha(i)=\alpha(j)=\alpha$. Then we know that
\begin{equation} \label{bigchunk}
\begin{split}
\Xi_\vec{i}^*\left(X_b\right)=&\frac{\left(\prod_{\beta\neq \alpha}\left(\Delta_\beta\left(\overline{u_{<i}}^{-1}x\overline{v^{-1}}\right)\right)^{-C_{\beta\alpha}}\right)\left(\prod_{\beta\neq \alpha}\left(\Delta_\beta\left(\overline{u}^{-1}x\overline{\left(v_{>j}\right)^{-1}}\right)\right)^{-C_{\beta\alpha}}\right)}{\Delta_\alpha\left(\overline{u_{<i}}^{-1}x\overline{v^{-1}}\right)\Delta_\alpha\left(\overline{u}^{-1}x\overline{\left(v_{>j}\right)^{-1}}\right)\prod_{\beta\neq \alpha}\left(\Delta_\beta\left(\overline{u}^{-1}x\overline{v^{-1}}\right)\right)^{-C_{\beta\alpha}}}\\
=&\frac{\Delta_\alpha\left(\overline{u_{<i}}^{-1}x\overline{v^{-1}}\right)\Delta_\alpha\left(\overline{u}^{-1}x\overline{\left(v_{>j}\right)^{-1}}\right)\prod_\beta\left(\Delta_\beta\left(\overline{u}^{-1}x\overline{v^{-1}}\right)\right)^{C_{\beta\alpha}}}{\left(\Delta_\alpha\left(\overline{u}^{-1}x\overline{v^{-1}}\right)\right)^2\left(\prod_\beta\left(\Delta_\beta\left(\overline{u_{<i}}^{-1}x\overline{v^{-1}}\right)\right)^{C_{\beta\alpha}}\right)\left(\prod_\beta\left(\Delta_\beta\left(\overline{u}^{-1}x\overline{\left(v_{>j}\right)^{-1}}\right)\right)^{C_{\beta\alpha}}\right)}.
\end{split}
\end{equation}

This time the easy ones are the cluster variables $X_a$ that are in neither the $u$-part nor the $v$-part, for which there is at most one on each level. From the Proposition \ref{u-v} we know that their exponents do not change as we move triangles across. Suppose $X_a$ is on level $\gamma$ to begin with. Then by simple computation we get that
\[
\deg_{X_a}\Xi_\vec{i}^*\left(X_b\right)=2\inprod{H^\gamma}{\omega_\alpha}+\delta_{\alpha\gamma}-2\inprod{H^\gamma}{\omega_\alpha}-2\delta_{\alpha\gamma}=-\delta_{\alpha\gamma}=-\delta_{ab}.
\]

Therefore to finish case (ii), we only need to show that $\deg_{X_a}\Xi_\vec{i}^*\left(X_b\right)=0$ for any $X_a$ on level $\gamma$ that is in either the $u$-part or the $v$-part. Due to the symmetry of the arguments, we will only consider the subcase where $X_a$ is in the $u$-part of the reduced word. Suppose in the final picture of $\overline{u}^{-1}x\overline{v^{-1}}$ the exponent of $X_a$ in the circle on level $\beta$ is $p_\beta$ for each $\beta$. Then we learn from (the symmetric version of) Proposition \ref{3.17} that in the final picture of $\overline{u}^{-1}x\overline{\left(v_{>j}\right)^{-1}}$, the exponent of $X_a$ in the circle on level $\beta$ is also $p_\beta$ for each $\beta$. Let $q_\beta$ be the counterparts of $p_\beta$ for $\overline{u_{<i}}^{-1}x\overline{v^{-1}}$. Then from equation \eqref{bigchunk} we can deduce that
\begin{equation}\label{degreeofchunk}
\begin{split}
\deg_{X_a} \Xi_\vec{i}^*\left(X_b\right)=&\inprod{\sum_\beta q_\beta H^\beta}{\omega_\alpha}+\inprod{\sum_\beta p_\beta H^\beta}{\omega_\alpha}+p_\alpha-2\inprod{\sum_\beta p_\beta H^\beta}{\omega_\alpha}-q_\alpha-p_\alpha\\
=&\left(\sum_\beta q_\beta \left(C^{-1}\right)_{\beta\alpha}\right)-\left(\sum_\beta p_\beta \left(C^{-1}\right)_{\beta\alpha}\right)-q_\alpha.
\end{split}
\end{equation}

For the final picture of $\overline{u_{<i}}^{-1}x\overline{v^{-1}}$, we need to divide it into two subcases depending on whether or not $X_a$ is in the $u_{<i}$-part. In particular, if $X_a$ is not in the $u_{<i}$-part then $\gamma\neq \alpha$, because $\alpha(i)$ is supposely the last opposite simple root on level $\alpha$. 

If $X_a$ is not in $u_{<i}$-part, we know from Proposition \ref{u-v} that $p_\beta=q_\beta=\delta_{\beta\gamma}$ and particularly $p_\alpha=\delta_{\alpha\gamma}=0$; plugging these into Equation \eqref{degreeofchunk} we see immediately that $\deg_{X_a}\Xi_\vec{i}^*\left(X_b\right)=0$.

If $X_a$ is in $u_{<i}$-part, we know from Proposition \ref{3.21} that $p_\beta=q_\beta-C_{\gamma\beta}q_\gamma$. Plug this into Equation \eqref{degreeofchunk} we see that
\begin{align*}
\deg_{X_a}\Xi_\vec{i}^*\left(X_b\right)=&\left(\sum_\beta q_\beta\left(C^{-1}\right)_{\beta\alpha}\right)-\left(\sum_\beta \left(q_\beta-C_{\gamma\beta}q_\gamma\right)\left(C^{-1}\right)_{\beta\alpha}\right)-q_\alpha\\
=&\left(\sum_\beta q_\beta\left(C^{-1}\right)_{\beta\alpha}\right)-\left(\sum_\beta q_\beta\left(C^{-1}\right)_{\beta\alpha}\right)+q_\alpha-q_\alpha\\
=&0.
\end{align*}
This finally completes the whole proof of Proposition \ref{1sthalf}.
\end{proof}

\subsection{Cluster Nature of \texorpdfstring{$\Xi_\vec{i}$}{}}\label{3.3} In this subsection we will continue to let $\vec{i}$ be the reduced word we have used in last subsection and prove the claim that the map $\Xi_\vec{i}:=\psi_\vec{i}\circ \chi_\vec{i}$, as a birational automorphism on $\mathcal{X}^{u,v}$, is indeed a cluster transformation (see Definition \ref{clustertran}).

To simplify notations in this subsection, we will denote letters of the reduced word $\vec{i}$ as follows:
\[
\vec{i}:=(\underbrace{-\alpha(1),\dots, -\alpha(m)}_\text{$u$-part}, \underbrace{\beta(1),\dots, \beta(n)}_\text{$v$-part}).
\]
In this way, $\left(\alpha(1),\dots, \alpha(m)\right)$ will be a reduced word for the Weyl group element $u$, and $\left(\beta(1),\dots, \beta(n)\right)$ will be a reduced word for the Weyl group element $v$. 

Let's recap what we know about $\Xi_\vec{i}$ so far. From the definition $\Xi_\vec{i}:=\psi_\vec{i}\circ \chi_\vec{i}$ and Proposition \ref{clustertwist} we know that the top half of the following diagram commutes, and from Proposition \ref{twistflag} we know that the bottom half of the following diagram commutes, so we can conclude that the whole diagram commutes.
\[
\xymatrix{\mathcal{X}^{u,v} \ar@{-->}[r]^{\Xi_\vec{i}} \ar[d]_\cong^\chi & \mathcal{X}^{u,v} \ar[d]^\cong_\chi \\
H\backslash G^{u,v}/H \ar[r]^{\iota\circ \tw } \ar[d]_\cong & H\backslash G^{u,v}/H \ar[d]^\cong \\
\conf^{u,v}(\mathcal{B})\ar[r]_\xi & \conf^{u,v}(\mathcal{B})}
\]
In particular, the vertical composite maps are mapping $\left(X_a\right)$ to the configuration $\left[\vcenter{\vbox{\xymatrix{B_+ \ar[r]^u \ar@{-}[d] & x.B_+ \ar@{-}[d] \\
B_- \ar[r]_{v^*} & x.B_-}}}\right]$ for any $x$ with $H\backslash x/H=\chi_\vec{i}\left(X_a\right)$.

Let's also recall from \eqref{etaxi} what the map $\xi:\conf^{u,v}\left(\mathcal{B}\right)\rightarrow \conf^{u,v}(\mathcal{B})$ does. Starting with a configuration of four Borel subgroups $\left[B_1,B_2,B_3,B_4\right]$, we break the two vertical edges using Corollary \ref{2.8}, and then take out the other square diagram, and then apply the $*$ involution to it.
\[
\left[\vcenter{\vbox{\xymatrix{B_1 \ar[r]^u \ar@{-}[d] & B_2 \ar@{-}[d] \\
B_3 \ar[r]_{v^*} & B_4}}}\right]
\quad \quad \rightsquigarrow \quad \quad 
\left[\vcenter{\vbox{\xymatrix{&B_1 \ar[r]^u \ar@{-}[dd] & B_2 \ar@{-}[dd] & \\
B_5 \ar[ur]^{u^l}\ar@{-}[urr] & & & B_6 \ar[ul]_{v} \\
& B_3 \ar[r]_{v^*} \ar@{-}[urr] \ar[ul]^{u^*} & B_4 \ar[ur]_{v^l} &}}}\right]
\quad \quad \rightsquigarrow \quad \quad 
\left[\vcenter{\vbox{\xymatrix{B_3^* \ar[r]^u \ar@{-}[d] & B_5^* \ar@{-}[d] \\
B_6^* \ar[r]_{v^*} & B_2^*}}}\right]
\]

Let's put aside the $*$ involution for now and focus on how to get the new square diagram $\vcenter{\vbox{\xymatrix{B_3 \ar[r]^{u^*} \ar@{-}[d] & B_5 \ar@{-}[d] \\
B_6 \ar[r]_{v} & B_2}}}$ from the old one $\vcenter{\vbox{\xymatrix{B_1 \ar[r]^u \ar@{-}[d] & B_2 \ar@{-}[d] \\
B_3 \ar[r]_{v^*} & B_4}}}$. We claim that this can be done via a sequence of ``tilting'', which we represent by the diagram below. One can think of the diagram below as a refinement of the hexagon diagram in the middle above: we just break the arrows $u^*$, $u$, $v^*$, and $v$ into even more pieces using Corollary \ref{2.8} and the reduced words $\left(\alpha(1),\dots, \alpha(m)\right)$ and $\left(\beta(1),\dots, \beta(n)\right)$ for $u$ and $v$ respectively. Note that in particular, the new Borel subgroups satisfy the relative position condition $\xymatrix{B_u^{(k)} \ar@{-}[r] & B_{u^*}^{(k)}}$ and $\xymatrix{B_v^{(l)} \ar@{-}[r] & B_{v^*}^{(l)}}$ for all $1\leq k\leq m-1$ and $1\leq l\leq n-1$.
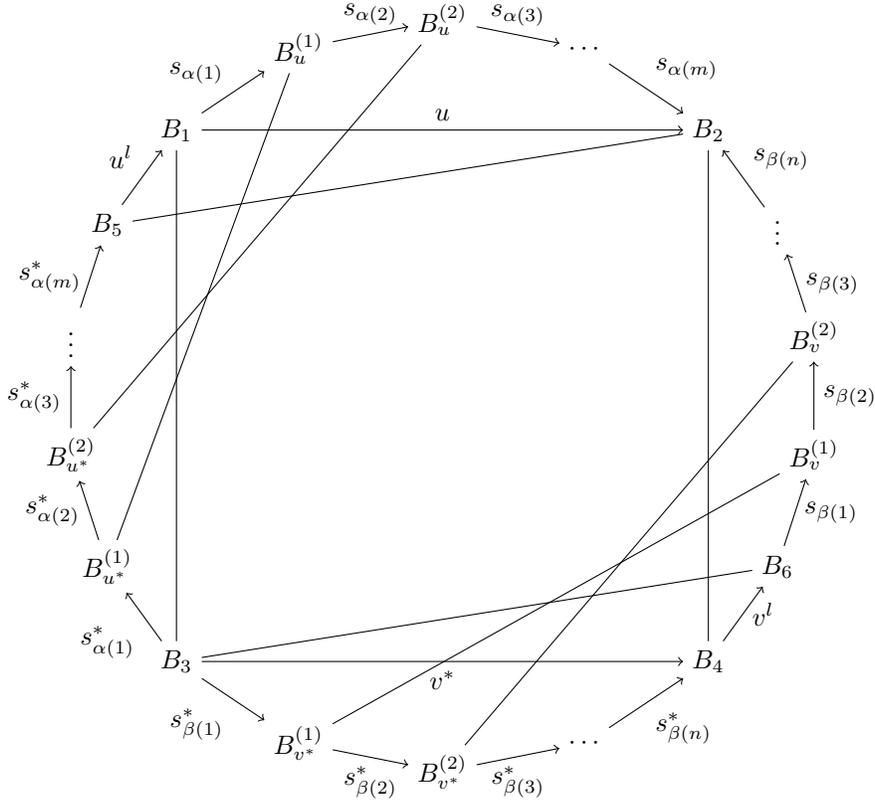
\begin{equation}\label{diagram}
\tikz{
\node (u0) at +(135:5) [] {$B_1$};
\node (u1) at +(112.5:5) [] {$B_u^{(1)}$};
\node (u2) at +(90:5) [] {$B_u^{(2)}$};
\node (u3) at +(67.5:5) [] {$\cdots$};
\node (um) at +(45:5) [] {$B_2$};
\node (u*0) at +(-135:5) [] {$B_3$};
\node (u*1) at +(-153:5) [] {$B_{u^*}^{(1)}$};
\node (u*2) at +(-171:5) [] {$B_{u^*}^{(2)}$};
\node (u*3) at +(171:5) [] {$\vdots$};
\node (u*m) at +(153:5) [] {$B_5$};
\node (v*l) at +(-45:5) [] {$B_4$};
\node (v*1) at +(-112.5:5) [] {$B_{v^*}^{(1)}$};
\node (v*2) at +(-90:5) [] {$B_{v^*}^{(2)}$};
\node (v*3) at +(-67.5:5) [] {$\cdots$};
\node (v0) at +(-27:5) [] {$B_6$};
\node (v1) at +(-9:5) [] {$B_v^{(1)}$};
\node (v2) at +(9:5) [] {$B_v^{(2)}$};
\node (v3) at +(27:5) [] {$\vdots$};
\draw [->] (u0) -- (u1) node [midway, above left] {$s_{\alpha(1)}$};
\draw [->] (u1) -- (u2) node [midway, above] {$s_{\alpha(2)}$};
\draw [->] (u2) -- (u3) node [midway, above] {$s_{\alpha(3)}$};
\draw [->] (u3) -- (um) node [midway, above right] {$s_{\alpha(m)}$};
\draw [->] (u*0) -- (u*1) node [midway, below left] {$s_{\alpha(1)}^*$};
\draw [->] (u*1) -- (u*2) node [midway, left] {$s_{\alpha(2)}^*$};
\draw [->] (u*m) -- (u0) node [midway, above left] {$u^l$};
\draw [->] (u*2) -- (u*3) node [midway, left] {$s_{\alpha(3)}^*$};
\draw [->] (u*3) -- (u*m) node [midway, left] {$s_{\alpha(m)}^*$};
\draw [->] (u*0) -- (v*l) node [midway, below] {$v^*$};
\draw (um) -- (v*l);
\draw (u0) -- (u*0);
\draw [->] (u0) -- (um) node [midway, above] {$u$};
\draw (u*1) -- (u1);
\draw (u*2) -- (u2);
\draw (u*m) -- (um);
\draw [->] (u*0) -- (v*1) node [midway, below left] {$s_{\beta(1)}^*$};
\draw [->] (v*1) -- (v*2) node [midway, below] {$s_{\beta(2)}^*$};
\draw [->] (v*2) -- (v*3) node [midway, below] {$s_{\beta(3)}^*$};
\draw [->] (v*3) -- (v*l) node [midway, below right] {$s_{\beta(n)}^*$};
\draw [->] (v*l) -- (v0) node [midway, right] {$v^l$};
\draw [->] (v0) -- (v1) node [midway, right] {$s_{\beta(1)}$};
\draw [->] (v1) -- (v2) node [midway, right] {$s_{\beta(2)}$};
\draw [->] (v2) -- (v3) node [midway, right] {$s_{\beta(3)}$};
\draw [->] (v3) -- (um) node [midway, above right] {$s_{\beta(n)}$};
\draw (u*0) -- (v0);
\draw (v*1) -- (v1);
\draw (v*2) -- (v2);
}
\end{equation}

But how do we realize such tilting for the quadruple $\vcenter{\vbox{\xymatrix{B_+ \ar[r]^u \ar@{-}[d] & x.B_+ \ar@{-}[d] \\
B_- \ar[r]_{v^*} & x.B_-}}}$ corresponding to $H\backslash x/H=\chi_\vec{i}\left(X_a\right)$? First, we observe that since the map $\xi$ is well-defined for configurations of quadruples of Borel subgroups, we have the freedom to use the diagonal $G$-action to move the initial quadruple of Borel subgroups into more convenient relative positions. Note that since the reduced word $\vec{i}$ has its $u$-part before its $v$-part, we know that any lift $x$ of $H\backslash x/H=\psi_\vec{i}\left(X_a\right)$ is already Gaussian decomposable. Therefore we can draw a vertical line across the string diagram, separating the factors occur in the defining product $\psi_\vec{i}\left(X_a\right)$ into the $u$-part and the $v$-part; we hence define $x_-$ to be the product of the factors in the $u$-part and $x_+$ be the product of the factors in the $v$-part. Note that such separation is not unique because the $H$-factors in the middle can go in either the $u$-part of the $v$-part; but regardless of the choices, we always have 
\[
x_-\in B_+uB_+\cap B_- \quad \text{and} \quad x_+\in B_+\cap B_-vB_-.
\] 
Since we also know that $\xi$ does not depend on the lift $x$, to simplify notation we may also assume that the frozen cluster $\mathcal{X}$-variables associated to the open strings are all 1.

\begin{exmp}\label{exmp 3.20} Let's keep a running example for demonstration purpose. Consider the group $\PGL_4$ with the pair of Weyl group elements $\left(w_0,w_0\right)$. Let $\vec{i}$ be the reduced word $\left(-1,-2,-1,-3,-2,-1,3,2,1,3,2,3\right)$. The string diagram associated to this reduced word looks like the following.
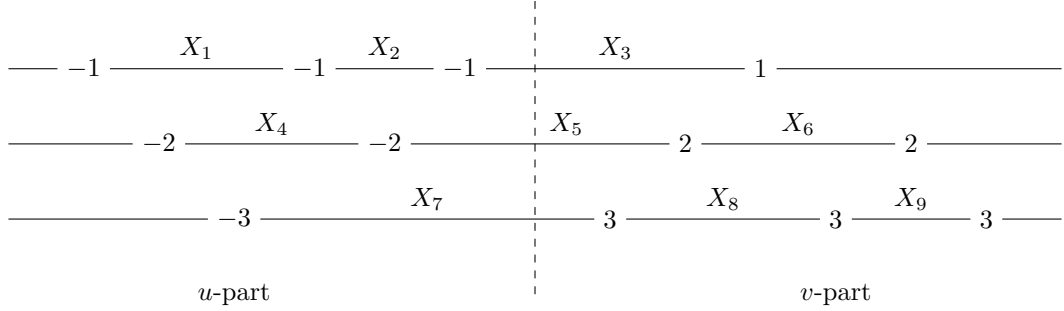
\[
\tikz{
\node (11) at (1,2) [] {$-1$};
\node (21) at (2,1) [] {$-2$};
\node (31) at (3,0) [] {$-3$};
\node (12) at (4,2) [] {$-1$};
\node (22) at (5,1) [] {$-2$};
\node (13) at (6,2) [] {$-1$};
\node (32) at (8,0) [] {$3$};
\node (23) at (9,1) [] {$2$};
\node (14) at (10,2) [] {$1$};
\node (33) at (11,0) [] {$3$};
\node (24) at (12,1) [] {$2$};
\node (34) at (13,0) [] {$3$};
\draw (0,2) -- (11) -- node [above] {$X_1$} (12) -- node [above] {$X_2$} (13) -- node [above] {$X_3$} (14) -- (14,2);
\draw (0,1) -- (21) -- node [above] {$X_4$} (22) -- node [above right] {$X_5$} (23) -- node [above] {$X_6$} (24) -- (14,1);
\draw (0,0) -- (31) -- node [above] {$X_7$} (32) -- node [above] {$X_8$} (33) -- node [above] {$X_9$} (34) -- (14,0);
\draw [dashed] (7,-1) -- (7,3);
\node at (3,-1) [] {$u$-part};
\node at (11,-1) [] {$v$-part};
}
\]
Then according to the description above, our choice of lift $x$ of $H\backslash x/H=\chi_\vec{i}\left(X_a\right)$ (with all frozen variable assigned to be 1) is factored as $x_-x_+$ where
\[
x_-=e_{-1}X_1^{H^1}e_{-2}X_4^{H^2} e_{-3}e_{-1} X_2^{H^1} e_{-2} X_7^{H^3}e_{-1}=\begin{pmatrix} X_1X_2X_4X_7 & 0 & 0 & 0 \\ 
X_4X_7\left(1+X_2+X_1X_2\right) & X_4X_7 & 0 & 0 \\
X_7 \left(1+X_4+X_2X_4\right) & X_7 \left(1+X_4\right) & X_7 & 0 \\
X_7 & X_7 & X_7 & 1\end{pmatrix},
\]
\begin{align*}
x_+=&X_5^{H^2}X_3^{H^1}  e_3e_2X_8^{H^3}e_1X_6^{H^2}e_3e_2X_9^{H^3}e_3\\
=&\begin{pmatrix} X_3X_5X_6X_8X_9 & X_3X_5X_6X_8X_9 & X_3X_5X_6X_8X_9 & X_3X_5X_6X_8X_9 \\ 
0 & X_5X_6X_8X_9 & X_5X_8x_9\left(1+X_6\right) & X_5X_8\left(1+X_9+X_6X_9\right) \\
0 & 0 & X_8 X_9 & 1+X_8+X_8X_9 \\
0 & 0 & 0 & 1
\end{pmatrix} .
\end{align*}
\end{exmp}

Now back to the general case. If we act on the initial quadruple by $x_-^{-1}$, we see that as configurations,
\[
\left[\vcenter{\vbox{\xymatrix{B_+ \ar[r]^u \ar@{-}[d] & x.B_+ \ar@{-}[d] \\
B_- \ar[r]_{v^*} & x.B_+}}}\right]=\left[\vcenter{\vbox{\xymatrix{x_-^{-1}.B_+ \ar[r]^u \ar@{-}[d] & B_+ \ar@{-}[d] \\
B_- \ar[r]_{v^*} & x_+.B_-}}}\right].
\]
So instead of tilting the initial quadruple of Borel subgroups, we decide to tilt the quadruple 
\[
\vcenter{\vbox{\xymatrix{B_1 \ar[r]^u \ar@{-}[d] & B_2 \ar@{-}[d] \\
B_3 \ar[r]_{v^*} & B_4}}}=\vcenter{\vbox{\xymatrix{x_-^{-1}.B_+ \ar[r]^u \ar@{-}[d] & B_+ \ar@{-}[d] \\
B_- \ar[r]_{v^*} & x_+.B_-}}}.
\]

Let's describe how we can find $B_5$ first, and to do so we will focus on the upper left semicircle. Start with the edge 
\[
\left(\xymatrix{B_1\ar@{-}[r] & B_3}\right) \quad = \quad \left( \xymatrix{x_-^{-1}.B_+\ar@{-}[r] & B_-}\right) \quad =\quad \left(  \xymatrix{x_-^{-1}.B_+\ar@{-}[r] & x_-^{-1}.B_-}\right),
\]
we first change the $x_-^{-1}$ in front of the last pair of Borel subgroups to $\left(e_{\alpha(1)}e_{-\alpha(1)}.x_-\right)^{-1}$ and define
\[
\left(\xymatrix{B_u^{(1)}\ar@{-}[r] & B_{u^*}^{(1)}} \right)\quad := \quad \left( \xymatrix{\left(e_{\alpha(1)}e_{-\alpha(1)}^{-1}x_-\right)^{-1}. B_+ \ar@{-}[r] & \left(e_{\alpha(1)}e_{-\alpha(1)}^{-1}x_-\right)^{-1}.B_- }\right).
\]
We should check the relative position conditions $\xymatrix{B_1 \ar[r]^{s_{\alpha(1)}} & B_u^{(1)}}$ and $\xymatrix{B_3 \ar[r]^{s_{\alpha(1)}^*} & B_{u^*}^{(1)}}$ to make sure we are on the right track. But this is not too bad; note that
\begin{equation}\label{relative}
\begin{split}
\left(x_-^{-1}.B_+, \left(e_{\alpha(1)}e_{-\alpha(1)}^{-1}x_-\right)^{-1}.B_+\right)\sim \left(B_+, \left(e_{-\alpha(1)}e_{\alpha(1)}^{-1}\right).B_+\right)\sim \left(B_+, e_{-\alpha(1)}.B_+\right), \\
\left(x_-^{-1}.B_-, \left(e_{\alpha(1)}e_{-\alpha(1)}^{-1}x_-\right)^{-1}.B_-\right)\sim \left(B_-,\left(e_{-\alpha(1)}e_{\alpha(1)}^{-1}\right).B_-\right) \sim \left(B_-,e_{\alpha(1)}^{-1}.B_-\right);
\end{split}
\end{equation}
then our desired relative position conditions follow from the fact that
\[
e_{-\alpha(1)}\in B_+s_{\alpha(1)}B_+\cap N_- \quad \text{and} \quad e_{\alpha(1)}^{-1}\in N_+\cap B_-s_{\alpha(1)}B_-.
\]

But we should also say the reason why we choose $B_u^{(1)}$ and $B_{u^*}^{(1)}$ like this. Let's turn to the $u$-part of the original string diagram, which is also the string diagram for $x_-$ and see what happens when we multiply $x_-$ by $e_{\alpha(1)}e_{-\alpha(1)}^{-1}$ on the left. Note that with all frozen $\mathcal{X}$-variables set to 1, the first non-trivial factor in the defining product for $x_-$ is $e_{-\alpha(1)}$; thus multiplication by $e_{\alpha(1)}e_{-\alpha(1)}^{-1}$ on the left is the same as changing this factor from $e_{-\alpha(1)}$ to $e_{\alpha(1)}$. Then we can use a sequence of cluster mutations to move this new factor $e_{\alpha(1)}$ all the way to the right of the whole $u$-part. This implies that $e_{\alpha(1)}e_{-\alpha(1)}^{-1}x_-$ is in the image of the map 
\[
\chi_{\left(-\alpha(2), -\alpha(3), \dots, -\alpha(m), \alpha(1)\right)}.
\]
Note that this is very analogous to the triangle flipping and moving we have done in last subsection! Guided by such observation, we might lead to think that the next pair $\xymatrix{B_u^{(2)}\ar@{-}[r] & B_{u^*}^{(2)}}$ can be obtained by multiplying $e_{\alpha(2)}e_{-\alpha(2)}^{-1}$ on the left of $e_{\alpha(1)}e_{-\alpha(1)}^{-1}x_-$ and then move the new $\alpha(2)$ node to the right using cluster mutations; however, this is incorrect.

This is incorrect because there is a subtlety here: since $e_{\alpha(1)}e_{-\alpha(1)}^{-1}x_-\in G_{ad}$ is at the group level not the double quotient level, when we apply a sequence of cluster mutations to move the node $\alpha(1)$ from left to right, we inevitably introduce some non-trivial frozen variables on the left of the string diagram. Fortunately this is not too much a problem: we can let $h_1\in H$ be the product of the factors corresponding to these frozen variables, and $e_{\alpha(1)}e_{-\alpha(1)}^{-1}x_-$ can then be factored as $h_1e_{-\alpha(2)}\dots$; then instead of multiplying by $e_{\alpha(2)}e_{-\alpha(2)}^{-1}$ on the left we can multiply by $e_{\alpha(2)}e_{-\alpha(2)}^{-1}h_1^{-1}$ and play the same game again. In other words, we would define 
\begin{align*}
&\left(\xymatrix{B_u^{(2)} \ar@{-}[r] & B_{u^*}^{(2)}}\right)\\
:=\quad &\left(\xymatrix{\left(e_{\alpha(2)}e_{-\alpha(2)}^{-1}h_1^{-1}e_{\alpha(1)}e_{-\alpha(1)}^{-1}x_-\right)^{-1}.B_+ \ar@{-}[r] & \left(e_{\alpha(2)}e_{-\alpha(2)}^{-1}h_1^{-1}e_{\alpha(1)}e_{-\alpha(1)}^{-1}x_-\right)^{-1}.B_-}\right).
\end{align*}
One can verify the relative position conditions $\xymatrix{B_u^{(1)}\ar[r]^{s_{\alpha(2)}} & B_u^{(2)}}$ and $\xymatrix{B_{u^*}^{(1)}\ar[r]^{s_{\alpha(2)}^*} & B_{u^*}^{(2)}}$ in a computation similar to \eqref{relative}; to save space, let's just do $\xymatrix{B_u^{(1)}\ar[r]^{s_{\alpha(2)}} & B_u^{(2)}}$ here:
\begin{align*}
\left(\left(e_{\alpha(1)}e_{-\alpha(1)}^{-1}x_-\right)^{-1}.B_+,\left(e_{\alpha(2)}e_{-\alpha(2)}^{-1}h_1^{-1}e_{\alpha(1)}e_{-\alpha(1)}^{-1}x_-\right)^{-1}.B_+ \right)\sim & \left(B_+,\left(e_{\alpha(2)}e_{-\alpha(2)}^{-1}h_1^{-1}\right)^{-1}.B_+\right)\\
\sim & \left(B_+,e_{-\alpha(2)}.B_+\right),
\end{align*}
and the relative position condition follows from the fact that $e_{-\alpha(2)}\in B_+s_{\alpha(2)}B_+\cap N_-$.

\begin{exmp} Let's continue to use the example \ref{exmp 3.20} to see how we compute $h_1$. If we start with the factorization $x_-=e_{-1}X_1^{H^1}e_{-2}X_4^{H^2}e_{-3}e_{-1} X_2^{H^1} e_{-2} X_7^{H^3}e_{-1}$ and multiply $e_1e_{-1}^{-1}$ on the left, we see that
\[
e_1e_{-1}^{-1}x_-=e_1X_1^{H^1}e_{-2}X_4^{H^2}e_{-3}e_{-1} X_2^{H^1} e_{-2} X_7^{H^3}e_{-1},
\]
which corresponds to the following string diagram (note that the first node changes from $-1$ to $1$ due to the left multiplication by $e_1e_{-1}^{-1}$.
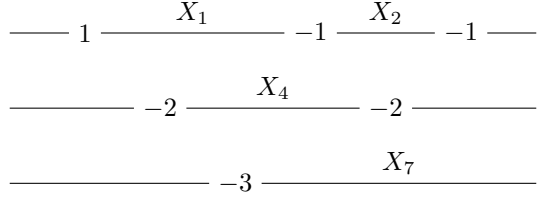
\[
\tikz{
\node (11) at (1,2) [] {$1$};
\node (21) at (2,1) [] {$-2$};
\node (31) at (3,0) [] {$-3$};
\node (12) at (4,2) [] {$-1$};
\node (22) at (5,1) [] {$-2$};
\node (13) at (6,2) [] {$-1$};
\draw (0,2) -- (11) -- node [above] {$X_1$} (12) -- node [above] {$X_2$} (13) -- (7,2);
\draw (0,1) -- (21) -- node [above] {$X_4$} (22) -- (7,1);
\draw (0,0) -- (31) -- node [above] {$X_7$} (7,0);
}
\]

We can use cluster mutations to reparametrize $e_1e_{-1}^{-1}x_-$ using the following string diagram (see Proposition \ref{birational} and Appendix \ref{B}).
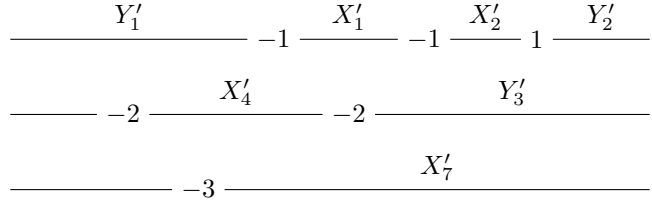
\[
\tikz{
\node (21) at (2,1) [] {$-2$};
\node (31) at (3,0) [] {$-3$};
\node (12) at (4,2) [] {$-1$};
\node (22) at (5,1) [] {$-2$};
\node (13) at (6,2) [] {$-1$};
\node (14) at (7.5,2) [] {$1$};
\draw (0.5,2) -- node [above] {$Y'_1$} (12) -- node [above] {$X'_1$} (13) -- node [above] {$X'_2$} (14) -- node [above] {$Y'_2$} (9,2);
\draw (0.5,1) -- (21) -- node [above] {$X'_4$} (22) -- node [above] {$Y'_3$} (9,1);
\draw (0.5,0) -- (31) -- node [above] {$X'_7$} (9,0);
}
\]
The new variables are related to the old one as below.
\begin{align*}
Y'_1=& 1+X_1 & X'_1=&\frac{1+X_2+X_1X_2}{X_1} & X'_2=& \frac{1}{X_2\left(1+X_1\right)} & Y'_2=&1+X_2+X_1X_2 \\
X'_4=&\frac{X_1X_4}{1+X_1} & Y'_3=&\frac{X_2\left(1+X_1\right)}{1+X_2+X_1X_2} & X'_7=&X_7 & &
\end{align*}
Note that before the frozen variables on the left are all 1, but after the sequence of mutations that moves the node 1 from left to right, we get a non-zero frozen variable on the top left that is $Y'_1$. Therefore we will need to set $h_1=\left(Y'_1\right)^{H^1}=\begin{pmatrix} 1+X_1 & 0 & 0 & 0 \\ 0 & 1 & 0 & 0 \\ 0 & 0 & 1 & 0 \\ 0 & 0 & 0 & 1\end{pmatrix}$ and apply $e_2e_{-2}^{-1}h_1^{-1}$ when we change the next left most node from $-2$ to $2$.
\end{exmp}

\begin{rmk}\label{left-mult} Note that by left multiplication $x_-\mapsto e_1e_{-1}^{-1}x_-$ we have moved from $\PGL_4^{w_0,e}$ to $\PGL_4^{s_1w_0,s_1}$. This may seem odd at the first glance, but we should point out that if we quotient out $H$ on the left, the seed for $H\left\backslash \PGL_4^{w_0,e}\right.$ (defined by the $u$-part of $\vec{i}$) and the seed for $H\left\backslash \PGL_4^{s_1w_0,s_1}\right.$ (defined by the string diagrams above) are actually the same, which can be encoded by the following quiver.
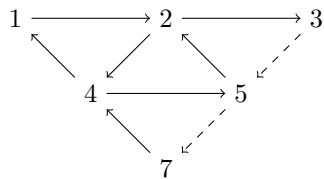
\[
\tikz{
\node (1) at (0,2) [] {$1$};
\node (2) at (2,2) [] {$2$};
\node (3) at (4,2) [] {$3$};
\node (4) at (1,1) [] {$4$};
\node (5) at (3,1) [] {$5$};
\node (7) at (2,0) [] {$7$};
\draw [->] (1) -- (2);
\draw [->] (2) -- (3);
\draw [->] (2) -- (4);
\draw [->] (4) -- (1);
\draw [->] (4) -- (5);
\draw [->] (7) -- (4);
\draw [->] (5) -- (2);
\draw [dashed, ->] (3) -- (5);
\draw [dashed, ->] (5) -- (7);
}
\]
This is because the node that changes from $-1$ to $1$ via the left multiplication $e_1e_{-1}^{-1}$ is the left-most node and it does not contribute to the unfrozen part of the seed. In other words, left multiplying by $e_1e_{-1}^{-1}$ can be seen as a \textbf{cluster isomorphism} once we quotient out $H$ on the left. This happens again when we left multiply $e_1e_{-1}^{-1}x_-$ by $e_2e_{-2}^{-1}h_1^{-1}$, which changes the underlying variety from the $H\left\backslash \PGL_4^{s_1w_0,s_1}\right.$ to $H\left\backslash \PGL_4^{s_2s_1w_0,s_1s_2}\right.$. A similar fact also applies to changing the right most node from a simple root to its opposite and at the same time quotienting out $H$ on the right. This observation is a key ingredient we use to prove the cluster nature of $\Xi_\vec{i}$ in Proposition \ref{2ndhalf}.
\end{rmk}

Now back to the general story. We can continue the tilting process recursively: after multiplying $e_{\alpha(k)}e_{-\alpha(k)}^{-1}h_{k-1}^{-1}$ on the left of 
\[
e_{\alpha(k-1)}e_{-\alpha(k-1)}^{-1}h_{k-2}^{-1}\dots e_{\alpha(1)}e_{-\alpha(1)}^{-1}x_-,
\]
and defining the pair $\xymatrix{B_u^{(k)}\ar@{-}[r] & B_{u^*}^{(k)}}$, we obtain a new element $h_k\in H$ which is the product of factors corresponding to the new frozen $\mathcal{X}$-variables emerged from the sequence of cluster mutations that moves the node $\alpha(k)$ behind the remaining opposite simple root nodes. Then we just need to define
\begin{align*}
&\left(\xymatrix{B_u^{(k+1)}\ar@{-}[r] & B_{u^*}^{(k+1)}}\right)\\
:=\quad & \left(\xymatrix{ \left(e_{\alpha(k)}e_{-\alpha(k)}^{-1}h_{k-1}^{-1}\dots e_{\alpha(1)}e_{-\alpha(1)}^{-1}x_-\right)^{-1}.B_+ \ar@{-}[r] & \left(e_{\alpha(k)}e_{-\alpha(k)}^{-1}h_{k-1}^{-1}\dots e_{\alpha(1)}e_{-\alpha(1)}^{-1}x_-\right)^{-1}.B_- }\right)
\end{align*}
By computation similar to \eqref{relative} one can verify the relative position conditions $\xymatrix{B_u^{(k-1)} \ar[r]^{s_{\alpha(k)}} & B_u^{(k)}}$ and $\xymatrix{B_{u^*}^{(k-1)} \ar[r]^{s_{\alpha(k)}^*} & B_{u^*}^{(k)}}$ for all $1\leq k\leq m$.

At the end of this process, we will get the pair 
\begin{align*}
&\left(\xymatrix{B_u^{(m)}\ar@{-}[r] & B_{u^*}^{(m)}}\right)\\
:=\quad & \left(\xymatrix{ \left(e_{\alpha(m)}e_{-\alpha(m)}^{-1}h_{m-1}^{-1}\dots e_{\alpha(1)}e_{-\alpha(1)}^{-1}x_-\right)^{-1}.B_+ \ar@{-}[r] & \left(e_{\alpha(m)}e_{-\alpha(m)}^{-1}h_{m-1}^{-1}\dots e_{\alpha(1)}e_{-\alpha(1)}^{-1}x_-\right)^{-1}.B_- }\right),
\end{align*}

In order to say for sure that $B_{u^*}^{(m)}$ is the right choice for $B_5$, we need to verify the final relative position conditions
\[
B_u^{(m)}=B_+ \quad \text{and} \quad \xymatrix{B_{u^*}^{(m)} \ar[r]^{u^l} & x_-^{-1}.B_+}.
\]

It is not hard to see that the former claim $B_u^{(m)} =B_+$ is equivalent to the following lemma.

\begin{lem}\label{B+} The product $e_{\alpha(m)}e_{-\alpha(m)}^{-1}h_{m-1}^{-1}e_{\alpha(m-1)}e_{-\alpha(m-1)}^{-1}h_{m-2}^{-1}\dots h_1^{-1}e_{\alpha(1)}e_{-\alpha(1)}^{-1}x_-$ lies in $B_+$.
\end{lem}
\begin{proof} This can be seen easily from the string diagram of $x_-$ (i.e., the $u$-part of the original string diagram): recall that the left multiplication by $e_{\alpha(1)}e_{-\alpha(1)}^{-1}$ changes the node $-\alpha(1)$ to $\alpha(1)$, and then the next left multiplication by $e_{\alpha(2)}e_{-\alpha(2)}h_1^{-1}$ changes the node $-\alpha(2)$ to $\alpha(2)$, and so on. At the end there are no opposite simple root nodes left and hence the above product must be in the image of the map 
\[
\chi_{\left(\alpha(m), \alpha(m-1),\dots, \alpha(2), \alpha(1)\right)},
\]
which is inside $B_+$ by definition.
\end{proof}

For the latter claim $\xymatrix{\left(e_{\alpha(m)}e_{-\alpha(m)}^{-1}h_{m-1}^{-1}\dots h_1^{-1}e_{\alpha(1)}e_{-\alpha(1)}^{-1}x_-\right)^{-1}.B_- \ar[r]^(0.8){u^l} & x_-^{-1}.B_+}$, it suffices to show that the product
\[
\left(\overline{w}_0e_{\alpha(m)}e_{-\alpha(m)}^{-1}h_{m-1}^{-1}\dots h_1^{-1}e_{\alpha(1)}e_{-\alpha(1)}^{-1}x_-\right)x_-^{-1}=\overline{w}_0e_{\alpha(m)}e_{-\alpha(m)}^{-1}h_{m-1}^{-1}\dots h_1^{-1}e_{\alpha(1)}e_{-\alpha(1)}^{-1}
\]
lies inside $B_+u^lB_+$ (see Proposition \ref{eqposition}). But since $\overline{w}_0=\overline{u^l}\overline{u}$, it suffices to show the following lemma.

\begin{lem} The product $\overline{u}e_{\alpha(m)}e_{-\alpha(m)}^{-1}h_{m-1}^{-1}\dots h_1^{-1}e_{\alpha(1)}e_{-\alpha(1)}^{-1}$ lies in $B_+$.
\end{lem}
\begin{proof} To show this, we recall that $\overline{s}_\alpha:=e_\alpha^{-1}e_{-\alpha}e_\alpha^{-1}$; thus
\[
\overline{u}e_{\alpha(m)}e_{-\alpha(m)}^{-1}h_{m-1}^{-1}\dots h_1^{-1}e_{\alpha(1)}e_{-\alpha(1)}^{-1}=\overline{s_{\alpha(1)}\dots s_{\alpha(m-1)}} e_{\alpha(m)}^{-1}h_{m-1}^{-1}e_{\alpha(m-1)}e_{-\alpha(m-1)}^{-1}\dots h_1^{-1}e_{\alpha(1)}e_{-\alpha(1)}^{-1}.
\]
But then since $(\alpha(1),\dots, \alpha(m))$ is a reduced word of $u$, $s_{\alpha(1)}\dots s_{\alpha(m-1)}$ maps the simple root $\alpha(m)$ to a positive root; this implies that if we conjugate $e_{\alpha(m)}^{-1}$, which is an element in the unipotent subgroup corresponding to the simple root $\alpha(m)$, by the lift $\overline{s_{\alpha(1)}\dots s_{\alpha(m-1)}}$, we will get some element $n$ in the unipotent subgroup corresponding to the positive root $s_{\alpha(1)}\dots s_{\alpha(m-1)}(\alpha(m))$. In addition, since the lift $\overline{s_{\alpha(1)}\dots s_{\alpha(m-1)}}$ is an element in the normalizer subgroup $N_GH$, there exists some $h\in H$ such that $h\overline{s_{\alpha(1)}\dots s_{\alpha(m-1)}}=\overline{s_{\alpha(1)}\dots s_{\alpha(m-1)}}h_{m-1}^{-1}$. Thus we can deduce that 
\[
\overline{s_{\alpha(1)}\dots s_{\alpha(m-1)}} e_{\alpha(m)}^{-1}h_{m-1}^{-1}=n\overline{s_{\alpha(1)}\dots s_{\alpha(m-1)}}h_{m-1}^{-1}=nh\overline{s_{\alpha(1)}\dots s_{\alpha(m-1)}}.
\]
The proof is then finished by induction on $m$.
\end{proof}

So now we have successfully found $B_5$, which is
\[
B_5=\left(e_{\alpha(m)}e_{-\alpha(m)}^{-1}h_{m-1}^{-1}e_{\alpha(m-1)}e_{-\alpha(m-1)}^{-1}h_{m-2}^{-1}\dots h_1^{-1}e_{\alpha(1)}e_{-\alpha(1)}^{-1}x_-\right)^{-1}.B_-
\]
By applying a symmetric process to the lower right semicircle in Diagram \ref{diagram}, we can also find $B_6$, we will give the answer below, and leave the actual computation as an exercise for the readers:
\[
B_6=\left(x_+e_{\beta(n)}^{-1}e_{-\beta(n)}t_{n-1}^{-1}e_{\beta(n-1)}^{-1}e_{-\beta(n-1)}t_{n-2}^{-1}\dots t_1^{-1}e_{\beta(1)}^{-1}e_{-\beta(1)}\right).B_+,
\]
where $t_i$ are again elements of $H$ that arise as product of factors corresponding to non-trivial frozen $\mathcal{X}$-variables emerged from the moving (cluster mutation) process.

\begin{rmk} If one choose a lift $x$ of $H\backslash x/H$ that does not have frozen $\mathcal{X}$-variables set to 1, then one will have two extra $H$ elements $h_0$ and $t_n$ added to the expressions for $B_5$ and $B_6$ at the appropriate places (namely between $\dots e_{-\alpha(1)}^{-1}h_0^{-1}x_-$ and $x_+t_n^{-1}e_{\beta(n)}^{-1}\dots$).
\end{rmk}

Now we are ready to prove our main goal of this subsection.

\begin{prop}\label{2ndhalf} $\Xi_\vec{i}:\mathcal{X}^{u,v}\dashrightarrow \mathcal{X}^{u,v}$ is a cluster transformation.
\end{prop}
\begin{proof} We have seen from the discussion above that $\xi$ maps the configuration $\left[\vcenter{\vbox{\xymatrix{x_-^{-1}.B_+ \ar[r]^u \ar@{-}[d] & B_+ \ar@{-}[d] \\
B_- \ar[r]_{v^*} & x_+.B_-}}}\right]$ to $\left[\vcenter{\vbox{\xymatrix{ B_- \ar[r]^(0.4){u^*} \ar@{-}[d] & \left(zx_-\right)^{-1}.B_- \ar@{-}[d]\\ \left(x_+w\right).B_+
 \ar[r]_(0.6){v} & B_+}}}\right]^*$,
where 
\[
z:=e_{\alpha(m)}e_{-\alpha(m)}^{-1}h_{m-1}^{-1}\dots h_1^{-1}e_{\alpha(1)}e_{-\alpha(1)}^{-1}
\]
and 
\[
w:=e_{\beta(n)}^{-1}e_{-\beta(n)}t_{n-1}^{-1}\dots t_1^{-1}e_{\beta(1)}^{-1}e_{-\beta(1)}.
\]
Recall from lemma \ref{B+} that the product $zx_-$ lies in $B_+$; a symmetric statement of this is also true: the product $x_+w$ lies in $B_-$. Therefore the above configuration is equivalent to 
\[
\left[\vcenter{\vbox{\xymatrix{\left(x_+w\right). B_- \ar[r]^{u^*} \ar@{-}[d] & \left(zx_-\right)^{-1}.B_- \ar@{-}[d]\\
\left(x_+w\right).B_+ \ar[r]_{v} & \left(zx_-\right)^{-1}.B_+}}}\right]^*.
\]
But we also know that $B_\pm=w_0.B_\mp$; therefore the above configuration is equivalent to
\[
\left[\vcenter{\vbox{\xymatrix{\left(x_+w\right)\overline{w}_0. B_+\ar[r]^{u^*} \ar@{-}[d] & \left(zx_-\right)^{-1}\overline{w}_0.B_+ \ar@{-}[d]\\
\left(x_+w\right)\overline{w}_0.B_- \ar[r]_{v} & \left(zx_-\right)^{-1}\overline{w}_0.B_-}}}\right]^*,
\]
which is also equivalent to
\[
\left[\vcenter{\vbox{\xymatrix{ B_+ \ar[r]^{u^*} \ar@{-}[d] & 
\overline{w}_0^{-1}y^{-1}\overline{w}_0.B_+
 \ar@{-}[d]\\
B_- \ar[r]_{v} & \overline{w}_0^{-1}y^{-1}\overline{w}_0.B_-}}}\right]^*
\]
where we define $y:=zxw$. 

Recall that $B_\pm^*=B_\pm$ and the $*$ involution is defined as $x^*:=\overline{w}_0 \left(x^{-1}\right)^t \overline{w}_0^{-1}$; applying the $*$ involution to the configuration above, we see that the image configuration is equivalent to
\[
\left[\vcenter{\vbox{\xymatrix{B_+ \ar[r]^{u} \ar@{-}[d] & y^t.B_+ \ar@{-}[d]\\
B_- \ar[r]_{v^*} & y^{t}.B_-}}}\right].
\]

Now let's remind ourselves about the following commutative diagram again.
\[
\xymatrix{\mathcal{X}^{u,v} \ar@{-->}[r]^{\Xi_\vec{i}} \ar[d]_\cong^\chi & \mathcal{X}^{u,v} \ar[d]^\cong_\chi \\
H\backslash G^{u,v}/H \ar[r]^{\iota\circ \tw } \ar[d]_\cong & H\backslash G^{u,v}/H \ar[d]^\cong \\
\conf^{u,v}(\mathcal{B})\ar[r]_\xi & \conf^{u,v}(\mathcal{B})}
\]
Our discussion above shows that on the middle level, generically we have
\[
\iota\circ \tw \left(H\backslash x/H\right)=H\left\backslash y^t\right/H.
\]

But what does this mean to the cluster $\mathcal{X}$-variables on the cluster coordinate chart $\mathcal{X}_\vec{i}^{u,v}$? Recall that when we multiply $e_{\alpha(1)}e_{-\alpha(1)}^{-1}$ on the left of $x$, we change the left most node on the string diagram from $-\alpha(1)$ to $\alpha(1)$. But since $\mathcal{X}_\vec{i}^{u,v}$ does not have any frozen variables (because of the double quotient by $H$), such left multiplication by $e_{\alpha(1)}e_{-\alpha(1)}$ in fact a cluster isomorphism (see Remark \ref{left-mult}). Then we move this node to the right of the $u$-part of the string diagram using a sequence of cluster mutations. Then we multiply $e_{\alpha(2)}e_{-\alpha(2)}^{-1}h_1^{-1}$ on the left again, which changes the left most node from $-\alpha(2)$ to $\alpha(2)$ and is again a cluster isomorphism. Then we again move the new node $\alpha(2)$ to the right of the remaining $u$-part of the string diagram using cluster mutations... The same thing happens on the $v$-part as well. Therefore by factoring $H\backslash y/H$ as
\begin{align*}
 H\left\backslash \left(e_{\alpha(m)}e_{-\alpha(m)}^{-1}h_{m-1}^{-1}\right)\right.\dots& \left(e_{\alpha(2)}e_{-\alpha(2)}^{-1}h_1^{-1}\right)\left(e_{\alpha(1)}e_{-\alpha(1)}^{-1}\right)x\left(e_{\beta(n)}^{-1}e_{-\beta(n)}\right) \\
&\left.\left(t_{n-1}^{-1}e_{\beta(n-1)}^{-1}e_{-\beta(n-1)}^{-1}\right)\dots\left( t_1^{-1}e_{\beta(1)}^{-1}e_{-\beta(1)}\right)\right/H
\end{align*}
we see that multiplying each additional factor on the left or on the right of $H\backslash x/H$ can be realized by a cluster isomorphism followed by a sequence of cluster mutations. 
\[
\tikz{
\node (0) at (0,2) [] {$H\backslash x/H$};
\node (1) at (6,2) [] {$H\left\backslash e_{\alpha(1)}e_{-\alpha(1)}^{-1}x\right/H$};
\node (2) at (12,2) [] {$H\left\backslash e_{\alpha(1)}e_{-\alpha(1)}^{-1}x\right/H$};
\node (3) at (0,0) [] {$H\left\backslash e_{\alpha(2)}e_{-\alpha(2)}h_1^{-1}e_{\alpha(1)}e_{-\alpha(1)}^{-1}x\right/H$};
\node (4) at (9,0) [] {$H\left\backslash e_{\alpha(2)}e_{-\alpha(2)}h_1^{-1}e_{\alpha(1)}e_{-\alpha(1)}^{-1}x\right/H$};
\node (5) at (13,0) [] {$\cdots$};
\draw [->] (0) -- node [above] {cluster isomorphism} (1);
\draw [->] (1) -- node [above] {cluster mutations} (2);
\draw [->] (2) -- node [below right] {cluster isomorphism} (3);
\draw [->] (3) -- node [below] {cluster mutations} (4);
\draw [->] (4) -- node [above] {$\cdots$} (5); 
}
\]

After the prescribed sequence of mutations, one gets a cluster where the $\mathcal{X}$-coordinates and the quiver are associated to the factorization of $y\in G_{ad}^{v^{-1},u^{-1}}$ using the reduced word
\[
\vec{j}=\left(\alpha(m), \alpha(m-1), \dots, \alpha(1),-\beta(n),-\beta(n-1),\dots, -\beta(1)\right).
\]
But this is in fact okay because we still need to take transposition. We claim that transposition can be realized as a cluster isomorphism! Note that if we write $y$ as a product according to the string diagram associated to the reduced word $\vec{j}$ and then apply transposition to each amalgamation factor, we see that  transposition does not change the cluster $\mathcal{X}$-variables, but switches the simple root nodes to corresponding opposite simple root nodes and vice versa and reverses the order of the nodes completely. Therefore we can interpret the same cluster $\mathcal{X}$-coordinates and the same quiver as giving the factorization of $y^t\in G_{ad}^{u,v}$ using the reduced word
\[
\vec{k}=\left(\beta(1),\beta(2),\dots, \beta(n), -\alpha(1), -\alpha(2),\dots, \alpha(m)\right). 
\]

But then since $\vec{k}$ is a reduced word of our pair of Weyl group elements $(u,v)$, there is a sequence of cluster mutations that restores our original reduced word $\vec{i}$. The composition of all the cluster mutations and cluster isomorphisms so far exhibits $\Xi_\vec{i}$ as a cluster transformation (see Remark \ref{clustertran'}), and hence the proof is finished. 
\end{proof}

\begin{rmk} Note that when factoring $\Xi_\vec{i}$ as a composition of cluster mutations and cluster isomorphisms, we actually go through cluster coordinate charts that do not belong to the double Bruhat cell $H\backslash G^{u,v}/H$ (see also Remark \ref{left-mult}); but this is okay as long as we get back to $H\backslash G^{u,v}/H$ at the very end.
\end{rmk}

We see in Diagram \eqref{diagram} that the diameter $\left(\xymatrix{ B_2\ar@{-}[r] & B_3}\right) \quad = \quad \left(\xymatrix{ B_+ \ar@{-}[r] & B_-}\right)$ is contained in both the initial quadruple and final quadruple during tilting. In fact, if we draw our initial quadruple as 
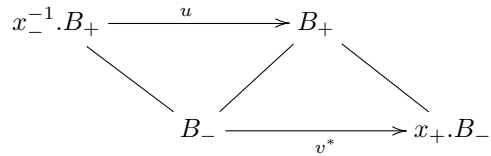
\[
\xymatrix{x_-^{-1}.B_+ \ar[rr]^u \ar@{-}[dr] & & B_+ \ar@{-}[dl] \ar@{-}[dr] & \\
& B_- \ar[rr]_{v^*} & & x_+.B_-}
\]
and place the string diagram associated to $\vec{i}$ horizontally across such parallelogram with opposite simple root nodes contained in the left triangle and simple root nodes contained in the right triangle, 
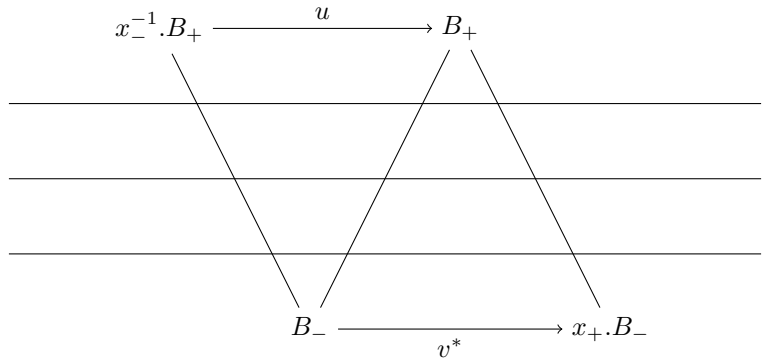
\[
\tikz{
\node (1) at (0,4) [] {$x_-^{-1}.B_+$};
\node (2) at (4,4) [] {$B_+$};
\node (3) at (2,0) [] {$B_-$};
\node (4) at (6,0) [] {$x_+.B_-$};
\draw [->] (1) -- node [above] {$u$} (2);
\draw [->] (3) -- node [below] {$v^*$} (4);
\draw (1) -- (3) -- (2) -- (4);
\draw (-2,1) -- (8,1);
\draw (-2,2) -- (8,2);
\draw (-2,3) -- (8,3);
}
\]
we can describe the whole ``cluster factorization'' of $\Xi_\vec{i}$ via triangle flipping and moving.
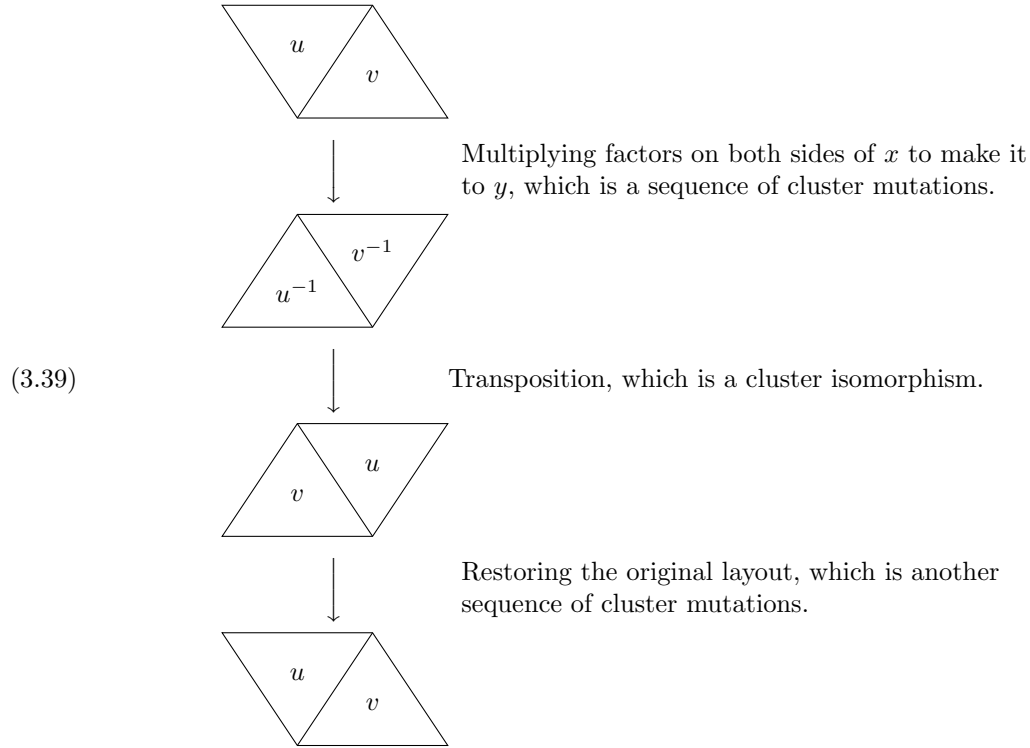
\begin{equation} \label{triangleflipping}
\begin{split}
\tikz{
\draw (0,1.5) -- (2,1.5) -- (3,0) -- (1,0) -- cycle;
\draw (1,0) -- (2,1.5);
\node at (1,0.75) [above] {$u$};
\node at (2,0.75) [below] {$v$};
}&
\\
\bigg\downarrow \quad \quad \quad \quad&\begin{array}{l}\text{Multiplying factors on both sides of $x$ to make it}\\ \text{to $y$, which is a sequence of cluster mutations.}
\end{array}
\\
\tikz{
\draw (0,0) -- (2,0) -- (3,1.5) -- (1,1.5) -- cycle;
\draw (1,1.5) -- (2,0);
\node at (1,0.75) [below] {$u^{-1}$};
\node at (2,0.75) [above] {$v^{-1}$};
}& 
\\
\bigg\downarrow \quad \quad\quad \quad & \text{Transposition, which is a cluster isomorphism.}
\\
\tikz{
\draw (0,0) -- (2,0) -- (3,1.5) -- (1,1.5) -- cycle;
\draw (1,1.5) -- (2,0);
\node at (1,0.75) [below] {$v$};
\node at (2,0.75) [above] {$u$};
}& 
\\
\bigg\downarrow \quad \quad \quad \quad&\begin{array}{l}\text{Restoring the original layout, which is another}\\
\text{sequence of cluster mutations.}\end{array}
\\
\tikz{
\draw (0,1.5) -- (2,1.5) -- (3,0) -- (1,0) -- cycle;
\draw (1,0) -- (2,1.5);
\node at (1,0.75) [above] {$u$};
\node at (2,0.75) [below] {$v$};
}&
\end{split}
\end{equation}
You may have wondered why we have used triangles of the form $\tikz[baseline=0ex, scale=0.5]{\draw (-0.75,0.75) -- (0.75,0.75) -- (0,-0.75) -- cycle;}$ in the $u$-part and triangles of the form $\tikz[baseline=0ex, scale=0.5]{\draw (-0.75,-0.75) -- (0.75,-0.75) -- (0,0.75) -- cycle;}$ in the $v$-part when computing the leading powers of variables in $\Xi_\vec{i}^*\left(X_a\right)$; well, the above sequence of diagrams is the reason.

\subsection{Recovering \texorpdfstring{$\DT$ from $\Xi$}{}}\label{3.4}

By now we have finished all the technical part of the proof of our main theorem. What's left is just to apply the intertwining maps to deduce the commutative diagram on the left from the one on the right in \ref{commutative diagram}. To summarize all the results, we combine the commutative diagrams in \ref{commutative diagram} together with the intertwining maps into a 3-dimensoinal commutative diagram, and recall the relevant definitions and propositions.
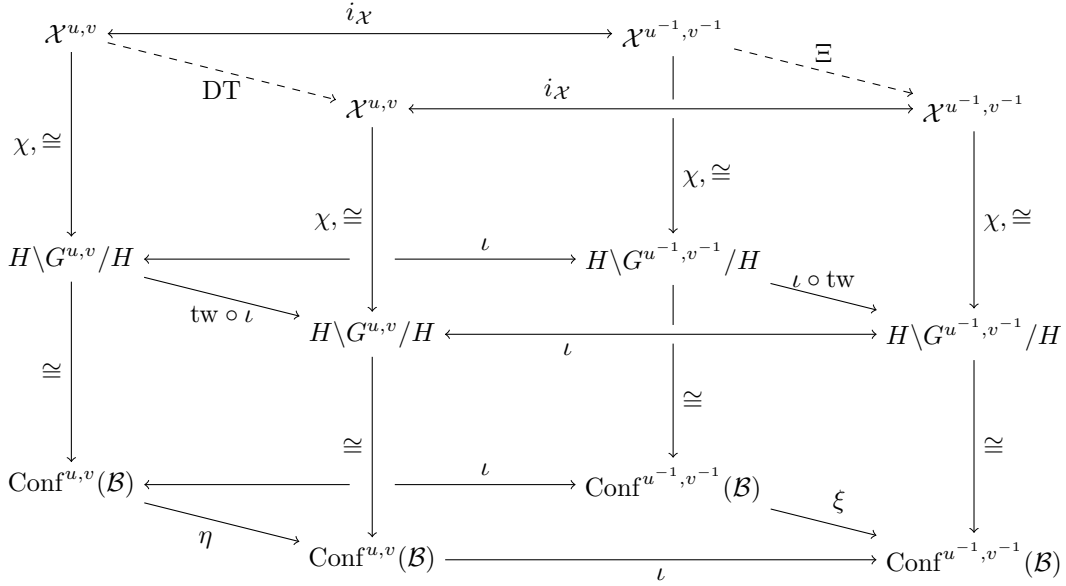
\[
\tikz{
\node (l1) at (0,0) [] {$\conf^{u,v}(\mathcal{B})$};
\node (m1) at (0,3) [] {$H\backslash G^{u,v}/H$};
\node (u1) at (0,6) [] {$\mathcal{X}^{u,v}$};
\node (l2) at (4,-1) [] {$\conf^{u,v}(\mathcal{B})$};
\node (m2) at (4,2) [] {$H\backslash G^{u,v}/H$};
\node (u2) at (4,5) [] {$\mathcal{X}^{u,v}$};
\node (l3) at (8,0) [] {$\conf^{u^{-1},v^{-1}}(\mathcal{B})$};
\node (m3) at (8,3) [] {$H\backslash G^{u^{-1},v^{-1}}/H$};
\node (u3) at (8,6) [] {$\mathcal{X}^{u^{-1},v^{-1}}$};
\node (l4) at (12,-1) [] {$\conf^{u^{-1},v^{-1}}(\mathcal{B})$};
\node (m4) at (12,2) [] {$H\backslash G^{u^{-1},v^{-1}}/H$};
\node (u4) at (12,5) [] {$\mathcal{X}^{u^{-1},v^{-1}}$};
\node (cr1) at (8,5) [] {$\quad $};
\node (cr2) at (4,3) [] {$\quad $};
\node (cr3) at (4,0) [] {$\quad$};
\node (cr4) at (8,2) [] {$\quad$};
\draw [->] (l1) -- node [below left] {$\eta$} (l2);
\draw [->] (m1) -- node [left] {$\cong$} (l1);
\draw [->] (m2) -- node [left] {$\cong$} (l2);
\draw [->] (m1) -- node [below] {$\tw\circ \iota$} (m2);
\draw [->] (u1) -- node [left] {$\chi, \cong$} (m1);
\draw [->] (u2) -- node [left] {$\chi, \cong$} (m2);
\draw [dashed, ->] (u1) -- node [below] {$\DT$} (u2);
\draw [<->] (u1) -- node [above] {$i_\mathcal{X}$} (u3);
\draw [<->] (u2) -- node [above left] {$i_\mathcal{X}\quad \quad \quad $} (u4);
\draw [dashed, ->] (u3) -- node [above] {$\Xi$} (u4);
\draw [->] (u3) -- (cr1) -- node [right] {$\chi,\cong$} (m3);
\draw [->] (u4) -- node [right] {$\chi,\cong$} (m4);
\draw [->] (cr2) -- node [above] {$\iota$} (m3);
\draw [->] (cr2) -- (m1);
\draw [<->] (m2) -- node [below left] {$\iota\quad \quad \quad$} (m4);
\draw [->] (m3) -- node [above] {$\iota\circ \tw$} (m4);
\draw [->] (cr3) -- node [above] {$\iota$} (l3);
\draw [->] (cr3) -- (l1);
\draw [->] (l2) -- node [below] {$\iota$} (l4);
\draw [->] (m3) -- (cr4) -- node [right] {$\cong$} (l3);
\draw [->] (m4) -- node [right] {$\cong$} (l4);
\draw [->] (l3) -- node [above right] {$\xi$} (l4);
}
\]

\begin{itemize}
\item Commutativity of the lower vertical squares both on the left and on the right is due to Proposition \ref{twistflag} .
\item Commutativitiy of the upper vertical square on the right is due to Definition \ref{defnxi} and Proposition \ref{clustertwist}.
\item Commutativity of the upper vertical squares both in the front and at the back is due to Proposition \ref{iota}.
\item Commutativity of the lower vertical squares both in the front and at the back is due to Proposition \ref{iotacommute}.
\item Propositions \ref{1sthalf} and \ref{2ndhalf} show that the map $\Xi$ defined in Definition \ref{defnxi} is the same as the cluster transformation $\Xi$ defined in \ref{dtxi}.
\item Commutativity of the top square then follows from Definition \ref{dtxi}.
\item Commutativity of the horizontal square on the middle level follows from the fact that $\iota$ is of order 2.
\item Commutativity of the bottom square is due to Proposition \ref{xieaaintertwine}. 
\end{itemize}

Using all the result we have so far, we can now prove the remaining part of our main theorem (Theorem \ref{mainthm}).

\vspace{0.5cm}

\noindent\textit{Proof of Part (1) of Theorem \ref{mainthm}.} The only remaining square in the above 3-dimensional diagram absent on the list above is the upper vertical one on the left, and it must commute as well due to the commutativity of the rest of the diagram. Furthermore we know that $\DT$ is a cluster transformation because $\Xi$ is a cluster transformation and $i_\mathcal{X}$ intertwines cluster transformations. Thus our main theorem is proved. \qed

\begin{rmk} \label{maxgreen} At the end, we would like to state one more way of computing the Donaldson-Thomas transformation in the case of $H\backslash G^{u,v}/H$. Recall that in Diagram \eqref{triangleflipping} we have factored $\Xi$ into a composition of a sequence of cluster mutations followed by a cluster isomorphism followed by another sequence of cluster mutations. We can surely intertwine that picture using the anti-involution $\iota$ to get a similar ``cluster factorization'' of the Donaldson-Thomas transformation $\DT$. Note that diagrammatically, $\iota$ is a cluster isomorphism that changes $
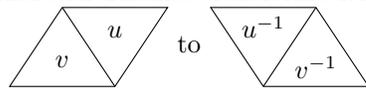
\tikz[baseline=3ex,scale=0.7]{
\draw (0,0) -- (2,0) -- (3,1.5) -- (1,1.5) -- cycle;
\draw (1,1.5) -- (2,0);
\node at (1,0.75) [below] {$v$};
\node at (2,0.75) [above] {$u$};
}$ to $\tikz[baseline=3ex,scale=0.7]{
\draw (0,1.5) -- (2,1.5) -- (3,0) -- (1,0) -- cycle;
\draw (1,0) -- (2,1.5);
\node at (1,0.75) [above] {$u^{-1}$};
\node at (2,0.75) [below] {$v^{-1}$};
}$. Therefore if we intertwine Diagram \eqref{triangleflipping} by $\iota$ we get the following.
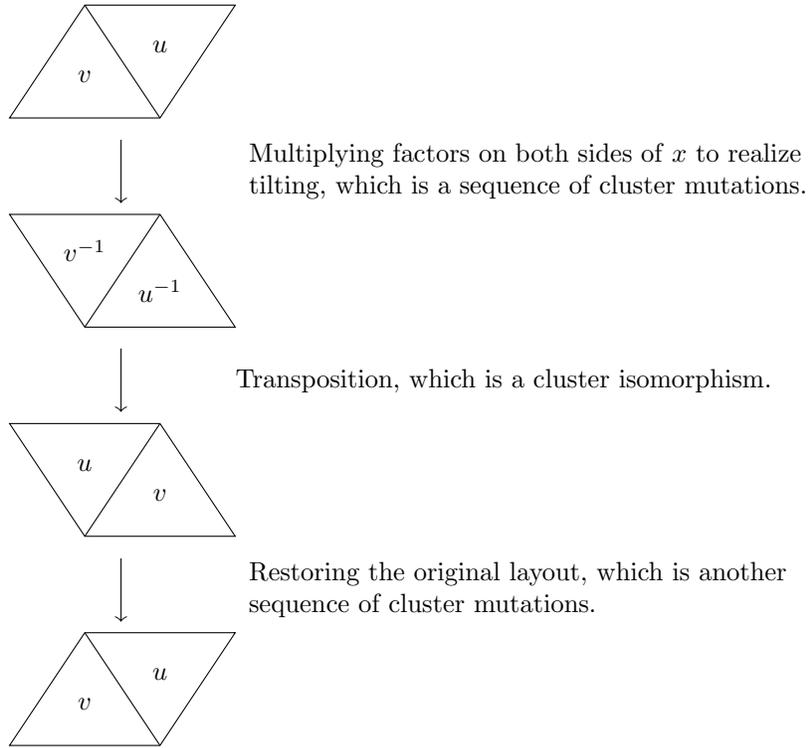
\[
\begin{split}
\tikz{
\draw (0,0) -- (2,0) -- (3,1.5) -- (1,1.5) -- cycle;
\draw (1,1.5) -- (2,0);
\node at (1,0.75) [below] {$v$};
\node at (2,0.75) [above] {$u$};
}
&
\\
\bigg\downarrow \quad \quad \quad \quad&\begin{array}{l}\text{Multiplying factors on both sides of $x$ to realize}\\ \text{tilting, which is a sequence of cluster mutations.}
\end{array}
\\
\tikz{
\draw (0,1.5) -- (2,1.5) -- (3,0) -- (1,0) -- cycle;
\draw (1,0) -- (2,1.5);
\node at (1,0.75) [above] {$v^{-1}$};
\node at (2,0.75) [below] {$u^{-1}$};
}
& 
\\
\bigg\downarrow \quad \quad\quad \quad & \text{Transposition, which is a cluster isomorphism.}
\\
\tikz{
\draw (0,1.5) -- (2,1.5) -- (3,0) -- (1,0) -- cycle;
\draw (1,0) -- (2,1.5);
\node at (1,0.75) [above] {$u$};
\node at (2,0.75) [below] {$v$};
}
& 
\\
\bigg\downarrow \quad \quad \quad \quad&\begin{array}{l}\text{Restoring the original layout, which is another}\\
\text{sequence of cluster mutations.}\end{array}
\\
\tikz{
\draw (0,0) -- (2,0) -- (3,1.5) -- (1,1.5) -- cycle;
\draw (1,1.5) -- (2,0);
\node at (1,0.75) [below] {$v$};
\node at (2,0.75) [above] {$u$};
}
&
\end{split}
\]

In a private conversation with L. Shen, he pointed out that this procedure can be used to construct a maximal green sequence, which is the combinatorial version of cluster Donaldson-Thomas transformation. For a more concrete example please see Appendix \ref{C}.
\end{rmk}

\setcounter{subsection}{0}
\renewcommand{\thesubsection}{\Alph{subsection}}

\section*{Appendix}

\subsection{A Sequence of 10 Mutations Corresponding to Move (2) in the \texorpdfstring{$G_2$}{} Case} \label{A} this case says that $(\alpha,\beta,\alpha,\beta,\alpha,\beta)\sim (\beta,\alpha,\beta, \alpha,\beta,\alpha)$. Without loss of generality let's assume that $C_{\alpha\beta}=-1$ and $C_{\beta\alpha}=-3$. The following pictures are the string diagrams and their corresponding quasi-quivers together with the sequence of seed mutations that transform one into the other. For more details and deeper reasons for why this is true, please see Fock and Goncharov's paper on amalgamation (\cite{FGamalgamation} Section 3.7).
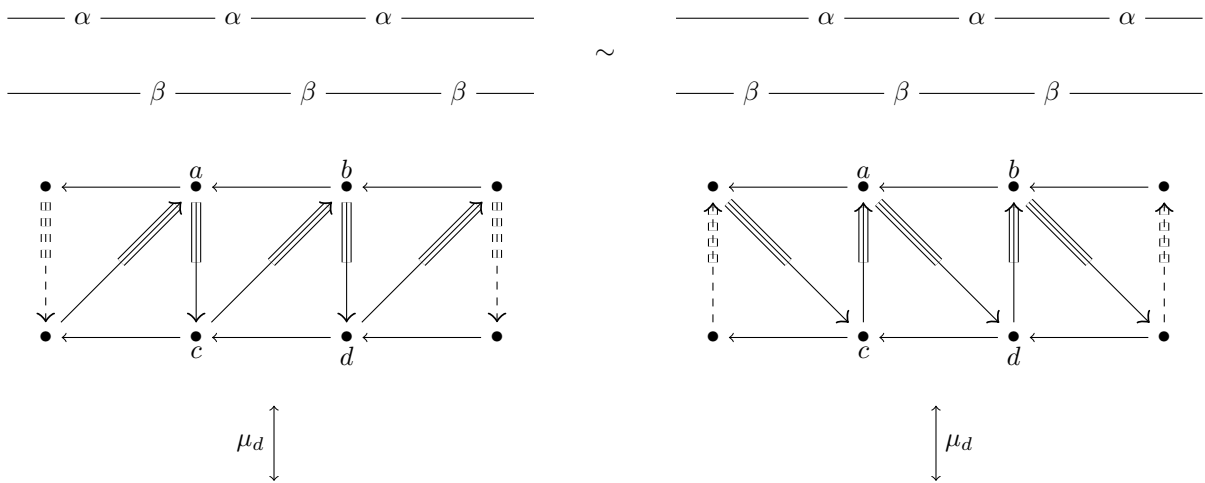
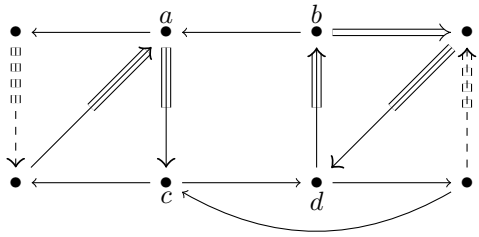
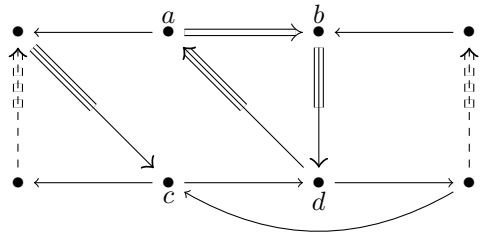
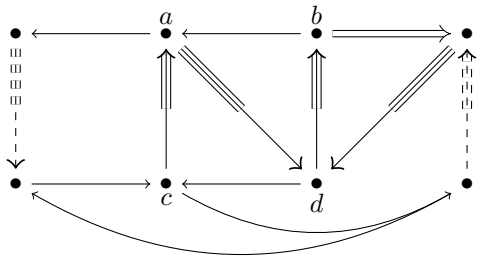
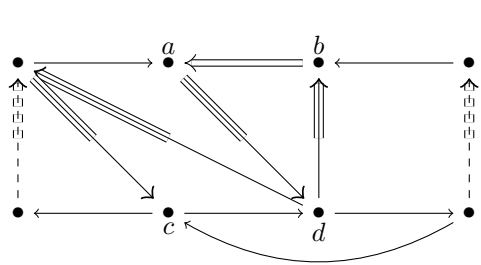
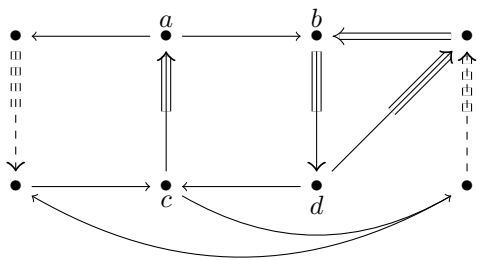
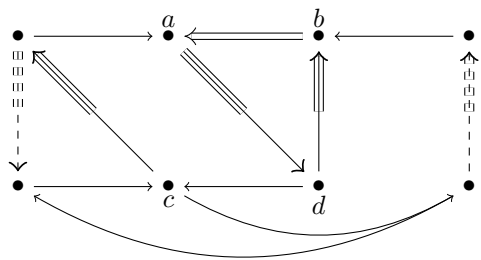
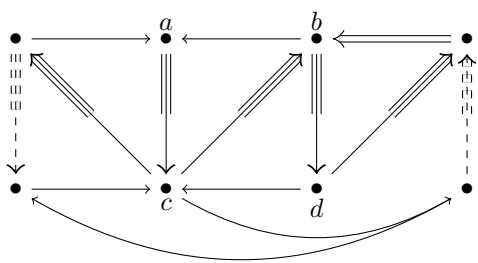
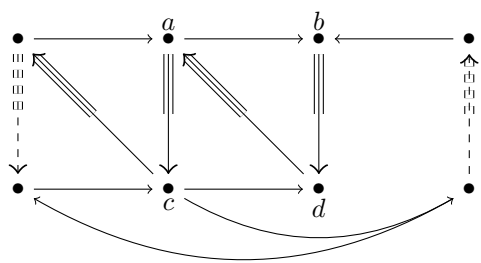
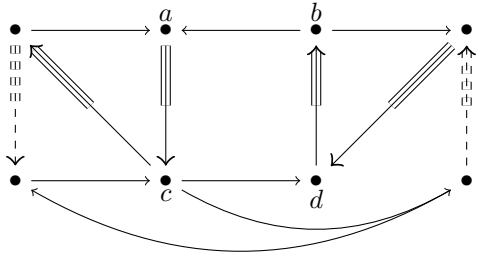
\[
\begin{tikzpicture}[baseline=3ex]
\node (1) at (1,1) [] {$\alpha$};
\node (2) at (3,1) [] {$\alpha$};
\node (3) at (5,1) [] {$\alpha$};
\node (4) at (2,0) [] {$\beta$};
\node (5) at (4,0) [] {$\beta$};
\node (6) at (6,0) [] {$\beta$};
\draw (0,0) -- (4) -- (5) -- (6) -- (7,0);
\draw (0,1) -- (1) -- (2) -- (3) --(7,1);
\end{tikzpicture}\quad\quad \sim \quad \quad 
\begin{tikzpicture}[baseline=3ex]
\node (1) at (1,0) [] {$\beta$};
\node (2) at (3,0) [] {$\beta$};
\node (3) at (5,0) [] {$\beta$};
\node (4) at (2,1) [] {$\alpha$};
\node (5) at (4,1) [] {$\alpha$};
\node (6) at (6,1) [] {$\alpha$};
\draw (0,1) -- (4) -- (5) -- (6) -- (7,1);
\draw (0,0) -- (1) -- (2) -- (3) --(7,0);
\end{tikzpicture}
\]
\[
\tikz{
\node (1) at (0,0) [] {$\bullet$};
\node (2) at (2,0) [] {$\bullet$};
\node (3) at (4,0) [] {$\bullet$};
\node (4) at (6,0) [] {$\bullet$};
\node (5) at (0,2) [] {$\bullet$};
\node (6) at (2,2) [] {$\bullet$};
\node (7) at (4,2) [] {$\bullet$};
\node (8) at (6,2) [] {$\bullet$};
\draw [->] (2) -- (1);
\draw [->] (3) -- (2);
\draw [->] (4) -- (3);
\draw [->] (6) -- (5);
\draw [->] (7) -- (6);
\draw [->] (8) -- (7);
\draw [x-->] (5) -- (1);
\draw [-x>] (1) -- (6);
\draw [x->] (6) -- (2);
\draw [-x>] (2) -- (7);
\draw [x->] (7) -- (3);
\draw [-x>] (3) -- (8);
\draw [x-->] (8) -- (4);
\node at (6) [above] {$a$};
\node at (7) [above] {$b$};
\node at (2) [below] {$c$};
\node at (3) [below] {$d$};
}
\quad \quad  \quad \quad \quad\quad \quad
\tikz{
\node (1) at (0,0) [] {$\bullet$};
\node (2) at (2,0) [] {$\bullet$};
\node (3) at (4,0) [] {$\bullet$};
\node (4) at (6,0) [] {$\bullet$};
\node (5) at (0,2) [] {$\bullet$};
\node (6) at (2,2) [] {$\bullet$};
\node (7) at (4,2) [] {$\bullet$};
\node (8) at (6,2) [] {$\bullet$};
\draw [->] (2) -- (1);
\draw [->] (3) -- (2);
\draw [->] (4) -- (3);
\draw [->] (6) -- (5);
\draw [->] (7) -- (6);
\draw [->] (8) -- (7);
\draw [--x>] (1) -- (5);
\draw [x->] (5) -- (2);
\draw [-x>] (2) -- (6);
\draw [x->] (6) -- (3);
\draw [-x>] (3) -- (7);
\draw [x->] (7) -- (4);
\draw [--x>] (4) -- (8);
\node at (6) [above] {$a$};
\node at (7) [above] {$b$};
\node at (2) [below] {$c$};
\node at (3) [below] {$d$};
}
\]
\[
 \tikz{\draw [<->] (0,0) -- node[left]{$\mu_d$} (0,1);} \quad\quad \quad\quad\quad \quad\quad \quad\quad \quad\quad \quad\quad \quad\quad \quad\quad \quad\quad \quad \quad\quad \quad\quad \quad \tikz{\draw [<->] (0,0) -- node[right]{$\mu_d$} (0,1);}
\]
\[
\tikz{
\node (1) at (0,0) [] {$\bullet$};
\node (2) at (2,0) [] {$\bullet$};
\node (3) at (4,0) [] {$\bullet$};
\node (4) at (6,0) [] {$\bullet$};
\node (5) at (0,2) [] {$\bullet$};
\node (6) at (2,2) [] {$\bullet$};
\node (7) at (4,2) [] {$\bullet$};
\node (8) at (6,2) [] {$\bullet$};
\draw [->] (2) -- (1);
\draw [->] (2) -- (3);
\draw [->] (3) -- (4);
\draw [->] (6) -- (5);
\draw [->] (7) -- (6);
\draw [double distance=2pt, -implies] (7) -- (8);
\draw [->] (4) edge[bend left] (2);
\draw [x-->] (5) -- (1);
\draw [-x>] (1) -- (6);
\draw [x->] (6) -- (2);
\draw [-x>] (3) -- (7);
\draw [x->] (8) -- (3);
\draw [--x>] (4) -- (8);
\node at (6) [above] {$a$};
\node at (7) [above] {$b$};
\node at (2) [below] {$c$};
\node at (3) [below] {$d$};
}
\quad \quad  \quad \quad \quad\quad \quad
\tikz{
\node (1) at (0,0) [] {$\bullet$};
\node (2) at (2,0) [] {$\bullet$};
\node (3) at (4,0) [] {$\bullet$};
\node (4) at (6,0) [] {$\bullet$};
\node (5) at (0,2) [] {$\bullet$};
\node (6) at (2,2) [] {$\bullet$};
\node (7) at (4,2) [] {$\bullet$};
\node (8) at (6,2) [] {$\bullet$};
\draw [->] (2) -- (1);
\draw [->] (2) -- (3);
\draw [->] (3) -- (4);
\draw [->] (6) -- (5);
\draw [double distance=2pt, -implies] (6) -- (7);
\draw [->] (8) -- (7);
\draw [->] (4) edge[bend left] (2);
\draw [--x>] (1) -- (5);
\draw [x->] (5) -- (2);
\draw [-x>] (3) -- (6);
\draw [x->] (7) -- (3);
\draw [--x>] (4) -- (8);
\node at (6) [above] {$a$};
\node at (7) [above] {$b$};
\node at (2) [below] {$c$};
\node at (3) [below] {$d$};
}
\]
\[
 \tikz{\draw [<->] (0,0) -- node[left]{$\mu_c$} (0,1);} \quad\quad \quad\quad\quad \quad\quad \quad\quad \quad\quad \quad\quad \quad\quad \quad\quad \quad\quad \quad \quad\quad \quad\quad \quad \tikz{\draw [<->] (0,0) -- node[right]{$\mu_a$} (0,1);}
\]
\[
\tikz{
\node (1) at (0,0) [] {$\bullet$};
\node (2) at (2,0) [] {$\bullet$};
\node (3) at (4,0) [] {$\bullet$};
\node (4) at (6,0) [] {$\bullet$};
\node (5) at (0,2) [] {$\bullet$};
\node (6) at (2,2) [] {$\bullet$};
\node (7) at (4,2) [] {$\bullet$};
\node (8) at (6,2) [] {$\bullet$};
\draw [->] (1) -- (2);
\draw [->] (3) -- (2);
\draw [->] (6) -- (5);
\draw [->] (7) -- (6);
\draw [double distance=2pt, -implies] (7) -- (8);
\draw [->] (2) edge[bend right] (4);
\draw [->] (4) edge[bend left] (1);
\draw [x-->] (5) -- (1);
\draw [x->] (6) -- (3);
\draw [-x>] (2) -- (6);
\draw [-x>] (3) -- (7);
\draw [x->] (8) -- (3);
\draw [--x>] (4) -- (8);
\node at (6) [above] {$a$};
\node at (7) [above] {$b$};
\node at (2) [below] {$c$};
\node at (3) [below] {$d$};
}
\quad \quad  \quad \quad \quad\quad \quad
\tikz{
\node (1) at (0,0) [] {$\bullet$};
\node (2) at (2,0) [] {$\bullet$};
\node (3) at (4,0) [] {$\bullet$};
\node (4) at (6,0) [] {$\bullet$};
\node (5) at (0,2) [] {$\bullet$};
\node (6) at (2,2) [] {$\bullet$};
\node (7) at (4,2) [] {$\bullet$};
\node (8) at (6,2) [] {$\bullet$};
\draw [->] (2) -- (1);
\draw [->] (2) -- (3);
\draw [->] (3) -- (4);
\draw [->] (5) -- (6);
\draw [double distance=2pt, -implies] (7) -- (6);
\draw [->] (8) -- (7);
\draw [->] (4) edge[bend left] (2);
\draw [--x>] (1) -- (5);
\draw [-x>] (3) -- (5);
\draw [x->] (5) -- (2);
\draw [x->] (6) -- (3);
\draw [-x>] (3) -- (7);
\draw [--x>] (4) -- (8);
\node at (6) [above] {$a$};
\node at (7) [above] {$b$};
\node at (2) [below] {$c$};
\node at (3) [below] {$d$};
}
\]
\[
 \tikz{\draw [<->] (0,0) -- node[left]{$\mu_b$} (0,1);} \quad\quad \quad\quad\quad \quad\quad \quad\quad \quad\quad \quad\quad \quad\quad \quad\quad \quad\quad \quad \quad\quad \quad\quad \quad \tikz{\draw [<->] (0,0) -- node[right]{$\mu_c$} (0,1);}
\]
\[
\tikz{
\node (1) at (0,0) [] {$\bullet$};
\node (2) at (2,0) [] {$\bullet$};
\node (3) at (4,0) [] {$\bullet$};
\node (4) at (6,0) [] {$\bullet$};
\node (5) at (0,2) [] {$\bullet$};
\node (6) at (2,2) [] {$\bullet$};
\node (7) at (4,2) [] {$\bullet$};
\node (8) at (6,2) [] {$\bullet$};
\draw [->] (1) -- (2);
\draw [->] (3) -- (2);
\draw [->] (6) -- (5);
\draw [->] (6) -- (7);
\draw [double distance=2pt, -implies] (8) -- (7);
\draw [->] (2) edge[bend right] (4);
\draw [->] (4) edge[bend left] (1);
\draw [x-->] (5) -- (1);
\draw [-x>] (2) -- (6);
\draw [x->] (7) -- (3);
\draw [-x>] (3) -- (8);
\draw [--x>] (4) -- (8);
\node at (6) [above] {$a$};
\node at (7) [above] {$b$};
\node at (2) [below] {$c$};
\node at (3) [below] {$d$};
}
\quad \quad  \quad \quad \quad\quad \quad
\tikz{
\node (1) at (0,0) [] {$\bullet$};
\node (2) at (2,0) [] {$\bullet$};
\node (3) at (4,0) [] {$\bullet$};
\node (4) at (6,0) [] {$\bullet$};
\node (5) at (0,2) [] {$\bullet$};
\node (6) at (2,2) [] {$\bullet$};
\node (7) at (4,2) [] {$\bullet$};
\node (8) at (6,2) [] {$\bullet$};
\draw [->] (1) -- (2);
\draw [->] (3) -- (2);
\draw [->] (5) -- (6);
\draw [double distance=2pt, -implies] (7) -- (6);
\draw [->] (8) -- (7);
\draw [->] (2) edge[bend right] (4);
\draw [->] (4) edge[bend left] (1);
\draw [x-->] (5) -- (1);
\draw [-x>] (2) -- (5);
\draw [x->] (6) -- (3);
\draw [-x>] (3) -- (7);
\draw [--x>] (4) -- (8);
\node at (6) [above] {$a$};
\node at (7) [above] {$b$};
\node at (2) [below] {$c$};
\node at (3) [below] {$d$};
}
\]
\[
 \tikz{\draw [<->] (0,0) -- node[left]{$\mu_a$} (0,1);} \quad\quad \quad\quad\quad \quad\quad \quad\quad \quad\quad \quad\quad \quad\quad \quad\quad \quad\quad \quad \quad\quad \quad\quad \quad \tikz{\draw [<->] (0,0) -- node[right]{$\mu_d$} (0,1);}
\]
\[
\tikz{
\node (1) at (0,0) [] {$\bullet$};
\node (2) at (2,0) [] {$\bullet$};
\node (3) at (4,0) [] {$\bullet$};
\node (4) at (6,0) [] {$\bullet$};
\node (5) at (0,2) [] {$\bullet$};
\node (6) at (2,2) [] {$\bullet$};
\node (7) at (4,2) [] {$\bullet$};
\node (8) at (6,2) [] {$\bullet$};
\draw [->] (1) -- (2);
\draw [->] (3) -- (2);
\draw [->] (5) -- (6);
\draw [->] (7) -- (6);
\draw [double distance=2pt, -implies] (8) -- (7);
\draw [->] (2) edge[bend right] (4);
\draw [->] (4) edge[bend left] (1);
\draw [x-->] (5) -- (1);
\draw [x->] (6) -- (2);
\draw [-x>] (2) -- (5);
\draw [-x>] (2) -- (7);
\draw [x->] (7) -- (3);
\draw [-x>] (3) -- (8);
\draw [--x>] (4) -- (8);
\node at (6) [above] {$a$};
\node at (7) [above] {$b$};
\node at (2) [below] {$c$};
\node at (3) [below] {$d$};
}
\quad \quad  \quad \quad \quad\quad \quad
\tikz{
\node (1) at (0,0) [] {$\bullet$};
\node (2) at (2,0) [] {$\bullet$};
\node (3) at (4,0) [] {$\bullet$};
\node (4) at (6,0) [] {$\bullet$};
\node (5) at (0,2) [] {$\bullet$};
\node (6) at (2,2) [] {$\bullet$};
\node (7) at (4,2) [] {$\bullet$};
\node (8) at (6,2) [] {$\bullet$};
\draw [->] (1) -- (2);
\draw [->] (2) -- (3);
\draw [->] (5) -- (6);
\draw [->] (6) -- (7);
\draw [->] (8) -- (7);
\draw [->] (2) edge[bend right] (4);
\draw [->] (4) edge[bend left] (1);
\draw [x-->] (5) -- (1);
\draw [-x>] (2) -- (5);
\draw [x->] (6) -- (2);
\draw [-x>] (3) -- (6);
\draw [x->] (7) -- (3);
\draw [--x>] (4) -- (8);
\node at (6) [above] {$a$};
\node at (7) [above] {$b$};
\node at (2) [below] {$c$};
\node at (3) [below] {$d$};
}
\]
\[
 \tikz{\draw [<->] (1,0) -- node[below left]{$\mu_d$} (0,1);} \quad\quad\quad   \quad\quad \quad\quad \quad\quad \quad \quad\quad \quad\quad \quad \tikz{\draw [<->] (-1,0) -- node[below right]{$\mu_b$} (0,1);}
\]
\[
\tikz{
\node (1) at (0,0) [] {$\bullet$};
\node (2) at (2,0) [] {$\bullet$};
\node (3) at (4,0) [] {$\bullet$};
\node (4) at (6,0) [] {$\bullet$};
\node (5) at (0,2) [] {$\bullet$};
\node (6) at (2,2) [] {$\bullet$};
\node (7) at (4,2) [] {$\bullet$};
\node (8) at (6,2) [] {$\bullet$};
\draw [->] (1) -- (2);
\draw [->] (2) -- (3);
\draw [->] (5) -- (6);
\draw [->] (7) -- (6);
\draw [->] (7) -- (8);
\draw [->] (2) edge[bend right] (4);
\draw [->] (4) edge[bend left] (1);
\draw [x-->] (5) -- (1);
\draw [x->] (6) -- (2);
\draw [-x>] (2) -- (5);
\draw [-x>] (3) -- (7);
\draw [x->] (8) -- (3);
\draw [--x>] (4) -- (8);
\node at (6) [above] {$a$};
\node at (7) [above] {$b$};
\node at (2) [below] {$c$};
\node at (3) [below] {$d$};
}
\]

\subsection{\texorpdfstring{Computation of $\DT$ for $H\backslash \PGL_3^{w_0,w_0}/H$}{}}\label{C}

In this section we will do an example of computation of the Donaldson-Thomas transformation in the case of $H\backslash\PGL_3^{w_0,w_0}/H$. The reduced word we will use is the following:
\[
\vec{i}=(2,1,2,-1,-2,-1).
\]
The associated string diagram looks like the following.
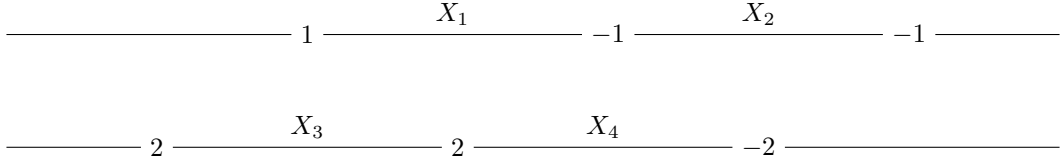
\[
\tikz{
\node (21) at (2,0) [] {$2$};
\node (11) at (4,1.5) [] {$1$};
\node (22) at (6,0) [] {$2$};
\node (12) at (8,1.5) [] {$-1$};
\node (23) at (10,0) [] {$-2$};
\node (13) at (12,1.5) [] {$-1$};
\draw (0,0) -- (21) -- node [above] {$X_3$} (22) -- node [above] {$X_4$} (23) -- (14,0);
\draw (0,1.5) -- (11) -- node [above] {$X_1$} (12) -- node [above] {$X_2$} (13) -- (14,1.5);
}
\]

We will compute $\DT_\vec{i}$ in two ways:
\begin{itemize}
    \item one using the factorization $\DT=i_\mathcal{X}\circ \Xi_{\vec{i}^\circ}\circ i_\mathcal{X}=i_\mathcal{X}\circ \psi_{\vec{i}^\circ}\circ \xi_{\vec{i}^\circ}\circ i_\mathcal{X}$;
\item the other one using a sequence of mutations also known as a maximal green sequence (see Remark \ref{maxgreen}).
\end{itemize}

Let's start with the first method. The opposite reduced word of $\vec{i}$ is 
\[
\vec{i}^\circ=(-1,-2,-1,2,1,2).
\]
The associated string diagram looks like the following (we have rewritten the cluster variables associated to the string diagram of $\vec{i}^\circ$ in terms of the original ones).
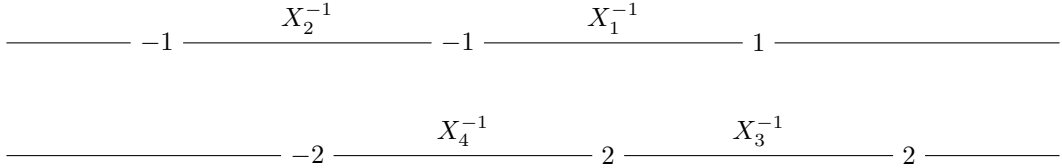
\[
\tikz{
\node (21) at (-2,0) [] {$2$};
\node (11) at (-4,1.5) [] {$1$};
\node (22) at (-6,0) [] {$2$};
\node (12) at (-8,1.5) [] {$-1$};
\node (23) at (-10,0) [] {$-2$};
\node (13) at (-12,1.5) [] {$-1$};
\draw (0,0) -- (21) -- node [above] {$X_3^{-1}$} (22) -- node [above] {$X_4^{-1}$} (23) -- (-14,0);
\draw (0,1.5) -- (11) -- node [above] {$X_1^{-1}$} (12) -- node [above] {$X_2^{-1}$} (13) -- (-14,1.5);
}
\]

Note that $\vec{i}^\circ$ is exactly the same seed we have used in Example \ref{exmp 2.69}. Therefore we can just do a substitution to get the image $\chi_{\vec{i}^\circ}\circ i_\mathcal{X}\left(X_a\right)$, which is just
\[
H\left\backslash \begin{pmatrix} 
\frac{1}{X_1X_2X_3X_4} & \frac{1}{X_1X_2X_3X_4} & \frac{1}{X_1X_2X_3X_4} \\
\frac{1+X_2}{X_1X_2X_3X_4} & \frac{1+X_2\left(1+X_1\right)}{X_1X_2X_3X_4} & \frac{1+X_2\left(1+X_1\left(1+X_3\right)\right)}{X_1X_2X_3X_4} \\
\frac{1}{X_1X_3X_4} & \frac{1+X_1}{X_1X_3X_4} & \frac{1+X_1\left(1+X_3\left(1+X_4\right)\right)}{X_1X_3X_4}
\end{pmatrix}\right/H.
\]

But then we need to take a lift from $H\backslash \PGL_3^{w_0,w_0}/H$ to $\SL_3^{w_0,w_0}$. One easy way to do this is to divide the matrix in between the two $H$-quotients by the cube root of its determinant. This gives the matrix
\[
\begin{pmatrix} 
\frac{1}{X_1^\frac{2}{3}X_2^\frac{2}{3}X_3^\frac{1}{3}X_4^\frac{1}{3}} & \frac{1}{X_1^\frac{2}{3}X_2^\frac{2}{3}X_3^\frac{1}{3}X_4^\frac{1}{3}} & \frac{1}{X_1^\frac{2}{3}X_2^\frac{2}{3}X_3^\frac{1}{3}X_4^\frac{1}{3}} \\
\frac{1+X_2}{X_1^\frac{2}{3}X_2^\frac{2}{3}X_3^\frac{1}{3}X_4^\frac{1}{3}} & \frac{1+X_2\left(1+X_1\right)}{X_1^\frac{2}{3}X_2^\frac{2}{3}X_3^\frac{1}{3}X_4^\frac{1}{3}} & \frac{1+X_2\left(1+X_1\left(1+X_3\right)\right)}{X_1^\frac{2}{3}X_2^\frac{2}{3}X_3^\frac{1}{3}X_4^\frac{1}{3}} \\
\frac{1}{X_1^\frac{2}{3}X_3^\frac{1}{3}X_4^\frac{1}{3}} & \frac{1+X_1}{X_1^\frac{2}{3}X_3^\frac{1}{3}X_4^\frac{1}{3}} & \frac{1+X_1\left(1+X_3\left(1+X_4\right)\right)}{X_1^\frac{2}{3}X_3^\frac{1}{3}X_4^\frac{1}{3}}
\end{pmatrix},
\]
and we name it $x$ for future reference. (Although entries of $x$ are not even rationals, we will still recover rational expressions at the end of the computation.)

To compute $\psi_{\vec{i}^\circ}\circ \chi_{\vec{i}^\circ}\circ i_\mathcal{X}\left(X_a\right)$, we have to first find out which generalized minors of $x$ we need to take, and then take their ratios as prescribed by the map $p_{\vec{i}^\circ}$. Fortunately we have done that in Example \ref{exmp 2.69} and we can just simply copy down the quiver together with the associated cluster $\mathcal{A}$-variables.
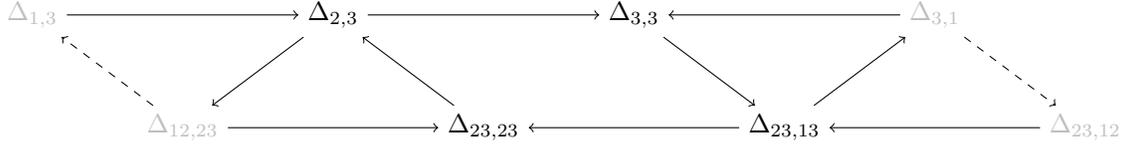
\[
\tikz{
\node (11) at (0,1.5) [lightgray] {$\Delta_{1,3}$};
\node (12) at (4,1.5) [] {$\Delta_{2,3}$};
\node (13) at (8,1.5) [] {$\Delta_{3,3}$};
\node (14) at (12, 1.5) [lightgray] {$\Delta_{3,1}$};
\node (21) at (2,0) [lightgray] {$\Delta_{12,23}$};
\node (22) at (6,0) [] {$\Delta_{23,23}$};
\node (23) at (10,0) [] {$\Delta_{23,13}$};
\node (24) at (14,0) [lightgray] {$\Delta_{23,12}$};
\draw [->] (11) -- (12);
\draw [->] (12) -- (13);
\draw [->] (14) -- (13);
\draw [->] (21) -- (22);
\draw [->] (23) -- (22);
\draw [->] (24) -- (23);
\draw [dashed, ->] (21) -- (11);
\draw [->] (12) -- (21);
\draw [->] (22) -- (12);
\draw [->] (13) -- (23);
\draw [->] (23) -- (14);
\draw [dashed, ->] (14) -- (24);
}
\]
But then since we need to compute $i_\mathcal{X}\circ \psi_{\vec{i}^\circ}\circ \chi_{\vec{i}^\circ}\circ i_\mathcal{X}\left(X_a\right)$ at the end, which will require us to take ratios of generalized minors of $x$ according to the quiver associated to $\vec{i}^\circ$ and then inverting them, we can just take the ratios of generalized minors of $x$ according to the quiver associated to $\vec{i}$, which is the following quiver instead.
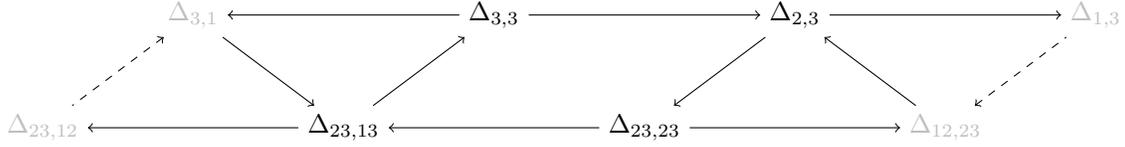
\[
\tikz{
\node (11) at (0,1.5) [lightgray] {$\Delta_{1,3}$};
\node (12) at (-4,1.5) [] {$\Delta_{2,3}$};
\node (13) at (-8,1.5) [] {$\Delta_{3,3}$};
\node (14) at (-12, 1.5) [lightgray] {$\Delta_{3,1}$};
\node (21) at (-2,0) [lightgray] {$\Delta_{12,23}$};
\node (22) at (-6,0) [] {$\Delta_{23,23}$};
\node (23) at (-10,0) [] {$\Delta_{23,13}$};
\node (24) at (-14,0) [lightgray] {$\Delta_{23,12}$};
\draw [<-] (11) -- (12);
\draw [<-] (12) -- (13);
\draw [<-] (14) -- (13);
\draw [<-] (21) -- (22);
\draw [<-] (23) -- (22);
\draw [<-] (24) -- (23);
\draw [dashed, <-] (21) -- (11);
\draw [<-] (12) -- (21);
\draw [<-] (22) -- (12);
\draw [<-] (13) -- (23);
\draw [<-] (23) -- (14);
\draw [dashed, <-] (14) -- (24);
}
\]
This gives the following pull-backs of cluster $\mathcal{X}$-variables:
\begin{align*}
 \DT^*\left(X_1\right)=&\frac{\Delta_{3,1}(x)\Delta_{2,3}(x)}{\Delta_{23,23}(x)}= \frac{1+X_2+X_1X_2+X_1X_2X_3}{X_1\left(1+X_3+X_3X_4+X_2X_3X_4\right)},\\
 \DT^*\left(X_2\right)=&\frac{\Delta_{23,23}(x)\Delta_{1,3}(x)}{\Delta_{3,3}(x)\Delta_{12,23}(x)}=\frac{1+X_4+X_2X_4+X_1X_2X_4}{X_2\left(1+X_1+X_1X_3+X_1X_3X_4\right)},\\
 \DT^*\left(X_3\right)=&\frac{\Delta_{23,12}(x)\Delta_{3,3}(x)}{\Delta_{3,1}(x)\Delta_{23,23}(x)}=\frac{1+X_1+X_1X_3+X_1X_3X_4}{X_3\left(1+X_4+X_2X_4+X_1X_2X_4\right)},\\
 \DT^*\left(X_4\right)=&\frac{\Delta_{23,13}(x)\Delta_{12,23}(x)}{\Delta_{2,3}(x)}=\frac{1+X_3+X_3X_4+X_2X_3X_4}{X_4\left(1+X_2+X_1X_2+X_1X_2X_3\right)}.
\end{align*}

Now let's apply the second method, which consists of three steps:
\begin{enumerate}
     \item Recursively change the nodes at the two ends of the string diagram to their opposite and then move them inward using cluster mutations, until all nodes are changed;
     \item Apply a transposition, which is a cluster isomorphism;
     \item Restore the string diagram back to the original one using cluster mutations again.
\end{enumerate}

Below is the execution of such process in this particular example. We see that at the end of such process we do obtain the same expressions for the pull-backs of the cluster $\mathcal{X}$-variables via $\DT$ as we had above.

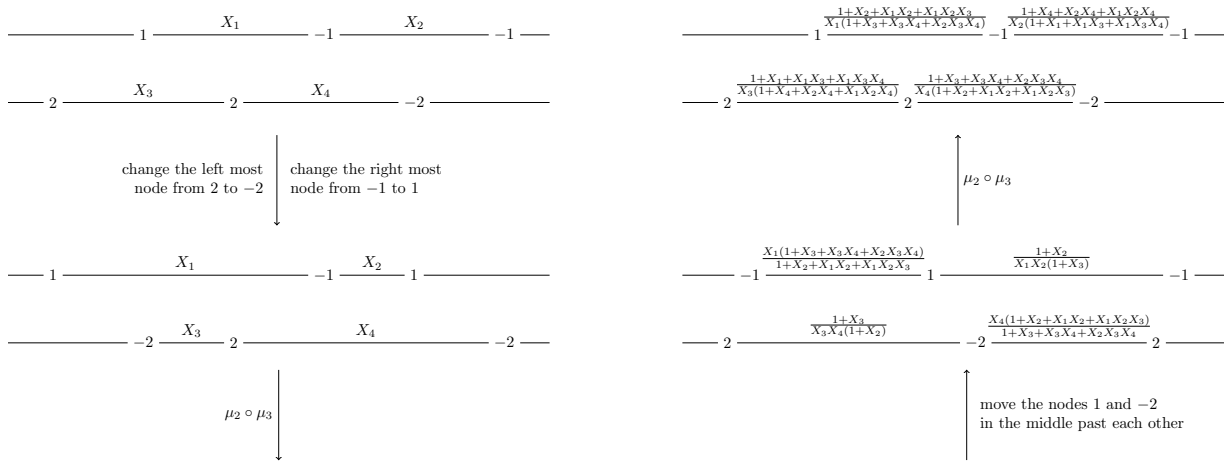
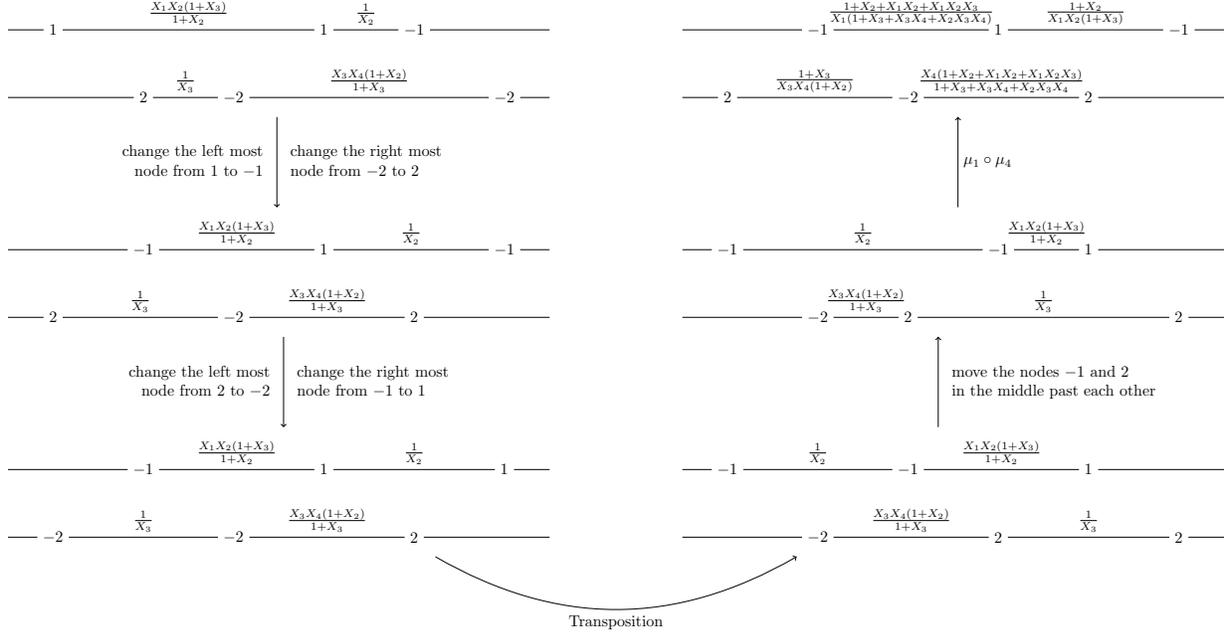
\[
\scalebox{0.6}{$\tikz{
\node (21) at (2,0) [] {$2$};
\node (11) at (4,1.5) [] {$1$};
\node (22) at (6,0) [] {$2$};
\node (12) at (8,1.5) [] {$-1$};
\node (23) at (10,0) [] {$-2$};
\node (13) at (12,1.5) [] {$-1$};
\draw (1,0) -- (21) -- node [above] {$X_3$} (22) -- node [above] {$X_4$} (23) -- (13,0);
\draw (1,1.5) -- (11) -- node [above] {$X_1$} (12) -- node [above] {$X_2$} (13) -- (13,1.5);
}$}
\quad \quad   \quad\quad \quad
\scalebox{0.6}{$\tikz{
\node (21) at (2,0) [] {$2$};
\node (11) at (4,1.5) [] {$1$};
\node (22) at (6,0) [] {$2$};
\node (12) at (8,1.5) [] {$-1$};
\node (23) at (10,0) [] {$-2$};
\node (13) at (12,1.5) [] {$-1$};
\draw (1,0) -- (21) -- node [above] {$\frac{1+X_1+X_1X_3+X_1X_3X_4}{X_3\left(1+X_4+X_2X_4+X_1X_2X_4\right)}$} (22) -- node [above] {$\frac{1+X_3+X_3X_4+X_2X_3X_4}{X_4\left(1+X_2+X_1X_2+X_1X_2X_3\right)}$} (23) -- (13,0);
\draw (1,1.5) -- (11) -- node [above] {$ \frac{1+X_2+X_1X_2+X_1X_2X_3}{X_1\left(1+X_3+X_3X_4+X_2X_3X_4\right)}$} (12) -- node [above] {$\frac{1+X_4+X_2X_4+X_1X_2X_4}{X_2\left(1+X_1+X_1X_3+X_1X_3X_4\right)}$} (13) -- (13,1.5);
}$}
\]
\[
\scalebox{0.6}{$\tikz{\draw [<-] (0,0) -- node[left]{$\begin{array}{r} \text{change the left most} \\ \text{node from $2$ to $-2$} \end{array}$} node [right] {$\begin{array}{l} \text{change the right most} \\ \text{node from $-1$ to $1$} \end{array}$}(0,2);}$} \quad\quad \quad\quad \quad\quad \quad\quad \quad\quad \quad\quad \quad\quad \quad\quad \quad\quad \quad  \scalebox{0.6}{$\tikz{\draw [->] (0,0) -- node[right]{$\mu_2\circ \mu_3$} (0,2);}$}\quad\quad \quad\quad
\]
\[
\scalebox{0.6}{$\tikz{
\node (21) at (4,0) [] {$-2$};
\node (11) at (2,1.5) [] {$1$};
\node (22) at (6,0) [] {$2$};
\node (12) at (8,1.5) [] {$-1$};
\node (23) at (12,0) [] {$-2$};
\node (13) at (10,1.5) [] {$1$};
\draw (1,0) -- (21) -- node [above] {$X_3$} (22) -- node [above] {$X_4$} (23) -- (13,0);
\draw (1,1.5) -- (11) -- node [above] {$X_1$} (12) -- node [above] {$X_2$} (13) -- (13,1.5);
}$}
\quad \quad   \quad\quad \quad
\scalebox{0.6}{$\tikz{
\node (21) at (-2.5,0) [] {$2$};
\node (11) at (-2,1.5) [] {$-1$};
\node (22) at (-6.5,0) [] {$-2$};
\node (12) at (-7.5,1.5) [] {$1$};
\node (23) at (-12,0) [] {$2$};
\node (13) at (-11.5,1.5) [] {$-1$};
\draw (-1,0) -- (21) -- node [above] {$\frac{X_4\left(1+X_2+X_1X_2+X_1X_2X_3\right)}{1+X_3+X_3X_4+X_2X_3X_4}$} (22) -- node [above] {$\frac{1+X_3}{X_3X_4\left(1+X_2\right)}$} (23) -- (-13,0);
\draw (-1,1.5) -- (11) -- node [above] {$\frac{1+X_2}{X_1X_2\left(1+X_3\right)}$} (12) -- node [above] {$\frac{X_1\left(1+X_3+X_3X_4+X_2X_3X_4\right)}{1+X_2+X_1X_2+X_1X_2X_3}$} (13) -- (-13,1.5);
}$} 
\]
\[
\quad\quad\quad\quad\quad \quad\quad\scalebox{0.6}{$\tikz{\draw [<-] (0,0) -- node[left]{$\mu_2\circ \mu_3$} (0,2);}$} \quad \quad \quad\quad \quad\quad\quad\quad \quad\quad \quad\quad \quad\quad \quad\quad \quad\quad \quad\quad \quad \quad\quad \quad\quad \quad\scalebox{0.6}{$\tikz{\draw [->] (0,0) -- node[right]{$\begin{array}{l} \text{move the nodes $1$ and $-2$} \\ \text{in the middle past each other}\end{array}$} (0,2);}$}
\]
\[
\scalebox{0.6}{$\tikz{
\node (21) at (4,0) [] {$2$};
\node (11) at (2,1.5) [] {$1$};
\node (22) at (6,0) [] {$-2$};
\node (12) at (8,1.5) [] {$1$};
\node (23) at (12,0) [] {$-2$};
\node (13) at (10,1.5) [] {$-1$};
\draw (1,0) -- (21) -- node [above] {$\frac{1}{X_3}$} (22) -- node [above] {$\frac{X_3X_4\left(1+X_2\right)}{1+X_3}$} (23) -- (13,0);
\draw (1,1.5) -- (11) -- node [above] {$\frac{X_1X_2\left(1+X_3\right)}{1+X_2}$} (12) -- node [above] {$\frac{1}{X_2}$} (13) -- (13,1.5);
}$} 
\quad \quad   \quad\quad \quad
\scalebox{0.6}{$\tikz{
\node (21) at (-4,0) [] {$2$};
\node (11) at (-2,1.5) [] {$-1$};
\node (22) at (-8,0) [] {$-2$};
\node (12) at (-6,1.5) [] {$1$};
\node (23) at (-12,0) [] {$2$};
\node (13) at (-10,1.5) [] {$-1$};
\draw (-1,0) -- (21) -- node [above] {$\frac{X_4\left(1+X_2+X_1X_2+X_1X_2X_3\right)}{1+X_3+X_3X_4+X_2X_3X_4}$} (22) -- node [above] {$\frac{1+X_3}{X_3X_4\left(1+X_2\right)}$} (23) -- (-13,0);
\draw (-1,1.5) -- (11) -- node [above] {$\frac{1+X_2}{X_1X_2\left(1+X_3\right)}$} (12) -- node [above] {$\frac{1+X_2+X_1X_2+X_1X_2X_3}{X_1\left(1+X_3+X_3X_4+X_2X_3X_4\right)}$} (13) -- (-13,1.5);
}$} 
\]
\[
\scalebox{0.6}{$\tikz{\draw [<-] (0,0) -- node[left]{$\begin{array}{r} \text{change the left most} \\ \text{node from $1$ to $-1$} \end{array}$} node [right] {$\begin{array}{l} \text{change the right most} \\ \text{node from $-2$ to $2$} \end{array}$}(0,2);}$}  \quad\quad \quad\quad \quad\quad \quad\quad \quad\quad \quad\quad \quad\quad \quad\quad \quad\quad \quad 
\scalebox{0.6}{$\tikz{\draw [->] (0,0) -- node[right]{$\mu_1\circ \mu_4$} (0,2);}$}\quad\quad \quad\quad
\]
\[
\scalebox{0.6}{$\tikz{
\node (21) at (2,0) [] {$2$};
\node (11) at (4,1.5) [] {$-1$};
\node (22) at (6,0) [] {$-2$};
\node (12) at (8,1.5) [] {$1$};
\node (23) at (10,0) [] {$2$};
\node (13) at (12,1.5) [] {$-1$};
\draw (1,0) -- (21) -- node [above] {$\frac{1}{X_3}$} (22) -- node [above] {$\frac{X_3X_4\left(1+X_2\right)}{1+X_3}$} (23) -- (13,0);
\draw (1,1.5) -- (11) -- node [above] {$\frac{X_1X_2\left(1+X_3\right)}{1+X_2}$} (12) -- node [above] {$\frac{1}{X_2}$} (13) -- (13,1.5);
}$} 
\quad \quad   \quad\quad \quad 
\scalebox{0.6}{$\tikz{
\node (21) at (-2,0) [] {$2$};
\node (11) at (-4,1.5) [] {$1$};
\node (22) at (-8,0) [] {$2$};
\node (12) at (-6,1.5) [] {$-1$};
\node (23) at (-10,0) [] {$-2$};
\node (13) at (-12,1.5) [] {$-1$};
\draw (-1,0) -- (21) -- node [above] {$\frac{1}{X_3}$} (22) -- node [above] {$\frac{X_3X_4\left(1+X_2\right)}{1+X_3}$} (23) -- (-13,0);
\draw (-1,1.5) -- (11) -- node [above] {$\frac{X_1X_2\left(1+X_3\right)}{1+X_2}$} (12) -- node [above] {$\frac{1}{X_2}$} (13) -- (-13,1.5);
}$} 
\]
\[
 \quad\quad\scalebox{0.6}{$\tikz{\draw [<-] (0,0) -- node[left]{$\begin{array}{r} \text{change the left most} \\ \text{node from $2$ to $-2$} \end{array}$} node [right] {$\begin{array}{l} \text{change the right most} \\ \text{node from $-1$ to $1$} \end{array}$}(0,2);}$}  
\quad\quad\quad\quad\quad \quad\quad\quad \quad\quad \quad\quad \quad  \quad\quad \quad\quad\quad  
\scalebox{0.6}{$\tikz{\draw [->] (0,0) -- node[right]{$\begin{array}{l} \text{move the nodes $-1$ and $2$} \\ \text{in the middle past each other}\end{array}$} (0,2);}$}
\]
\[
\scalebox{0.6}{$\tikz{
\node (21) at (2,0) [] {$-2$};
\node (11) at (4,1.5) [] {$-1$};
\node (22) at (6,0) [] {$-2$};
\node (12) at (8,1.5) [] {$1$};
\node (23) at (10,0) [] {$2$};
\node (13) at (12,1.5) [] {$1$};
\draw (1,0) -- (21) -- node [above] {$\frac{1}{X_3}$} (22) -- node [above] {$\frac{X_3X_4\left(1+X_2\right)}{1+X_3}$} (23) -- (13,0);
\draw (1,1.5) -- (11) -- node [above] {$\frac{X_1X_2\left(1+X_3\right)}{1+X_2}$} (12) -- node [above] {$\frac{1}{X_2}$} (13) -- (13,1.5);
}$} 
\quad \quad   \quad\quad \quad \scalebox{0.6}{$\tikz{
\node (21) at (-2,0) [] {$2$};
\node (11) at (-4,1.5) [] {$1$};
\node (22) at (-6,0) [] {$2$};
\node (12) at (-8,1.5) [] {$-1$};
\node (23) at (-10,0) [] {$-2$};
\node (13) at (-12,1.5) [] {$-1$};
\draw (-1,0) -- (21) -- node [above] {$\frac{1}{X_3}$} (22) -- node [above] {$\frac{X_3X_4\left(1+X_2\right)}{1+X_3}$} (23) -- (-13,0);
\draw (-1,1.5) -- (11) -- node [above] {$\frac{X_1X_2\left(1+X_3\right)}{1+X_2}$} (12) -- node [above] {$\frac{1}{X_2}$} (13) -- (-13,1.5);
}$} 
\]
\[
\scalebox{0.6}{$\tikz{
\draw [->] (-4,0) to [bend right] node [below] {Transposition} (4,0);
}$}
\]

It is not hard to verify that the mutation sequence
\[
\mu_2\circ \mu_3\circ \mu_1\circ \mu_4\circ \mu_2\circ \mu_3
\]
is a maximal green sequence for the following quiver.
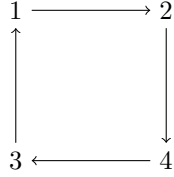
\[
\tikz{
\node (1) at (0,2) [] {$1$};
\node (2) at (2,2) [] {$2$};
\node (3) at (0,0) [] {$3$};
\node (4) at (2,0) [] {$4$};
\draw [->] (1) -- (2);
\draw [->] (2) -- (4);
\draw [->] (4) -- (3);
\draw [->] (3) -- (1);
}
\]

\begin{landscape}

\subsection{Cluster Identities on Double Bruhat Cells } The following table contains identities corresponding to move (1) and move (2) in the cases of $C_{\alpha\beta}C_{\beta\alpha}=1$ and $C_{\alpha\beta}C_{\beta\alpha}=2$. \label{B} 

\begin{table}[h]
\begin{center}
\resizebox{\columnwidth}{!}{
\begin{tabular}{|c|c|c|}\hline
&  $\mathcal{A}$   &  $\mathcal{X}$ \\ \hline
$\begin{array}{c} (\alpha,-\alpha) \\ \downarrow \\ (-\alpha,\alpha)\end{array}$ & $\Delta_\alpha\left(\overline{s}_\alpha^{-1}x\overline{s}_\alpha\right)=\displaystyle\frac{
\Delta_\alpha\left(\overline{s}_\alpha^{-1} x\right)\Delta_\alpha\left(x\overline{s}_\alpha\right) +\prod_{\beta\neq \alpha}\left( \Delta_\beta(x)\right)^{-C_{\beta\alpha}}}{\Delta_\alpha(x)}$ & $e_\alpha t^{H^\alpha} e_{-\alpha}=\displaystyle(1+t)^{H^\alpha}e_{-\alpha}\left(\frac{1}{t}\right)^{H^\alpha}
e_\alpha (1+t)^{H^\alpha}\prod_{\beta\neq \alpha} \left(\frac{1}{1+\frac{1}{t}}\right)^{H^\beta}$ \\ \hline
$\begin{array}{c} (\alpha,\beta,\alpha) \\
\downarrow \\
(\beta,\alpha,\beta)\end{array}$ &  $\Delta_\beta\left(x\overline{s}_\beta\right)=\displaystyle\frac{\Delta_\alpha(x)\Delta_\beta\left(x\overline{s}_\alpha\overline{s}_\beta\right)+\Delta_\alpha\left(x\overline{s}_\alpha\overline{s}_\beta\overline{s}_\alpha\right)\Delta_\beta(x)}{\Delta_\alpha\left(x\overline{s}_\alpha\right)}$ & $e_\alpha t^{H^\alpha} e_\beta e_\alpha=\displaystyle(1+t)^{H^\alpha}\left(\frac{1}{1+\frac{1}{t}}\right)^{H^\beta}e_\beta \left(\frac{1}{t}\right)^{H^\beta} e_\alpha e_\beta (1+t)^{H^\beta}\left(\frac{1}{1+\frac{1}{t}}\right)^{H^\alpha}$ \\ \hline
$\begin{array}{c} (\alpha,\beta,\alpha,\beta) \\ \downarrow \\ (\beta,\alpha,\beta,\alpha) \\ \quad \\ \text{with $C_{\alpha\beta}=-2$} \\ \text{and $C_{\beta\alpha}=-1$} \end{array}$ &
{$\!\begin{aligned} 
\Delta_\alpha\left(x\overline{s}_\alpha\right)=&\displaystyle\frac{1}{\Delta_\alpha\left(x\overline{s}_\beta\overline{s}_\alpha\right)\Delta_\beta\left(x\overline{s}_\beta\right)}\left(\left(\Delta_\alpha(x)\right)^2\Delta_\beta\left(x\overline{s}_\beta \overline{s}_\alpha \overline{s}_\beta\right)+\left(\Delta_\alpha\left(x\overline{s}_\beta\overline{s}_\alpha\right)\right)^2\Delta_\beta(x)\right. \\
& \left. + \Delta_\alpha(x)\Delta_\alpha\left(x\overline{s}_\beta\overline{s}_\alpha\overline{s}_\beta\overline{s}_\alpha\right)\Delta_\beta \left(x\overline{s}_\beta\right)\right)\\
\Delta_\beta\left(x\overline{s}_\alpha\overline{s}_\beta\right)=&\displaystyle\frac{1}{\left(\Delta_\alpha\left(x\overline{s}_\beta\overline{s}_\alpha\right)\right)^2\Delta_\beta\left(x\overline{s}_\beta\right)}\left(\left(\Delta_\alpha(x)\Delta_\beta\left(x\overline{s}_\beta\overline{s}_\alpha\overline{s}_\beta\right)+\Delta_\alpha\left(x\overline{s}_\beta\overline{s}_\alpha\overline{s}_\beta\overline{s}_\alpha\right)\Delta_\beta\left(x\overline{s}_\beta\right)\right)^2\right.\\
&\left.+\left(\Delta_\alpha\left(x\overline{s}_\beta\overline{s}_\alpha\right)\right)^2\Delta_\beta\left(x\overline{s}_\alpha\overline{s}_\alpha \overline{s}_\beta\right)\Delta_\beta(x)\right)
\end{aligned}$}
& $\begin{array}{l} e_\alpha e_\beta t_\alpha^{H^\alpha} t_\beta^{H^\beta} e_\alpha e_\beta=y_1^{H^\alpha}y_2^{H^\beta}e_\beta e_\alpha y_3^{H^\alpha}y_4^{H^\beta} e_\beta e_\alpha y_5^{H^\alpha}y_6^{H^\beta} \\
\quad \\
\text{where $y_1, y_2, \dots, y_6$ are functions of $t_\alpha$ and $t_\beta$ as described below.}
\end{array}$ \\ \hline
\end{tabular}
}
\end{center}
\end{table} 

\begin{align*}
    y_1=& \frac{1+t_\beta+2t_\alpha t_\beta+t_\alpha^2t_\beta}{1+t_\beta+t_\alpha t_\beta} & y_2=&\frac{t_\alpha^2 t_\beta}{1+t_\beta +2t_\alpha t_\beta+t_\alpha^2t_\beta} \\
    y_3=& \frac{t_\alpha}{1+t_\beta+2t_\alpha t_\beta+t_\alpha^2t_\beta} & y_4=& \frac{\left(1+t_\beta+t_\alpha t_\beta\right)^2}{t_\alpha^2 t_\beta} \\
    y_5=& 1+t_\beta+t_\alpha t_\beta & y_6=&\frac{t_\beta+t_\beta^2+2t_\alpha t_\beta^2+t_\alpha^2t_\beta^2}{\left(1+t_\beta+t_\alpha t_\beta\right)^2}
\end{align*}
\end{landscape}

\bibliographystyle{amsalpha-a}

\bibliography{biblio}

\end{document}